\documentclass[reqno]{amsart}
\usepackage{amssymb}
\usepackage{amsmath}
\usepackage{amsthm}
\usepackage[usenames]{color}
\usepackage{graphicx}
\usepackage{cite}
\usepackage{bbm}

\usepackage{tikz}
\usetikzlibrary{positioning}

\allowdisplaybreaks

\usepackage[colorlinks,linkcolor=blue,anchorcolor=red,citecolor=blue]{hyperref}
\usepackage[margin=1in]{geometry} 
\usepackage{marginnote}

\usepackage{color}

\newtheorem{lemma}{Lemma}[section]
\newtheorem{theorem}{Theorem}[section]

\newtheorem{proposition}{Proposition}[section]
\newtheorem{remark}{Remark}[section]
\newtheorem{corollary}{Corollary}[section]

\numberwithin{equation}{section}

\newcommand{\dis}{\displaystyle}

\newcommand{\R}{\mathbb{R}}

\newcommand{\CD}{\mathcal{D}}
\newcommand{\CE}{\mathcal{E}}

\newcommand{\CI}{\mathcal{I}}

\newcommand{\CL}{\mathcal{L}}

\newcommand{\ep}{\epsilon}

\newcommand{\al}{\alpha}
\newcommand{\be}{\beta}

\newcommand{\la}{\lambda}
\newcommand{\de}{\delta}
\newcommand{\si}{\sigma}
\newcommand{\pa}{\partial}
\newcommand{\ka}{\kappa}
\newcommand{\eps}{\epsilon}

\newcommand{\Ga}{\Gamma}

\newcommand{\vertiii}[1]{{\left\vert\kern-0.25ex\left\vert\kern-0.25ex\left\vert #1
		\right\vert\kern-0.25ex\right\vert\kern-0.25ex\right\vert}}

\makeatletter
\@namedef{subjclassname@2020}{%
	\textup{2020} Mathematics Subject Classification}
\makeatother

\begin{document}
	
	\title[Ionic KdV structure in weakly collisional plasmas]{Ionic KdV structure in weakly collisional plasmas
    }
	
	\author[R.-J. Duan]{Renjun Duan}
	\address[RJD]{Department of Mathematics, The Chinese University of Hong Kong, Shatin, Hong Kong, P.R.~China}
	\email{rjduan@math.cuhk.edu.hk}
	
	\author[Z.-G. Li]{Zongguang Li}
	\address[ZGL]{Department of Applied Mathematics, The Hong Kong Polytechnic University, Hung Hom, Hong Kong, P.R.~China}
	\email{zongguang.li@polyu.edu.hk}
	
	\author[D.-C. Yang]{Dongcheng Yang}
	\address[DCY]{School of Mathematics, South China University of Technology, Guangzhou, 510641, P.R.~China}
	\email{mathdcyang@scut.edu.cn}
	
	\author[T. Yang]{Tong Yang}
	\address[TY]{Institute for Math \& AI, Wuhan University, Wuhan, 430072, P.R.~China}
	\email{tongyang@whu.edu.cn}

	\begin{abstract}
    We consider the one-dimensional ions dynamics in weakly collisional plasmas governed by the Vlasov-Poisson-Landau system under the Boltzmann relation with the small collision frequency $\nu>0$. It is observed in physical experiments that the interplay of nonlinearities and dispersion may lead to the formation of ion acoustic solitons that are described by the Korteweg-de Vries equation. In this paper, to capture the ionic KdV structure in the weak-collision regime, we study the combined cold-ions limit and longwave limit of the rescaled VPL system depending on a small scaling parameter $\eps>0$. 
    
    The main goal is to justify the uniform convergence of the VPL solutions to the KdV solutions over any finite time interval as $\eps\to 0$ under restriction that $\eps^{3/2}\lesssim \nu \lesssim \eps^{1/2}$. The proof is based on the energy method near local Maxwellians for making use of the Euler-Poisson dynamics under the longwave scaling. The KdV profiles, in particular including both velocity field and electric potential, may have large amplitude, which induces the cubic velocity growth. To overcome the $\eps$-singularity in such multi-parameter limit problem, we design delicate velocity weighted energy functional and dissipation rate functional in the framework of macro-micro decomposition that is further incorporated with the Caflisch's decomposition. 
    
    As an application of our approach, the global-in-time existence of solutions near global Maxwellians when the KdV profile is degenerate to a constant equilibrium is also established under the same scaling with $\eps^{3}\lesssim \nu \lesssim \eps^{5/2}$. For the proof, the velocity weight is modified to depend on the solution itself, providing an extra quartic dissipation so as to obtain the global dynamics for most singular Coulomb potentials.
	\end{abstract}

	\date{\today}
	
	\subjclass[2020]{35Q83, 35Q84, 35Q53, 35B35; 35Q20, 76X05, 82D10, 35B40}
	

\keywords{Vlasov-Poisson-Landau system, Euler-Poisson dynamics, Korteweg-de Vries equations, weak collisions, cold-ions limit, longwave limit, macro-micro decomposition, Caflisch's decomposition, energy method}
	\maketitle
	\thispagestyle{empty}
	
	\tableofcontents
	\section{Introduction}

In the paper we consider the spatially one-dimensional motion of positively charged ions in plasmas governed by the following Vlasov-Poisson-Landau (VPL) system in the absence of a magnetic field,
	\begin{equation}
		\label{VPL}
		\left\{
		\begin{array}{rl}
			\dis\partial_{t}F+v_1\partial_{x}F
			-\partial_{x}\phi\partial_{v_1}F&\dis=\nu Q(F,F),
			\\[3mm]
			\dis-\partial^{2}_{x}\phi&\dis=\rho-(1+\phi),
		\end{array} \right.
	\end{equation}
where the unknowns are $F=F(t,x,v)\geq0$ being the density distribution function of ions with velocity $v=(v_{1},v_{2},v_{3})\in\mathbb{R}^{3}$ at time $t\geq0$ and position $x\in\mathbb{R}$, and $\phi=\phi(t,x)$ being the electric potential. Here, $\rho=\rho(t,x)=\int_{\mathbb{R}^{3}}F(t,x,v)\,dv$ denotes the ions density, and the binary collision between particles is described by the bilinear Landau operator acting only on velocity variable
	\begin{equation}
		\label{LandauOp}
		Q(F,G)(v)=\nabla_{v}\cdot\int_{\mathbb{R}^{3}}\Phi(v-v_*)\left\{F(v_*)\nabla_{v}G(v)
		-G(v)\nabla_{v_*}F(v_*)\right\}\,dv_*,
	\end{equation}
where we consider only the Coulomb interactions for which the Landau  kernel $\Phi(z)=[\Phi_{ij}(z)]$ with $z=v-v_*$  is a matrix-valued function given by
	\begin{equation*}
		\Phi_{ij}(z)=\frac{1}{|z|}\big(\delta_{ij}-\frac{z_iz_j}{|z|^{2}}\big),\quad 1\leq i,j\leq 3,
	\end{equation*}
with  $\delta_{ij}$ being the Kronecker delta. We use a parameter $\nu=\frac{1}{{\rm K\!n}}>0$ to describe the frequency or amplitude of collisions with ${\rm K\!n}$ understood to be the Knudsen number as in the Boltzmann theory. 

In plasma physics the motion of charged particles in the absence of a magnetic field is usually governed by the two-species isothermal compressible Euler-Poisson system for both positively charged ions and negatively charged electrons,
\begin{equation}\label{tfEP}
\left\{\begin{aligned}
  \pa_tn_\pm +\pa_x(n_\pm u_\pm)  &=0,   \\
  m_\pm n_\pm (\pa_t u_\pm +u_\pm \pa_x u_\pm)+T_\pm \pa_x n_\pm  &=\mp n_\pm \pa_x \phi, \\
  -\pa_x^2 \phi &= n_+-n_-, 
\end{aligned}\right.
\end{equation}
where $n_\pm$ denote number densities, $u_\pm$ velocities, $T_\pm$ temperatures and $m_\pm$ masses for ions and electrons, respectively. We assume $T_-=1$ and $T_+\geq 0$ with the case $T_+=0$ corresponding to the cold ions.  Since mass of electron $m_-$ is much lighter than that of ion $m_+$, the so-called Boltzmann relation 
\begin{equation}\label{BErel}
n_-=e^\phi,
\end{equation}
can be taken into account for giving the approximate electron number density in terms of the electric potential; we refer to \cite{Bastdos-Golse2018,FG,GGPS} for discussions and rigorous justifications of the Boltzmann relation from \eqref{tfEP} under the electron massless limit $m_-\to 0$. Through the paper, we adopt the linear approximation $n_-=1+\phi$ to the Boltzmann relation as in the Poisson equation of \eqref{VPL} for simplicity of presentation, but all the results can be carried over to the case of $n_-=e^\phi$ in \eqref{BErel}, because the perturbation method in  high-order Sobolev spaces will be employed. 

It is well known in plasma physics that the ionic motion exposes the soliton behavior. In particular, at the fluid level, the one-dimensional motion of ions governed by the collisionless Euler-Poisson system under the Boltzmann relation,
\begin{equation}\label{ep+}
\left\{\begin{aligned}
  \pa_tn_+ +\pa_x(n_+ u_+)  &=0,   \\
  m_+ n_+ (\pa_t u_+ +u_+ \pa_x u_+)+T_+ \pa_x n_+  &=- n_+ \pa_x \phi, \\
  -\pa_x^2 \phi &= n_+-e^{\phi}, 
\end{aligned}\right.
\end{equation}
 can be approximated by solutions of  the Kortweg-de Vries (KdV) equations describing a reversible dispersion process. Indeed, it is Gardner-Morikawa \cite{GaMo} that first introduced the following transformation
\begin{equation}
\label{gmt}
\eps^{1/2}(x-Vt)\to x,\quad \eps^{3/2}t\to t, 
\end{equation} 
with a traveling speed $V$ so as to derive the KdV equation:
\begin{equation}
\label{kdv-xi}
\pa_t \xi +V\xi\pa_x \xi +\frac{1}{2m_+V}\pa_x^3\xi=0,
\end{equation}
where it necessarily requires $V^2=\frac{1+T_+}{m_+}$. Thus \eqref{kdv-xi} can be regarded as a nonlinear dispersive approximation to the ionic Euler-Poisson system \eqref{ep+} under the Gardner-Morikawa transformation \eqref{gmt}, which will be called the {\it longwave limit} $\eps\to 0$ for later use; see also other references \cite{Washimi,Su-1}. The first rigorous mathematical proof for such longwave limit was given by Guo-Pu \cite{Guo-Pu} for either $T_+>0$ or $T_+=0$, through the classical energy method together with the Gronwall argument.

At the kinetic level, the collisionless Vlasov-Poisson system under the Boltzmann relation,
\begin{equation}\label{vp+}
\left\{\begin{aligned}
\pa_{t}F+v_1\partial_{x}F-\partial_{x}\phi\partial_{v_1}F&=0,\\
-\partial^{2}_{x}\phi&=\rho-e^{\phi},
\end{aligned}\right.
\end{equation}
corresponding to \eqref{VPL} with $\nu=0$, is usually taken to govern the one-dimensional motion of ions with slab symmetry in space. One can consider the {\it cold-ions limit} $\kappa\to 0$ under the transformation
\begin{equation}
\label{lim.coldi}
 \frac{1}{\kappa^3}F(t,x,\frac{v}{\kappa})\to F(t,x,v),
\end{equation} 
so that for $\kappa>0$ small enough, $F(t,x,v)$ is close to a Dirac delta function in velocity 
\begin{equation}
\label{def.mok}
\rho \delta_{v=(u,0,0)}
\end{equation} 
with $(\rho,u)$ satisfying the {\it pressureless} Euler-Poisson system of the form as given in \eqref{ep+} in the cold ions case $T_+=0$. We regard \eqref{def.mok} as the monokinetic solution to the Vlasov-Poisson \eqref{vp+} with the zero kinetic temperature. Furthermore, inspired by \cite{Guo-Pu} mentioned before, the pressureless Euler-Poisson system may converge to the KdV equations \eqref{kdv-xi} under the longwave limit $\eps\to 0$. Hence there should be a direct connection of the Vlasov-Poisson system \eqref{vp+} with the KdV equations \eqref{kdv-xi}  under the simultaneous limits $\kappa\to 0$ and $\eps\to 0$. Indeed, for the Vlasov-Poisson system \eqref{vp+}, Han-Kwan \cite{Han-Kwan} first proposed a combined cold-ions and longwave limit as in \eqref{lim.coldi} and \eqref{gmt} with a specific choice of $\kappa$ given by  $\kappa=\eps$; this is an important observation. The proof in \cite{Han-Kwan} is based on the relative entropy method to obtain the uniform stability estimates for the convergence of non-negative global weak solutions of the Vlasov-Poisson system to smooth solutions of the KdV equations.     

In contrast to the collisionless model \eqref{vp+}, we are interested in the KdV limit of the VPL system \eqref{VPL} when grazing collisions occur. Note that the Landau operator \eqref{LandauOp} with Coulomb interactions is the physically most relevant collision operator for describing the long time motion of kinetic plasmas. Motivated by \cite{Guo-Pu} and \cite{Han-Kwan}, one may study the issue in two distinct cases where collisions are either strong or weak to be explained as follows:\\

\begin{itemize}
  \item[(a)] {\it Strong-collision regime:} In this regime VPL converges to Euler-Poisson in the fluid limit ${\rm K\!n}\to 0$ which further converges to KdV in the longwave limit $\eps\to 0$; see the Figure \ref{fig1} below. In a direct way VPL may converge to  KdV under the combined limits with the restriction ${\rm K\!n}=O(\eps^\al)$ as $\eps\to 0$ for a constant parameter $\alpha>0$. To the best of our knowledge, the recent work \cite{DYY} provided the first rigorous justification for $3/2\leq \alpha\leq 5/2$.  
  
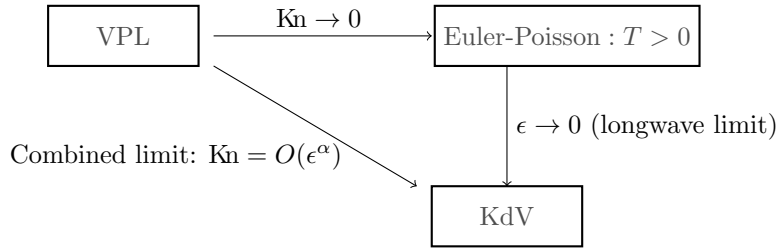
\begin{figure}[h]
	\begin{tikzpicture}	
		\node (A) at (-0.1,0) [draw, thick, fill=white, fill opacity=0.7,
		minimum width=2cm,minimum height=0.8cm] {$\mbox{VPL}$};
		
		\draw[->] (1.1,0) -- (4.0,0); 
		\node[above] at (2.5,0){${\rm K\!n} \to 0$};
		
		\node (B) at (5.8,0) [draw, thick, fill=white, fill opacity=0.7,
		minimum width=2cm,minimum height=0.8cm] {$\mbox{Euler-Poisson}: T>0$};
		\draw[->] (5,-0.4) -- (5,-2); 
		\node[right] at (5,-1.2){$\eps \to 0$ (longwave limit)};
		
		
		\node (D) at (5,-2.4) [draw, thick, fill=white, fill opacity=0.7,
		minimum width=2cm,minimum height=0.8cm] {$\mbox{KdV}$};
		
		\coordinate (M) at (1.1,-0.4);
		\coordinate (N) at (3.8,-2.0);
		\draw [->] (M) -- (N) {};
		\node[above] at (0.6,-1.9){Combined limit: ${\rm K\!n}=O(\eps^\alpha)$};
	\end{tikzpicture}
	\caption{From VPL to KdV in strong-collision regime}\label{fig1}	
\end{figure}

  \item[(b)] {\it Weak-collision regime:} In this regime VPL converges to pressureless Euler-Poisson in the cold ions limit $\kappa\to 0$ which further converges to KdV in the longwave limit $\eps\to 0$; see the Figure \ref{fig2} below. In a direct way VPL may converge to KdV under the combined limits with the restriction that $\kappa=\eps$ and $\nu=1/{\rm K\!n}=O(\eps^\beta)$ as $\eps\to 0$ for another constant parameter $\beta>0$. In contrast to \cite{DYY}, {\bf the goal of the current work is to justify the convergence of VPL to KdV under such simultaneous limits in the weak-collision regime for a certain range of $\beta>0$.} It turns out that the convergence to KdV is valid over finite time interval for $1/2 \leq \beta\leq 3/2$, and if the KdV profile is degenerate to zero then the convergence is valid globally in time for $5/2\leq \beta\leq 3$. The precise statements of those two results are to be given in Theorem \ref{TheoremKdV} and Theorem \ref{TheoremGlobal} in the next section, respectively.   
  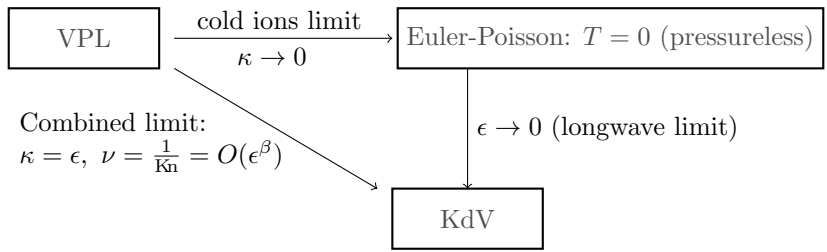
\begin{figure}[h]
	\begin{tikzpicture}	
		\node (A) at (-0.1,0) [draw, thick, fill=white, fill opacity=0.7,
		minimum width=2cm,minimum height=0.8cm] {$\mbox{VPL}$};
		
		\draw[->] (1.1,0) -- (4.0,0); 
		\node[above] at (2.5,0){cold ions limit};
		\node[below] at (2.4,0){$\kappa\to 0$};
		
		\node (B) at (6.9,0) [draw, thick, fill=white, fill opacity=0.7,
		minimum width=2cm,minimum height=0.8cm] {\mbox{Euler-Poisson: $T=0$ (pressureless)}};
		\draw[->] (5,-0.4) -- (5,-2); 
		\node[right] at (5,-1.2){$\eps \to 0$ (longwave limit)};
		
		
		\node (D) at (5,-2.4) [draw, thick, fill=white, fill opacity=0.7,
		minimum width=2cm,minimum height=0.8cm] {$\mbox{KdV}$};
		
		\coordinate (M) at (1.1,-0.4);
		\coordinate (N) at (3.8,-2.0);
		\draw [->] (M) -- (N) {};
		\node[above] at (0.8,-1.9){\parbox{10em}{Combined limit: \\
		$\kappa=\eps,\ \nu={ \frac{1}{\rm K\!n}}=O(\eps^\beta)$}};
	\end{tikzpicture}
	\caption{From VPL to KdV in weak-collision regime}\label{fig2}		
\end{figure}
\end{itemize} 

We further refer readers to \cite{Guo-Pu, Han-Kwan} and \cite{DYY} as well as references therein for the literature review related to the KdV limits in different situations. In what follows we only make a few comments on the main results of the current work:\\

\begin{itemize}
  \item Supplementary to \cite{Han-Kwan} for the collisionless Vlasov-Poisson system, we obtain the convergence of VPL to KdV in a regularized version with the vanishing Landau collision additionally involved. In fact, similar to the Boltzmann theory in the presence of collisions, due to the $H$-theorem, the asymptotic state of the VPL system is a local Maxwellian equilibrium with its macroscopic temperature converging to zero under the cold ions limit $\kappa\to 0$, and hence the local Maxwellian provides a smooth approximation to the Dirac delta function in the collisionless situation. Moreover, instead of the relative entropy method for weak solutions in \cite{Han-Kwan}, we are able to adopt the energy method for constructing classical solutions of the VPL system for the dispersive scaling parameter $\eps>0$. Because of the specific choice of $\kappa$ by $\kappa=\eps$, the system temperature after scaling is still a positive constant such that the method in \cite{DYY} could be applicable for $\nu=O(\eps^\beta)$ with a suitable range of $\beta>0$. We also remark that the energy method in case of cold ions used in \cite{Guo-Pu} gives us a clue for how to obtain the uniform bounds on the potential $\phi$.  \\  
  
  \item Although \cite{DYY} treats the strong-collision regime, we employ in the current work an essentially similar argument for treating the weak-collision regime via the construction of energy functional and energy dissipation functional in the macro-micro decomposition framework. However, there are still several main differences of the proof between the current work and \cite{DYY}. First of all, the amplitude of KdV solutions is assumed to be sufficiently small in \cite{DYY}, so the general strategy as in \cite{Duan-Yang-Yu-VPL} for dealing with solutions around small-amplitude rarefaction waves may be used. Instead we allow the KdV profile to be of arbitrarily large amplitude. According to the pioneering work by Caflisch \cite{Caflisch}, the cubic growth in large velocities becomes the main obstacle to be overcome through a decomposition. In the context of Landau collision or non-cutoff Boltzmann collision, we hence borrow another strategy recently developed in \cite{Duan-Li} for incorporating the Caflisch’s decomposition into the macro-micro decomposition for allowing the microscopic component to exhibit only the polynomial tail in large velocities. Here, for simplicity we still take the Gaussian tail $\exp(-c|v|^2)$ for $c>0$ small enough. In the meantime, due to scaling on velocity variable for the cold ions limit, singularity of the terms $\frac{1}{\eps}\pa_x F$ and $\frac{1}{\eps}\pa_x\phi \pa_{v_1}F$ in the rescaled system \eqref{reVPL} looks much stronger than that of those corresponding terms in the strong-collision regime. Another point is that if the KdV profile is degenerate to zero or equivalently the constant equilibrium is considered for the scaled system, we are able to obtain the global-in-time and uniform-in-$\eps$ estimates, and such estimates seem hard to achieve by the usual Hilbert expansion around global Maxwellians. In particular, the proof also relies on the appropriate velocity weight function in terms of the solution itself whenever the explicit time-decay of solutions is no longer avaivable, see \eqref{defw} later.  \\ 
\end{itemize}

The rest of this paper is arranged as follows: In Section \ref{secMain} we reformulate the problem with scalings under considerations in the weak-collision regime and present the main results together with detailed explanations to ideas of the proof. In Section \ref{secPre}, we collect some important properties of the Landau collision operator around local Maxwellians. In addition, to simplify our later calculations, some useful estimates on the microscopic correction are given. In Section \ref{secKdV}, we perform all the energy estimates required to obtain our KdV limit. In Section \ref{SecProof}, we prove the first theorem on the construction of solutions to the rescaled VPL system which converges to the KdV equation. In Section \ref{SecGlobal}, we use the framework of previous sections to study the VPL system under the same scalings as in the combined cold-ions limit and longwave limit for obtaining the global existence of solutions around global Maxwellians. Section \ref{sec.app} is an appendix giving the proofs of two lemmas that are left in Section \ref{secPre}.

\section{Main results}\label{secMain}

In order to derive the KdV equation from \eqref{VPL}, we perform the longwave scaling 
	\begin{equation}
		\label{scaling}
		\left\{
		\begin{array}{rl}
		\dis&\tilde{t}=\eps^{3/2}t,\quad\tilde{x}=\eps^{1/2}x, \quad \tilde{v}=\eps^{-1}v,\\
		&\tilde{F}(\tilde{t},\tilde{x},\tilde{v})=\eps^3F(t,x,v),\\
		&\tilde{\phi}(\tilde{t},\tilde{x})=\eps^{-1}\phi(t,x),
		\end{array} \right.
	\end{equation}
	as in \cite{Han-Kwan}, and still denote $F,t,x,v$ for $\tilde{F},\tilde{t},\tilde{x},\tilde{v}$ to get
	\begin{equation}\label{r1VPL}
		\left\{
		\begin{array}{rl}
			\dis \eps\partial_{t}F+\eps v_{1}\partial_{x}F-\partial_{x}\phi\partial_{v_{1}}F&\dis =\frac{\nu}{\eps^{3+1/2}}Q(F,F),
			\\[3mm]
			\dis -\eps^2\partial^{2}_{x}\phi+\eps\phi&\dis =\rho-1.
		\end{array} \right.
	\end{equation}
	Then a shift on $x$,
	\begin{equation}
		\label{shift}
		\left\{
		\begin{array}{rl}
			\dis&\bar{x}=x-t/\eps,\\
			&\overline{F}(t,\bar{x},v)=F(t,x,v),\\
			&\bar{\phi}(t,\bar{x})=\phi(t,x),
		\end{array} \right.
	\end{equation}
	yields
	\begin{equation}\label{rVPL}
		\left\{
		\begin{array}{rl}
			\dis \eps\partial_{t}F-\partial_{x}F+\eps v_{1}\partial_{x}F-\partial_{x}\phi\partial_{v_{1}}F&\dis =\frac{\nu}{\eps^{3+1/2}}Q(F,F),
			\\[3mm]
			\dis -\eps^2\partial^{2}_{x}\phi+\eps\phi&\dis =\rho-1,
		\end{array} \right.
	\end{equation}
	where we still denote $F,x$ for $\overline{F}, \bar{x}$ in the resulting equation. When $\nu$ is equal to zero, the work \cite{Han-Kwan} suggests that the electronic potential can behave like
	$
	\phi=\phi_0+O(\eps),
	$
	where $\phi_0$ solves the KdV equation 
	$$
	2\pa_t\phi_0+\pa_x^3\phi_0+3\phi_0\pa_x\phi_0=0.
	$$
	The longwave scaling \eqref{scaling} implies the cold ions assumption that the temperature for the original distribution function $f$ is low. Taking the collision effect into consideration, we shall construct solutions and prove the KdV limit for the original distribution $F$ that is close to the Dirac function $F\sim \rho(t,x)\de_{v=u(t,x)}$.
To simplify the notation, we set
\begin{equation}\label{def.delta}
\frac{1}{\de}:=\frac{\nu}{\eps^{3+1/2}},
 \end{equation}
 and
rewrite the above system \eqref{rVPL} to be
	\begin{equation}
		\label{reVPL}
		\left\{
		\begin{array}{rl}
			\dis \partial_{t}F-\frac{1}{\eps}\partial_{x}F+ v_{1}\partial_{x}F-\frac{1}{\eps}\partial_{x}\phi\partial_{v_{1}}F&\dis =\frac{1}{\eps\de}Q(F,F),
			\\[3mm]
			\dis -\eps^2\partial^{2}_{x}\phi+\eps\phi&\dis =\rho-1.
		\end{array} \right.
	\end{equation}
To the end we mainly focus on the rescaled system \eqref{rVPL} or \eqref{reVPL} with \eqref{def.delta} for studying the limit $\eps\to 0$ with $\nu=\eps^\beta$ for suitable $0<\beta<7/2$ such that $\delta\to 0$; see Remark \ref{rk.beta} for disccussions in case of $\beta>7/2$.
	
	\subsection{Decomposition of the system}
	It is worth pointing out that the Landau collision operator \eqref{LandauOp}  has five collision invariants:
	\begin{equation*}
		\psi_{0}(v)=1, \quad \psi_{i}(v)=v_{i},~(i=1,2,3),\quad \psi_{4}(v)=\frac{1}{2}|v|^{2},
	\end{equation*}
satisfying
	\begin{equation*}
		\int_{\mathbb{R}^{3}}\psi_{i}(v)Q(F,F)\,dv=0,\quad \mbox{for $i=0,1,2,3,4$.}
	\end{equation*}
Associated with any solution $F(t,x,v)$ of 	the VPL system \eqref{reVPL}, we define the macroscopic quantities, 
	the mass density $\rho(t,x)>0$, momentum $\rho(t,x)u(t,x)$, and
	energy density $\rho(\frac{3}{2}K\theta(t,x)+\frac 12|u(t,x)|^2)$, by
	\begin{equation}
		\label{DefMacro}
		\left\{
		\begin{array}{rl}
			\rho(t,x)&\dis \equiv\int_{\mathbb{R}^{3}}\psi_{0}(v)F(t,x,v)\,dv,
			\\[3mm]
			\rho(t,x) u_{i}(t,x)&\dis \equiv\int_{\mathbb{R}^{3}}\psi_{i}(v)F(t,x,v)\,dv, \quad \mbox{for $i=1,2,3$,}
			\\[3mm]
			\rho(t,x)(\frac{3}{2}K\theta(t,x)+\frac{1}{2}|u(t,x)|^{2})&\dis \equiv\int_{\mathbb{R}^{3}}\psi_{4}(v)F(t,x,v)\,dv,
		\end{array} \right.
	\end{equation} 
together with the corresponding local Maxwellian
\begin{equation}
		\label{DefLocalMax}
		M\equiv M_{[\rho,u,\theta]}(t,x,v):=\frac{\rho(t,x)}{\sqrt{(2\pi K\theta(t,x))^{3}}}\exp\big\{-\frac{|v-u(t,x)|^{2}}{2K\theta(t,x)}\big\},
\end{equation}
such that $\int_{\R^3}\psi_i(v) (F-M)\,dv=0$ for all collision invariants $\psi_{i}(v)$ $(0\leq i\leq 4)$.
Here we have denoted $u(t,x)=(u_{1},u_{2},u_{3})(t,x)$ and $\theta(t,x)$ to be the bulk velocity and temperature, respectively, and $K>0$ is the ideal gas constant. Through the paper we take $K=\frac{2}{3}$ for conveniences. 
	
	Let the inner product on  $L^2(\mathbb{R}^{3})$  be defined by
	$\langle h,g\rangle=\langle h,g\rangle_{L^2_v}:=\int_{\mathbb{R}^{3}}h(v)g(v)\,d v$, and the linearized collision operator $L_{M}$ around the local Maxwellian $M$ as in \eqref{DefLocalMax} be defined by 
\begin{equation}
\label{def.LM}
L_Mh:=Q(M,h)+Q(h,M).
\end{equation}
It is direct to verify that the orthonormal basis of the
	null space $\mathcal{N}$ of $L_{M}$ is
	\begin{equation*}
		\left\{
		\begin{array}{rl}
			&\dis \chi_{0}(v)=\frac{1}{\sqrt{\rho}}M,\\[2mm]
			&\dis  \chi_{i}(v)=\frac{v_{i}-u_{i}}{\sqrt{K\rho\theta}}M, \quad i=1,2,3,
			\\[2mm]
			&\dis \chi_{4}(v)=\frac{1}{\sqrt{6\rho}}(\frac{|v-u|^{2}}{K\theta}-3)M.
		\end{array} \right.
	\end{equation*}
	We set the macroscopic and microscopic projections  $P_{0}$ and $P_{1}$ to be
	\begin{equation}
		\label{DefProjection}
		P_{0}h_1\equiv\sum_{i=0}^{4}\langle h_1,\frac{\chi_{i}}{M}\rangle\chi_{i},\quad P_{1}h_1\equiv h_1-P_{0}h_1,
	\end{equation}
with the properties
	$$
	P_{0}P_{0}=P_{0},\quad
	P_{1}P_{1}=P_{1},\quad
	P_{1}P_{0}=P_{0}P_{1}=0.
	$$
A function $h_1(v)$ is called microscopic if
\begin{equation*}
	\langle h_1(v),\psi_{i}(v)\rangle=0, \quad \mbox{for $i=0,1,2,3,4$}.
\end{equation*}
Using the notations above, the solution $F(t,x,v)$ of the VPL system \eqref{reVPL} can be decomposed into its macroscopic part $M$ and microscopic part $G$ by
	\begin{equation*}
		F=M+G, \quad P_{0}F=M, \quad P_{1}F=G.
	\end{equation*}
    
		To further study $M$ and $G$ above, we will derive the equations to describe the fluid quantities $\rho,u,\theta$ and the non-fluid part $G$ by macro-micro decomposition. Multiplying \eqref{reVPL} by the collision invariants $\psi_{i}(v)$ $(i=0,1,2,3,4)$ and integrating the resulting equations with respect to
	$v\in\mathbb{R}^{3}$, together with the last equation in \eqref{reVPL}, we obtain the macroscopic system as follows:
	\begin{equation}\label{1macro}
		\left\{
		\begin{array}{rl}
			&\dis\partial_t\rho-\frac{1}{\eps}\partial_x\rho+\partial_x(\rho u_{1})=0,
			\\
			&\dis\partial_t(\rho u_{1})-\frac{1}{\eps}\partial_x(\rho u_{1})
			+\partial_x(\rho u_{1}^{2})+\partial_xp+\frac{1}{\eps}\rho\partial_x\phi=-\int_{\mathbb{R}^{3}} v^{2}_{1}\partial_xG\,dv,
			\\
			&\dis\partial_t(\rho u_{i})-\frac{1}{\eps}\partial_x(\rho u_{i})+\partial_x(\rho u_{1}u_{i})=-\int_{\mathbb{R}^{3}} v_{1}v_{i}\partial_xG\,dv, ~~i=2,3,
			\\
			&\dis\partial_t\{\rho (\theta+\frac{|u|^{2}}{2})\}-\frac{1}{\eps}\partial_x\{\rho (\theta+\frac{|u|^{2}}{2})\}
			+\partial_x\{\rho u_{1}(\theta+\frac{|u|^{2}}{2})+pu_{1}\}
			\\
			&\dis\hspace{2cm}+\frac{1}{\eps}\rho u_{1}\partial_x\phi
			=-\int_{\mathbb{R}^{3}} \frac{1}{2}v_{1}|v|^{2}\partial_xG\,dv,
			\\
			&\dis-\eps^2\partial^{2}_{x}\phi+\eps\phi =\rho-1,
		\end{array} \right.
	\end{equation}
 with the pressure $p=\frac{2}{3}\rho\theta$.
	Then we take $P_{1}$ of both side of \eqref{reVPL} to obtain
	\begin{align}\label{equG}
		\partial_tG-\frac{1}{\eps}\partial_xG+P_{1}(v_{1}\partial_xG)+P_{1}(v_{1}\partial_xM)-\frac{1}{\eps}\partial_x\phi\partial_{v_{1}}G
		=\frac{1}{\eps\delta}L_{M}G+\frac{1}{\eps\delta}Q(G,G),
	\end{align}
	where $L_M$ is defined in \eqref{def.LM}. Since $L_{M}: \mathcal{N}^\perp\to \mathcal{N}^\perp$ is invertible, the above equation implies
	\begin{equation}\label{reG}
		G=\eps\de L^{-1}_{M}[P_{1}(v_{1}\partial_xM)]+L^{-1}_{M}\Theta,
	\end{equation}
	with
	\begin{equation}\label{DefTheta}
		\Theta:=\eps\de\partial_tG-\de\partial_xG+\eps\de P_{1}(v_{1}\partial_xG)-\de\partial_x\phi\partial_{v_{1}}G-Q(G,G).
	\end{equation}
	For the first term on the right hand side of \eqref{reG}, one has the identities that
	\begin{align}
		-\int_{\R^3} v_i v_1\pa_x
		L^{-1}_M[P_1(v_1\pa_x M)] dv
		&\equiv\partial_x(\mu(\theta)\partial_xu_{i})+\frac{1}{3}\de_{1i}\partial_x(\mu(\theta)\partial_xu_{1}),\quad i=1,2,3,\label{id1}
		\\
		-\int_{\R^3} \frac{1}{2}|v|^{2} v_1\pa_x
		L^{-1}_M[P_{1}(v_1\pa_xM)] dv
		&\equiv\partial_x(\kappa(\theta)\partial_x\theta)+\partial_x(\mu(\theta)\{\frac{4}{3}u_{1}\partial_xu_{1}+u_{2}\partial_xu_{2}+u_{3}\partial_xu_{3}\})\label{id2},
	\end{align}
	where the viscosity coefficient $\mu(\theta)$ and heat conductivity coefficient $\ka(\theta)$ are given by 
		\begin{align*}
		\mu(\theta)&=- K\theta\int_{\R^3}\hat{B}_{ij}(\frac{v-u}{\sqrt{K\theta}})
		B_{ij}(\frac{v-u}{\sqrt{K\theta}})dv>0,\quad i\neq j,\quad
		\kappa(\theta)=-K^2\theta\int_{\R^3}\hat{A}_j(\frac{v-u}{\sqrt{K\theta}})
		A_j(\frac{v-u}{\sqrt{K\theta}})dv>0.
	\end{align*}
	In the above definition, we have used Burnett functions that
		\begin{equation}\label{defbur1}
		\hat{A}_j(v)=\frac{|v|^2-5}{2}v_j, \quad \hat{B}_{ij}(v)=v_iv_j-\frac{1}{3}\delta_{ij}|v|^2, 
	\end{equation}
	and
	\begin{equation}\label{defbur2}
		A_{j}(v)=L^{-1}_M[\hat{A}_j(v)M], \quad B_{ij}(v)=L^{-1}_M[\hat{B}_{ij}(v)M].
	\end{equation}
	We substitute \eqref{reG} into \eqref{1macro} and use \eqref{id1} and \eqref{id2} to get
	\begin{equation}
		\label{NSmacro}
		\left\{
		\begin{array}{rl}
			&\dis\partial_t\rho-\frac{1}{\eps}\partial_x\rho+\partial_x(\rho u_{1})=0,
			\\
			&\dis\partial_t(\rho u_{1})-\frac{1}{\eps}\partial_x(\rho u_{1})+\partial_x(\rho u_{1}^{2})+\partial_xp+\frac{1}{\eps}\rho\partial_x\phi=\frac{4}{3}\eps\de\partial_x(\mu(\theta)\partial_xu_{1})-\partial_x(\int_{\mathbb{R}^{3}} v^{2}_{1}L^{-1}_{M}\Theta \,dv),
			\\
			&\dis\partial_t(\rho u_{i})-\frac{1}{\eps}\partial_x(\rho u_{i})+\partial_x(\rho u_{1}u_{i})=\eps\de\partial_x(\mu(\theta)\partial_xu_{i})
			-\partial_x(\int_{\mathbb{R}^{3}} v_{1}v_{i}L^{-1}_{M}\Theta \,dv), ~~i=2,3,
			\\
			&\dis\partial_t\{\rho (\theta+\frac{|u|^{2}}{2})\}-\frac{1}{\eps}\partial_x\{\rho (\theta+\frac{|u|^{2}}{2})\}+\partial_x\{\rho u_{1}(\theta+\frac{|u|^{2}}{2})+pu_{1}\}+\frac{1}{\eps}\rho u_{1}\partial_x\phi=\eps\de\partial_x(\kappa(\theta)\partial_x\theta)
			\\
			&\dis\qquad+\frac{4}{3}\eps\de\partial_x(\mu(\theta)u_{1}\partial_xu_{1})+\eps\de\partial_x(\mu(\theta)u_{2}\partial_xu_{2}+\mu(\theta)u_{3}\partial_xu_{3})
			-\frac{1}{2}\partial_x(\int_{\mathbb{R}^{3}}v_{1}|v|^{2}L^{-1}_{M}\Theta \,dv),
			\\
			&\dis-\eps^2\partial^{2}_{x}\phi+\eps\phi =\rho-1.
		\end{array} \right.
	\end{equation}
Then we obtain the macroscopic system \eqref{NSmacro} and the microscopic equation \eqref{equG}.
	
	\subsection{Formal derivation to KdV equation}
	If we assume the non-fluid quantity $G$ in \eqref{1macro} has minor impact on the system and ignore $G$, $u_2$ and $u_3$ by letting $G=u_2=u_3=0$, one further gets the closed system,
	\begin{equation}\label{formalep}
		\left\{
		\begin{array}{rl}
			&\dis\partial_t\rho-\frac{1}{\eps}\partial_x\rho+\partial_x(\rho u_{1})=0,
			\\
			&\dis\partial_t u_{1}-\frac{1}{\eps}\partial_x u_{1}
			+u_{1}\partial_x u_{1}+\frac{2}{3}\frac{\partial_x\rho}{\rho}\theta+\frac{2}{3}\pa_x\theta+\frac{1}{\eps}\partial_x\phi=0,
			\\
			&\dis\partial_t\theta-\frac{1}{\eps}\partial_x\theta+u\pa_x\theta+\frac{2}{3}\theta\pa_xu
			=0,
			\\
			&\dis-\eps^2\partial^{2}_{x}\phi+\eps\phi =\rho-1.
		\end{array} \right.
	\end{equation}
Note that the rescaled Euler-Poisson system above is different from that in \cite{Guo-Pu}, but it is consistent with the one derived from the fluid dynamic limit system of the collisionless Euler-Poisson system under the cold-ions limit as in \cite{Han-Kwan}. Indeed, we perform the expansion:
	\begin{align}\label{expansion}
\left\{
\begin{array}{rl}		\rho&=1+\eps\rho_1+\eps^2\rho_2+\eps^3\rho_3+O(\ep^4),\\
		u_1&=u_{10}+\eps u_{11}+\eps^2u_{12}+\eps^3u_{13}+O(\ep^4),\\
		\theta&=\frac{3}{2}+\eps\theta_1+\eps^2\theta_2+\eps^3\theta_3+O(\ep^4),\\
		\phi&=\phi_0+\eps\phi_1+\eps^2\phi_2+\eps^3\phi_3+O(\ep^4),
        \end{array} \right.
	\end{align}
    where the zero-orders of $u_1$ and $\phi$ can be variable. 
 Substituting the above ansatz into \eqref{formalep} and matching the different powers of $\eps$, we obtain the equations for $\rho_i,u_{1i},\theta_i$ and $\phi_i$ as follows.
 \begin{enumerate}
 	\item From mass equation:
 	\begin{align}
 	O(1):\quad&-\pa_x\rho_1+\pa_xu_{10}=0,\label{mass1}\\
 	O(\eps):\quad&\pa_t\rho_1-\pa_x\rho_2+\pa_x(\rho_1u_{10})+\pa_xu_{11}=0,\label{mass2}\\
 	O(\eps^2):\quad&\pa_t\rho_2-\pa_x\rho_3+\pa_x(\rho_2u_{10}+\rho_1 u_{11})+\pa_xu_{12}=0.\label{mass3}
 	\end{align}
 	
 	 	\item From velocity equation:
 	\begin{align}
 		O(\eps^{-1}):\quad&-\pa_xu_{10}+\pa_x\phi_0=0,\label{velocity1}\\
 		O(1):\quad&\pa_t u_{10}-\pa_x u_{11}+u_{10}\pa_x u_{10}+\pa_x \phi_1=0,\label{velocity2}\\
 		O(\eps):\quad&\pa_t u_{11}-\pa_x u_{12}+u_{10}\pa_x u_{11}+u_{11}\pa_x u_{10}+\pa_x\rho_1+\frac{2}{3}\pa_x\theta_1+\pa_x \phi_2=0, \label{velocity3}\\
 		O(\eps^2):\quad&\pa_t u_{12}-\pa_x u_{13}+u_{10}\pa_x u_{12}+u_{11}\pa_x u_{11}+u_{12}\pa_x u_{10}+\pa_x\rho_2+\frac{2}{3}\theta_1\pa_x\rho_1+\frac{2}{3}\pa_x\theta_2+\pa_x \phi_3=0.\label{velocity4}	
 	\end{align}
  		
 	 	 	\item From temperature equation:
 	\begin{align}
	O(1):\quad&-\pa_x \theta_1+\pa_x u_{10}=0,\label{temperature1}\\
	O(\eps):\quad&\pa_t \theta_1-\pa_x \theta_2+u_{10}\pa_x \theta_1+\frac{2}{3}\theta_1 \pa_x u_{10}+\pa_x u_{11}=0,	\label{temperature2}\\
		O(\eps^2):\quad&\pa_t \theta_2-\pa_x \theta_3+u_{10}\pa_x \theta_2+u_{11}\pa_x \theta_1+\frac{2}{3}\theta_2 \pa_x u_{10}+\frac{2}{3}\theta_1 \pa_x u_{11}+\pa_x u_{12}=0.	\label{temperature3}
\end{align}
 
 \item From Poisson equation:
 	\begin{align}
	O(\eps):\quad&\phi_0=\rho_1,\label{Poisson1}\\
	O(\eps^2):\quad&-\pa_x^2 \phi_0+\phi_1=\rho_2,	\label{Poisson2}\\
	O(\eps^3):\quad&-\pa_x^2 \phi_1+\phi_2=\rho_3.	\label{Poisson3}
\end{align} 
 	\end{enumerate}	
We use \eqref{mass1}, \eqref{velocity1}, \eqref{temperature1} and \eqref{Poisson1} to get 
\begin{align}\label{1orderid}
	u_{10}=\rho_1=\theta_1=\phi_0.
\end{align}	
Plugging the above identity into \eqref{mass2}, together with \eqref{Poisson2} gives
\begin{align*}
	\pa_t\phi_0+\pa_x^3\phi_0-\pa_x\phi_1+\pa_x(\phi_0^2)+\pa_x u_{11}=0,
\end{align*}
	which, further combined with \eqref{velocity2} and \eqref{1orderid}, yields the KdV equation
	\begin{align}\label{KdV}
		2\pa_t\phi_0+\pa_x^3\phi_0+3\phi_0\pa_x\phi_0=0.
	\end{align}
Hence, we formally derive the KdV equations from  the fluid-type system \eqref{NSmacro}. Notice that we only use some of the equations in \eqref{mass1}--\eqref{Poisson3}. The others are still important since we need to use them to define the approximate fluid quantities in the higher orders. The existence of solutions to the KdV equation \eqref{KdV} is guaranteed by the following proposition. See \cite{Kenig} for the complete statement of the result and proof. 

\begin{proposition}
	Let $s\geq2$ be a sufficiently large integer. For any given $T>0$, the Cauchy problem on the KdV equation \eqref{KdV} with initial data $\phi_0(0,x)\in H^{s}(\mathbb{R})$ admits a unique
	smooth  solution
	\begin{equation*}
		\phi_0(t,x)\in L^{\infty}(-T,T;H^{s}(\mathbb{R})).
	\end{equation*}
\end{proposition}
The existence of other $(\rho_i,u_{1i},\theta_i,\phi_i)$ can be ensured by the following lemma. The proof will be given in the appendix.
\begin{lemma}\label{lemmacor}
Let $r\geq2$ be any integer. Then there exists a sufficiently large integer $s>r$ such that for any $T>0$ and $\phi_0(t,x)\in L^{\infty}(-T,T;H^{s}(\mathbb{R}))$, the Cauchy problem on equations \eqref{mass1}--\eqref{Poisson3} with any initial data $(\rho_{i},u_{1j},\theta_{i},\phi_j)(0,x)\in H^{r}(\mathbb{R})$ admits a unique
	smooth solution
	\begin{equation}\label{boundcor}
		(\rho_{i},u_{1j},\theta_{i},\phi_j)\in L^{\infty}(-T,T;H^{r}(\mathbb{R})),\quad i=2,3, \quad j=1,2,3.
	\end{equation}
\end{lemma}

\begin{remark}
Combining \eqref{scaling} and \eqref{shift}, we have performed the transform
$$
\eps^{3/2}t\rightarrow t,\quad \eps^{1/2}(x-t)\rightarrow x,\quad \eps^{-1}v\rightarrow v.
$$
If we set the transform for spacial variable to be
$$
\eps^{1/2}(x-Vt)\rightarrow x
$$
with the velocity parameter $V$, then \eqref{mass1}, \eqref{velocity1}, \eqref{temperature1}, \eqref{Poisson1} will become
\begin{align*}
	-V\pa_x\rho_1+\pa_xu_{10}=0,\quad -V\pa_xu_{10}+\pa_x\phi_0=0,\quad -V\pa_x \theta_1+\pa_x u_{10}=0,\quad \phi_0=\rho_1.
\end{align*}
To avoid trivial solutions, one sees that $V$ should satisfy $V^2-1=0$, which implies $V=\pm 1.$ Hence, in \eqref{shift}, we directly let $V=1$.
\end{remark}

\subsection{Approximate fluid quantities and perturbed system}	
In terms of the expansion \eqref{expansion}, we choose our approximation to the fluid quantities $(\rho,u,\theta,\phi)$ to be 
	\begin{align}\label{approx}
	\left\{
	\begin{array}{rl}		\bar{\rho}&=1+\eps\rho_1+\eps^2\rho_2+\eps^3\rho_3,\\
		\bar{u}_1&=u_{10}+\eps u_{11}+\eps^2u_{12}+\eps^3u_{13},\\
		\bar{u}_2&=\bar{u}_3=0,\\
		\bar{\theta}&=\frac{3}{2}+\eps\theta_1+\eps^2\theta_2+\eps^3\theta_3,\\
		\bar{\phi}&=\phi_0+\eps\phi_1+\eps^2\phi_2+\eps^3\phi_3.\end{array} \right.
\end{align}

\begin{remark}\label{rkkdvexp}
Here we should point out that for $\eps>0$ small enough, $\bar{\rho}$ and $\bar{\theta}$ are sufficiently close to constant states and have small variations, even though the KdV profile is of large amplitude. However, for $\bar{u}_1$ and $\bar{\phi}$, they may still have large amplitude and variations, which is one of the main difficulties to be overcome. Thus, the situation under consideration is essentially different from that in the previous work \cite{DYY}.
\end{remark}

Corresponding to \eqref{approx}, the local Maxwellian is given by
	\begin{equation}
    \label{2.38A}
		\overline{M}:=M_{[\bar{\rho},\bar{u},\bar{\theta}](t,x)}(v)
		=\frac{\bar{\rho}(t,x)}{(2\pi K\bar{\theta}(t,x))^{3/2}}\exp\big(-\frac{|v-\bar{u}(t,x)|^{2}}{2K\bar{\theta}(t,x)}\big).
	\end{equation}
	Moreover, we define the microscopic Chapman-Enskog correction
	\begin{equation}\label{DefbarG}
		\overline{G}=\eps\de L_{M}^{-1}\{ P_{1}v_{1}M(\frac{|v-u|^{2}
			\partial_x\bar{\theta}}{2K\theta^{2}}+\frac{(v-u)\cdot\partial_x\bar{u}}{K\theta})\}.
	\end{equation}
	The perturbation is denoted by 
	\begin{equation}\label{Defpert}
		\left\{
		\begin{array}{rl}
			&\dis(\widetilde{\rho},\widetilde{u},\widetilde{\theta},\widetilde{\phi})(t,x)
			=({\rho}-\bar{\rho},u-\bar{u},\theta-\overline{\theta},\phi-\overline{\phi})(t,x),
			\\[3mm]
			&\dis h(t,x,v)=G(t,x,v)-\overline{G}(t,x,v)\ \text{with}\ \ G=F-M.
		\end{array} \right.
	\end{equation}
	We substitute the first identity above into \eqref{NSmacro} to get
	\begin{align}
		\label{perturbeq}
		\left\{
		\begin{array}{rl}
			&\dis \partial_{t}\widetilde{\rho}-\frac{1}{\eps}\pa_x \widetilde{\rho}
			+\partial_x(\bar{\rho}\widetilde{u}_{1})
			+\partial_x(\widetilde{\rho}u_{1})=-\eps^3R_1,
			\\
			&\dis \partial_t\widetilde{u}_{1}-\frac{1}{\eps}\partial_x\widetilde{u}_{1}
			+u_{1}\partial_x\widetilde{u}_{1}+\widetilde{u}_{1}\partial_x\bar{u}_{1}+\frac{2}{3}\partial_x\widetilde{\theta}
			+\frac{2}{3}(\frac{\theta}{\rho}-\frac{\bar{\theta}}{\bar{\rho}})\partial_x\rho+\frac{2}{3}\frac{\bar{\theta}}{\bar{\rho}}\partial_x\widetilde{\rho}
			+\frac{1}{\eps}\partial_x\widetilde{\phi}
			\\
			&\dis\qquad=\eps\de\frac{4}{3\rho}\partial_x(\mu(\theta)\partial_xu_{1})-\frac{1}{\rho}\partial_x(\int_{\mathbb{R}^{3}} v^{2}_{1}L^{-1}_{M}\Theta\,dv)-\eps^3\frac{1}{\bar{\rho}}R_2,\\
			&\dis\partial_t\widetilde{u}_{i}-\frac{1}{\eps}\partial_x\widetilde{u}_{i}+u_{1}\partial_x\widetilde{u}_{i}=\eps\de\frac{1}{\rho}\partial_x(\mu(\theta)\partial_xu_{i})-\frac{1}{\rho}\partial_x(\int_{\mathbb{R}^3} v_{1}v_{i}L^{-1}_{M}\Theta\,dv), ~~i=2,3,
			\\
			&\dis\partial_t\widetilde{\theta}-\frac{1}{\eps}\partial_x\widetilde{\theta}
			+\frac{2}{3}\bar{\theta}\partial_x\widetilde{u}_{1}+\frac{2}{3}\widetilde{\theta}\partial_xu_{1}
			+u_1\partial_x\widetilde{\theta}+\widetilde{u}_{1}\partial_x\bar{\theta}
			\\
			&\dis\qquad=\eps\de\frac{1}{\rho}\partial_x(\kappa(\theta)\partial_x\theta) 
			+\eps\de\frac{4}{3\rho}\mu(\theta)(\partial_xu_{1})^2
			+\eps\de\frac{1}{\rho}\mu(\theta)[(\partial_xu_{2})^2+(\partial_xu_{3})^2]
			\\
			&\dis\qquad\quad+\frac{1}{\rho}u\cdot\partial_x(\int_{\mathbb{R}^3} v_{1}v L^{-1}_{M}\Theta\, dv)-\frac{1}{\rho}\partial_x(\int_{\mathbb{R}^3}v_{1}\frac{|v|^{2}}{2}L^{-1}_{M}\Theta\, dv)-\eps^3 R_3,
			\\
			&\dis-\eps^2\pa_x^2\widetilde{\phi}+\eps\widetilde{\phi}=\widetilde{\rho}-\eps^4R_4,
		\end{array} \right.
	\end{align}
where the remainder terms are denoted by
\begin{align*}
	R_1&=\pa_t\rho_3+\pa_x\big(\rho_3(u_{10}+\eps u_{11}+\eps^2u_{12}+\eps^3u_{13})+\rho_2( u_{11}+\eps u_{12}+\eps^2u_{13})+\rho_1(u_{12}+\eps u_{13})+\pa_xu_{13}\big),\\
	R_2&=(1+\eps\rho_1+\eps^2\rho_2+\eps^3\rho_3)(\pa_tu_{13}+u_{10}\pa_xu_{13}+u_{11}\pa_x(u_{12}+\eps u_{13})+\frac{2}{3}\pa_x\theta_3\notag\\
	&\qquad\qquad\qquad\qquad\qquad\qquad+u_{12}\pa_x(u_{11}+\eps u_{12}+\eps^2 u_{13})+u_{13}\pa_x(u_{10}+\eps u_{11}+\eps^2 u_{12}+\eps^3 u_{13}))\notag\\
	&\qquad-(\rho_1+\eps\rho_2+\eps^2\rho_3)(\pa_x\rho_2+\frac{2}{3}\theta_1\pa_x\rho_1)-(\rho_2+\eps\rho_3)\pa_x\rho_1+\pa_x\rho_3+\frac{2}{3}\theta_1\pa_x(\rho_2+\eps\rho_3)\notag\\
	&\qquad+\frac{2}{3}(\theta_2+\eps\theta_3)\pa_x(\rho_1+\eps\rho_2+\eps^2\rho_3)
    +\frac{2}{3}(1+\eps\rho_1+\eps^2\rho_2+\eps^3\rho_3)\pa_x\theta_3,  
    \\
R_3&=\pa_t\theta_3+u_{10}\pa_x\theta_3+u_{11}\pa_x(\theta_2+\eps\theta_3)+u_{12}\pa_x(\theta_1+\eps\theta_2+\eps^2\theta_3)+\pa_xu_{13}\notag\\
	&\qquad+\frac{2}{3}\theta_1\pa_x(u_{12}+\eps u_{13})+\frac{2}{3}\theta_2\pa_x(u_{11}+\eps u_{12}+\eps^2 u_{13})+\frac{2}{3}\theta_3\pa_x(u_{10}+\eps u_{11}+\eps^2 u_{12}+\eps^3 u_{13})
    \\
    &\qquad+u_{13}\pa_x(\eps\theta_1+\eps^2\theta_2+\eps^3\theta_3),\\
	R_4&=-\pa_x^2(\phi_2+\eps\phi_3)+\phi_3.
\end{align*}
This would be the major fluid-type system.
	
Then we substitute the second line in \eqref{Defpert} into \eqref{equG} to get the microscopic equation
		\begin{align}\label{eqh}
		&\partial_th-\frac{1}{\eps}\partial_xh+v_{1}\partial_xh-\frac{1}{\eps}\partial_x\phi\partial_{v_{1}}h	-\frac{1}{\eps\de}L_M h
		\notag\\
		=&\frac{1}{\eps\de}Q(G,G)+P_{0}(v_{1}\partial_xh)-P_{1}\{v_{1}M(\frac{|v-u|^{2}
			\partial_x\widetilde{\theta}}{2K\theta^{2}}+\frac{(v-u)\cdot\partial_x\widetilde{u}}{K\theta})\}
		\notag\\
		&+\frac{1}{\eps}\partial_x\overline{G}+\frac{1}{\eps}\partial_x\phi\partial_{v_{1}}\overline{G}
		-P_{1}(v_{1}\partial_x\overline{G})
		-\partial_t\overline{G},
	\end{align}
		with the initial data
	\begin{align*}
		h_0(x,v):=h(0,x,v)=G(0,x,v)-\overline{G}(0,x,v).
	\end{align*}
As explained in Remark \ref{rkkdvexp}, we are forced to apply the Caflisch's decomposition for solving \eqref{eqh}. Thus we introduce a global Maxwellian $\mu$ by 
	\begin{align}\label{Defmu}
		\mu=\mu(v)=e^{-\frac{|v|^2}{4}},
	\end{align}
which has the slower decay in large velocities than $\overline{M}$. 

\begin{remark}
We remark that such $\mu$ in \eqref{Defmu} could be replaced by a general Gaussian form $\exp (-c|v|^2)$ with $0<c\leq 1/4$, or a stretched exponential $\exp(-c\langle v\rangle^s)$ with $0<s<2$ and $c>0$, or a much heavier-tailed polynomial $\langle v\rangle^m$. However, we do not persue such generality in the current work; interested readers may refer to a recent work \cite{Duan-Li} for details.
\end{remark}

We further decompose the solution $h$ as
\begin{equation}
\label{2.46A}
h(t,x,v)=\sqrt{\overline{M}}f(t,x,v)+\sqrt{\mu}g(t,x,v),
\end{equation}
with $g$ and $f$ satisfying
	\begin{align}\label{eqg}
		&\partial_tg-\frac{1}{\eps}\partial_xg+v_{1}\partial_xg	-\frac{1}{\eps\de}L_Dg+\frac{(\pa_t-\frac{1}{\eps}\pa_x+v_1\pa_x)\sqrt{\overline{M}}}{\sqrt{\mu}}f-\frac{1}{\eps}\frac{\partial_x\phi\partial_{v_{1}}(\sqrt{\mu}g+\sqrt{\overline{M}}f)}{\sqrt{\mu}}
		\notag\\
		=&\frac{1}{\eps\de}\Gamma(g+\frac{\sqrt{\overline{M}}}{\sqrt{\mu}}f+\frac{\overline{G}}{\sqrt{\mu}},g+\frac{\sqrt{\overline{M}}}{\sqrt{\mu}}f+\frac{\overline{G}}{\sqrt{\mu}})+\frac{1}{\eps\de}\{\Gamma(\frac{M-\overline{M}}{\sqrt{\mu}},\frac{\sqrt{\overline{M}}}{\sqrt{\mu}}f)+\Gamma(\frac{\sqrt{\overline{M}}}{\sqrt{\mu}}f,\frac{M-\overline{M}}{\sqrt{\mu}})\},
	\end{align}
	and
	\begin{align}\label{eqf}
		&\partial_tf-\frac{1}{\eps}\partial_xf+v_{1}\partial_xf	-\frac{1}{\eps\de}\CL_{\overline{M}} f-\frac{1}{\eps\de}\frac{\sqrt{\mu}}{\sqrt{\overline{M}}}L_Bg\notag\\
		=&\frac{P_{0}(v_{1}\partial_x(\sqrt{\overline{M}}f)+v_{1}\sqrt{\mu}\partial_xg)}{\sqrt{\overline{M}}}-\frac{1}{\sqrt{\overline{M}}}P_{1}\{v_{1}M(\frac{|v-u|^{2}
			\partial_x\widetilde{\theta}}{2K\theta^{2}}+\frac{(v-u)\cdot\partial_x\widetilde{u}}{K\theta})\}\notag\\
		&+\frac{1}{\eps}\frac{\partial_x\overline{G}}{\sqrt{\overline{M}}}+\frac{1}{\eps}\frac{\partial_x\phi\partial_{v_{1}}\overline{G}}{\sqrt{\overline{M}}}
		-\frac{P_{1}(v_{1}\partial_x\overline{G})}{\sqrt{\overline{M}}}
	-\frac{\partial_t\overline{G}}{\sqrt{\overline{M}}}.
	\end{align}
We supplement \eqref{eqg} and \eqref{eqf} with initial data
	\begin{align*}
		g_0(x,v):=g(0,x,v)=\frac{h_0(x,v)}{\sqrt{\mu}(v)},\quad f_0(x,v):=f(0,x,v)=0,
	\end{align*}
where the slow-velocity-decay part $\sqrt{\mu}g(t,x,v)$ carries the whole initial data of $h$. Here we have used the following notations:
\begin{equation}\label{DefGa}
\Gamma(g_1,g_2):=\frac{1}{\sqrt{\mu}}Q(\sqrt{\mu}g_1,\sqrt{\mu}g_2),
\end{equation}
and 
\begin{align}
\CL_{\overline{M}}f:=\frac{1}{\sqrt{\overline{M}}}Q(\overline{M},\sqrt{\overline{M}}f)+\frac{1}{\sqrt{\overline{M}}}Q(\sqrt{\overline{M}}f,\overline{M}),
\label{DefLM}
\end{align}
and
\begin{align}
-L_Dg:=-\frac{1}{\sqrt\mu}Q(M,\sqrt{\mu}g)-\frac{1}{\sqrt\mu}Q(\sqrt{\mu}g,M)+A_1\chi_{|v|\leq A_2}g,\quad L_Bg:=A_1\chi_{|v|\leq A_2}g.\label{DefL}
\end{align}
In the above definition, $\chi_{|v|\leq A_2}=\chi_{A_2}(v)$ is a smooth cut-off function such that $\chi_{A_2}(v)=0$ when $|v|>2A_2$ and $\chi_{A_2}(v)=1$ when $|v|\leq A_2$. The constants $A_1$ and $A_2$ are chosen in Lemma \ref{LemmaL} such that $L_D$ has good dissipative property and $L_B$ is bounded. It is straightforward to verify that if $(g,f)$ satisfies \eqref{eqg} and \eqref{eqf}, then $h(t,x,v)=\sqrt{\overline{M}}f(t,x,v)+\sqrt{\mu}g(t,x,v)$ solves \eqref{eqh}.

\subsection{Notations}
Throughout the paper we shall use $\langle \cdot , \cdot \rangle$  to denote the standard $L^{2}$ inner product in $\mathbb{R}_{v}^{3}$
with its corresponding $L^{2}$ norm $|\cdot|_2$, and $( \cdot , \cdot )$ to denote $L^{2}$ inner product in
$\mathbb{R}_{x}$ or $\mathbb{R}_{x}\times \mathbb{R}_{v}^{3}$  with its corresponding $L^{2}$ norm $\|\cdot\|$.
In order to control the polynomial growth caused by the forcing term in the VPL system, we should define the time-velocity weight
\begin{equation}\label{Defw}
	w(\alpha,\beta)(t,v):=\langle v\rangle^{2(l-\al-|\beta|)}\exp\left(\frac{q_1}{(1+t)^{q_2}}\frac{\langle v\rangle^2}{2}\right),\quad  l\geq \al+|\beta|,
\end{equation}
and the pure velocity weight
\begin{equation}
\label{DefW}
	W(\alpha,\beta)(t,v):=\langle v\rangle^{2(l-\al-|\beta|)},\quad  l\geq \al+|\beta|,
\end{equation}
where $\langle v\rangle=\sqrt{1+|v|^{2}}$  and $ q_1,q_2\in[0,1)$ will be chosen later.
The Landau collision frequency is given by
\begin{equation*}
	\sigma_{ij}(v):=(\Phi_{ij}\ast \mu)(v)=\int_{{\mathbb R}^3}\Phi_{ij}(v-v_{\ast})\mu(v_{\ast})\,dv_{\ast}.
\end{equation*}
Here $[\sigma_{ij}(v)]_{1\leq i,j\leq 3}$ is a positive-definite self-adjoint matrix. Given integer $\al\geq0$ and multiple index $\be=(\be_1,\be_2,\be_3)$, we define the differential operator
$$
\pa^\al_\be:=\pa^\al_x\pa^\be_v.
$$
Moreover, we denote the weighted $L^2$ energy and dissipation norms as follows:
\begin{align*}
|\partial^\alpha_\beta g|_{w(\al',\be')}^2:&= \int_{{\mathbb R}^3}|w(\al',\be')|^{2}|\partial^\alpha_\beta g|^2\,dv,\quad
\|\partial^\alpha_\beta g\|_{w(\al',\be')}^2:=\int_{{\mathbb R}}|\partial^\alpha_\beta g|_{w(\al',\be')}^2\,dx,
\\
	|\partial^\alpha_\beta g|^2_{\sigma,w(\al',\be')}:&=\sum_{i,j=1}^3\int_{{\mathbb
			R}^3}|w(\al',\be')|^{2}[\sigma_{ij}\partial_{v_i}\partial^\alpha_\beta g\partial_{v_j}\partial^\alpha_\beta g+\sigma_{ij}\frac{v_i}{2}\frac{v_j}{2}(\partial^\alpha_\beta g)^2]\,dv,
\end{align*}
and
$$
 \|\partial^\alpha_\beta g\|^2_{\sigma,w(\al',\be')}:=\int_{\R}|\partial^\alpha_\beta g|^2_{\sigma,w(\al',\be')}dx.
$$
If $(\al',\be')=(\al,\be)$, we write $w$ instead of $w(\al',\be')$ in the above definition. Moreover, $\|\partial^\alpha_\beta f\|_{W}^2$ and $\|\partial^\alpha_\beta f\|_{\si,W}^2$ are defined in the corresponding way.
In particular, we write $\|g\|_{\sigma}=\|g\|_{\sigma,1}$.
It is shown in \cite[Lemma 5,p.315]{Strain-Guo} that
\begin{equation}\label{boundlandaunorm}
	|g|_{\sigma,w}\sim |w\langle v\rangle^{-\frac{1}{2}}g|_2+\Big|w\langle v\rangle^{-\frac{3}{2}}\nabla_vg\cdot\frac{v}{|v|}\Big|_2+\Big|w\langle v\rangle^{-\frac{1}{2}}\nabla_vg\times\frac{v}{|v|}\Big|_2.
\end{equation}
Here we say $A\sim B$ if there exists some $C>1$ such that $\frac{1}{C}B\leq A\leq CB.$ 

Note that $h(t,x,v)=\sqrt{\overline{M}}f(t,x,v)+\sqrt{\mu}g(t,x,v)$ is purely microscopic since $G$ and $\overline{G}$ are both microscopic, which however, does not imply that $\sqrt{\overline{M}}f$ or $\sqrt{\mu}g$ is microscopic. Thus, we need some notations to capture the macro and micro parts of $f$ {corresponding to the linearized collision term $\CL_{\overline{M}} f$ in \eqref{eqf}}.
The orthonormal basis of the null space $\overline{\mathcal{N}}$ of $\CL_{\overline{M}}$, given by	
\begin{equation*} 			
\left\{ 		\begin{array}{rl} 			
&\dis \bar{\chi}_{0}(v)=\frac{1}{\sqrt{\bar{\rho}}}\overline{M},\\[2mm] 			
&\dis  \bar{\chi}_{i}(v)=\frac{v_{i}-\bar{u}_{i}}{\sqrt{K\bar{\rho}\bar{\theta}}}\overline{M}, \quad i=1,2,3, 			\\[2mm] 			&\dis \bar{\chi}_{4}(v)=\frac{1}{\sqrt{6\bar{\rho}}}(\frac{|v-\bar{u}|^{2}}{K\bar{\theta}}-3)\overline{M}. 		\end{array} \right. 	
\end{equation*}
Let us define $\mathbf{P}_0$ as the orthogonal projection from
$L^2_v$ onto $\overline{\mathcal{N}}$, then we have
\begin{align}\label{defP0}
	\mathbf{P}_0f(t,x,v)=\{a(t,x)\frac{1}{\sqrt{\bar{\rho}}}+b(t,x)\cdot \frac{v-\bar{u}}{\sqrt{K\bar{\rho}\bar{\theta}}} + c(t,x)\frac{1}{\sqrt{6\bar{\rho}}}(\frac{|v-\bar{u}|^{2}}{K\bar{\theta}}-3)\}\overline{M}^{1/2}(v),
\end{align}
with
\begin{align*}
	\dis \left\{\begin{array}{rl}
	\dis	a(t,x)&\dis=\int_{\R^3}f(t,x,v)\frac{1}{\sqrt{\bar{\rho}}}\overline{M}^{1/2}(v)\,dv,\\
	\dis	b(t,x)&\dis=\int_{\R^3}f(t,x,v)\frac{v-\bar{u}}{\sqrt{K\bar{\rho}\bar{\theta}}}\overline{M}^{1/2}(v)\,dv,\\
	\dis	c(t,x)&\dis=\int_{\R^3}f(t,x,v)\frac{1}{\sqrt{6\bar{\rho}}}(\frac{|v-\bar{u}|^{2}}{K\bar{\theta}}-3)\overline{M}^{1/2}(v)\,dv.
	\end{array}\right.
\end{align*}
Correspondingly, we write the microscopic projection as 
\begin{equation*}
\mathbf{P}_1f=f-\mathbf{P}_0f. 
\end{equation*} 
Notice that $\sqrt{\overline{M}}f(t,x,v)+\sqrt{\mu}g(t,x,v)$ is microscopic, then
$$
\int_{\R^3}(\sqrt{\overline{M}}f+\sqrt{\mu}g)\frac{1}{\sqrt{\bar{\rho}}}\,dv=
\int_{\R^3}(\sqrt{\overline{M}}f+\sqrt{\mu}g)\frac{v-\bar{u}}{\sqrt{K\bar{\rho}\bar{\theta}}}\,dv
=\int_{\R^3}(\sqrt{\overline{M}}f+\sqrt{\mu}g)\frac{1}{\sqrt{6\bar{\rho}}}(\frac{|v-\bar{u}|^{2}}{K\bar{\theta}}-3)\,dv=0.
$$
It follows from this and \eqref{defP0} that
\begin{align*}
\mathbf{P}_0f=-\{\frac{1}{\sqrt{\bar{\rho}}}\int_{\R^3}\frac{1}{\sqrt{\bar{\rho}}}\sqrt\mu g \,dv+&\frac{v-\bar{u}}{\sqrt{K\bar{\rho}\bar{\theta}}}\cdot\int_{\R^3}\frac{v-\bar{u}}{\sqrt{K\bar{\rho}\bar{\theta}}}\sqrt\mu g \,dv
\notag\\&+\frac{1}{\sqrt{6\bar{\rho}}}(\frac{|v-\bar{u}|^{2}}{K\bar{\theta}}-3)\int_{\R^3}\frac{1}{\sqrt{6\bar{\rho}}}(\frac{|v-\bar{u}|^{2}}{K\bar{\theta}}-3)\}\sqrt\mu g \,dv\}\overline{M}^{1/2}.
\end{align*}
Hence we can recover the macroscopic dissipation of $f$ by $g$.
\subsection{Main theorems}
We define the instant energy functional
\begin{align}\label{DefE}
	\mathcal{E}(t):=&\sum_{\alpha\leq 1}\{\|\partial^{\alpha}_x(\widetilde{\rho},\widetilde{u},\widetilde{\theta},\widetilde{\phi})(t)\|^{2}+\eps\|\pa^\al_x \pa_x\widetilde{\phi}\|^2
	+\|\partial^{\alpha}_xg(t)\|_{w}^{2}+\|\partial^{\alpha}_xf(t)\|_W^{2}\}
	\nonumber\\
	&\hspace{0.5cm}+\eps^2 \de^2\{\|\pa^2_x(\widetilde{\rho},\widetilde{u},\widetilde{\theta},\widetilde{\phi})(t)\|^{2}+\eps\|\pa^3_x\widetilde{\phi}\|^2
	+\|\pa^2_xg(t)\|_{w}^{2}+\|\pa^2_xf(t)\|_W^{2}\}\notag\\
	&\hspace{0.5cm}+\sum_{\alpha+|\beta|\leq 2,|\beta|\geq1}\{\|\partial^{\alpha}_{\beta}g(t)\|_{w}^{2}+\|\partial^{\alpha}_{\beta}f(t)\|_W^{2}\},
\end{align}
and the dissipation rate functional
\begin{align}\label{DefD}
	\mathcal{D}(t)&:=\eps\de\sum_{1\leq\alpha\leq 2}\{
	\|\partial^{\alpha}_x(\widetilde{\rho},\widetilde{u},\widetilde{\theta},\widetilde{\phi})(t)\|^{2}+\eps\|\pa^\al_x \pa_x\widetilde{\phi}\|^2\}
	+\frac{1}{\eps\de}\sum_{\alpha+|\beta|\leq 2,|\beta|\geq1}\{\|\partial^{\alpha}_{\beta}g(t)\|_{\sigma,w}^{2}+\|\partial^{\alpha}_{\beta}f(t)\|_{\sigma,W}^{2}\}
	\nonumber\\
	&\quad+\frac{1}{\eps\de}\sum_{\al\leq 1}\{\|\partial^{\alpha}_xg(t)\|_{\sigma,w}^{2}+\|\partial^{\alpha}_xf(t)\|_{\sigma,W}^{2}\}+\eps\de\sum_{\al=2}\{\|\partial^{\alpha}_xg(t)\|_{\sigma,w}^{2}+\|\partial^{\alpha}_xf(t)\|_{\sigma,W}^{2}\}.
\end{align}
We require the solution of KdV equation to be bounded by some $\eta\geq 0$ in the sense that 
\begin{align}\label{2.61AA}
	&\|(\rho_i,u_{1i},\theta_i,\phi_i)\|_{L^{\infty}}\leq C,\quad \|\partial_x(\rho_i,u_{1i},\theta_i,\phi_i)\|_{H^k}
	+\|\partial_t(\rho_i,u_{1i},\theta_i,\phi_i)\|_{H^k}\leq \eta,\quad i=0,1,2,3,
\end{align}
where $k\geq 4$, $\rho_0=1$ and $\theta_0=3/2.$ Note that when the solution to KdV equation is constant state, then $\rho_1,u_{10},\theta_1,\phi_0$ are constants and we let all other $\rho_i,u_{1i},\theta_i,\phi_i$ and remainders $R_i$ vanish, which implies $\eta=0$. In the paper, to give a more special structure for constant KdV solution, we also track the dependence in $\eta$ by some constant $C(\eta)\geq 0$ with $C(0)=0$ corresponding to the case when $\eta=0$. 

By using \eqref{2.61AA} and the definition of $(R_1,R_2,R_3,R_4)$ in \eqref{perturbeq}, one has
\begin{align}\label{boundkdv}
\|(R_1,R_2,R_3,R_4)\|^2_{H^4}+\|\partial_t(R_1,R_2,R_3,R_4)\|^2_{H^4}\leq C(\eta).
\end{align}

With the above preparations, we can state the main results of this paper as follows.

\begin{theorem}\label{TheoremKdV}
	Let $T>0$ be given and \begin{align*}
		\left\{
		\begin{array}{rl}		\bar{\rho}&=1+\eps\rho_1+\eps^2\rho_2+\eps^3\rho_3,\\
			\bar{u}_1&=u_{10}+\eps u_{11}+\eps^2u_{12}+\eps^3u_{13},\\
			\bar{u}_2&=\bar{u}_3=0,\\
			\bar{\theta}&=\frac{3}{2}+\eps\theta_1+\eps^2\theta_2+\eps^3\theta_3,\\
			\bar{\phi}&=\phi_0+\eps\phi_1+\eps^2\phi_2+\eps^3\phi_3,\end{array} \right.
	\end{align*}
	where $\phi_0=\rho_1=u_{10}=\theta_1\in L^{\infty}(-T,T;H^{s}(\mathbb{R}))$ is the solution to KdV equation \eqref{KdV} for large $s$ such that \eqref{boundkdv} holds for any $\eta>0$, and other  $\rho_i,u_{1i},\theta_i,\phi_i$ are corrections defined by \eqref{mass1}--\eqref{Poisson3} on $(t,x)\in [0,T]\times \R$.  
    For any $0<c_0<1/2$, there are constants  $0<c_1\ll 1$ and $0<\nu_0\ll 1$ independent of $\eps$ and $\nu$,  such that for 
	\begin{align}\label{nueps}
	\eps^{{\frac{3}{2}}-c_0}\leq \nu\leq\eps^{\frac{1}{2}+c_0}\ \text{or}\ \nu=c_1\eps^{\frac{1}{2}}\ \text{or}\ \nu=\frac{1}{c_1}\eps^{\frac{3}{2}},
	\end{align} and $\nu\leq \nu_0$, if it holds that $F_0(x,v)\geq 0$ and
	\begin{align*}
		\mathcal{E}(0)\leq \eps^4,
	\end{align*} 
	where we have chosen $q_1\in(0,1)$ small enough and some given $q_2>0$ and $l>3$ in velocty weight functions of the energy functional $\CE(t)$  defined by \eqref{DefE},
	then the Cauchy problem on the Vlasov-Poisson-Landau system \eqref{rVPL} admits a unique  solution $(F(t,x,v),\phi(t,x))$  over $[0,T]\times \R\times \R^3$ satisfying
\begin{equation*}
\sup_{0\leq t\leq T}\mathcal{E}(t)
\leq A\eps^4,
\end{equation*}
where constant $A>1$ is independent of $\eps$.	In particular, it holds that 
	\begin{equation}
    \label{2.63A}
	\sup_{0\leq t\leq T}\{\|\frac{F(t,x,v)-M_{[\bar{\rho},\bar{u},\bar{\theta}](t,x)}(v)}{\sqrt{\mu}}\|_{L^2_xL^2_v}+\|\frac{F(t,x,v)-M_{[\bar{\rho},\bar{u},\bar{\theta}](t,x)}(v)}{\sqrt{\mu}}\|_{L^\infty_xL^2_v}\}\leq C\eps^2,
	\end{equation}
	and 
	\begin{equation*}
	\sup_{0\leq t\leq T}\{\|\phi(t,x)-\phi_0(t,x)\|_{L^2_x}+\|\phi(t,x)-\phi_0(t,x)\|_{L_{x}^{\infty}}\}\leq C\eps^2.
	\end{equation*}
\end{theorem}

Regarding the original VPL system \eqref{VPL}, we immediately have the following result.

\begin{corollary}
Let $T>0$, $(\bar{\rho},\bar{u},\bar{\theta},\bar{\phi})$ be defined in Theorem \ref{TheoremKdV} and $(\eps,\nu)$ satisfy \eqref{nueps}. There exists $\nu_0>0$ such that for well prepared initial data and $0<\nu<\nu_0$, the Cauchy problem on the Vlasov-Poisson-Landau system \eqref{VPL} admits a unique  solution $(F(t,x,v),\phi(t,x))$  over $[0,\eps^{-\frac{3}{2}}T]\times \R\times \R^3$ satisfying
\begin{align*}
	\sup_{0\leq t\leq \eps^{-\frac{3}{2}}T}\{&\eps^{\frac{7}{4}}\|\frac{F(t,x,v)-\eps^{-3}M_{[\bar{\rho},\bar{u},\bar{\theta}](\eps^{\frac{3}{2}}t,\eps^{\frac{1}{2}}(x-t))}(\eps^{-1}v)}{\sqrt{\mu}(\eps^{-1}v)}\|_{L^2_xL^2_v}\notag\\
	&+\eps^{\frac{3}{2}}\|\frac{F(t,x,v)-\eps^{-3}M_{[\bar{\rho},\bar{u},\bar{\theta}](\eps^{\frac{3}{2}}t,\eps^{\frac{1}{2}}(x-t))}(\eps^{-1}v)}{\sqrt{\mu}(\eps^{-1}v)}\|_{L^\infty_xL^2_v}\}\leq C\eps^{2},
\end{align*}
and 
\begin{align*}
	\sup_{0\leq t\leq \eps^{-\frac{3}{2}}T}\{&\eps^{-\frac{3}{4}}\|\phi(t,x)-\eps\bar{\phi}(\eps^{\frac{3}{2}}t,\eps^{\frac{1}{2}}(x-t))\|_{L^2_x}\notag\\
	&+\eps^{-1}\|\phi(t,x)-\eps\bar{\phi}(\eps^{\frac{3}{2}}t,\eps^{\frac{1}{2}}(x-t))\|_{L_{x}^{\infty}}\}\leq C\eps^{2}.
\end{align*}
Moreover, for the macroscopic component $(\rho,u,\theta)(t,x)$ of $F(t,x,v)$, it holds that
\begin{align*}
\sup_{0\leq t\leq \eps^{-\frac{3}{2}}T}\{&\|(\rho,\eps^{-1}u,\eps^{-2}\theta)(t,x)-(\bar{\rho},\bar{u},\bar{\theta})(\eps^{\frac{3}{2}}t,\eps^{\frac{1}{2}}(x-t))\|_{L^2_x}\\
&+\eps^{-\frac{1}{4}}\|(\rho,\eps^{-1}u,\eps^{-2}\theta)(t,x)-(\bar{\rho},\bar{u},\bar{\theta})(\eps^{\frac{3}{2}}t,\eps^{\frac{1}{2}}(x-t))\|_{L^\infty_x}\}\leq C\eps^{\frac{7}{4}}.
\end{align*}
\end{corollary}

Our estimates hold when the KdV solution becomes a constant state. In this case, $C(\eta)=0$  with $\eta=0$, which leads to a global solution to VPL system. However, we need to modify our energy and dissipation functionals as in \eqref{DefEG} and \eqref{DefDG} given in the proof later.

\begin{theorem}\label{TheoremGlobal}
	Let $(\bar{\rho},\bar{u},\bar{\theta},\bar{\phi})=(1,0,\frac{3}{2},0)$. For any $0<c_0<1/2$, there are small constants $c_1>0$ and $\nu_0>0$ independent of $\eps$ and $\nu$ such that for
	\begin{align}\label{Gnueps}
		\eps^{3-c_0}\leq \nu\leq\eps^{\frac{5}{2}+c_0}\ \text{or}\ \nu=c_1\eps^{\frac{5}{2}}\ \text{or}\ \nu=\frac{1}{c_1}\eps^{3},
	\end{align} and $\nu\leq \nu_0$, if it holds that $F_0(x,v)\geq 0$ and
	\begin{align*}
		\mathcal{E}(0)\leq \frac{1}{C_1}\eps^{4+\frac{3}{4}},
	\end{align*} 
where the constant $C_1>1$ is independent of $\eps$,  then the Cauchy problem on the Vlasov-Poisson-Landau system \eqref{rVPL} admits a unique global solution $(F(t,x,v),\phi(t,x))$ over $[0,\infty]\times \R\times \R^3$ satisfying
\begin{equation}\label{globalenergy}
\mathcal{E}(t)+\int^{t}_{0}\mathcal{D}(s) \,ds
\leq 2\eps^{4+\frac{3}{4}},\quad 0\leq t<\infty.
\end{equation}
Here $\mathcal{E}(t)$ and $\mathcal{D}(t)$ are defined in \eqref{DefEG} and \eqref{DefDG} respectively.
\end{theorem}

We also have the following corollary for the original VPL system \eqref{VPL}.
\begin{corollary}
Let $\eps$ and $\nu$ satisfy \eqref{Gnueps}. There exists $\nu_0>0$ such that for well prepared initial data and $0<\nu<\nu_0$, the Cauchy problem on the Vlasov-Poisson-Landau system \eqref{VPL} admits a unique global solution $(F(t,x,v),\phi(t,x))$  over $[0,\infty)\times \R\times \R^3$ satisfying
	\begin{align*}
		\sup_{0\leq t<\infty}\{&\eps^{\frac{7}{4}}\|\frac{F(t,x,v)-\eps^{-3}M_{[1,0,3/2]}(\eps^{-1}v)}{\sqrt{\mu}(\eps^{-1}v)}\|_{L^2_xL^2_v}\notag\\
		&+\eps^{\frac{3}{2}}\|\frac{F(t,x,v)-\eps^{-3}M_{[1,0,3/2]}(\eps^{-1}v)}{\sqrt{\mu}(\eps^{-1}v)}\|_{L^\infty_xL^2_v}\}\leq C\eps^{\frac{19}{8}},
	\end{align*}
	and 
	\begin{align*}
		\sup_{0\leq t<\infty}\{&\eps^{-\frac{3}{4}}\|\phi(t,x)\|_{L^2_x}+\eps^{-1}\|\phi(t,x)\|_{L_{x}^{\infty}}\}\leq C\eps^{\frac{19}{8}}.
	\end{align*}
Moreover, for the macroscopic component $(\rho,u,\theta)(t,x)$ of $F(t,x,v)$, it holds that
\begin{align*}
\sup_{0\leq t<\infty}\{\|(\rho,\eps^{-1}u,\eps^{-2}\theta)(t,x)-(1,0,\frac{3}{2})\|_{L^2_x}+\eps^{-\frac{1}{4}}\|(\rho,\eps^{-1}u,\eps^{-2}\theta)(t,x)-(1,0,\frac{3}{2})\|_{L^\infty_x}\}\leq C\eps^{\frac{17}{8}}.
\end{align*}
\end{corollary}

\begin{remark}
	Our framework is supposed to be applicable for proving similar results for the general VPL system of ions where the electronic field satisfies
	$
	-\partial^{2}_{x}\phi=\rho-e^{\phi},
	$
	or, after scaling, $
	-\eps^2\partial^{2}_{x}\phi=\rho-e^{\eps\phi}.
	$ The approach for handling this general relation strongly depends on the current case for the linearized Boltzmann relation, and the way for bounding remainders beyond the linear part is also standard but requires more calculations, see the study on VPL and Euler-Poisson systems for ions \cite{Bastdos-Golse2018,Duan-Yang-Yu-VPL,Guo-Pausader}. To catch the key points of the proof, we only consider the  linearized Boltzmann relation in this paper for brevity.
\end{remark}

\begin{remark}\label{rk.beta}
The restriction condition \eqref{nueps} or \eqref{Gnueps} is employed to ensure that $\nu\sim \eps^\beta$ with $0<\beta< 7/2$, maintaining the dominance of collisions in the rescaled VPL system \eqref{r1VPL}. In the significantly weaker collision regime where $\beta>7/2$, the Landau damping effect may prevail, making the justification of KdV from VPL in this regime more subtle and challenging. This will be addressed in our future work.
\end{remark}

\subsection{Ideas of the proof}
The KdV limit for VPL system is determined by the Euler-Poisson structure, which is observed for either collisionless case (Vlasov-Poisson) in \cite{Han-Kwan}, or the collision dominated cases (VPL or Euler-Poisson) in \cite{DYY,Guo-Pu}. Due to the complexity of Landau collisional operator, in comparison with Vlasov-Poisson system, we choose the macro-micro decomposition around local Maxwellian to capture the Euler-Poisson dynamics, as well as the microscopic part dominated by the linearized Landau operator. The lower order macroscopic energy can be estimated by the Navier-Stokes-Poisson type structure with small viscosity in terms of $(\eps,\de)$. On the other hand, we observe from \eqref{reVPL} that there are terms with singularities in small parameter, such as $\eps^{-1}\partial_{x}\phi\partial_{v_{1}}F,\eps^{-1}\de^{-1}Q(F,F)$, or the extra term $\eps^{-1}\partial_{x}F$ compared to VPL system. Moreover, the amplitude for our bulk velocity $u$ and electronic field $\phi$ could be large, since they are very close to the KdV solution and we have dropped any smallness assumption on the KdV solutions, which makes the local Maxwellian $M$ far away from any global Maxwellian $\mu$ in $L^\infty$ sense. In what follows we explain the important ideas in this paper on how to solve those troubles.

\subsubsection{Decomposition on microscopic quantity for large amplitude reference state}	
In the sense of taking the limit, the macroscopic part $M$ of our solution $F$ is close to $(2\pi)^{-3/2}\exp\{-|v-u_0|^2/2\}$, where $u_0=(u_{10},0,0)$ with $u_{10}$ being the bounded solution to the KdV equation with possibly large amplitude. If we still use the global Maxwellian $\mu$ as the reference state for linearization like the previous work on the hydrodynamic limit problem for kinetic equation in \cite{Duan-Yang-Yu,Guo-Jang-Jiang-2010,Liu-Yang-Yu} and so on, we will find that in our estimates for the microscopic quantity $h$, there should occur some terms in the form of $(\eps\de)^{-1}(\Ga(M-\mu,\pa^\al_xh),w^2\pa^\al_xh)$, which is difficult to control 
using
$$(\eps\de)^{-1}|(\Ga(M-\mu,\pa^\al_xh),w^2\pa^\al_xh)|
\lesssim \|M-\mu\|_{L^\infty_xL^2_v}(\eps\de)^{-1}\|w\pa^\al_xh\|_{\si}^2\lesssim \|M-\mu\|_{L^\infty_xL^2_v}\CD(t)$$ with the dissipation norm $\CD(t)$ 
since 
$$
\|M-\mu\|_{L^\infty_xL^2_v}\lesssim \|(\bar{\rho}-1,\bar{u},\bar{\theta}-\frac{3}{2})\|_{L^\infty}+\cdots
$$
could be very large due to our assumption that the KdV solution $\bar{u}$ is only bounded but not necessarily small. To overcome this difficulty, we choose the local Maxwellian $\overline{M}$ to be the reference state for linearization where it is close enough to $M$ with an $O(\eps)$ rate, which controls the nonlinear terms mentioned above by $$
(\eps\de)^{-1}|(\Ga(M-\overline{M},\pa^\al_xh),w^2\pa^\al_xh)|
\leq C\|(\rho-\bar{\rho},u-\bar{u},\theta-\bar{\theta})\|_{L^\infty}(\eps\de)^{-1}\|w\pa^\al_xh\|_{\si}^2
\lesssim O(\eps)\CD(t).
$$ 
However, in this setting, to obtain the hypocoercivity for the symmetric linear operator, it is expected that our solution has the form $F=M+\sqrt{\overline{M}}f+\cdots,$ which lead to another well-known difficulty that in derivation for the equation of $f$, when the transport operator $\pa_t+v_1\pa_x$ acts on $\sqrt{\overline{M}}$, it produces the growth of $v_1|v|^2$. In order to cancel this growth in $v$, we make use of the Caflisch's decomposition on the microscopic part such that the solution becomes $F=M+\sqrt{\overline{M}}f+\sqrt{\mu}g+\cdots,$ and we need to construct both equations \eqref{eqg} and \eqref{eqf} for $f$ and $g$. The construction should have two main purposes. The first one is to absorb the polynomial growth in $v$, and the second is to obtain the dissipation for both $f$ and $g$. The first goal is achieved by choosing the global Maxwellian $\mu$ to be much smaller than $\overline{M}$ such that when we put the growth of $v$ with the form $(\pa_t+v_1\pa_x)\sqrt{\overline{M}}f$ in the equation of $g$ in \eqref{eqg}, it can be controlled by 
$$
|(\frac{(\pa_t+v_1\pa_x)\sqrt{\overline{M}}}{\sqrt{\mu}}f,g)|\lesssim O(\eps)\CD(t).
$$
For the second goal, it is natural to see the linear part of $f$ should be determined by the symmetric, positive linearized operator $-\CL_{\overline{M}}f:=-\Gamma_{\overline{M}}(\sqrt{\overline{M}},f)-\Gamma_{\overline{M}}(f,\sqrt{\overline{M}})$ which gives the coercivity. But the linearized operator for $g$ would be $-\frac{1}{\sqrt\mu}(Q(M,\sqrt{\mu}g)+Q(\sqrt{\mu}g,M))$, which is neither positive nor symmetric. To obtain the dissipation for $g$, we decompose it into two parts
$$
	\frac{1}{\sqrt\mu}(Q(M,\sqrt{\mu}g)+Q(\sqrt{\mu}g,M))=L_Dg+L_Bg,
$$
where $L_D$ is dissipative and $L_B$ preserves only bounded velocity part. Then we let the linear behavior of $g$ be described by $L_D$ and move $L_B g$ to the equation of $f$ such that this part can be bounded by 
$$
\frac{1}{\eps\de}|(\frac{\sqrt{\mu}}{\sqrt{\overline{M}}}L_Bg, f)|\leq \ka\frac{1}{\eps\de}\|f\|_{\sigma}^2+C_\ka\frac{1}{\eps\de}\|g\|_{\sigma}^2.
$$
The first term, which is contained in the dissipation functional, can be chosen small and the second one is the dissipation of $g$ which can be absorbed by taking linear combination with the estimate for $g$. Such decomposition method to cancel the increase is inspired by Caflisch \cite{Caflisch} in the study of Euler limit for Boltzmann equation without ensuring the positivity of the solution, which we fix the scheme to guarantee. See also \cite{DL-arma,DL-cmaa,DL-hard} for kinetic shear flow where they use the similar thoughts to control the growth of velocity. In comparison, our structure of the decomposed microscopic equations for the rescaled Vlasov-Poisson-Landau system produces more trouble terms such as the singular in $\eps$ electronic term $\eps^{-1}\partial_x\phi\partial_{v_{1}}(\sqrt{\mu}g+\sqrt{\overline{M}}f)$ which also gives extra polynomial in $v$ and the term that increases the order of $x$ derivative $P_{0}(v_{1}\partial_x[\sqrt{\overline{M}}f]+v_{1}\sqrt{\mu}\partial_xg)$. It is also tricky to determine the position of those terms in our decomposition and control them. More importantly, for the equations of $(g,f)$, it is difficult to obtain the highest order estimate due to the aforementioned term with extra $\pa_x$ that always gives some $\al+1$ order derivative term when we apply $\pa^\al_x$ to the equation and try to get the $\al$ order derivative estimate. Thus, we need to develop other strategies to solve these problems.

\subsubsection{Construction for the weight functions}	
The choice of weight functions to bound the microscopic quantities is subtle. In view of our decomposition for the microscopic part $G=\overline{G}+\sqrt{\overline{M}}f+\sqrt{\mu}g$, the electronic term in the equation \eqref{eqg} will produce a term
$$
\frac{1}{\eps}\sqrt{\mu}^{-1}\partial_x\phi\partial_{v_{1}}(\sqrt{\mu}g+\sqrt{\overline{M}}f)
$$
which causes the first-order polynomial increase in $v_1$ behaving like $\frac{1}{\eps}\partial_x\phi v_1 g$, and a term $\frac{1}{\eps}\partial_x\phi\pa_{v_1} g$ with the effect of $\pa_{v_1}$. Motivated by \cite{Duan-Yang-Zhao-M3,Duan-Yang-Zhao}, we choose the time-velocity weight 
$$
w(\alpha,\beta)(t,v)=\langle v\rangle^{2(l-\al-|\beta|)}\exp\left(\frac{q_1}{(1+t)^{q_2}}\frac{\langle v\rangle^2}{2}\right)
$$
to control such two terms since the transport operator acting on this weight gives the dissipation 
$$
-\partial_{t}w^{2} =q_{1}q_{2}(1+t)^{-(1+q_2)}\langle v\rangle^2 w^{2},
$$
such that it absorbs the large velocity growth for finite time. However, we will find this weight is inconvenient to control the equation for $f$. As mentioned above, the linear behaviors of $g$ and $f$ are largely determined by $L_D$ and $\CL_{\overline{M}}$ where the weighted coercivities are obtained by 
$$
-( L_Dg,w^2(0,0) g)
\geq c\| g\|^2_{\sigma,w} \quad  \text{and}\ -(\CL_{\overline{M}} f,w^2(0,0) f)
\geq c\| f\|^2_{\sigma,w}-C\|f\|^2_{\sigma}.
$$
Therefore, for $f$, we not only need to make estimates with the weight, but also should obtain the one without the weight. If we still use $w$ in the equation of $f$, then the estimates with and without $w$, the choice of three different parameters $q_1,q_2,l$  in $w$ than the one for $g$, and the high order derivatives acting on $\sqrt{\overline{M}}$ will make it very complicated to calculate. We simplify the case by putting the term $\frac{1}{\eps}\partial_x\phi\partial_{v_{1}}(\sqrt{\overline{M}}f)/\sqrt{\mu}$ on the equation of $g$ such that it becomes $\frac{1}{\eps}\mu^{-1/2}\partial_x\phi\partial_{v_{1}}(\sqrt{\mu}g+\sqrt{\overline{M}}f)$ which avoids the increase of $v$ and is controlled by $O(\eps)\CD(t).$ Without such term, the weighted estimate for equation of $f$ is very similar to Landau equation and we can bound it using the pure velocity weight $\langle v\rangle^{2(l-\al-|\beta|)}$ as in \cite{Guo-2002}. We should also mention that the case for global solution is different since the dissipation provided by $w$ tends to zero when time goes to infinity. Then we construct the time, velocity and parameters dependent weight 
$$
\textbf{w}(\alpha,\beta)(t,v)=\dis \langle v\rangle^{2(l-\al-|\beta|)}\exp\left((q_1-q_2\int^t_0 q_3(s)\,ds)\langle v\rangle\right),
$$
which provides the dissipation 
$$
-\partial_{t}\textbf{w}^{2} =2q_{2}q_3(t)\langle v\rangle \textbf{w}^{2}
$$
with $q_3$ being dependent of some norms of $\widetilde{\phi}$ and two small parameters $\eps$ and $\de$. This weight is useful for solving the term $\frac{1}{\eps}\partial_x\phi\pa_{v_1} g$ with extra $\pa_{v_1}$. Notice that there is the other trouble term $\frac{1}{\eps}\partial_x\phi v_1 g$ with the additional $v_1$. A classic approach by Guo in \cite{Guo-JAMS} makes use of another weight $e^{\phi}$ and the cancellation 
$$
(v_1\partial_\be^{\alpha} \pa_x g,\textbf{w}^2e^{\phi}\partial_\be^{\alpha} g)+(\frac{v_1}{2} \partial_x\phi\pa^\al_\be g,\textbf{w}^2e^{\phi}\partial^\al_\be g)=0.
$$
In our setting, we notice that the coefficient of the transport term $v_{1}\partial_xg$ and the electronic term $\eps^{-1}v_{1}\partial_x\phi g$ is different where there is a significant singularity $\eps^{-1}$ in the electronic term. To carry out the cancellation, we modify the exponential in $\phi$ weight to be $e^{\phi/\eps}$. Due to the singularity in the power of this weight, our construct of the time, velocity and parameters dependent $\textbf{w}$ need to satisfy a couple of requirement. The first one is to provide dissipation that is strong enough to absorb $\eps^{-1}\partial_x\phi\pa_{v_1} g$ and all singularities which come from the power of the weight $e^{\phi/\eps}$ when $\pa_t$ and $\pa_x$ act on it. The second one is to guarantee the boundedness $\int^\infty_0 q_3(s)\,ds<\infty$ and the positivity $q_1-q_2\int^\infty_0 q_3(s)\,ds>0$ uniformly on both small parameters $\eps$ and $\de$. These requirements highly rely on how we design the time function $q_3(t)$ and its dependency on $\eps$ and $\de$. Also we need to make appropriate a priori assumption and keep the structure of our equations, energy functional and most of the estimates being very similar to the ones for KdV limit. See \eqref{defw}, \eqref{Gapriori} and Lemma \ref{Gleg} in Section \ref{SecGlobal}.

\subsubsection{Singularities in $\eps$ and $\de$}
We observe that there are many terms including both fluid and non-fluid quantities that contain singularities in $\eps$ and $\de$, such as $\eps^{-1}\partial_x\widetilde{\phi}$ in the equation of $u_1$ in \eqref{perturbeq}, $\eps^{-1}\partial_x(\widetilde{\rho},\widetilde{u},\widetilde{\theta})$ on the left hand side of \eqref{perturbeq} and $\eps^{-1}\partial_x\phi\partial_{v_{1}}(\sqrt{\mu}g+\sqrt{\overline{M}}f)\sqrt{\mu}^{-1}$ in the microscopic equation \eqref{eqg}. Note that our KdV approximation appears on the zero order correction and it is only bounded but not necessarily small, which makes it very difficult to cancel by the background state in our weak collision setting compared to the fluid dominated case in \cite{DYY} where their $\nu$ tends to infinity and the approximated solution naturally provides $O(\eps)$ by $(\bar\rho,\bar u,\bar\theta,\bar\phi)=(1+O(\eps),O(\eps),3/2+O(\eps),O(\eps))$. We should use various of approaches to deal with these singularities which cause troubles mainly on linear terms. The first important observation is that we can cancel the singularity induced by the shift $\bar{x}=x-t/\eps$ in \eqref{shift} by tracing back along this direction. For example, when we try to obtain the first order estimate on $\pa_x\widetilde{u}$ by applying $\pa_x$ to the second equation in \eqref{perturbeq}, the $\eps\de\pa^2_x\pa_tG$ in $\pa^2_x\Theta$ defined in \eqref{DefTheta} on the right hand side of \eqref{perturbeq} gives a term which behaves as
$$
\eps\de|(\pa^3_x
h,\partial_t\widetilde{u}_1)|=\eps\de|(\pa^2_x
h,\pa_x\partial_t\widetilde{u}_1)|\lesssim (\eps\de)^{1/2} \|\pa_x\partial_t\widetilde{u}_1\|\sqrt{\CD(t)}.
$$
Then for the bound of $(\eps\de)^{1/2}\|\pa_x\partial_t\widetilde{u}_1\|$, because of the $\frac{1}{\eps}\partial_x\widetilde{u}_1$ in the equation of $\widetilde{u}_1$, one can only get 
$$(\eps\de)^{1/2}\|\pa_x\partial_t\widetilde{u}_1\|\lesssim \frac{1}{\eps}(\eps\de)^{1/2}\|\pa^2_x \widetilde{u}_1\|+\cdots\lesssim \frac{1}{\eps}\sqrt{\CD(t)}+\cdots,
$$
where the $O(\eps^{-1})$ remains unbounded. On the other hand, we notice that there is the other term $-\eps\de\pa^3_xG$ in $\pa^2_x\Theta$ that has the same issue which produces the trouble term 
$
-\de(\pa^3_x
h,\partial_x\widetilde{u}_1)=\de (\pa^2_x
h,\pa^2_x\widetilde{u}_1).
$ The combination of these two terms gives 
$$
-\eps\de(\pa^2_x
h,\pa_x\pa_t\widetilde{u}_1)+\de(\pa^2_x
h,\pa^2_x\widetilde{u}_1)=-\eps\de(\pa^2_x
h,\pa_x(\pa_t-\frac{1}{\eps}\pa_x)\widetilde{u}_1)
$$
and $\pa_x(\pa_t-\frac{1}{\eps}\pa_x)\widetilde{u}_1$ has much better properties since the $O(\eps^{-1})$ is transferred to $\frac{1}{\eps}\pa^2_x\widetilde\phi$ which has less singularity using our energy estimates on electronic field.
The second ingredient is on the structure of scaled Euler-Poisson type equation in the last line of \eqref{perturbeq}. The perturbed density quantity $\widetilde{\rho}$ is coupled with the perturbed electronic field which induces $O(\eps)$ and $O(\eps^2)$ terms that can be used to balance $O(\eps^{-1})$ or $O(\eps^{-2})$. See \eqref{pa2phipau3} and \eqref{pa2phipau41}. For the $O(\eps^{-1})$ in the microscopic equation, for instance, $|\frac{1}{\eps}(\partial_x\phi\partial_{v_1} g,w^2 g)|$ from the equation of $g$ in \eqref{eqg}, we can use the interpolation method among the microscopic dissipation, the small parameters and the weighted energy, and the choice of relation between $\eps$ and $\de$ to cancel it. Considering the above term, as seen in the later proof \eqref{I31}, the interpolation and the definition of energy functional will lead to 
$$
|\frac{1}{\eps}(\partial_x\phi\partial_{v_1} g,w^2 g)|\lesssim 
\ka\frac{1}{\eps\delta}\|g\|_{\sigma,w}^2 	+C_\ka(C(\eta)+\eps)\frac{\de^{1/2}}{\eps}\|\langle v\rangle g\|_w^2,
$$
where $\ka$ is small and $\|\langle v\rangle g\|_w$ can be well controlled by the additional dissipation induced by weight function. Then we can choose $\de$ small enough to balance with $\eps$, while keeping the Knudsen number $1/\nu$ large.

\subsubsection{Estimates on highest order}
In both macroscopic and microscopic system, there are additional terms that increase the order of $x$ derivatives, see $\frac{1}{\rho}\partial_x(\int_{\mathbb{R}^{3}} v^{2}_{1}L^{-1}_{M}\Theta\,dv)$ on the right hand side of \eqref{perturbeq} and $P_{0}(v_{1}\partial_xh)$ in \eqref{eqh}. When we take inner product to obtain the energy estimates, those terms remain there. Thus, we need special consideration on the second order energy estimates. A series of recent studies on the hydrodynamic limit from kinetic equations (Landau equation, Vlasov-Poisson-Landau/Boltzmann systems) to compressible fluid equations (compressible Euler equation, Euler-Poisson system) in \cite{Duan-Yang-Yu,Duan-Yang-Yu-VPL} indicates that it is possible to use the original equation for the highest order. In the current paper, if we directly turn to the equation for $F=M+\overline{G}+\sqrt{\mu}g+\sqrt{\overline{M}}f,$ in order to give the dissipation for $f$ from $\CL_{\overline{M}}f$, we should apply $\pa^2_x$ and take inner product with $\pa^2_xF/\overline{M}$, which induces a term $\|\mu g/\overline{M}\|$ with exponential increase in $v$ since $\mu$ has much smaller decay in $v$ than $\overline{M}$. We resolve the equation for $F$ into two parts. In the design of the equation for $g$, we try to avoid any term with extra $\pa_x$ and put those terms in the equation of $f$. Then we may get estimates for any order of $g$, from zero to the highest order, using its equation. For the fluid quantities and $f$, we take difference of equations of $F$ and $\sqrt{\mu}g$ to get the other equation for $F-\sqrt{\mu}g\sim M+\overline{G}+\sqrt{\overline{M}}f$, then apply $\pa^2_x$ and take inner product with $\pa^2_x(M+\overline{G}+\sqrt{\overline{M}}f)/\overline{M}$ to estimate the second order energy for $(\widetilde{\rho},\widetilde{u},\widetilde{\theta},\widetilde{\phi},f).$ Notice that if there are some nonlinear terms such as $Q(\cdot,\cdot)$ in the equation of $F-\sqrt{\mu}g$, then the above process require us to estimate many crossing terms involving the nonlinear operator from $(\pa^2_xQ(\cdot,\cdot),\pa^2_x(M+\overline{G}+\sqrt{\overline{M}}f)/\overline{M})$ and we also need to track how the second order derivative in $x$ acts on $M,\sqrt{\overline{M}}$ and the elements in the nonlinear operator. To lighten our burden on the tedious calculations, in our construction of equations for $f,g$, we put almost all nonlinear terms to the equation of $g$. Hence, we only need to bound much simpler terms in the form of $(\pa^2_xQ(\cdot,\cdot),\pa^2_x g).$

Moreover, the energy setting for second order is different. When we carry out the aforementioned approach to analyze highest order energy, it is necessary to control
 $$(\eps\de)^{-1}|\CL_{\overline{M}}\pa^2f,\pa^2_x M|\sim (\eps\de)^{-1}\|\pa^2_x f\|_\si \|\pa^2_x(\widetilde{\rho},\widetilde{u},\widetilde{\theta})\|,$$
 where the fluid dissipation $\eps\de\|\pa^2_x(\widetilde{\rho},\widetilde{u},\widetilde{\theta})\|^2$ is obtained from the Navier-Stokes-Poisson structure in the macroscopic system. Hence, we see if one bounds the above term by Cauchy-Schwarz inequality, there will be extra $O(\eps^{-2}\de^{-2})$ singularity as coefficient for the microscopic dissipation $(\eps\de)^{-1}\|\pa^2_x f\|^2_\si.$ We enlarge the functional space by setting the highest order energy to be $\eps^2\de^2\|(\widetilde{\rho},\widetilde{u},\widetilde{\theta})\|^2+\eps^2\de^2\times\it{Microscopic\ quantities}$, where the corresponding microscopic dissipation becomes $\eps^2\de^2\times(\eps\de)^{-1}\|\pa^2_x f\|^2_\si=\eps\de\|\pa^2_x f\|^2_\si$, such that the coefficient $\eps^2\de^2$ balances the singularity by
$$
\eps^2\de^2\times(\eps\de)^{-1}|(\CL_{\overline{M}}\pa^2f,\pa^2_x M)|\lesssim \eps\de\|\pa^2_x f\|^2_\si+ \eps\de\|\pa^2_x(\widetilde{\rho},\widetilde{u},\widetilde{\theta})\|^2.
$$

\subsubsection{Other difficulties}
The enlargement of highest order energy functional space could cause singular terms in $\eps$ and $\de$ in the lower order analysis. Especially the $L^2$ bound of the nonlinear fluid quantities forces us to use Sobolev embedding such that part of the $L^\infty$ bound for lower order terms leads to $L^2$ norm for highest order, for instance, $
|(\partial_x\widetilde{u}_1\pa_x\widetilde{\rho},\pa_x\widetilde{\rho})|\lesssim \|\partial_x\widetilde{u}_1\|^{-1/2}\|\partial^2_x\widetilde{u}_1\|^{-1/2}\|\pa_x\widetilde{\rho}\|^2\lesssim (\eps\de)^{-1/2}\CE(t)^{3/2}.$ Then it is crucial to make an appropriate a priori assumption by careful choice on $r$ for $\CE(t)\leq \eps^r.$ Our collision is weak described by $0<\nu\ll 1$ implying that $\de=\eps^{7/2}/\nu$ should not be too small, thus we need to determine some $r$ which is not that small. In addition, we expect that $r$ is not too large such that the solution space is small. Also, large $r$ require us to expand the approximated solution to high order corrections on $\eps$, which makes the approximated equation too complicated. We derive the expansion for the approximate equation to $O(\eps^2)$ and see that $r=4$ is a suitable choice.

One of the key points when using the macro-micro decomposition around local Maxwellians is that the non-fluid quantity $h=G-\overline{G}$ is purely microscopic, then up to $(\eps\de)^{-1}$ singularity, one can get the clean coercivity $-(\eps\de)^{-1}(\CL h,h)\gtrsim (\eps\de)^{-1}\|h\|^2_\si,$ without the appearance of similar terms to singular macroscopic $\|(\eps\de)^{-1}\mathbf{P}_0 h\|^2_\si$ when one applies Guo's approach \cite{Guo-2002} by setting $F=\mu+\sqrt{\mu}h$. The singular $O(\ep^{-1}\de^{-1})$ macroscopic term shown above is difficult to be absorbed by macroscopic dissipation which is only $O(\eps\de)$, such as $\eps\de\|\pa^\al_x(\widetilde{\rho},\widetilde{u},\widetilde{\theta})\|^2\sim\eps\de\|\pa^\al_xP_0 F\|^2$ by the Navier-Stokes-Poisson structure. However, if we decompose $h$ into $h=\sqrt{\overline{M}}f+\sqrt{\mu}g$, neither $f$ nor $g$ can be proved to be purely microscopic and there would be the appearance of singular $O(\ep^{-1}\de^{-1})$ macroscopic term. An important observation is that $g$ is dominated by $L_D$ which gives $-(\eps\de)^{-1}(L_Dg,g)\gtrsim \|(\eps\de)^{-1}g\|^2_\si$ without assuming $g$ to be microscopic, while $h$ is microscopic, which leads to $P_0(\sqrt{\overline{M}}f)=-P_0(\sqrt{\mu}g)$. Then we can recover the macroscopic dissipation of $f$ by $g$ that avoids passing the singularity of microscopic part to fluid quantities.

When we consider the global estimates near global Maxwellian, the direct inner product method will be difficult to control terms in the form like $(u\pa_x\widetilde{\rho},\widetilde{\rho})$ by $O(\eps)\CD(t)$ due to the lack of zero order dissipation on the fluid part. Motivated by \cite{Liu-Yang-Yu}, we use the entropy and entropy flux method. For lower order, we expect to get all the estimates from the previous sections, then we track the dependence on KdV solutions in the calculations such that all trouble terms which stop us from getting the uniform global-in-time estimates vanish when our KdV solutions tend to constant state.

\subsection{Related literature}
For the derivation of KdV equation from either VPL or Euler-Poisson system and the rigorous mathematical proof, one may refer to \cite{Han-Kwan,Guo-Pu,Su-1,Washimi}. In two and three dimensional cases, the Zakharov-Kuznetsov are derived in \cite{Lannes,Pu}. There are many results of long wave limits towards KdV from other equations, we mention \cite{Chiron,BGSS1,BGSS2} for Nonlinear Schr\"odinger equations, \cite{Ben} for general hyperbolic systems, and \cite{Alvarez-Samaniego,Bona,Craig,Schneider} for the water waves.

When the collision frequency is small, we expect the VPL system to have similar properties to the collisionless case. The weak collisional limit of VPL system with uniform-in-$\nu$ Landau damping and enhancing dissipation is proved in \cite{Chaturvedi-Luk-Nguyen}, see also \cite{Bedrossian,Bedrossian-Zhao-Zi} for Vlasov-Poisson-Fokker-Planck system. On the contrary, for small Knudsen number which implies large collision frequency, the VPL system should be dominated by the compressible Euler-Poisson equation. For the studies on hydrodynamic limit of kinetic equations, one may refer to \cite{Caflisch,Duan-Yang-Yu-2,Duan-Li} for compressible Euler limit of Boltzmann or Landau equations, \cite{Duan-Yang-Yu,Duan-Yang-Yu-1,Duan-Yang-Yu-VPL} for rarefaction or contact waves for Landau equation or VPL system, \cite{Duan-Yang-Yu-VMB,Guo-Jang,Lei-Liu-Xiao-Zhao} for compressible Euler-Poisson/Maxwell limit of Vlasov-Poisson/Maxwell-Boltzmann/Landau systems. When Knudsen number is constant, the well-posedness of VPL system around equilibrium is proved in \cite{Guo-JAMS}, see also \cite{Duan-Yang-Zhao,Strain-Guo}. In addition, we mention \cite{Bobylev-Potapenko} for longwave asymptotics for the VPL system.


\section{Preliminaries}\label{secPre}
In this section, we give some useful estimates that will be used for many times in the later parts including the bound for microscopic correction $\overline{G}$, the coercivity for $\CL_{\overline{M}}$ and $L_D$, and the control for nonlinear collision operator. Before we start the major part of this section, we first make the {\bf a priori assumption}  \begin{equation}\label{apriori} 	  \sup_{0\leq t\leq T}\mathcal{E}(t)\leq A\eps^4,  \end{equation} where $A>1$ will be chosen later. Since $A$ is independent of $\eps$ and $\de$, we let $\eps$ and $\de$ to be small such that  
\begin{align}  
\label{smalleps} 	  
A\eps^{\frac{1}{2}}+A\de^{\frac{1}{2}}<1.
\end{align}
Now we give some bounded estimate for microscopic correction $\overline{G}$.
\begin{lemma}\label{lem3.1}
Let $\overline{G}$ and $w=w(\alpha,\beta)$ be defined in \eqref{DefbarG} and \eqref{Defw} respectively with small constant $q_1\in(0,1]$ and $q_2>0$. 		
	Under the a priori assumption \eqref{apriori}, for any $k\geq 0$, $1\leq\alpha\leq 2$ and $\be\geq0$, one has
	\begin{equation}\label{boundbarG0}
		|\langle v\rangle^{k}e^{\frac{|v|^2}{8}}\pa^\be_v(\frac{\overline{G}}{\sqrt{\mu}})|_{w}
		\leq C\eps\de(|\pa_x\bar{u}|+|\pa_x\bar{\theta}|),
	\end{equation}
and
	\begin{align}\label{boundbarG}
		&|\langle v\rangle^{k}e^{\frac{|v|^2}{8}}\partial^{\alpha}_{\beta}(\frac{\overline{G}}{\sqrt{\mu}})|_w+
		|\langle v\rangle^{k}e^{\frac{|v|^2}{8}}\partial^{\alpha}_{\beta}(\frac{\overline{G}}{\sqrt{\mu}})|_{\sigma,w}
		\leq C\eps\de\sum^{\al}_{i=1}(|\pa^i_x\bar{u}|+|\pa^i_x\bar{\theta}|)
		(|\partial^{\alpha+1-i}_xu|+|\partial^{\alpha+1-i}_x\theta|)
        \notag\\
       &\quad\quad+C\eps\de(|\pa_x\bar{u}|+|\pa_x\bar{\theta}|)
		(|\partial_xu|^2+|\partial_x\theta|^2) +C\eps\de(|\pa^{\al+1}_x\bar{u}|+|\pa^{\al+1}_x\bar{\theta}|).
	\end{align}
\end{lemma}
\begin{proof}
In view of \eqref{defbur1}	and \eqref{defbur2}, we rewrite
	\begin{equation}\label{ReG}
		\overline{G}=\eps\de \{\frac{\sqrt{K}\partial_x\bar{\theta}}{\sqrt{\theta}}A_{1}(\frac{v-u}{\sqrt{K\theta}})
		+\partial_x\bar{u}_{1}B_{11}(\frac{v-u}{\sqrt{K\theta}})\}.
	\end{equation}	
From the similar arguments in \cite{Duan-Yang-Yu-1,Duan-Yang-Yu-2,DYY}, we have
\begin{equation}\label{boundAB}
	|\pa^\be_vA_{j}(\frac{v-u}{\sqrt{K\theta}})|+|\partial_{\beta}B_{ij}(\frac{v-u}{\sqrt{K\theta}})|
	\leq C_{a,\beta}M^{1-a},
\end{equation}
for any $0<a<1$ and $|\be|\geq0$. By choosing $q_1\in(0,1]$ small enough in \eqref{Defw}, then for any $|\beta|\geq0$, $k\geq0$ and sufficiently small $a>0$, we claim that
\begin{equation}
\label{3.5AA}
|\langle v\rangle^k e^{\frac{|v|^2}{8}}w(0,\beta)\mu^{-\frac{1}{2}}M^{1-a}|_2+|\langle v\rangle^k e^{\frac{|v|^2}{8}}w(0,\beta)\mu^{-\frac{1}{2}}M^{1-a}|_\sigma\leq C.
\end{equation}
Let $\beta_{1}=(1,0,0)$, a direct calculations gives 	\begin{equation}\label{pabebarG} 		\partial^{\beta_{1}}_v\overline{G}=\eps\de\{\frac{\sqrt{K}\partial_x\bar{\theta}}{\sqrt{\theta}} 		\partial_{v_{1}}A_{1}(\frac{v-u}{\sqrt{K\theta}})\frac{1}{\sqrt{K\theta}} 		+\partial_x\bar{u}_{1}\partial_{v_{1}}B_{11}(\frac{v-u}{\sqrt{K\theta}})\frac{1}{\sqrt{K\theta}}\}. 	\end{equation}	
Then \eqref{boundbarG0} holds for $\be=0$ and $\be=\be_1$
in terms of \eqref{ReG}, \eqref{boundAB}, \eqref{3.5AA}, \eqref{pabebarG}, $|u|\leq C$ and $4/3<\theta<2$. The other cases follow from the similar calculation as in \eqref{pabebarG}.

We take $\pa_x$ to \eqref{ReG} to get
	\begin{align*}
		\partial_{x}\overline{G}=&\eps\de\Big\{\frac{\sqrt{K}\partial^2_x\bar{\theta}}{\sqrt{\theta}}A_{1}(\frac{v-u}{\sqrt{K\theta}})
		-\frac{\sqrt{K}\partial_x\bar{\theta}\partial_x\theta}{2\sqrt{\theta^{3}}}A_{1}(\frac{v-u}{\sqrt{K\theta}})
		\notag\\
		&\qquad-\frac{\sqrt{K}\partial_x\bar{\theta}}{\sqrt{\theta}}
		\nabla_{v}A_{1}(\frac{v-u}{\sqrt{K\theta}})\cdot\frac{\partial_xu}{\sqrt{K\theta}}
		-\frac{\sqrt{K}\partial_x\bar{\theta}\partial_x\theta}{\sqrt{\theta}}
		\nabla_{v}A_{1}(\frac{v-u}{\sqrt{K\theta}})\cdot\frac{v-u}{2\sqrt{K\theta^{3}}}
		\notag\\
		&\qquad+\partial^2_x\bar{u}_{1}B_{11}(\frac{v-u}{\sqrt{K\theta}})
		-\frac{\partial_x\bar{u}_{1}\partial_xu}{\sqrt{K\theta}}\cdot\nabla_{v}B_{11}(\frac{v-u}{\sqrt{K\theta}})-\frac{\partial_x\bar{u}_{1}\partial_x\theta(v-u)}{2\sqrt{K\theta^{3}}}\cdot\nabla_{v}B_{11}(\frac{v-u}{\sqrt{K\theta}})
		\Big\}.
	\end{align*}
Then for $\al=1$, $\be=0$, we can see that \eqref{boundbarG} holds from the above identity and \eqref{boundAB}. The other cases can be obtained by similar approach and are omitted here for simplicity.
\end{proof}

The estimates of  the nonlinear collision operator $\Ga$ in \eqref{DefGa} can be proved as in \cite[Lemma 9]{Strain-Guo}, see also \cite[Lemma 2.3]{Wang}. The proof is omitted for brevity. 	
\begin{lemma}
Let $\Gamma(g_1,g_2)$ defined in \eqref{DefGa} and $w=w(\alpha,\beta)$ in \eqref{Defw} with sufficiently small $q_1\in(0,1]$ and  $q_2>0$. 
	For any $\al,\be\geq0$ and some small $a>0$, it holds that
	\begin{equation}\label{controlGa}
		\langle\partial^\alpha_x \Gamma(g_1,g_2), g_3\rangle\leq C\sum_{\alpha'\leq\alpha}|\mu^a\partial^{\alpha'}_xg_1|_2| \partial^{\alpha-\alpha'}_xg_2|_\sigma|  g_3|_\sigma,
	\end{equation}
	and
	\begin{equation}\label{controlpaGa}
		\langle\partial^\alpha_\beta \Gamma(g_1,g_2), w^2(\alpha,\beta)g_3\rangle\leq
		C\sum_{\alpha'\leq\alpha}\sum_{\bar{\beta}\leq\beta'\leq\beta}|\mu^a\partial^{\alpha'}_{\bar{\beta}}g_1|_2|  \partial^{\alpha-\alpha'}_{\beta-\beta'}g_2|_{\sigma,w(\al,\be)}|g_3|_{\sigma,w}.
	\end{equation}
\end{lemma}
Then we provide coercivity estimates for the linearized operator $\CL_{\overline{M}}$ defined in \eqref{DefLM}. For completeness we give the proof of Lemma
\ref{leL} in the Appendix \ref{sec.app} later.
\begin{lemma}\label{leL} 
Let $\CL_{\overline{M}}$ be defined in \eqref{DefLM}, there exists  $c>0$ such that 
\begin{equation} 
\label{controlLbarM} -\langle \CL_{\overline{M}} f, f\rangle\geq c|\mathbf{P}_1 f|_{\sigma}^2. 
\end{equation} 
Moreover, for $W=W(\alpha,\beta)$ give in \eqref{DefW}, one has 
\begin{equation} 
\label{controlpaL} -\langle \CL_{\overline{M}} f,W^2(0,0) f\rangle\geq c| f|_{\sigma,W}^2-C|\chi_{|v|\leq C}f|_2^2. 
\end{equation} 
and for $|\al|\geq 1$, 
\begin{align} 		 
\label{controlpaLM} 		 
-\langle\partial^\alpha_x \CL_{\overline{M}} f,W^2(\alpha,0)\partial^\alpha_x f\rangle \geq&  c|\partial^\alpha_x f|_{\sigma,W}^2-C|\chi_{|v|\leq C}\partial^\alpha_x f|_2^2-C(\eta)\sum_{\al'<\al}|\pa^{\al'}_x f|^2_{\si,W}. 	 
\end{align} 
For  $|\be|>0$ and small $\ka>0$, there exists $C_\ka>0$ such that
\begin{align} 
\label{controlLM} -\langle\partial^\alpha_\beta \CL_{\overline{M}}f,W^2(\alpha,\beta)\partial^\alpha_\beta f\rangle \geq& c|\partial^\alpha_\beta f|_{\sigma,W}^2-\ka\sum_{|\beta'|=|\beta|}|\partial^\alpha_{\beta'} f|_{\sigma,W}^2 -C_\ka\sum_{|\beta'|<|\beta|}|\partial^\alpha_{\beta'} f|_{\sigma,W}^2
\notag\\ 
&-C(\eta)\sum_{\al'<\al,\beta'\leq\beta}|\pa^{\al'}_{\be'} f|^2_{\si,W}. \end{align} 
Here $C(\eta)$ depends only on $\eta$ defined in 
\eqref{boundkdv}, and $\chi_{|v|\leq C}$ is a general cutoff function depending on $C$.   
\end{lemma}

The properties of the linear operator $L_D$ defined in \eqref{DefL} can be summarized as follows. 
\begin{lemma}\label{LemmaL}
Let $w=w(\alpha,\beta)$ be defined in \eqref{Defw} with $q_1\in(0,1]$ small enough and some $q_2>0$ and $l>\al-|\beta|-\frac{1}{4}$, there exists $A_1,A_2>0$ large enough such that 
\begin{equation}
\label{coLD0}
-\langle L_Dg,w^2(0,0) g\rangle \geq c| g|_{\si,w}^2,
\end{equation}
and
\begin{align}
	-\langle\pa^\al_x L_Dg,w^2(\al,0)\pa^\al_x g\rangle
		\geq& c|\pa^\al_x g|_{\si,w}^2-C\sum_{\al'<\al}(|\mu^a \frac{\pa^{\al-\al'}_xM}{\sqrt{\mu}}|_2| \partial^{\al'}_x g|_{\si,w}\notag\\
		&\qquad\qquad\qquad+|\mu^a\partial^{\al'}_xg|_2|w(\al,0) \frac{\pa^{\al-\al'}_xM}{\sqrt{\mu}})|_\sigma)|\pa^\al_x g|_{\si,w},\label{coLDx}
	\end{align}
	and 
	\begin{align}\label{coLD}
		&-\langle\pa^\al_\be L_Dg,w^2(\al,\be)\pa^\al_\be g\rangle
		\geq c|\pa^\al_\be g|_{\si,w}^2-\ka\sum_{|\beta'|=|\beta|}|\pa^\al_x\partial^{\beta'}_v g|_{\si,w}^2
		-C_\ka\sum_{|\beta'|<|\beta|}|\pa^\al_x\partial^{\beta'}_v g|_{\si,w}^2
        \notag\\
		&\quad-C\sum_{\al'<\al,\bar{\beta}\leq\beta'\leq\beta}(|\mu^a\partial^{\bar{\beta}}_v [\frac{\pa^{\al-\al'}_xM}{\sqrt{\mu}}]|_2| \partial^{\al'}_{\beta-\beta'} g|_{\si,w}
    +|\mu^a\partial^{\al'}_{\bar{\beta}}g|_2|w(\al,\be)\pa^{\beta-\beta'}_v [\frac{\pa^{\al-\al'}_xM}{\sqrt{\mu}}]|_\sigma)|\pa^\al_\be g|_{\si,w},
	\end{align}
for any  $0<\ka<1$,  where $c$ depends  on $\rho,u$ and $\theta$. 
\end{lemma}
\begin{proof}
From the definition of $L_D$ in \eqref{DefL}, one has
\begin{align*}
-\langle\pa^\be_v L_Dg,w^2(0,\be)\pa^\be_v g\rangle=&
-\langle\pa^\be_v\{\frac{1}{\sqrt\mu}Q(M,\sqrt{\mu}g)+\frac{1}{\sqrt\mu}Q(\sqrt{\mu}g,M)\},w^2(0,\be)\pa^\be_v g\rangle
\notag\\
&+\langle A_1\pa^\be_v (\chi_{|v|\leq A_2}g),w^2(0,\be)\pa^\be_v g\rangle.
\end{align*}
From similar arguments in \cite[Lemmas 8,9]{Strain-Guo}, it holds that
\begin{align*}
-\langle\pa^\be_v L_Dg,w^2(0,\be)\pa^\be_v g\rangle
\geq& c|\pa^\be_v g|_{\si,w}^2-C_2|\pa^\be_v g|_2^2
\notag\\
&-\ka\sum_{|\beta'|=|\beta|}|\partial^{\beta'}_v g|_{\si,w}^2
	-C_\ka\sum_{|\beta'|<|\beta|}|\partial^{\beta'}_v g|_{\si,w}^2+A_1|\chi_{|v|\leq A_2}\pa^\be_v g|_2^2.
\end{align*}
for any $\ka>0$ and some $C_2>1$, where the constant $0<c<1$ depends on $\rho,u$ and $\theta$.
Then we prove that $c/2|\pa^\be_v g|_{\si,w}^2-C_2|\pa^\be_v g|^2+A_1|\chi_{|v|\leq A_2}\pa^\be_v g|_2^2\geq0$ for large $A_1$ and $A_2$. Direct calculation shows
\begin{align*}
& \frac{c}{2}|\pa^\be_v g|_{\si,w}^2-C_2|\pa^\be_v g|_2^2+A_1|\chi_{|v|\leq A_2}\pa^\be_v g|_2^2
\\
&\geq \frac{c}{4}\int_{\R^3}\big(\langle v\rangle^{2l-2|\be|-1/2}(1-\chi_{|v|\leq A_2})+\chi_{|v|\leq A_2}\big)|\pa^\be_v g|^2\,dv-C_2| \pa^\be_v g|_2 ^2+A_1|\chi_{|v|\leq A_2}\pa^\be_v g|_{2}^2 
    \\
&=\int_{\R^3}\big(\frac{c}{4}-C_2\langle v\rangle^{-(2l-2|\be|-1/2)}\big)\langle v\rangle^{2l-2|\be|-1/2}(1-\chi_{|v|\leq A_2})|\pa^\be_v g|^2\,dv\\
	&\quad+\int_{\R^3}\big(\frac{c}{4}\langle v\rangle^{2l-2|\be|-1/2}-C_2+A_1\big)\chi_{|v|\leq A_2}|\pa^\be_v g|^2\,dv\geq 0.
\end{align*}
Here we have required $2l>2|\be|+1/2$ and taken suitably large $A_1>0$ and $A_2>0$ such  that $(\frac{c}{4}--C_2\langle v\rangle^{-(2l-2|\be|-1/2)})(1-\chi_{|v|\leq A_2})\geq0$ and $\frac{c}{4}-C_2+A_1>0$.
This leads to 
\begin{align}
	-\langle\pa^\be_v L_Dg,w^2(0,\be)\pa^\be_v g\rangle&\geq \frac{c}{2}|\pa^\be_v g|_{\si,w}^2-\ka\sum_{|\beta'|=|\beta|}|\partial^{\beta'}_v g|_{\si,w}^2
	-C_\ka\sum_{|\beta'|<|\beta|}|\partial^{\beta'}_v g|_{\si,w}^2.
    \label{coLDv}
\end{align}
Similarly,  the desired estimate \eqref{coLD0} holds true.

Applying $\pa^\al_\be$ to $L_D$ and using \eqref{coLDv}, one gets
\begin{align*}
-\langle \pa^\al_\be L_Dg,w^2(\al,\be)\pa^\al_\be g\rangle=&\langle-\pa^\al_\be\{\frac{1}{\sqrt\mu}Q(M,\sqrt{\mu}g)+\frac{1}{\sqrt\mu}Q(\sqrt{\mu}g,M)\}+A_1\pa^\al_\be\chi_{|v|\leq A_2}g,w^2(\al,\be)\pa^\al_\be g\rangle
\notag\\
	\geq&\langle-\pa^\be_v\{\frac{1}{\sqrt\mu}Q(M,\sqrt{\mu}\pa^\al_x g)+\frac{1}{\sqrt\mu}Q(\sqrt{\mu}\pa^\al_x g,M)\}+A_1\pa^\be_v\chi_{|v|\leq A_2}\pa^\al_x g,w^2(\al,\be)\pa^\al_\be g\rangle
    \notag\\
&-C\sum_{\al'<\al}\Big|\langle-\pa^\be_v\{\frac{1}{\sqrt\mu}Q(\pa^{\al-\al'}_xM,\sqrt{\mu}\pa^{\al'}_x g)+\frac{1}{\sqrt\mu}Q(\sqrt{\mu}\pa^{\al'}_x g,\pa^{\al-\al'}_xM)\},w^2(\al,\be)\pa^\al_\be g\rangle\Big|\notag\\
\geq & c|\pa^\al_\be g|_{\si,w}^2-\ka\sum_{|\beta'|=|\beta|}|\pa^\al_x\partial^{\beta'}_v g|_{\si,w}^2
	-C_\ka\sum_{|\beta'|<|\beta|}|\pa^\al_x\partial^{\beta'}_v g|_{\si,w}^2
    \notag\\
	&-C\sum_{\al'<\al}\Big|\langle\pa^\be_v\Ga(\frac{\pa^{\al-\al'}_xM}{\sqrt{\mu}},\pa^{\al'}_x g)+\Ga(\pa^{\al'}_x g,\frac{\pa^{\al-\al'}_xM}{\sqrt{\mu}})\},w^2(\al,\be)\pa^\al_\be g\rangle \Big|,
\end{align*}
which, combined with \eqref{controlpaGa}, gives \eqref{coLD}. Moreover, \eqref{coLDx} can be obtained by similar calculations above and \eqref{controlGa}. This completes the proof of Lemma \ref{LemmaL}.
\end{proof}

\section{Energy estimates}\label{secKdV}
Before we start the major part of this section, we now prove some bounds of the fluid quantities involving time derivatives.
\begin{lemma}\label{lem4.1A}
Under the a priori assumption \eqref{apriori}, it holds that
\begin{align}\label{patphi}
\eps^2\|\partial_t\pa_x\widetilde{\phi}\|^2
+\eps\|\partial_t\widetilde{\phi}\|^2
		\leq C\|\pa^2_x\widetilde{\phi}\|^2
		+C\frac{1}{\eps}\CE(t)+C(\eta)\eps^4,	
\end{align}
and
\begin{align}\label{patpaxphi}
	\eps^2\|\partial_t\pa^2_x\widetilde{\phi}\|^2+\eps\|\partial_t\pa_x\widetilde{\phi}\|^2
	\leq C\|\pa^3_x\widetilde{\phi}\|^2+C\frac{1}{\eps}\|\pa^2_x(\widetilde{\rho},\widetilde{u}_1,\widetilde{\phi})\|^2+CA\frac{\eps^2}{\de}\CE(t)+C(\eta)A\eps^3,	
\end{align}
where $C(\eta)\geq0$ depends only on $\eta$ with $C(0)=0$.
\end{lemma}
\begin{proof}
Applying $\pa_t$ to the last equation of \eqref{perturbeq}, then taking the inner product with $\partial_t\widetilde{\phi}$ and using the first equation in \eqref{perturbeq}, we reach at 
	\begin{align}\label{1patphi}
		-\eps^2(\partial^{2}_{x}\partial_t\widetilde{\phi},\partial_t\widetilde{\phi})+\eps(\partial_t\widetilde{\phi},\partial_t\widetilde{\phi})
		=&(\partial_t\widetilde{\rho},\partial_t\widetilde{\phi})
		-\eps^4(\partial_tR_{4},\partial_t\widetilde{\phi})
		\notag\\
		=&(\frac{1}{\eps}\pa_x \widetilde{\rho}
		-\partial_x(\bar{\rho}\widetilde{u}_{1})
		-\partial_x(\widetilde{\rho}u_{1})-\eps^3R_1,\partial_t\widetilde{\phi})
		-\eps^4(\partial_tR_{4},\partial_t\widetilde{\phi}).
	\end{align}
	Using the last equation of \eqref{perturbeq} again, we have
	\begin{align*}
		(\frac{1}{\eps}\pa_x \widetilde{\rho},\partial_t\widetilde{\phi})&=(\frac{1}{\eps}\pa_x (-\eps^2\pa_x^2\widetilde{\phi}+\eps\widetilde{\phi}+\eps^4R_4),\partial_t\widetilde{\phi})\notag\\
		&=\eps(\pa_x^2\widetilde{\phi},\partial_t\pa_x\widetilde{\phi})+(\pa_x\widetilde{\phi},\pa_t\widetilde{\phi})+\eps^3(\pa_xR_4,\pa_t\widetilde{\phi}).
	\end{align*}
	Then by the Cauchy-Schwarz inequality and \eqref{boundkdv}, we obtain
	\begin{align}\label{1patphi1}
		|(\frac{1}{\eps}\pa_x \widetilde{\rho},\partial_t\widetilde{\phi})|\leq \frac{1}{2}(\eps^2\|\partial_t\pa_x\widetilde{\phi}\|^2+\eps\|\partial_t\widetilde{\phi}\|^2)+C\frac{1}{\eps}\|\pa_x\widetilde{\phi}\|^2+C\|\pa^2_x\widetilde{\phi}\|^2
        +C(\eta)\eps^4.
	\end{align} 
By \eqref{approx} and \eqref{2.61AA}, we get
\begin{equation}
\label{4.5A}
\|(\bar{\rho}-1,\bar{\theta}-\frac{3}{2})\|_{L^{\infty}}+\|\partial_x(\bar{\rho},\bar{\theta})\|_{H^4}
+\|\partial_t(\bar{\rho},\bar{\theta})\|_{H^4}\leq C(\eta)\eps, \quad \|\partial_x(\bar{u},\bar{\phi})\|_{H^4}
+\|\partial_t(\bar{u},\bar{\phi})\|_{H^4}\leq C(\eta).
\end{equation}
By the Sobolev embedding, \eqref{2.61AA}, \eqref{apriori}, \eqref{4.5A} and \eqref{smalleps}, we get
	\begin{align}\label{1patphi2}
		|(\partial_x(\bar{\rho}\widetilde{u}_{1})
		+\partial_x(\widetilde{\rho}u_{1}),
        \partial_t\widetilde{\phi})|\leq& \frac{\eps}{4}\|\partial_t\widetilde{\phi}\|^2+C\frac{1}{\eps}\|\partial_x(\rho\widetilde{u}_{1})
		+\partial_x(\widetilde{\rho}\bar{u}_{1})\|^2\notag\\
		\leq&\frac{\eps}{4}\|\partial_t\widetilde{\phi}\|^2+C\frac{1}{\eps}\big(\|\rho\|_{L^\infty}^2\|\pa_x\widetilde{u}_{1}\|^2\notag\\
		&\qquad\qquad\qquad\qquad+\|\pa_x\rho\|^2\|\widetilde{u}_{1}\|_{L^\infty}^2+\|\pa_x\bar{u}_1\|^2_{L^\infty}\|\widetilde{\rho}\|^2\|\bar{u}_1\|^2_{L^\infty}\|\pa_x\widetilde{\rho}\|^2\big)\notag\\
		\leq&\frac{\eps}{4}\|\partial_t\widetilde{\phi}\|^2+\frac{C}{\eps}(1+C(\eta))\CE(t).
	\end{align}
	Moreover, by \eqref{boundkdv}, it is straightforward to get 
	\begin{align}\label{1patphi3}
		|\eps^3R_1+\eps^4\pa_tR_4,\pa_t\widetilde{\phi})|\leq \frac{\eps}{8}\|\partial_t\widetilde{\phi}\|^2+C(\eta)\eps^5.
	\end{align}
	Hence, the estimate \eqref{patphi} follows from \eqref{1patphi}, \eqref{1patphi1}, \eqref{1patphi2} and \eqref{1patphi3}.
	
In the following, we prove that estimate \eqref{patpaxphi} holds true. We first apply $\pa_t\pa_x$ to the last equation of \eqref{perturbeq} and take the inner product to get
	\begin{align}\label{1patpaxphi}
\eps^2\|\partial_t\pa^2_x\widetilde{\phi}\|^2+\eps\|\partial_t\pa_x\widetilde{\phi}\|^2=&(\partial_{x}\partial_t\widetilde{\rho},\partial_{x}\partial_t\widetilde{\phi})-\eps^4(\partial_{x}\partial_tR_{4},\partial_{x}\partial_t\widetilde{\phi})\notag\\
		=&(\frac{1}{\eps}\pa^2_x \widetilde{\rho}
		-\partial^2_x(\bar{\rho}\widetilde{u}_{1})
		-\partial^2_x(\widetilde{\rho}u_{1})-\eps^3\pa_xR_1,\partial_t\pa_x\widetilde{\phi})
		-\eps^4(\partial_t\pa_xR_{4},\partial_t\pa_x\widetilde{\phi}).
	\end{align}
	By the last equation of \eqref{perturbeq}, \eqref{boundkdv}, \eqref{2.61AA} and \eqref{apriori}, one gets
	\begin{align}\label{1patpaxphi1}
		|(\frac{1}{\eps}\pa^2_x \widetilde{\rho},\pa_x\partial_t\widetilde{\phi})|&=|(\frac{1}{\eps}\pa^2_x (-\eps^2\pa_x^2\widetilde{\phi}+\eps\widetilde{\phi}+\eps^4R_4),\pa_x\partial_t\widetilde{\phi})|\notag\\
		&\leq |\eps(\pa_x^3\widetilde{\phi},\partial_t\pa^2_x\widetilde{\phi})|+|(\pa^2_x\widetilde{\phi},\pa_t\pa_x\widetilde{\phi})|+|\eps^3(\pa^2_xR_4,\pa_t\pa_x\widetilde{\phi})|\notag\\
		&\leq \frac{1}{2}(\eps^2\|\partial_t\pa^2_x\widetilde{\phi}\|^2+\eps\|\partial_t\pa_x\widetilde{\phi}\|^2)+C\|\pa^3_x\widetilde{\phi}\|^2+C\frac{1}{\eps}\|\pa^2_x\widetilde{\phi}\|^2+C(\eta)\eps^4.
	\end{align}
Using the similar calculation \eqref{1patphi2} and the fact that
	$$
	\|\pa_x\widetilde{u}_1\|^2_{L^\infty}\leq C\|\pa_x\widetilde{u}_1\|\|\pa^2_x\widetilde{u}_1\|\leq CA\eps^2\frac{\eps^2}{\eps\de}=CA\frac{\eps^3}{\de},
	$$ it holds that
		\begin{align}\label{1patpaxphi2}
		|(\partial^2_x(\bar{\rho}\widetilde{u}_{1}+\widetilde{\rho}u_{1}),\partial_t\pa_x\widetilde{\phi})|\leq& \frac{\eps}{4}\|\partial_t\pa_x\widetilde{\phi}\|^2+\frac{C}{\eps}\|\partial^2_x(\rho\widetilde{u}_{1})
		+\partial^2_x(\widetilde{\rho}\bar{u}_{1})\|^2\notag\\
		\leq&\frac{\eps}{4}\|\partial_t\pa_x\widetilde{\phi}\|^2+\frac{C}{\eps}\|\partial^2_x(\widetilde{\rho},\widetilde{u}_{1})\|^2
        +CA\frac{\eps^2}{\de}\CE(t)+C(\eta)A\eps^3.
	\end{align}
By \eqref{boundkdv}, we directly see that
	\begin{align}\label{1patpaxphi3}
		|\eps^3\pa_xR_1+\eps^4\pa_t\pa_xR_4,\pa_t\pa_x\widetilde{\phi})|\leq \frac{\eps}{8}\|\partial_t\pa_x\widetilde{\phi}\|^2+C(\eta)\eps^5.
	\end{align}
	Thus, \eqref{patpaxphi} follows from \eqref{1patpaxphi}, \eqref{1patpaxphi1}, \eqref{1patpaxphi2} and \eqref{1patpaxphi3}.
    This completes the proof of Lemma \ref{lem4.1A}.
\end{proof}
	We also need to bound the time derivatives of $\widetilde{\rho}$, $\widetilde{u}$ and $\widetilde{\theta}$.
	\begin{lemma}\label{lem4.2}
Under the a priori assumption \eqref{apriori}, it holds 
for any small $a>0$ that
		\begin{equation}\label{estpat}
			\|(\partial_t\widetilde{\rho},\partial_t\widetilde{u},\partial_t\widetilde{\theta})\|^{2}
			\leq C\frac{1}{\eps^2}\|\partial_x(\widetilde{\rho},\widetilde{u},\widetilde{\theta},\widetilde{\phi})\|^{2}
            +C\|\mu^{a}\partial_x(g,f)\|^{2}
            +C(\eta)\CE(t)+C(\eta)(\eps^4+\de^4),
		\end{equation}
		and 
		\begin{align}\label{estpatpax}
\|\partial_x\partial_t(\widetilde{\rho},\widetilde{u},\widetilde{\theta})\|^{2}\leq&
			C\frac{1}{\eps^2}\|\partial^2_x(\widetilde{\rho},\widetilde{u},\widetilde{\theta},\widetilde{\phi})\|^{2}+C\|\mu^{a}\partial^2_x(g,f)\|^{2}
+C\eps^2(\|\mu^{a}\partial_x(g,f)\|^{2}
+\|\partial_x(\widetilde{\rho},\widetilde{u},\widetilde{\theta},\widetilde{\phi})\|^{2})
\notag\\
&+C(\eta)\CE(t)
+C(\eta)(\de^4+\eps^4),
		\end{align}
 where $C(\eta)\geq0$ is a constant which depends only on $\eta$ with $C(0)=0$.
	\end{lemma}
	\begin{proof}
By \eqref{1macro}, \eqref{Defpert} and \eqref{approx}, we arrive at
\begin{align}
		\label{perturbeqA}
		\left\{
		\begin{array}{rl}
			&\dis \partial_{t}\widetilde{\rho}-\frac{1}{\eps}\pa_x \widetilde{\rho}
			+\partial_x(\bar{\rho}\widetilde{u}_{1})
			+\partial_x(\widetilde{\rho}u_{1})=-\eps^3R_1,
			\\
			&\dis \partial_t\widetilde{u}_{1}-\frac{1}{\eps}\partial_x\widetilde{u}_{1}
			+u_{1}\partial_x\widetilde{u}_{1}+\widetilde{u}_{1}\partial_x\bar{u}_{1}+\frac{2}{3}\partial_x\widetilde{\theta}
			+\frac{2}{3}(\frac{\theta}{\rho}-\frac{\bar{\theta}}{\bar{\rho}})\partial_x\rho+\frac{2}{3}\frac{\bar{\theta}}{\bar{\rho}}\partial_x\widetilde{\rho}
			+\frac{1}{\eps}\partial_x\widetilde{\phi}
			\\
			&\dis\qquad=-\frac{1}{\rho}\int_{\mathbb{R}^{3}} v^{2}_{1}\partial_xG\,dv-\eps^3\frac{1}{\bar{\rho}}R_2,\\
			&\dis\partial_t\widetilde{u}_{i}-\frac{1}{\eps}\partial_x\widetilde{u}_{i}+u_{1}\partial_x\widetilde{u}_{i}=-\frac{1}{\rho}\int_{\mathbb{R}^{3}} v_{1}v_{i}\partial_xG\,dv, ~~i=2,3,
			\\
			&\dis\partial_t\widetilde{\theta}-\frac{1}{\eps}\partial_x\widetilde{\theta}
			+\frac{2}{3}\bar{\theta}\partial_x\widetilde{u}_{1}+\frac{2}{3}\widetilde{\theta}\partial_xu_{1}
			+u_1\partial_x\widetilde{\theta}+\widetilde{u}_{1}\partial_x\bar{\theta}
			\\
			&\dis\qquad=\frac{1}{\rho}u\cdot\partial_x(\int_{\mathbb{R}^3} v_{1}vG \, dv)-\frac{1}{\rho}\partial_x(\int_{\mathbb{R}^3}v_{1}\frac{|v|^{2}}{2}G\, dv)-\eps^3 R_3,
			\\
			&\dis-\eps^2\pa_x^2\widetilde{\phi}+\eps\widetilde{\phi}=\widetilde{\rho}-\eps^4R_4.
		\end{array} \right.
	\end{align}
By taking the inner product 
of the second equation of \eqref{perturbeqA} with  $\partial_t\widetilde{u}_1$, one has
		\begin{align*}
			(\partial_t\widetilde{u}_{1},\partial_t\widetilde{u}_{1})
			&=-(-\frac{1}{\eps}\partial_x\widetilde{u}_{1}
+u_{1}\partial_x\widetilde{u}_{1}+\widetilde{u}_{1}\partial_x\bar{u}_{1}+\frac{2}{3}\partial_x\widetilde{\theta}
			+\frac{2}{3}(\frac{\theta}{\rho}-\frac{\bar{\theta}}{\bar{\rho}})\partial_x\rho+\frac{2\bar{\theta}}{3\bar{\rho}}\partial_x\widetilde{\rho},\partial_t\widetilde{u}_{1})-\frac{1}{\eps}(\partial_x\widetilde{\phi},\partial_t\widetilde{u}_{1})
			\notag\\
			&\quad
			-(\frac{1}{\rho}\int_{\mathbb{R}^3} v^2_{1}\partial_x\overline{G}\,dv,\partial_t\widetilde{u}_{1})-(\frac{1}{\rho}\int_{\mathbb{R}^3} v^2_{1}\partial_x h\,dv,\partial_t\widetilde{u}_{1})-(\eps^3\frac{1}{\bar{\rho}}R_2,\partial_t\widetilde{u}_{1}).
		\end{align*}
Here we have used $G=\overline{G}+h$. A direct application of Cauchy-Schwarz inequality, \eqref{2.61AA} and \eqref{apriori} gives
		\begin{align}\label{patest}
		&|(-\frac{1}{\eps}\partial_x\widetilde{u}_{1}
		+u_{1}\partial_x\widetilde{u}_{1}+\widetilde{u}_{1}\partial_x\bar{u}_{1}+\frac{2}{3}\partial_x\widetilde{\theta}
		+\frac{2}{3}(\frac{\theta}{\rho}-\frac{\bar{\theta}}{\bar{\rho}})\partial_x\rho+\frac{2\bar{\theta}}{3\bar{\rho}}\partial_x\widetilde{\rho},\partial_t\widetilde{u}_{1})|+\frac{1}{\eps}|(\partial_x\widetilde{\phi},\partial_t\widetilde{u}_{1})|\notag\\
		\leq&	\frac{1}{2}\|\partial_t\widetilde{u}_{1}\|^{2}+C\frac{1}{\eps^2}\|\partial_x(\widetilde{\rho},\widetilde{u}_1,\widetilde{\theta},\widetilde{\phi})\|^{2}+C(\|u_1\|^2_{L^\infty}\|\pa_x\widetilde{u}_1\|^2+\|\widetilde{u}_1\|^2\|\pa_x\bar{u}_1\|^2_{L^\infty}+\|\frac{\theta}{\rho}-\frac{\bar{\theta}}{\bar{\rho}}\|^2_{L^\infty}\|\partial_x\rho\|^2)
        \notag\\
		\leq&	\frac{1}{2}\|\partial_t\widetilde{u}_{1}\|^{2}+C\frac{1}{\eps^2}\|\partial_x(\widetilde{\rho},\widetilde{u}_1,\widetilde{\theta},\widetilde{\phi})\|^{2}+C(\eta)\CE(t).
		\end{align}
Recalling $h=\sqrt{\overline{M}}f+\sqrt{\mu}g$ given in \eqref{2.46A},
for any small $a>0$, we have
\begin{align*}
|(\frac{1}{\rho}\int_{\mathbb{R}^3} v^2_{1}\partial_x h\,dv,\partial_t\widetilde{u}_{1})|
\leq \frac{1}{8}\|\partial_t\widetilde{u}_{1}\|^{2}+
C\|\mu^{a}\partial_x(g,f)\|^{2}+ C(\eta)\|\mu^{a}f\|^{2}.
\end{align*}
Using \eqref{boundbarG}, \eqref{boundkdv}, \eqref{4.5A} and \eqref{apriori}, we obtain 
		\begin{align*}
	|(\frac{1}{\rho}\int_{\mathbb{R}^3} v^2_{1}\partial_x\overline{G}\,dv,\partial_t\widetilde{u}_{1})|+|(\eps^3\frac{1}{\bar{\rho}}R_2,\partial_t\widetilde{u}_{1})|
			\leq&\frac{1}{4}\|\partial_t\widetilde{u}_{1}\|^{2}+C(\eta)\eps^2\de^2+C(\eta)\CE(t)+C(\eta)\eps^4.
		\end{align*}
Hence, we collect the above estimates to get 
$$\|\partial_t\widetilde{u}_1\|^{2}\leq  C\frac{1}{\eps^2}\|\partial_x(\widetilde{\rho},\widetilde{u},\widetilde{\theta},\widetilde{\phi})\|^{2}
+C\|\mu^{a}\partial_x(g,f)\|^{2}
+C(\eta)\CE(t)+C(\eta)\eps^4+C(\eta)\eps^2\de^2.
$$
The same bound for $\partial_t\widetilde{\rho}$, $\partial_t\widetilde{u}_2$, $\partial_t\widetilde{u}_3$ and $\partial_t\widetilde{\theta}$ can be deduced similarly and we omit the proof for simplicity. Therefore, the desired estimate \eqref{estpat} holds.
		
Next, by taking $\partial_x$ to the second equation of \eqref{perturbeqA} and then taking the inner product with $\partial_x\partial_t\widetilde{u}_1$, one gets
		\begin{align}
        \label{4.16A}
(\partial_x\partial_t\widetilde{u}_{1},\partial_x\partial_t\widetilde{u}_{1})&=-(\partial_x[-\frac{1}{\eps}\partial_x\widetilde{u}_{1}+u_{1}\partial_x\widetilde{u}_{1}+\widetilde{u}_{1}\partial_x\bar{u}_{1}+\frac{2}{3}\partial_x\widetilde{\theta}
			+\frac{2}{3}(\frac{\theta}{\rho}-\frac{\bar{\theta}}{\bar{\rho}})\partial_x\rho+\frac{2}{3}\frac{\bar{\theta}}{\bar{\rho}}\partial_x\widetilde{\rho}]+\frac{1}{\eps}\partial^2_x\widetilde{\phi},\partial_x\partial_t\widetilde{u}_{1})
			\notag\\
			&\quad
			-(\partial_x[\frac{1}{\rho}\int_{\mathbb{R}^3} v^2_{1}\partial_x(\overline{G}+h)\,dv],\partial_x\partial_t\widetilde{u}_{1})-(\eps^3	\partial_x[\frac{1}{\bar{\rho}}R_2],
			\partial_x\partial_t\widetilde{u}_{1}).
		\end{align}
		Similar calculation as in \eqref{patest} gives
		\begin{align*}
			&\big|(\partial_x[-\frac{1}{\eps}\partial_x\widetilde{u}_{1}
			+u_{1}\partial_x\widetilde{u}_{1}+\widetilde{u}_{1}\partial_x\bar{u}_{1}+\frac{2}{3}\partial_x\widetilde{\theta}
			+\frac{2}{3}(\frac{\theta}{\rho}-\frac{\bar{\theta}}{\bar{\rho}})\partial_x\rho+\frac{2}{3}\frac{\bar{\theta}}{\bar{\rho}}\partial_x\widetilde{\rho}]+\frac{1}{\eps}\partial^2_x\widetilde{\phi},\partial_x\partial_t\widetilde{u}_{1}) \big|\notag\\
			&\leq	\frac{1}{2}\|\partial_t\pa_x\widetilde{u}_{1}\|^{2}+C\frac{1}{\eps^2}\|\partial^2_x(\widetilde{\rho},\widetilde{u},\widetilde{\theta},\widetilde{\phi})\|^{2}+C\eps^2\|\partial_x(\widetilde{\rho},\widetilde{u},\widetilde{\theta},\widetilde{\phi})\|^{2}
            +C(\eta)\CE(t).
		\end{align*}
By using $h=\sqrt{\overline{M}}f+\sqrt{\mu}g$, \eqref{boundbarG}, \eqref{boundkdv}, \eqref{4.5A} and \eqref{apriori}, we obtain that the last line in \eqref{4.16A} is bound by
			\begin{align*}
\frac{1}{4}\|\partial_t\pa_x\widetilde{u}_{1}\|^{2}+C(\eta)\eps^2\de^2+C(\eta)\CE(t)+C(\eta)\eps^4+C\|\mu^{a}\partial^2_x(g,f)\|^{2}+C\eps^2\|\mu^{a}\partial_x(g,f)\|^{2},
	\end{align*}
for any small $a>0$. From the above estimates, we have
\begin{align*}
\|\partial_t\pa_x\widetilde{u}_{1}\|^{2}\leq& C\frac{1}{\eps^2}\|\partial^2_x(\widetilde{\rho},\widetilde{u},\widetilde{\theta},\widetilde{\phi})\|^{2}+C\|\mu^{a}\partial^2_x(g,f)\|^{2}
+C\eps^2(\|\mu^{a}\partial_x(g,f)\|^{2}
+\|\partial_x(\widetilde{\rho},\widetilde{u},\widetilde{\theta},\widetilde{\phi})\|^{2})
\notag\\
&+C(\eta)\CE(t)
+C(\eta)(\de^4+\eps^4).
\end{align*}
Hence, the desired estimate \eqref{estpatpax} follows from this and the similar arguments for $\widetilde{\rho}$, $\widetilde{u}_2$, $\widetilde{u}_3$ and $\widetilde{\theta}$. This completes the proof of 
Lemma \ref{lem4.2}.
\end{proof}
	
\subsection{Lower order estimates on fluid quantities}
In this subsection, we obtain the energy estimates for macroscopic quantities up to order one. Note that the global zero order estimate will be treated using energy flux pair approach in Section \ref{SecGlobal} with $\eta=0$. Then we only track the dependence on $\eta$ for the first order.
\begin{lemma}\label{le1strho}
	Under the a priori assumption \eqref{apriori}, for any small $\ka>0$, there exists some $C_\ka>0$ such that
		\begin{align}\label{zerorho}
		\frac{1}{2}\frac{d}{dt}(\widetilde{\rho},\frac{2\bar{\theta}}{3\bar{\rho}^{2}}\widetilde{\rho})+(\pa_x\widetilde{u}_1,\frac{2\bar{\theta}}{3\bar{\rho}}\widetilde{\rho})
		\leq& C(\eta)\mathcal{E}(t)+C\eps^4,
	\end{align}
and
	\begin{align}\label{1strho}
\frac{1}{2}\frac{d}{dt}(\pa_x\widetilde{\rho},\frac{2\bar{\theta}}{3\bar{\rho}^{2}}\pa_x\widetilde{\rho})
+(\pa^2_x\widetilde{u}_1,\frac{2\bar{\theta}}{3\bar{\rho}}\pa_x\widetilde{\rho})
\leq& C\ka\eps\de\|\partial^2_x\widetilde{u}_1\|^2+C(\eta)\mathcal{E}(t)
\notag\\&
\quad+C_\ka A\min\{(\frac{\eps}{\de}+\frac{\eps^2}{\de^2})\CD(t),(\frac{\eps^{3/2}}{\sqrt{\de}}+\frac{\eps^{3}}{\de})\CE(t)\}+C(\eta)\eps^4,
	\end{align}
where $C(\eta)\geq0$ depends only on $\eta$ with $C(0)=0$.
	\end{lemma}
	\begin{proof}
We prove \eqref{1strho} first. By applying $\pa_x$ to the first equation  of \eqref{perturbeqA} and then taking the inner product of the resulting equation with $\frac{2\bar{\theta}}{3\bar{\rho}^{2}}\pa_x\widetilde{\rho}$, one gets
\begin{align}\label{iprho}
	&\frac{1}{2}\frac{d}{dt}(\pa_x\widetilde{\rho},\frac{2\bar{\theta}}{3\bar{\rho}^{2}}\pa_x\widetilde{\rho})
	-\frac{1}{2}(\pa_x\widetilde{\rho},\partial_t(\frac{2\bar{\theta}}{3\bar{\rho}^{2}})\pa_x\widetilde{\rho})
	-\frac{1}{\eps}(\pa^2_x\widetilde{\rho},\frac{2\bar{\theta}}{3\bar{\rho}^{2}}\pa_x\widetilde{\rho})
	+(\bar{\rho}\pa^2_x\widetilde{u}_{1},\frac{2\bar{\theta}}{3\bar{\rho}^{2}}\pa_x\widetilde{\rho})+(\pa_x\bar{\rho}\partial_x\widetilde{u}_{1},
	\frac{2\bar{\theta}}{3\bar{\rho}^{2}}\pa_x\widetilde{\rho})
	\notag\\
	=&-(\pa_x(\widetilde{u}_1\partial_x\bar{\rho}),\frac{2\bar{\theta}}{3\bar{\rho}^{2}}\pa_x\widetilde{\rho})
	-(\pa_x(u_1\partial_x\widetilde{\rho}),\frac{2\bar{\theta}}{3\bar{\rho}^{2}}\pa_x\widetilde{\rho})-(\pa_x(\widetilde{\rho}\partial_xu_{1}),\frac{2\bar{\theta}}{3\bar{\rho}^{2}}\pa_x\widetilde{\rho})
	-(\eps^3\pa_xR_1,\frac{2\bar{\theta}}{3\bar{\rho}^{2}}\pa_x\widetilde{\rho}).
\end{align}		
It follows from integration by parts, \eqref{4.5A}, \eqref{approx}, \eqref{boundkdv} and \eqref{apriori} that
	\begin{align}\label{1rho1}
		&\big| \frac{1}{2}(\pa_x\widetilde{\rho},\partial_t(\frac{2\bar{\theta}}{3\bar{\rho}^{2}})\pa_x\widetilde{\rho})
		+\frac{1}{\eps}(\pa^2_x\widetilde{\rho},\frac{2\bar{\theta}}{3\bar{\rho}^{2}}\pa_x\widetilde{\rho}) \big|=\big| \frac{1}{2}(\pa_x\widetilde{\rho},\partial_t(\frac{2\bar{\theta}}{3\bar{\rho}^{2}})\pa_x\widetilde{\rho})
		-\frac{1}{2\eps}(\pa_x\widetilde{\rho},\pa_x(\frac{2\bar{\theta}}{3\bar{\rho}^{2}})\pa_x\widetilde{\rho}) \big|\notag\\
		\leq& C(\|\partial_t(\bar{\rho},\bar{\theta})\|_{L^{\infty}}+\frac{1}{\eps}\|\partial_x(\bar{\rho},\bar{\theta})\|_{L^{\infty}})\|\pa_x\widetilde{\rho}\|^2\leq C(\eta)\CE(t).
	\end{align}
	Similarly, one has
		\begin{align}\label{1rho2}
		&\big| (\pa_x\bar{\rho}\partial_x\widetilde{u}_{1},
		\frac{2\bar{\theta}}{3\bar{\rho}^{2}}\pa_x\widetilde{\rho})+(\pa_x(\widetilde{u}_1\partial_x\bar{\rho}),\frac{2\bar{\theta}}{3\bar{\rho}^{2}}\pa_x\widetilde{\rho}) \big|
		\leq C(\eta)\CE(t).
	\end{align}
	For the second term on the right hand side of \eqref{iprho}, it holds that
	\begin{align}\label{1rho3}
		| (\pa_x(u_1\partial_x\widetilde{\rho}),\frac{2\bar{\theta}}{3\bar{\rho}^{2}}\pa_x\widetilde{\rho})|
		\leq& | (\pa_x(\bar{u}_1\partial_x\widetilde{\rho}),\frac{2\bar{\theta}}{3\bar{\rho}^{2}}\pa_x\widetilde{\rho})|+| (\pa_x\widetilde{u}_1\partial_x\widetilde{\rho},\frac{2\bar{\theta}}{3\bar{\rho}^{2}}\pa_x\widetilde{\rho})|+| (\widetilde{u}_1\partial^2_x\widetilde{\rho},\frac{2\bar{\theta}}{3\bar{\rho}^{2}}\pa_x\widetilde{\rho})|\notag\\
		\leq&C(\|\partial_x(\bar{\rho},\bar{u}_1,\bar{\theta})\|_{L^{\infty}}+\|\partial_x\widetilde{u}_1\|_{L^{\infty}}+\|\partial_x(\widetilde{u}_1\frac{2\bar{\theta}}{3\bar{\rho}^{2}})\|_{L^{\infty}})\|\pa_x\widetilde{\rho}\|^2\notag\\
		\leq&C(\eta)\CE(t)+C\min\{\frac{1}{\ep\de}\sqrt{\CE(t)}\CD(t),\frac{1}{\sqrt{\eps\de}}\CE(t)^{\frac{3}{2}}\}\notag\\
		\leq&C(\eta)\CE(t)+C\sqrt{A}\min\{\frac{\eps}{\de}\CD(t),\frac{\eps^{3/2}}{\sqrt{\de}}\CE(t)\}.
	\end{align}
Here the third inequality  in \eqref{1rho3} holds since
$$   \|\partial_x\widetilde{u}_1\|_{L^{\infty}}\|\pa_x\widetilde{\rho}\|^2\leq C\|\partial_x\widetilde{u}_1\|^\frac{1}{2}\|\partial^2_x\widetilde{u}_1\|^\frac{1}{2}\|\pa_x\widetilde{\rho}\|^2\leq C\min\{\frac{1}{\ep\de}\sqrt{\CE(t)}\CD(t),\frac{1}{\sqrt{\eps\de}}\CE(t)^{\frac{3}{2}}\},
$$
by the Sobolev embedding, \eqref{DefE} and \eqref{DefD}. Likewise, we have for any $0<\ka<1$ that
	\begin{align}\label{1rho4}
		\big| (\pa_x(\widetilde{\rho}\partial_xu_{1}),\frac{2\bar{\theta}}{3\bar{\rho}^{2}}\pa_x\widetilde{\rho})\big|
		\leq&C(\eta)\CE(t)+C\sqrt{A}\min\{\frac{\eps}{\de}\CD(t),\frac{\eps^{3/2}}{\sqrt{\de}}\CE(t)\}+C| (\widetilde{\rho}\partial^2_x\widetilde{u}_{1},\frac{2\bar{\theta}}{3\bar{\rho}^{2}}\pa_x\widetilde{\rho})|
        \notag\\
		\leq&C\ka\eps\de\|\partial^2_x\widetilde{u}_1\|^2+C(\eta)\CE(t)+CA\min\{\frac{\eps}{\de}\CD(t),\frac{\eps^{3/2}}{\sqrt{\de}}\CE(t)\}+C_\ka\frac{1}{\eps\de}\|\widetilde{\rho}\|^2_{L^\infty}\|\pa_x\widetilde{\rho}\|^2
        \notag\\
		\leq&C\ka\eps\de\|\partial^2_x\widetilde{u}_1\|^2+C(\eta)\CE(t)
        +C_\ka A\min\{(\frac{\eps}{\de}+\frac{\eps^2}{\de^2})\CD(t),(\frac{\eps^{3/2}}{\sqrt{\de}}+\frac{\eps^{3}}{\de})\CE(t)\}.
	\end{align}
	In the last inequality of \eqref{1rho4}, we have used the fact that
	$$
	 \frac{1}{\eps\de}\|\widetilde{\rho}\|^2_{L^\infty}\|\pa_x\widetilde{\rho}\|^2\leq CA\frac{\eps^3}{\de}\|\pa_x\widetilde{\rho}\|^2\leq CA\min\{\frac{\eps^2}{\de^2}\CD(t),\frac{\eps^{3}}{\de}\CE(t)\}.
	$$
By \eqref{boundkdv} and \eqref{DefE}, we get
	\begin{align}\label{1rho5}
		|(\eps^3\pa_xR_1,\frac{2\bar{\theta}}{3\bar{\rho}^{2}}\pa_x\widetilde{\rho})|\leq C(\eta)\eps^3\|\pa_x\widetilde{\rho}\|\leq C(\eta)\CE(t)+C(\eta)\eps^6.
	\end{align}
Hence, the desired estimate \eqref{1strho} follows from \eqref{1rho1}, \eqref{1rho2}, \eqref{1rho3}, \eqref{1rho4} and \eqref{1rho5}. 

The proof of \eqref{zerorho} is similar. By taking the inner product on both side of the first equation in \eqref{perturbeq} with $\frac{2\bar{\theta}}{3\bar{\rho}^{2}}\widetilde{\rho}$, we can arrive at
\begin{align*}
	&\frac{1}{2}\frac{d}{dt}(\widetilde{\rho},\frac{2\bar{\theta}}{3\bar{\rho}^{2}}\widetilde{\rho})
	+(\bar{\rho}\pa_x\widetilde{u}_{1},\frac{2\bar{\theta}}{3\bar{\rho}^{2}}\widetilde{\rho})
	\notag\\
	=&\frac{1}{2}(\widetilde{\rho},\partial_t(\frac{2\bar{\theta}}{3\bar{\rho}^{2}})\widetilde{\rho}) +\frac{1}{\eps}(\pa_x\widetilde{\rho},\frac{2\bar{\theta}}{3\bar{\rho}^{2}}\widetilde{\rho})
    -(\widetilde{u}_1\partial_x\bar{\rho},\frac{2\bar{\theta}}{3\bar{\rho}^{2}}\widetilde{\rho})
	-(u_1\partial_x\widetilde{\rho},\frac{2\bar{\theta}}{3\bar{\rho}^{2}}\widetilde{\rho})
    \notag\\
    &-(\widetilde{\rho}\partial_xu_{1},\frac{2\bar{\theta}}{3\bar{\rho}^{2}}\widetilde{\rho})
	-(\eps^3R_1,\frac{2\bar{\theta}}{3\bar{\rho}^{2}}\widetilde{\rho})
    \notag\\
    \leq&C(\|\pa_t(\bar{\rho},\bar{\theta})\|_{L^\infty}+\frac{1}{\eps}
     \|\pa_x(\bar{\rho},\bar{\theta})\|_{L^\infty}+
    \|\pa_x(\bar{\rho},\bar{u}_1,\bar{\theta})\|_{L^\infty})
    \|(\widetilde{\rho},\widetilde{u})\|^2+C    \|\widetilde{\rho}\|_{L^\infty}\|\pa_x\widetilde{u}_1\|\|\widetilde{\rho}\|+C(\eta)\eps^3\|\widetilde{\rho}\|
\notag\\     \leq& C(\eta)\CE(t)+C(\eta)A\eps^5+CA^{\frac{3}{2}}\eps^6,
\end{align*}
which, together with our assumption \eqref{smalleps}, gives the desired estimate \eqref{zerorho}. Hence, we complete the proof of Lemma \ref{le1strho}.
	\end{proof}
	
For the velocity quantity $\pa_x\widetilde{u}$, we have the following lemma where the energy of the electronic field is obtained simultaneously. It is crucial for us to deal with the $O(\eps^{-1})$ and $O(\eps^{-2})$ terms caused by the electronic field and singular shift on $x$.

\begin{lemma}\label{le1stu}
Under the a priori assumption \eqref{apriori}, 
for any small $\ka>0$,
there exists constants $0<c<1$ and $C_{\ka}>0$ such that
	\begin{align}
		&\frac{1}{2}\frac{d}{dt}\{\|\pa_x\widetilde{u}_1\|^{2}+(\frac{1}{\bar\rho}\pa_x\widetilde{\phi},\pa_x\widetilde{\phi})+\eps(\frac{1}{\bar\rho}\pa^2_x\widetilde{\phi},\pa^2_x\widetilde{\phi})\}+
		(\frac{2\bar{\theta}}{3\bar{\rho}}\pa^2_x\widetilde{\rho},\pa_x\widetilde{u}_1)+\frac{2}{3}(\pa^2_x\widetilde{\theta},\pa_x\widetilde{u}_1)+c\ep\de\|\pa^2_x\widetilde{u}_1\|^{2}\notag\\
		&\quad+\ep\de\frac{d}{dt}\int_{\mathbb{R}}\int_{\mathbb{R}^{3}}
		\big\{\pa_x[\frac{1}{\rho}\partial_{x}(K\theta B_{11}(\frac{v-u}{\sqrt{K\theta}})\frac{1}{M}h)]
		\pa_x\widetilde{u}_1\big\}\,dv\,dx
		\notag\\
		\leq & C\ka\eps\de\|\pa^2_x(\widetilde{\rho},     \widetilde{u},\widetilde{\theta})\|^{2}+
        C_\ka(\de+\eps+\frac{\de}{\sqrt\eps})\CD(t)+C_\ka\eps\de\|\pa^2_x(g,f)\|_\sigma^2
        \notag\\
		&+C_\ka A\min\{(\frac{\eps}{\de}+\frac{\eps^2}{\de^2})\CD(t),(\frac{\eps^{3/2}}{\sqrt{\de}}+\frac{\eps^{3}}{\de})\CE(t)\}+C_\ka C(\eta)\CE(t)+C_\ka C(\eta)(\eps^4+\eps^3\de+\eps\de^3),\label{1stu}
	\end{align}
and for $i=2,3$,
\begin{align}\label{1stui}
	&\frac{1}{2}\frac{d}{dt}\|\pa_x\widetilde{u}_i\|^{2}
	+c\ep\de\|\pa^2_x\widetilde{u}_i\|^{2}
	+\ep\de\frac{d}{dt}\int_{\mathbb{R}}\int_{\mathbb{R}^{3}}
	\big\{\pa_x[\frac{1}{\rho}\partial_{x}(K\theta B_{1i}(\frac{v-u}{\sqrt{K\theta}})\frac{1}{M}h)]
	\pa_x\widetilde{u}_i\big\}\,dv\,dx\notag\\
	\leq & C\ka\eps\de\|\pa^2_x(\widetilde{\rho},     \widetilde{u},\widetilde{\theta})\|^{2}+
    C_\ka(\de+\eps+\frac{\de}{\sqrt\eps})\CD(t)+C_\ka\eps\de\|\pa^2_x(g,f)\|_\sigma^2
    \notag\\
	&+C_\ka A\min\{(\frac{\eps}{\de}+\frac{\eps^2}{\de^2})\CD(t),(\frac{\eps^{3/2}}{\sqrt{\de}}+\frac{\eps^{3}}{\de})\CE(t)\}+C_\ka C(\eta)\CE(t)+C_\ka C(\eta)(\eps^4+\eps^3\de+\eps\de^3),
\end{align}
for $i=2,3$, where $C(\eta)\geq0$ depends only on $\eta$ with $C(0)=0$.
\end{lemma}	
\begin{proof}
 Applying $\pa_x$ to the second equation  of \eqref{perturbeq} and then taking the inner product of the resulting equation with $\pa_x\widetilde{u}_1$, it holds that
\begin{align}\label{ipu}
	&\frac{1}{2}\frac{d}{dt}\|\pa_x\widetilde{u}_1\|^{2}-\frac{1}{\eps}(\pa^2_x\widetilde{u}_1,\pa_x\widetilde{u}_1)
	+
	(\frac{2\bar{\theta}}{3\bar{\rho}}\partial^2_x\widetilde{\rho},\pa_x\widetilde{u}_1)+(\frac{2}{3}\pa^2_x\widetilde{\theta},\pa_x\widetilde{u}_1)+\frac{1}{\eps}(\pa^2_x\widetilde{\phi},\pa_x\widetilde{u}_1)
	\notag\\
	=&-(\pa_x(\frac{2\bar{\theta}}{3\bar{\rho}})\partial_x\widetilde{\rho},\pa_x\widetilde{u}_1)
	-(\pa_x(u_1\partial_x\widetilde{u}_1),\pa_x\widetilde{u}_1)
	-(\pa_x(\widetilde{u}_1\partial_x\bar{u}_1),\pa_x\widetilde{u}_1)-(\pa_x[\frac{2}{3}(\frac{\theta}{\rho}-\frac{\bar{\theta}}{\bar{\rho}})\partial_x\rho],\pa_x\widetilde{u}_1)
	\notag\\
	&+\eps\de(\pa_x[\frac{4}{3\rho}\partial_x(\mu(\theta)\partial_xu_1)],\pa_x\widetilde{u}_1)
	-(\pa_x[\frac{1}{\rho}\partial_x(\int_{\mathbb{R}^{3}} v_{1}^2L^{-1}_{M}\Theta\,dv)],\pa_x\widetilde{u}_1)-\eps^3(\pa_x[\frac{1}{\bar{\rho}}R_2],\pa_x\widetilde{u}_1).
\end{align}
In the following we compute \eqref{ipu} term by term. First note that
$(\pa^2_x\widetilde{u}_1,\pa_x\widetilde{u}_1)=0$, then we perform similar calculations in the proof of Lemma \ref{le1strho} to get
\begin{align}\label{1stufluid}
	&|(\pa_x(\frac{2\bar{\theta}}{3\bar{\rho}})\partial_x\widetilde{\rho},\pa_x\widetilde{u}_1)
	+(\pa_x(u_1\partial_x\widetilde{u}_1),\pa_x\widetilde{u}_1)
	+(\pa_x(\widetilde{u}_1\partial_x\bar{u}_1),\pa_x\widetilde{u}_1)|\notag\\
	&+|(\pa_x[\frac{2}{3}(\frac{\theta}{\rho}-\frac{\bar{\theta}}{\bar{\rho}})\partial_x\rho],\pa_x\widetilde{u}_1)|+\eps^3|(\pa_x[\frac{1}{\bar{\rho}}R_2],\pa_x\widetilde{u}_1)|\notag\\
	\leq& C\ka\eps\de\|\partial^2_x\widetilde{u}_1\|^2+C(\eta)\mathcal{E}(t)+C_\ka A\min\{(\frac{\eps}{\de}+\frac{\eps^2}{\de^2})\CD(t),(\frac{\eps^{3/2}}{\sqrt{\de}}+\frac{\eps^{3}}{\de})\CE(t)\}+C(\eta)\eps^4.
\end{align}
Now we only need to consider the last term on the left hand side of  \eqref{ipu}, and the fifth, sixth terms on the right hand side. For the last term on the left hand side, we first use the first equation of \eqref{perturbeq} to get
\begin{align*}
	\partial_x\widetilde{u}_{1}=-\frac{1}{\bar{\rho}}\partial_{t}\widetilde{\rho}-\frac{1}{\bar{\rho}}\widetilde{u}_{1}\partial_x\bar{\rho}+\frac{1}{\eps}\frac{1}{\bar{\rho}}\pa_x \widetilde{\rho}-\frac{1}{\bar{\rho}}\partial_x(\widetilde{\rho}u_{1})-\frac{1}{\bar{\rho}}\eps^3R_1,
\end{align*} 
which yields
\begin{align}\label{pa2phipau}
	\frac{1}{\eps}(\pa^2_x\widetilde{\phi},\pa_x\widetilde{u}_1)=\frac{1}{\eps}(\pa^2_x\widetilde{\phi},-\frac{1}{\bar{\rho}}\partial_{t}\widetilde{\rho}-\frac{1}{\bar{\rho}}\widetilde{u}_{1}\partial_x\bar{\rho}+\frac{1}{\eps}\frac{1}{\bar{\rho}}\pa_x \widetilde{\rho}-\frac{1}{\bar{\rho}}\partial_x(\widetilde{\rho}u_{1})-\frac{1}{\bar{\rho}}\eps^3R_1).
\end{align}
Then by the last equation in \eqref{perturbeq}, it holds that
\begin{align*}
	\frac{1}{\eps}(\pa^2_x\widetilde{\phi},-\frac{1}{\bar{\rho}}\partial_{t}\widetilde{\rho})
	=&\frac{1}{\eps}(\pa^2_x\widetilde{\phi},\frac{1}{\bar{\rho}}\partial_{t}(\eps^2\pa_x^2\widetilde{\phi}-\eps\widetilde{\phi}-\eps^4R_4))\notag\\
	=&\frac{1}{2}\frac{d}{dt}\{(\frac{1}{\bar{\rho}}\pa_x\widetilde{\phi},\pa_x\widetilde{\phi})+\eps(\frac{1}{\bar{\rho}}\pa^2_x\widetilde{\phi},\pa^2_x\widetilde{\phi})\}-\eps(\pa^2_x\widetilde{\phi},\partial_{t}(\frac{1}{\bar{\rho}})\pa^2_x\widetilde{\phi})+(\pa_x\widetilde{\phi},\pa_x(\frac{1}{\bar{\rho}})\pa_t\widetilde{\phi})\notag\\
	&-(\pa_x\widetilde{\phi},\pa_t(\frac{1}{\bar{\rho}})\pa_x\widetilde{\phi})+\eps^3(\pa_x\widetilde{\phi},\pa_x[\frac{1}{\bar{\rho}}\partial_{t}R_4]).
\end{align*}
By using \eqref{patphi}, \eqref{4.5A}, \eqref{boundkdv} and \eqref{DefE}, we obtain
$$
|\eps(\pa^2_x\widetilde{\phi},\partial_{t}(\frac{1}{\bar{\rho}})\pa^2_x\widetilde{\phi})|+|(\pa_x\widetilde{\phi},\pa_t(\frac{1}{\bar{\rho}})\pa_x\widetilde{\phi})|+|\eps^3(\pa_x\widetilde{\phi},\pa_x[\frac{1}{\bar{\rho}}\partial_{t}R_4])|\leq C(\eta)\CE(t)+ C(\eta)\eps^4,
$$
and
$$
|(\pa_x\widetilde{\phi},\pa_x(\frac{1}{\bar{\rho}})\pa_t\widetilde{\phi})|\leq C(\eta)\|\pa_x\widetilde{\phi}\|^2+C(\eta)\eps^2\|\pa_t\widetilde{\phi}\|^2\leq C(\eta)\CE(t)+C(\eta)\eps^4.
$$
It follows from the above three estimates that
\begin{align}\label{pa2phipau1}
	-\frac{1}{\eps}(\pa^2_x\widetilde{\phi},\frac{1}{\bar{\rho}}\partial_{t}\widetilde{\rho})
	\geq&\frac{1}{2}\frac{d}{dt}\{(\frac{1}{\bar{\rho}}\pa_x\widetilde{\phi},\pa_x\widetilde{\phi})+\eps(\frac{1}{\bar{\rho}}\pa^2_x\widetilde{\phi},\pa^2_x\widetilde{\phi})\}-C(\eta)\CE(t)-C(\eta)\eps^4.
\end{align}
On the other hand, it is direct to obtain
\begin{align}\label{pa2phipau2}
	|\frac{1}{\eps}(\pa^2_x\widetilde{\phi},-\frac{1}{\bar{\rho}}\widetilde{u}_{1}\partial_x\bar{\rho})|&\leq |\frac{1}{\eps}(\pa_x\widetilde{\phi},\frac{1}{\bar{\rho}}\partial_x\bar{\rho}\pa_x\widetilde{u}_{1})|+|\frac{1}{\eps}(\pa_x\widetilde{\phi},\pa_x[\frac{\partial_x\bar{\rho}}{\bar{\rho}}]\widetilde{u}_{1})|\leq C(\eta)\CE(t).
\end{align}
We turn to the third term on the right hand side of \eqref{pa2phipau}. By combining with the last equation in \eqref{perturbeq}, it holds that
\begin{align}
\label{pa2phipau3}
&|\frac{1}{\eps^2}(\pa^2_x\widetilde{\phi},\frac{1}{\bar{\rho}}\pa_x \widetilde{\rho})|
=|\frac{1}{\eps^2}(\pa^2_x\widetilde{\phi},\frac{1}{\bar{\rho}} (-\eps^2\pa_x^3\widetilde{\phi}+\eps\pa_x\widetilde{\phi}+\eps^4\pa_x R_4))|
\notag\\
&\leq C|(\pa_x^2\widetilde{\phi},\pa_x(\frac{1}{\bar{\rho}})\pa_x^2\widetilde{\phi})|+\frac{1}{\eps}|(\pa_x\widetilde{\phi},\pa_x(\frac{1}{\bar{\rho}})\pa_x\widetilde{\phi})|+\eps^2
    |(\pa_x\widetilde{\phi},\pa_x[\frac{1}{\bar{\rho}} \pa_x R_4])|
    \notag\\
	&\leq C(\eta)\CE(t)+C(\eta)\eps^4,
\end{align}
where in the last inequality we have used \eqref{patphi}, \eqref{4.5A}, \eqref{boundkdv} and \eqref{DefE}.

For fourth term on the right hand side of \eqref{pa2phipau}, we also use the last equation of \eqref{perturbeq} to get
\begin{align*}
	\big|\frac{1}{\eps}(\pa^2_x\widetilde{\phi},-\frac{1}{\bar{\rho}}\partial_x(\widetilde{\rho}u_{1}))\big|\leq& \frac{1}{\eps}|(\pa^2_x\widetilde{\phi},\frac{\bar{u}_1+\widetilde{u}_1}{\bar{\rho}}(-\eps^2\pa_x^3\widetilde{\phi}+\eps\pa_x\widetilde{\phi}+\eps^4\pa_xR_4))|\notag\\
	&\quad+\frac{1}{\eps}|(\pa^2_x\widetilde{\phi},\frac{\pa_x(\bar{u}_1+\widetilde{u}_1)}{\bar{\rho}}(-\eps^2\pa_x^2\widetilde{\phi}+\eps\widetilde{\phi}+\eps^4R_4))|.
\end{align*}
All terms involving $\bar{u}_1$ can be bounded in the similar way as in \eqref{pa2phipau3}, which leads to
\begin{align}\label{pa2phipau41}
	\big|\frac{1}{\eps}(\pa^2_x\widetilde{\phi},-\frac{1}{\bar{\rho}}\partial_x(\widetilde{\rho}u_{1}))\big|\leq&C(\eta)\CE(t)+C(\eta)\eps^4+ \frac{1}{\eps}|(\pa^2_x\widetilde{\phi},\frac{\widetilde{u}_1}{\bar{\rho}}(-\eps^2\pa_x^3\widetilde{\phi}+\eps\pa_x\widetilde{\phi}+\eps^4\pa_xR_4))|\notag\\
	&\quad+\frac{1}{\eps}|(\pa^2_x\widetilde{\phi},\frac{\pa_x\widetilde{u}_1}{\bar{\rho}}(-\eps^2\pa_x^2\widetilde{\phi}+\eps\widetilde{\phi}+\eps^4R_4))|.
\end{align}
We only estimate the last term above and the rest can be bounded similarly. Using the Sobolev embedding inequality, \eqref{DefE}, \eqref{DefD} and \eqref{apriori}, one has
\begin{align}\label{pa2phipau42}
	\frac{1}{\eps}|(\pa^2_x\widetilde{\phi},\frac{\pa_x\widetilde{u}_1}{\bar{\rho}}(-\eps^2\pa_x^2\widetilde{\phi}+\eps\widetilde{\phi}+\eps^4R_4))|\leq& C\eps\|\pa_x\widetilde{u}_1\|^\frac{1}{2}\|\pa^2_x\widetilde{u}_1\|^\frac{1}{2}\|\pa^2_x\widetilde{\phi}\|^2\notag\\
	&+C\|\pa_x\widetilde{u}_1\|\|\pa^2_x\widetilde{\phi}\|\|\pa_x\widetilde{\phi}\|^\frac{1}{2}\|\widetilde{\phi}\|^\frac{1}{2}+C(\eta)\eps^3\|\pa^2_x\widetilde{\phi}\|\|\pa_x\widetilde{u}_1\|\notag\\
	\leq&C(\eta)\CE(t)+CA\min\{\frac{\eps}{\de}\CD(t),\frac{\eps^{3/2}}{\sqrt{\de}}\CE(t)\}.
\end{align}
The combination of \eqref{pa2phipau41} and \eqref{pa2phipau42} gives
\begin{align}\label{pa2phipau4}
	\big|\frac{1}{\eps}(\pa^2_x\widetilde{\phi},-\frac{1}{\bar{\rho}}\partial_x(\widetilde{\rho}u_{1}))\big|\leq&C(\eta)\CE(t)+C(\eta)\eps^4+CA\min\{\frac{\eps}{\de}\CD(t),\frac{\eps^{3/2}}{\sqrt{\de}}\CE(t)\}.
\end{align}
The last term on the right hand side of \eqref{pa2phipau} is more bounded by
\begin{align}\label{pa2phipau5}
	|\frac{1}{\eps}(\pa^2_x\widetilde{\phi},\frac{1}{\bar{\rho}}\eps^3R_1)|\leq |\frac{1}{\eps}(\pa_x\widetilde{\phi},\pa_x[\frac{1}{\bar{\rho}}\eps^3R_1])|\leq C(\eta)\CE(t)+C(\eta)\eps^4.
\end{align}
It follows from \eqref{pa2phipau}, \eqref{pa2phipau1}, \eqref{pa2phipau2}, \eqref{pa2phipau3}, \eqref{pa2phipau4} and \eqref{pa2phipau5} that
\begin{align}\label{1stuphi}
	\frac{1}{\eps}(\pa^2_x\widetilde{\phi},\pa_x\widetilde{u}_1)\geq & \frac{1}{2}\frac{d}{dt}\{(\frac{1}{\bar{\rho}}\pa_x\widetilde{\phi},\pa_x\widetilde{\phi})+\eps(\frac{1}{\bar{\rho}}\pa^2_x\widetilde{\phi},\pa^2_x\widetilde{\phi})\}\notag\\
	&\qquad-C(\eta)\CE(t)-C(\eta)\eps^4-CA\min\{\frac{\eps}{\de}\CD(t),\frac{\eps^{3/2}}{\sqrt{\de}}\CE(t)\}.
\end{align}
We now consider the term $\ep\de(\pa_x[\frac{4}{3\rho}\partial_x(\mu(\theta)\partial_xu_1)],\pa_x\widetilde{u}_1)$ in \eqref{ipu}. First split it into
\begin{align*}
	\ep\de(\pa_x[\frac{4}{3\rho}\partial_x(\mu(\theta)\partial_xu_1)],\pa_x\widetilde{u}_1)=\ep\de(\pa_x[\frac{4}{3\rho}\partial_x(\mu(\theta)\partial_x\bar{u}_1)],\pa_x\widetilde{u}_1)+\ep\de(\pa_x[\frac{4}{3\rho}\partial_x(\mu(\theta)\partial_x\widetilde{u}_1)],\pa_x\widetilde{u}_1).
\end{align*} 
Since both $\mu(\theta)$ and $\kappa(\theta)$ are smooth functions of $\theta$, there exists a constant $C>1$ 
such that $\mu(\theta),\kappa(\theta)\in[C^{-1},C]$. The first term on the right hand side above is bounded by
\begin{align*}
	\ep\de(\pa_x[\frac{4}{3\rho}\partial_x(\mu(\theta)\partial_x\bar{u}_1)],\pa_x\widetilde{u}_1)&\leq C(\eta)\ep\de\|\pa_x\widetilde{u}_1\|(1+\|\pa_x(\rho,\theta)\|+\|\pa_x(\rho,\theta)\|_{L^{\infty}}\|\pa_x\theta\|+\|\pa^2_x\theta\|)
    \notag\\
    &\leq C(\eta)\CE(t)+C(\eta)\eps^2\de^2,
\end{align*}
according to \eqref{4.5A}, \eqref{DefE} and \eqref{apriori} as well as the smallness of $\eps$ and $\de$.
For the second term, we first integrate by parts, then use \eqref{apriori} and \eqref{smalleps} to get
\begin{align*}
	\eps\de(\pa_x[\frac{4}{3\rho}\partial_x(\mu(\theta)\partial_x\widetilde{u}_1)],\pa_x\widetilde{u}_1)
	=&-\eps\de(\frac{4}{3\rho}\partial_x(\mu(\theta)\partial_x\widetilde{u}_1),\pa^2_x\widetilde{u}_1)
	\notag\\
	\leq& -\eps\de(\frac{4}{3\rho}\mu(\theta)\pa^2_x\widetilde{u}_1,\pa^2_x\widetilde{u}_1)
	+|\eps\de(\frac{4}{3\rho}\partial_x[\mu(\theta)]\partial_x\widetilde{u}_1,\pa^2_x\widetilde{u}_1)|\notag\\
	\leq& -\eps\de(\frac{4}{3\rho}\mu(\theta)\pa^2_x\widetilde{u}_1,\pa^2_x\widetilde{u}_1)
	+C\eps\de\|\pa_x\theta\|_{L^\infty}\|\pa_x\widetilde{u}_1\|\|\pa^2_x\widetilde{u}_1\|
    \notag\\
	\leq& -\eps\de(\frac{4}{3\rho}\mu(\theta)\pa^2_x\widetilde{u}_1,\pa^2_x\widetilde{u}_1)+C(\eta)\CE(t)
	+C\eps\CD(t).
\end{align*}
We collect the above two estimates to obtain
\begin{equation}\label{1studiss}
	\ep\de(\pa_x[\frac{4}{3\rho}\partial_x(\mu(\theta)\partial_xu_1)],\pa_x\widetilde{u}_1)\leq -\eps\de(\frac{4}{3\rho}\mu(\theta)\pa^2_x\widetilde{u}_1,\pa^2_x\widetilde{u}_1)
	+C\eps\CD(t)+C(\eta)\CE(t)+C(\eta)\eps^2\de^2.
\end{equation}
Next we bound the term involving $L^{-1}_{M}\Theta$ in \eqref{ipu}. Recalling the Burnett functions defined in \eqref{defbur1} and \eqref{defbur2}, and using the identity
\begin{align}\label{reB}
	\int_{\mathbb{R}^{3}}v_{1}v_{j}L^{-1}_{M}\Theta \,dv
	=\int_{\mathbb{R}^{3}} L^{-1}_{M}\{K\theta\hat{B}_{1j}(\frac{v-u}{\sqrt{K\theta}})M\}\frac{\Theta}{M}\,dv
	=K\theta\int_{\mathbb{R}^{3}}B_{1j}(\frac{v-u}{\sqrt{K\theta}})\frac{\Theta}{M}\,dv,
\end{align}
one has from integration by parts that
\begin{align}\label{I1I2}
	-(\pa_x[\frac{1}{\rho}\partial_x(\int_{\mathbb{R}^{3}} v_{1}^2L^{-1}_{M}\Theta\,dv)],\pa_x\widetilde{u}_1)=I_1+I_2,
\end{align}
where
\begin{align*}
	I_1=&(\pa_x[\frac{1}{\rho}\int_{\mathbb{R}^{3}}K\theta B_{11}(\frac{v-u}{\sqrt{K\theta}})\frac{\Theta}{M}\, dv],\pa^2_x\widetilde{u}_1),
	\end{align*}
and
\begin{align}\label{defI2}
	I_2&=(\pa_x[\partial_{x}(\frac{1}{\rho})\int_{\mathbb{R}^{3}}K\theta B_{11}(\frac{v-u}{\sqrt{K\theta}})\frac{\Theta}{M}\, dv],\pa_x\widetilde{u}_1).
\end{align}
Consider $I_1$ first. We use \eqref{DefTheta} to rewrite
\begin{align}\label{defI1}
	I_1=&\int_{\mathbb{R}}\int_{\mathbb{R}^{3}}\pa_x[\frac{1}{\rho}K\theta B_{11}(\frac{v-u}{\sqrt{K\theta}})\frac{1}{M}\{(\eps\de\partial_tG-\de\partial_xG)+\eps\de P_{1}(v_{1}\partial_xG)-\de\partial_x\phi\partial_{v_{1}}G-Q(G,G)\}]\,\pa^2_x\widetilde{u}_1\,dv\,dx
	\notag\\
	:=&I_{11}+I_{12}+I_{13}+I_{14}.
\end{align}
Recalling $G=\overline{G}+h$ given by \eqref{Defpert}, we further decompose $I_{11}$ into
\begin{align}\label{defI11}
	I_{11}=&\eps\de\int_{\mathbb{R}}\int_{\mathbb{R}^{3}}\pa_x[\frac{1}{\rho}K\theta B_{11}(\frac{v-u}{\sqrt{K\theta}})\frac{\partial_t\overline{G}-\eps^{-1}\pa_x\overline{G}}{M}]\,\pa^2_x\widetilde{u}_1\,dv\,dx\notag\\
	&+\eps\de\int_{\mathbb{R}}\int_{\mathbb{R}^{3}}\pa_x[\frac{1}{\rho}K\theta B_{11}(\frac{v-u}{\sqrt{K\theta}})\frac{(\partial_th-\eps^{-1}\pa_xh)}{M}]\,\pa^2_x\widetilde{u}_1 \,dv\,dx
	\notag\\
	:=&I_{111}+I_{112}.
\end{align}
Then we easily see that
\begin{align*}
	I_{111}=&\eps\de\int_{\mathbb{R}}\int_{\mathbb{R}^{3}}\pa_x[\frac{1}{\rho}K\theta B_{11}(\frac{v-u}{\sqrt{K\theta}})]\frac{\partial_t\overline{G}-\eps^{-1}\pa_x\overline{G}}{M}\,\pa^2_x\widetilde{u}_1\,dv\,dx\notag\\
	&+\eps\de\int_{\mathbb{R}}\int_{\mathbb{R}^{3}}[\frac{1}{\rho}K\theta B_{11}(\frac{v-u}{\sqrt{K\theta}})\frac{\partial_t\pa_x\overline{G}-\eps^{-1}\pa^2_x\overline{G}}{M}]\,\pa^2_x\widetilde{u}_1\,dv\,dx.
\end{align*}
For the first term above, we use \eqref{ReG}, \eqref{boundAB}, \eqref{4.5A}, \eqref{estpat} and \eqref{apriori} to get
\begin{align*}
	&\eps\de\int_{\mathbb{R}}\int_{\mathbb{R}^{3}}\pa_x[\frac{1}{\rho}K\theta B_{11}(\frac{v-u}{\sqrt{K\theta}})]\frac{\partial_t\overline{G}-\eps^{-1}\pa_x\overline{G}}{M}\,\pa^2_x\widetilde{u}_1\,dv\,dx\notag\\
	\leq&C\eps\de\|\pa_x[\frac{1}{\rho}K\theta B_{11}(\frac{v-u}{\sqrt{K\theta}})M^{-1/2}]\|_{L^\infty_xL^2_v}\|\frac{\partial_t\overline{G}-\eps^{-1}\pa_x\overline{G}}{\sqrt{M}}\|_{L^2_xL^2_v}\|\pa^2_x\widetilde{u}_1\|
    \notag\\
	\leq&C\ka\eps\de\|\pa^2_x\widetilde{u}_1\|^2+C_\ka\eps\de\|\pa_x(\rho,u,\theta)\|^2_{L^\infty}\|\frac{\partial_t\overline{G}-\eps^{-1}\pa_x\overline{G}}{\sqrt{M}}\|^2_{L^2_xL^2_v}
    \notag\\
	\leq&C\ka\eps\de\|\pa^2_x\widetilde{u}_1\|^2
    +C_\ka(\eps+\de)\CD(t)
    +C_\ka C(\eta)(\eps^4+\eps\de^3),
\end{align*}
for any $0<\ka<1$. Here we have used the fact that
\begin{align*}
\|\frac{\partial_t\overline{G}-\eps^{-1}\pa_x\overline{G}}{\sqrt{M}}\|^2_{L^2_xL^2_v}
\leq& C\de^2\|\pa_x(\bar{u},\bar{\theta})
\|^2_{L^\infty}	\|(\partial_xu,\partial_x\theta) \|^2+C\de^2\|(\pa^{2}_x\bar{u},\pa^{2}_x\bar{\theta})\|^2
\notag\\
&+C\eps^{2}\de^2\|\pa_x(\bar{u},\bar{\theta}) \|^2_{L^\infty}	\|(\partial_tu,\partial_t\theta) \|^2+C\de^2\|(\pa_x\pa_t\bar{u},\pa_x\pa_t\bar{\theta})\|^2,
\end{align*}
which can be obtained with similar arguments as in \eqref{boundbarG}.
Similarly, we also have
\begin{align*}
\eps\de&\int_{\mathbb{R}}\int_{\mathbb{R}^{3}}[\frac{1}{\rho}K\theta B_{11}(\frac{v-u}{\sqrt{K\theta}})\frac{\partial_t\pa_x\overline{G}-\eps^{-1}\pa^2_x\overline{G}}{M}]\,\pa^2_x\widetilde{u}_1\,dv\,dx
    \notag\\
    &\leq \ka\eps\de\|\pa^2_x\widetilde{u}_1\|^2+ C_\ka\eps\de\|\frac{\partial_t\pa_x\overline{G}-\eps^{-1}\pa^2_x\overline{G}}{\sqrt{M}}\|^2
    \notag\\
	&\leq C\ka\eps\de\|\pa^2_x\widetilde{u}_1\|^2     +C_\ka(\eps+\de)\CD(t)     +C_\ka C(\eta)(\eps^4+\eps\de^3),
\end{align*}
which immediately gives
\begin{align}\label{I111}
	I_{111}\leq C\ka\eps\de\|\pa^2_x\widetilde{u}_1\|^2     +C_\ka(\eps+\de)\CD(t)     +C_\ka C(\eta)(\eps^4+\eps\de^3).
\end{align}
We should pay extra attention to $I_{112}$ since it contains $\pa_t\pa_x$ without providing extra $\eps\de$ as $\overline{G}$ to cancel the singularity $\eps^{-2}$ from the right hand side of \eqref{estpatpax}. It follows from integration by parts that
\begin{align}\label{defI112}
	I_{112}&=\eps\de\frac{d}{dt}\int_{\mathbb{R}}\int_{\mathbb{R}^{3}}\pa_x
	\big\{[\frac{1}{\rho}K\theta B_{11}(\frac{v-u}{\sqrt{K\theta}})\frac{1}{M}]h\big\}\pa^2_x\widetilde{u}_1\,dv\,dx
	\nonumber\\
	&\quad-\eps\de\int_{\mathbb{R}}\int_{\mathbb{R}^{3}}\pa_x	\big\{h(\partial_t-\eps^{-1}\pa_x)[\frac{1}{\rho}K\theta B_{11}(\frac{v-u}{\sqrt{K\theta}})\frac{1}{M}]\big\}\pa^2_x\widetilde{u}_1\,dv\,dx
	\nonumber\\
	&\quad+\eps\de\int_{\mathbb{R}}\int_{\mathbb{R}^{3}}\pa^2_x
	\big\{[\frac{1}{\rho}K\theta B_{11}(\frac{v-u}{\sqrt{K\theta}})\frac{1}{M}]h\big\}\pa_x(\partial_t-\eps^{-1}\pa_x)\widetilde{u}_1\,dv\,dx.
\end{align}
Noticing $h=\sqrt{\overline{M}}f(t,x,v)+\sqrt{\mu}g(t,x,v)$, the second line in \eqref{defI112} can be bounded by
\begin{align}\label{1I112}
	&\big|\eps\de\int_{\mathbb{R}}\int_{\mathbb{R}^{3}}\pa_x	\big\{h(\partial_t-\eps^{-1}\pa_x)[\frac{1}{\rho}K\theta B_{11}(\frac{v-u}{\sqrt{K\theta}})\frac{\sqrt{\mu}}{M}]\big\}\pa^2_x\widetilde{u}_1\,dv\,dx\big|
    \notag\\
	\leq &C\eps\de\|\langle v\rangle^{-\frac{1}{2}}\frac{\sqrt{\overline{M}}f+\sqrt{\mu}g}{\sqrt{\mu}}\|_{L^\infty_xL^2_v}\|\pa_x(\partial_t-\eps^{-1}\pa_x)[\frac{1}{\rho}K\theta B_{11}(\frac{v-u}{\sqrt{K\theta}})\frac{\mu}{M}]\|_{L^2_xL^2_v}\|\pa^2_x\widetilde{u}_1\|
    \notag\\
	&+C\eps\de\|\langle v\rangle^{-\frac{1}{2}}\pa_x[\frac{\sqrt{\overline{M}}f+\sqrt{\mu}g}{\sqrt{\mu}}]\|_{L^2_xL^2_v}\|(\partial_t-\eps^{-1}\pa_x)[\frac{1}{\rho}K\theta B_{11}(\frac{v-u}{\sqrt{K\theta}})\frac{\mu}{M}]\|_{L^\infty_xL^2_v}\|\pa^2_x\widetilde{u}_1\|\notag\\
\leq& C\ka\eps\de\|\pa^2_x\widetilde{u}_1\|^2     +C_\ka(\eps+\de)\CD(t)     +C_\ka C(\eta)(\eps^4+\eps\de^3).
\end{align}
In the above inequality, we have used $|\langle v\rangle^{m}\overline{M}^{\frac{1}{2}}\mu^{-\frac{1}{2}}|\leq C$ for any $m>0$, \eqref{boundAB}, \eqref{4.5A}, \eqref{DefE}, \eqref{DefD}, \eqref{apriori}, Lemma \ref{lem4.2} and
\begin{align*}
&\|\pa_x(\partial_t-\eps^{-1}\pa_x)[\frac{1}{\rho}K\theta B_{11}(\frac{v-u}{\sqrt{K\theta}})\frac{\mu}{M}]\|_{L^2_xL^2_v}
\notag\\
\leq&C(\|\pa_x\pa_t(\rho,u,\theta)\|+\|\pa_x(\rho,u,\theta)\|_{L^\infty}\|\pa_t(\rho,u,\theta)\|)
\notag\\
&+C\frac{1}{\eps}(\|\pa^2_x(\rho,u,\theta)\|+\|\pa_x(\rho,u,\theta)\|_{L^\infty}\|\pa_x(\rho,u,\theta)\|)
\notag\\
\leq& C(\eta)\frac{1}{\eps}+CA\frac{1}{\de}+C(\eta).
\end{align*}
For the last term on the right hand side of \eqref{defI112}, we first use the similar arguments as in \eqref{1I112} to get
\begin{align}\label{2I112}
	&|\eps\de\int_{\mathbb{R}}\int_{\mathbb{R}^{3}}\pa^2_x
	\big\{[\frac{1}{\rho}K\theta B_{11}(\frac{v-u}{\sqrt{K\theta}})\frac{1}{M}]h\big\}\pa_x(\partial_t-\eps^{-1}\pa_x)\widetilde{u}_1\,dv\,dx|\notag\\
	\leq& |\eps\de\int_{\mathbb{R}}\int_{\mathbb{R}^{3}}
	\frac{1}{\rho}K\theta B_{11}(\frac{v-u}{\sqrt{K\theta}})\frac{1}{M}\pa^2_xh\pa_x(\partial_t-\eps^{-1}\pa_x)\widetilde{u}_1\,dv\,dx|
    \notag\\
    &+ C\ka\eps\de\|\pa^2_x(\widetilde{\rho},\widetilde{u},\theta)\|^2+C_\ka(\de+\eps)\CD(t)+C(\eta)(\eps^3\de+\eps^4).
\end{align}
It remains to estimate the first term on the right hand side, which, combined with the second equation in \eqref{perturbeqA}, shows that
\begin{align}\label{3I112}
	&\eps\de|\int_{\mathbb{R}}\int_{\mathbb{R}^{3}}\frac{1}{\rho}K\theta B_{11}(\frac{v-u}{\sqrt{K\theta}})\frac{1}{M}\pa^2_xh\pa_x(\partial_t-\eps^{-1}\pa_x)\widetilde{u}_1\,dv\,dx|\notag\\
	=&\eps\de|\int_{\mathbb{R}}\int_{\mathbb{R}^{3}}\frac{1}{\rho}K\theta B_{11}(\frac{v-u}{\sqrt{K\theta}})\frac{1}{M}\pa^2_xh\pa_x(u_{1}\partial_x\widetilde{u}_{1}+\widetilde{u}_{1}\partial_x\bar{u}_{1}+\frac{2}{3}\partial_x\widetilde{\theta}
	+\frac{2}{3}(\frac{\theta}{\rho}-\frac{\bar{\theta}}{\bar{\rho}})\partial_x\rho+\frac{2}{3}\frac{\bar{\theta}}{\bar{\rho}}\partial_x\widetilde{\rho}\notag\\
	&\qquad\qquad\qquad\qquad\qquad\qquad\qquad\qquad\qquad
	+\frac{1}{\eps}\partial_x\widetilde{\phi}+\frac{1}{\rho}\int_{\mathbb{R}^3} \mathbf{v}^2_{1}\partial_xG(t,x,\mathbf{v})\,d\mathbf{v}+\eps^3	\partial_x[\frac{1}{\bar{\rho}}R_2])\,dv\,dx|,
\end{align}
where $\mathbf{v}=(\mathbf{v}_1,\mathbf{v}_2,\mathbf{v}_3)$. We perform similar calculations as in \eqref{1I112} to get 
\begin{align}\label{4I112}
	&\eps\de\int_{\mathbb{R}}\int_{\mathbb{R}^{3}}\frac{1}{\rho}K\theta B_{11}(\frac{v-u}{\sqrt{K\theta}})\frac{1}{M}\pa^2_xh\pa_x(u_{1}\partial_x\widetilde{u}_{1}+\widetilde{u}_{1}\partial_x\bar{u}_{1}+\frac{2}{3}\partial_x\widetilde{\theta}\notag\\
	&\qquad\qquad\qquad\qquad\qquad\qquad\qquad\qquad\qquad
	+\frac{2}{3}(\frac{\theta}{\rho}-\frac{\bar{\theta}}{\bar{\rho}})\partial_x\rho+\frac{2}{3}\frac{\bar{\theta}}{\bar{\rho}}\partial_x\widetilde{\rho}
	+\eps^3	\partial_x[\frac{1}{\bar{\rho}}R_2])\,dv\,dx
    \notag\\
\leq&C\ka\eps\de\|\pa^2_x(\widetilde{\rho},\widetilde{u},\widetilde{\theta})\|^{2}+C_\ka\eps\de\|\langle v\rangle^{-10}\mu^{-1/2}\pa^2_xh \|^{2}+C_\ka(\de+\eps)\CD(t)
    +C_\ka C(\eta)\CE(t)+C_\ka C(\eta)\eps^3\de
    \notag\\
	\leq&C\ka\eps\de\|\pa^2_x(\widetilde{\rho},\widetilde{u},\widetilde{\theta})\|^{2}
    +C_\ka(\eps+\de)\CD(t)+C_\ka\eps\de\|\pa^2_x(g,f)\|_{\sigma}^{2}+C_\ka C(\eta)(\CE(t)+\eps^3\de),
\end{align}
for any $0<\ka<1$. Here in the last inequality we have used
$h=\sqrt{\overline{M}}f(t,x,v)+\sqrt{\mu}g(t,x,v)$.
On the other hand, it holds that
\begin{align}\label{41I112}
	&\eps\de|\int_{\mathbb{R}}\int_{\mathbb{R}^{3}}\frac{1}{\rho}K\theta B_{11}(\frac{v-u}{\sqrt{K\theta}})\frac{1}{M}\pa^2_xh\pa_x(\frac{1}{\rho}\int_{\mathbb{R}^3} \mathbf{v}^2_{1}\partial_x\overline{G}(t,x,\mathbf{v})\,d\mathbf{v})\,dv\,dx|
    \notag\\
	\leq&\eps\de|\int_{\mathbb{R}}\int_{\mathbb{R}^{3}}\frac{1}{\rho}K\theta B_{11}(\frac{v-u}{\sqrt{K\theta}})\frac{1}{M}\pa^2_xh\frac{1}{\rho}\int_{\mathbb{R}^3} \mathbf{v}^2_{1}\partial^2_x\overline{G}(t,x,\mathbf{v})\,d\mathbf{v}\,dv\,dx|\notag\\
	&+\eps\de|\int_{\mathbb{R}}\int_{\mathbb{R}^{3}}\frac{1}{\rho}K\theta B_{11}(\frac{v-u}{\sqrt{K\theta}})\frac{1}{M}\pa^2_xh\pa_x(\frac{1}{\rho})\int_{\mathbb{R}^3} \mathbf{v}^2_{1}\partial_x\overline{G}(t,x,\mathbf{v})\,d\mathbf{v}\,dv\,dx|
    \notag\\
	\leq& C\eps\de\|\langle v\rangle^{-10}\mu^{-1/2}\pa^2_x( \sqrt{\overline{M}}f+\sqrt{\mu}g )\|\big(\|\partial^2_x\overline{G}\|+\|\pa_x\rho\|_{L^\infty_x}\|\partial_x\overline{G}\|\big)\notag\\
	\leq& C\eps\CD(t)+C(\eta)\eps^3\de,
\end{align}
and
\begin{align}\label{5I112}
	\eps\de|\int_{\mathbb{R}}\int_{\mathbb{R}^{3}}\frac{1}{\rho}K\theta B_{11}(\frac{v-u}{\sqrt{K\theta}})\frac{1}{M}\pa^2_xh\pa_x(\frac{1}{\rho}\int_{\mathbb{R}^3} \mathbf{v}^2_{1}\partial_xh(t,x,\mathbf{v})\,d\mathbf{v})\,dv\,dx|
	\leq C\eps\de\|\pa^2_x(g,f)\|_\sigma^2+C\eps\CD(t).
\end{align}
We turn to the term containing $\pa_x\widetilde{\phi}$ in \eqref{3I112} now. Integration by parts gives 
\begin{align}\label{6I112}
	&\de\Big|\int_{\mathbb{R}}\int_{\mathbb{R}^{3}}\frac{1}{\rho}K\theta B_{11}(\frac{v-u}{\sqrt{K\theta}})\frac{1}{M}\pa^2_xh\pa^2_x\widetilde{\phi}\,dv\,dx\Big|\notag\\
	\leq&\de|\int_{\mathbb{R}}\int_{\mathbb{R}^{3}}\pa_x(\frac{1}{\rho}K\theta B_{11}(\frac{v-u}{\sqrt{K\theta}})\frac{1}{M})\pa_xh\pa^2_x\widetilde{\phi}\,dv\,dx|+\de|\int_{\mathbb{R}}\int_{\mathbb{R}^{3}}\frac{1}{\rho}K\theta B_{11}(\frac{v-u}{\sqrt{K\theta}})\frac{1}{M}\pa_xh\pa^3_x\widetilde{\phi}\,dv\,dx|\notag\\
	\leq& C(\eps+\de)\CD(t)+C\de\times\frac{1}{\sqrt{\de\eps}}\|\langle v\rangle^{-\frac{1}{2}}\pa_x(g,f)\|\times\eps\sqrt \de\|\pa^3_x\widetilde{\phi}\|\times\frac{1}{\sqrt{\eps}}
    \notag\\
	\leq& C(\eps+\de+\frac{\de}{\sqrt\eps})\CD(t),
\end{align}
where in the last two inequalities we have used \eqref{boundlandaunorm},
\eqref{DefD}, \eqref{4.5A} and \eqref{apriori}.
It follows from \eqref{defI112}, \eqref{1I112}, \eqref{2I112}, \eqref{3I112}, \eqref{4I112}, \eqref{41I112}, \eqref{5I112} and \eqref{6I112} that 
\begin{align*}
	I_{112}\leq& \eps\de\frac{d}{dt}\int_{\mathbb{R}}\int_{\mathbb{R}^{3}}\pa_x
	\big\{\frac{1}{\rho}K\theta B_{11}(\frac{v-u}{\sqrt{K\theta}})\frac{1}{M}h\big\}\pa^2_x\widetilde{u}_1\,dv\,dx
+C\ka\eps\de\|\pa^2_x(\widetilde{\rho},\widetilde{u},\widetilde{\theta})\|^{2}
    \notag\\
	&\qquad +C_\ka(\de+\eps+\frac{\de}{\sqrt\eps})\CD(t)+C_\ka\eps\de\|\pa^2_x(g,f)\|_\sigma^{2}+C_\ka C(\eta)\CE(t)+C_\ka C(\eta)(\eps^4+\eps^3\de+\eps\de^3),
\end{align*}
which, combined with \eqref{defI11} and \eqref{I111}, yields
\begin{align}\label{I11}
	I_{11}\leq &\eps\de\frac{d}{dt}\int_{\mathbb{R}}\int_{\mathbb{R}^{3}}\pa_x
	\big\{\frac{1}{\rho}K\theta B_{11}(\frac{v-u}{\sqrt{K\theta}})\frac{1}{M}h\big\}\pa^2_x\widetilde{u}_1\,dv\,dx
+C\ka\eps\de\|\pa^2_x(\widetilde{\rho},\widetilde{u},\widetilde{\theta})\|^{2}
    \notag\\
    &+C_\ka(\de+\eps+\frac{\de}{\sqrt\eps})\CD(t)
	+C_\ka\eps\de\|\pa^2_x(g,f)\|_\sigma^{2}+C_\ka C(\eta)\CE(t)+C_\ka C(\eta)(\eps^4+\eps^3\de+\eps\de^3).
\end{align}
Similar to \eqref{I11}, we can obtain
\begin{align}\label{I1213}
	I_{12}+I_{13}\leq& C\ka\eps\de\|\pa^2_x(\widetilde{\rho},\widetilde{u},\widetilde{\theta})\|^{2}+
    C_\ka(\de+\eps+\frac{\de}{\sqrt\eps})\CD(t)\notag\\
	&+C_\ka\eps\de\|\pa^2_x(g,f)\|_\sigma^{2}+C_\ka C(\eta)\CE(t)+C_\ka C(\eta)(\eps^4+\eps^3\de+\eps\de^3).
\end{align}
It remains to estimate $I_{14}$ now. Using \eqref{controlpaGa}, \eqref{boundbarG0}, \eqref{boundbarG}, \eqref{4.5A}, \eqref{DefE}, \eqref{DefD} and \eqref{apriori}, we have
\begin{align}\label{I14}
	I_{14}=&-\int_{\mathbb{R}}\int_{\mathbb{R}^{3}}\pa_x[\frac{1}{\rho}K\theta B_{11}(\frac{v-u}{\sqrt{K\theta}})\frac{\sqrt{\mu}}{M}\Gamma(\frac{G}{\sqrt{\mu}},\frac{G}{\sqrt{\mu}})]\,\pa^2_x\widetilde{u}_1\,dv\,dx
	\notag\\
	\leq& C
	\int_{\mathbb{R}}|\pa_x(\rho,u,\theta)|\mu^a\frac{G}{\sqrt{\mu}}|_{2}|\frac{G}{\sqrt{\mu}}|_{\sigma}|\pa^2_x\widetilde{u}_1|\,dx
	\notag\\
	&+C
	\int_{\mathbb{R}}|\mu^a\frac{\pa_xG}{\sqrt{\mu}}|_{2}|\frac{G}{\sqrt{\mu}}|_{\sigma}|\pa^2_x\widetilde{u}_1|\,dx
	+C
	\int_{\mathbb{R}}|\mu^a\frac{G}{\sqrt{\mu}}|_{2}|\frac{\pa_xG}{\sqrt{\mu}}|_{\sigma}|\pa^2_x\widetilde{u}_1|\,dx\notag\\
	\leq&C(\eps+\de)\CD(t)+C(\eta)\CE(t)+C(\eta)\eps^3\de,
\end{align}
where the last inequality is achieved by
\begin{align*}
&\int_{\mathbb{R}}|\mu^a\frac{\pa_xG}{\sqrt{\mu}}|_{2}|\frac{G}{\sqrt{\mu}}|_{\sigma}|\pa^2_x\widetilde{u}_1|\,dx
\leq C\big\||\frac{G}{\sqrt{\mu}}|_{\sigma}\big\|_{L^\infty}
\big\||\mu^a\frac{\pa_xG}{\sqrt{\mu}}|_{2}\big\|
\|\pa^2_x\widetilde{u}_1\|
\notag\\
&\leq C(\|\frac{G}{\sqrt{\mu}}\|_{\sigma}+
\|\frac{\pa_xG}{\sqrt{\mu}}\|_{\sigma})
(\big\|\mu^a\frac{\pa_x\overline{G}}{\sqrt{\mu}}\big\|
+\big\|\mu^a\frac{\pa_x(\sqrt{\overline{M}}f+\sqrt{\mu}g)}{\sqrt{\mu}}\big\|)\|\pa^2_x\widetilde{u}_1\|
\notag\\ 
&\leq C(C(\eta)\de\eps+
\|(f,g)\|_{\sigma}+\|\pa_x(f,g)\|_{\sigma})
(C(\eta)\de\eps+ \|\langle v\rangle^{-\frac{1}{2}}f\|+\|\langle v\rangle^{-\frac{1}{2}}\pa_x(f,g)\|)\|\pa^2_x\widetilde{u}_1\|
\notag\\ 	
&\leq C(\eps+\de)\CD(t)+C(\eta)\CE(t)+C(\eta)\eps^3\de.
\end{align*}
The combination of \eqref{defI1}, \eqref{I11}, \eqref{I1213} and \eqref{I14} gives
\begin{align*}
	I_{1}\leq &\eps\de\frac{d}{dt}\int_{\mathbb{R}}\int_{\mathbb{R}^{3}}\pa_x
	\big\{\frac{1}{\rho}K\theta B_{11}(\frac{v-u}{\sqrt{K\theta}})\frac{1}{M}h\big\}\pa^2_x\widetilde{u}_1\,dv\,dx
+C\ka\eps\de\|\pa^2_x(\widetilde{\rho},\widetilde{u},\widetilde{\theta})\|^{2}
    \notag\\
	&\qquad+C_\ka(\de+\eps+\frac{\de}{\sqrt\eps})\CD(t)+C_\ka\eps\de\|\pa^2_x(g,f)\|_\sigma^{2}+C_\ka C(\eta)\CE(t)+C_\ka C(\eta)(\eps^4+\eps^3\de+\eps\de^3).
\end{align*}
The estimate for $I_2$, which is defined in \eqref{defI2}, can be obtained in the similar way that
\begin{align*}
	I_{2}\leq &\eps\de\frac{d}{dt}\int_{\mathbb{R}}\int_{\mathbb{R}^{3}}\pa_x
	\big\{\pa_x(\frac{1}{\rho})K\theta B_{11}(\frac{v-u}{\sqrt{K\theta}})\frac{1}{M}h\big\}\pa_x\widetilde{u}_1\,dv\,dx
    +C\ka\eps\de\|\pa^2_x(\widetilde{\rho},\widetilde{u},\widetilde{\theta})\|^{2}
    \notag\\
	&\qquad+C_\ka(\de+\eps+\frac{\de}{\sqrt\eps})\CD(t)+C_\ka\eps\de\|\pa^2_x(g,f)\|_\sigma^{2}+C_\ka C(\eta)\CE(t)+C_\ka C(\eta)(\eps^4+\eps^3\de+\eps\de^3).
\end{align*}
It follows from the above two inequalities and \eqref{I1I2} that
\begin{align}\label{1stutheta}
	&-(\pa_x[\frac{1}{\rho}\partial_x(\int_{\mathbb{R}^{3}} v_{1}^2L^{-1}_{M}\Theta\,dv)],\pa_x\widetilde{u}_1)
    \notag\\
	\leq &\eps\de\frac{d}{dt}\int_{\mathbb{R}}\int_{\mathbb{R}^{3}}
	\big\{\pa_x
	[\frac{1}{\rho}K\theta B_{11}(\frac{v-u}{\sqrt{K\theta}})\frac{1}{M}h]\pa^2_x\widetilde{u}_1+\pa_x[\pa_x(\frac{1}{\rho})K\theta B_{11}(\frac{v-u}{\sqrt{K\theta}})\frac{1}{M}h]\pa_x\widetilde{u}_1\big\}\,dv\,dx
    \notag\\
	&+C\ka\eps\de\|\pa^2_x(\widetilde{\rho},     \widetilde{u},\widetilde{\theta})\|^{2}+C_\ka(\de+\eps+\frac{\de}{\sqrt\eps})\CD(t)+C_\ka\eps\de\|\pa^2_x(g,f)\|_\sigma^{2}+C_\ka C(\eta)(\CE(t)+\eps^4+\eps^3\de+\eps\de^3)
    \notag\\
	= &-\eps\de\frac{d}{dt}\int_{\mathbb{R}}\int_{\mathbb{R}^{3}}
	\big\{\pa_x[\frac{1}{\rho}\partial_{x}(K\theta B_{11}(\frac{v-u}{\sqrt{K\theta}})\frac{1}{M}h)]
	\pa_x\widetilde{u}_1\big\}\,dv\,dx
    +C\ka\eps\de\|\pa^2_x(\widetilde{\rho},
    \widetilde{u},\widetilde{\theta})\|^{2}
    \notag\\
	&\qquad+C_\ka(\de+\eps+\frac{\de}{\sqrt\eps})\CD(t)+C_\ka\eps\de\|\pa^2_x(g,f)\|_\sigma^{2}+C_\ka C(\eta)\CE(t)+C_\ka C(\eta)(\eps^4+\eps^3\de+\eps\de^3).
\end{align}
Hence, the desired estimate \eqref{1stu} follows from \eqref{ipu}, \eqref{1stufluid}, \eqref{1stuphi}, \eqref{1studiss} and \eqref{1stutheta}. On the other hand, by performing similar calculations, we can prove \eqref{1stui} holds without giving all the details. This completes the proof of Lemma
\ref{le1stu}.
\end{proof}
The zero order estimates for $\widetilde{u}$ is more straightforward.
\begin{lemma}
Under the a priori assumption \eqref{apriori}, for any small $\ka>0$, there exists constants $0<c<1$ and $C_{\ka}>0$ such that
\begin{align}
		&\frac{1}{2}\frac{d}{dt}\{\|\widetilde{u}_1\|^{2}+(\frac{1}{\bar\rho}\widetilde{\phi},\widetilde{\phi})+\eps(\frac{1}{\bar\rho}\pa_x\widetilde{\phi},\pa_x\widetilde{\phi})\}+
	(\frac{2\bar{\theta}}{3\bar{\rho}}\pa_x\widetilde{\rho},\widetilde{u}_1)+\frac{2}{3}(\pa_x\widetilde{\theta},\widetilde{u}_1)+c\ep\de\|\pa_x\widetilde{u}_1\|^{2}\notag\\
	&\quad+\ep\de\frac{d}{dt}\int_{\mathbb{R}}\int_{\mathbb{R}^{3}}
	\big\{\frac{1}{\rho}\partial_{x}(K\theta B_{11}(\frac{v-u}{\sqrt{K\theta}})\frac{1}{M}h)
	\widetilde{u}_1\big\}\,dv\,dx
	\notag\\
	\leq &
     C\ka\eps\de\|\pa_x(\widetilde{\rho},     \widetilde{u},\widetilde{\theta})\|^{2}
    +C_\ka(\de+\eps+\frac{\de}{\sqrt\eps})\CD(t)+C_\ka C(\eta)(\CE(t)+\eps^4+\eps^3\de+\eps\de^3)+C\eps^4,
    \label{zerou}
\end{align}
and 
\begin{align}\label{zeroui}
	&\frac{1}{2}\frac{d}{dt}\|\widetilde{u}_i\|^{2}
	+c\ep\de\|\pa_x\widetilde{u}_i\|^{2}
	+\ep\de\frac{d}{dt}\int_{\mathbb{R}}\int_{\mathbb{R}^{3}}
	\big\{\frac{1}{\rho}\partial_{x}(K\theta B_{1i}(\frac{v-u}{\sqrt{K\theta}})\frac{1}{M}h)
	\widetilde{u}_i\big\}\,dv\,dx
    \notag\\
	\leq &
     C\ka\eps\de\|\pa_x(\widetilde{\rho},     \widetilde{u},\widetilde{\theta})\|^{2}
    +C_\ka(\de+\eps+\frac{\de}{\sqrt\eps})\CD(t)+C_\ka C(\eta)(\CE(t)+\eps^3\de+\eps\de^3)+C\eps^4.
\end{align}
\end{lemma}
\begin{proof}
The proof is similar to Lemma \ref{le1stu} and hence are omitted here
for simplicity.
\end{proof}

The following lemma shows the estimate of the temperature $\pa_x\widetilde{\theta}$.
\begin{lemma}\label{le1sttheta}
Under the a priori assumption \eqref{apriori}, for any small $\ka>0$, there exists constants $0<c<1$ and $C_{\ka}>0$ such that
\begin{align}\label{1sttheta}
	&\frac{1}{2}\frac{d}{dt}(\pa_x\widetilde{\theta},\frac{1}{\bar{\theta}}\pa_x\widetilde{\theta})+\frac{2}{3}(\pa^2_x\widetilde{u}_1,\pa_x\widetilde{\theta})
	+c\eps\de\|\pa^2_x\widetilde{\theta}\|^2+\eps\de\frac{d}{dt}\int_{\mathbb{R}}\int_{\mathbb{R}^{3}}
	\big\{\pa_x[\frac{1}{\rho}\partial_{x}
	((K\theta)^{\frac{3}{2}}A_{1}(\frac{v-u}{\sqrt{K\theta}})\frac{1}{M}h)]\frac{1}{\bar{\theta}}\pa_x\widetilde{\theta}\big\}\,dv\,dx
	\notag\\	
	&+\eps\de\sum^{3}_{j=1}\frac{d}{dt}\int_{\mathbb{R}}\int_{\mathbb{R}^{3}}
	\big\{\pa_x[\frac{1}{\rho}
	\partial_{x}u_j K\theta B_{1j}(\frac{v-u}{\sqrt{K\theta}})\frac{1}{M}h]\frac{1}{\bar{\theta}}\pa_x\widetilde{\theta}\big\}\,dv\,dx
	\notag\\
	\leq& C\ka\eps\de\|\pa^2_x(\widetilde{\rho},     \widetilde{u},\widetilde{\theta})\|^{2}
    +C_\ka(\de+\eps+\frac{\de}{\sqrt\eps})\CD(t)+C_\ka\eps\de\|\pa^2_x(g,f)\|_\sigma^2
    \notag\\
	&+C_\ka A\min\{(\frac{\eps}{\de}+\frac{\eps^2}{\de^2})\CD(t),(\frac{\eps^{3/2}}{\sqrt{\de}}+\frac{\eps^{3}}{\de})\CE(t)\}+C_\ka C(\eta)\CE(t)+C_\ka C(\eta)(\eps^4+\eps^3\de+\eps\de^3),
\end{align}
where $C(\eta)\geq0$ depends only on $\eta$ with $C(0)=0$.
\end{lemma}	
\begin{proof}
We have from the fourth equation of \eqref{perturbeq} that
	\begin{align}\label{iptheta}
		&\frac{1}{2}\frac{d}{dt}(\pa_x\widetilde{\theta},\frac{1}{\bar{\theta}}\pa_x\widetilde{\theta})
        \notag 		\\
        =&\frac{1}{2}(\pa_x\widetilde{\theta},\partial_{t}(\frac{1}{\bar{\theta}})\pa_x\widetilde{\theta})
		+\frac{1}{\eps}(\pa^2_x\widetilde{\theta},\frac{1}{\bar{\theta}}\pa_x\widetilde{\theta})-\frac{2}{3}(\pa_x[\bar{\theta}\partial_x\widetilde{u}_{1}],\frac{1}{\bar{\theta}}\pa_x\widetilde{\theta})-\frac{2}{3}(\pa_x(\widetilde{\theta}\partial_xu_1),\frac{1}{\bar{\theta}}\pa_x\widetilde{\theta})\notag
		\\
		&-(\pa_x(u_1\partial_x\widetilde{\theta}),\frac{1}{\bar{\theta}}\pa_x\widetilde{\theta})
		-(\pa_x(\widetilde{u}_1\partial_{x}\bar{\theta}),\frac{1}{\bar{\theta}}\pa_x\widetilde{\theta})
		+\eps\de(\pa_x[\frac{1}{\rho}\partial_x(\kappa(\theta)\partial_x\theta)],\frac{1}{\bar{\theta}}\pa_x\widetilde{\theta})
		\notag\\
		&+\eps\de(\pa_x\{\frac{4}{3\rho}\mu(\theta)(\partial_xu_{1})^2
		+\frac{1}{\rho}\mu(\theta)[(\partial_xu_{2})^2+(\partial_xu_{3})^2]\},\frac{1}{\bar{\theta}}\pa_x\widetilde{\theta})
        -\eps^3(\pa_xR_3,\frac{1}{\bar{\theta}}\pa_x\widetilde{\theta})
		\notag\\
		&+(\pa_x[\frac{1}{\rho}u\cdot\partial_x(\int_{\mathbb{R}^3} v_{1}v L^{-1}_{M}\Theta\, dv)] 		,\frac{1}{\bar{\theta}}\pa_x\widetilde{\theta})-(\pa_x[\frac{1}{\rho}\partial_x(\int_{\mathbb{R}^3}v_{1}\frac{|v|^{2}}{2}L^{-1}_{M}\Theta\, dv)],\frac{1}{\bar{\theta}}\pa_x\widetilde{\theta}).
	\end{align}
Most of the terms above have similar structures to the corresponding ones in the proofs of Lemma \ref{le1strho} and Lemma \ref{le1stu}. Thus, we mainly focus on the different terms. Similar calculations as in \eqref{1rho1}, \eqref{1rho2}, \eqref{1rho3}, \eqref{1rho4}, \eqref{1rho5} and \eqref{1studiss}  show that the second, third, fourth lines in \eqref{iptheta} are bounded by
\begin{align}\label{1theta1}
&-c\eps\de\|\pa^2_x\widetilde{\theta}\|^2-\frac{2}{3}(\pa^2_x\widetilde{u}_1,\pa_x\widetilde{\theta})+C_\ka A\min\{(\frac{\eps}{\de}+\frac{\eps^2}{\de^2})\CD(t),(\frac{\eps^{3/2}}{\sqrt{\de}}+\frac{\eps^{3}}{\de})\CE(t)\}
\notag\\
& +C\ka\eps\de\|\pa^2_x(\widetilde{\rho},     \widetilde{u},\widetilde{\theta})\|^{2}
+C_\ka(\de+\eps)\CD(t)+C(\eta)\mathcal{E}(t)+C(\eta)(\eps^4+\eps^2\de^2).
\end{align}
For the last two terms in \eqref{iptheta},
we first use \eqref{reB} and the property
\begin{equation*}
	\int_{\mathbb{R}^{3}}(\frac{1}{2}v_{i}|v|^{2}-v_{i}u\cdot v)L^{-1}_{M}\Theta \,dv
	=\int_{\mathbb{R}^{3}} L^{-1}_{M}\{(K\theta)^{\frac{3}{2}}\hat{A}_{i}(\frac{v-u}{\sqrt{K\theta}})M\}\frac{\Theta}{M}\,dv
	=(K\theta)^{\frac{3}{2}}\int_{\mathbb{R}^{3}}A_{i}(\frac{v-u}{\sqrt{K\theta}})\frac{\Theta}{M}\,dv
\end{equation*}
to rewrite
\begin{align*}
	&(\pa_x[\frac{1}{\rho}u\cdot\partial_x(\int_{\mathbb{R}^3} v_{1}v L^{-1}_{M}\Theta\, dv)]
	,\frac{1}{\bar{\theta}}\pa_x\widetilde{\theta})
	-(\pa_x[\frac{1}{\rho}\partial_x(\int_{\mathbb{R}^3}v_{1}\frac{|v|^{2}}{2}L^{-1}_{M}\Theta\, dv)],\frac{1}{\bar{\theta}}\pa_x\widetilde{\theta})	
	\notag\\	
	=&-(\pa_x[\frac{1}{\rho}\partial_{x}(\int_{\mathbb{R}^{3}}
	(\frac{1}{2}|v|^{2}v_1-u\cdot vv_{1})L^{-1}_{M}\Theta\,dv)],\frac{1}{\bar{\theta}}\pa_x\widetilde{\theta})-
	(\pa_x[\frac{1}{\rho}\int_{\mathbb{R}^{3}}
	\partial_{x}u\cdot vv_{1}L^{-1}_{M}\Theta\,dv],\frac{1}{\bar{\theta}}\pa_x\widetilde{\theta})
	\notag\\
	=&-\int_{\mathbb{R}}\int_{\mathbb{R}^{3}}\pa_x[\frac{1}{\rho}\partial_{x}((K\theta)^{\frac{3}{2}}A_{1}(\frac{v-u}{\sqrt{K\theta}})\frac{\Theta}{M})]\,\frac{1}{\bar{\theta}}\pa_x\widetilde{\theta}\,dv\,dx-\sum^{3}_{j=1}\int_{\mathbb{R}}\int_{\mathbb{R}^{3}}\pa_x[\frac{1}{\rho}
	\partial_{x}u_j K\theta B_{1j}(\frac{v-u}{\sqrt{K\theta}})\frac{\Theta}{M}]\,\frac{1}{\bar{\theta}}\pa_x\widetilde{\theta}\,dv\,dx.
\end{align*}
We notice that both two terms on the right hand side above enjoy the similar structure to $I_2$ defined in \eqref{defI2}, and, after integration by parts, $I_1$ defined in \eqref{defI2}. Hence, the approach of how we get \eqref{I11}, \eqref{I1213} and \eqref{I14} gives
\begin{align}\label{1theta4}
&-\int_{\mathbb{R}}\int_{\mathbb{R}^{3}}\pa_x[\frac{1}{\rho}\partial_{x}((K\theta)^{\frac{3}{2}}A_{1}(\frac{v-u}{\sqrt{K\theta}})\frac{\Theta}{M})]\,\frac{1}{\bar{\theta}}\pa_x\widetilde{\theta}\,dv\,dx-\sum^{3}_{j=1}\int_{\mathbb{R}}\int_{\mathbb{R}^{3}}\pa_x[\frac{1}{\rho}
\partial_{x}u_j K\theta B_{1j}(\frac{v-u}{\sqrt{K\theta}})\frac{\Theta}{M}]\,\frac{1}{\bar{\theta}}\pa_x\widetilde{\theta}\,dv\,dx
\notag\\
\leq& 	-\eps\de\frac{d}{dt}\int_{\mathbb{R}}\int_{\mathbb{R}^{3}}
\big\{\pa_x[\frac{1}{\rho}\partial_{x}
((K\theta)^{\frac{3}{2}}A_{1}(\frac{v-u}{\sqrt{K\theta}})\frac{1}{M}h)+\sum^{3}_{j=1}\frac{1}{\rho}
\partial_{x}u_j K\theta B_{1j}(\frac{v-u}{\sqrt{K\theta}})\frac{1}{M}h]\frac{1}{\bar{\theta}}\pa_x\widetilde{\theta}\big\}\,dv\,dx
\notag\\
&+C\ka\eps\de\|\pa^2_x(\widetilde{\rho},     \widetilde{u},\widetilde{\theta})\|^{2}
+C_\ka(\de+\eps+\frac{\de}{\sqrt\eps})\CD(t)+C_\ka\eps\de\|\pa^2_x(g,f)\|_\sigma^{2}+C_\ka C(\eta)\CE(t)+C_\ka C(\eta)(\eps^4+\eps^3\de+\eps\de^3).
\end{align}
Hence, \eqref{1sttheta} follows from \eqref{iptheta}, \eqref{1theta1} and \eqref{1theta4}. We thus complete the proof of Lemma \ref{le1sttheta}.
\end{proof}	
The zero order estimate of $\widetilde{\theta}$ can be obtained in the same way.
\begin{lemma}
Under the a priori assumption \eqref{apriori}, 
for any small $\ka>0$, there exists constants $0<c<1$ and $C_{\ka}>0$ such that
	\begin{align}\label{zerotheta}
		&\frac{1}{2}\frac{d}{dt}(\widetilde{\theta},\frac{1}{\bar{\theta}}\widetilde{\theta})+\frac{2}{3}(\pa_x\widetilde{u}_1,\widetilde{\theta})
		+c\eps\de\|\pa_x\widetilde{\theta}\|^2+\eps\de\frac{d}{dt}\int_{\mathbb{R}}\int_{\mathbb{R}^{3}}
		\big\{\frac{1}{\rho}\partial_{x}
		((K\theta)^{\frac{3}{2}}A_{1}(\frac{v-u}{\sqrt{K\theta}})\frac{1}{M}h)\frac{1}{\bar{\theta}}\widetilde{\theta}\big\}\,dv\,dx
		\notag\\	
		&+\eps\de\sum^{3}_{j=1}\frac{d}{dt}\int_{\mathbb{R}}\int_{\mathbb{R}^{3}}
		\big\{\frac{1}{\rho}
		\partial_{x}u_j K\theta B_{1j}(\frac{v-u}{\sqrt{K\theta}})\frac{1}{M}h\frac{1}{\bar{\theta}}\widetilde{\theta}\big\}\,dv\,dx
		\notag\\
		\leq&
        C\ka\eps\de\|\pa_x(\widetilde{\rho},     \widetilde{u},\widetilde{\theta})\|^{2}
        +C_\ka(\de+\eps+\frac{\de}{\sqrt\eps})\CD(t)+C_\ka C(\eta) (\CE(t)+\eps^3\de+\eps\de^3)+C\eps^4.
	\end{align}
\end{lemma}
\begin{proof}
	The proof is very similar to Lemma \ref{le1sttheta} and hence omitted here.
\end{proof}
We see from the above three lemmas that in order to complete our estimates on the fluid part, it is also required to get the dissipation of mass and electric field from the first two equations in \eqref{perturbeq} where the $O(\eps^{-1})$ terms in the transport terms will remain since good properties in previous energy lemmas like $\eps^{-1}(\pa^2_x u_1,\pa_x u_1)$ no longer appear here. We present a key observation to use the interaction between the two singular transport operators on the density and velocity to cancel the $O(\eps^{-1})$ terms. 
\begin{lemma}\label{lem4.8}
		Under the a priori assumption \eqref{apriori}, it holds that
	\begin{align}\label{zerodissrho}
		&\eps\de\frac{d}{dt}(\widetilde{u}_1,\pa_x\widetilde{\rho})
		+c\eps\de(\|\pa_x\widetilde{\rho}\|^2+\|\pa_x\widetilde{\phi}\|^2
		+\eps\|\pa^2_x\widetilde{\phi}\|^2)
        \notag\\
		&\leq C\eps\de\|\pa_x(\widetilde{u},\widetilde{\theta})\|^2+ C\eps\CD(t)+C(\eta)\CE(t)
		+C(\eta)(\eps^4+\eps^2\de^2),
	\end{align}
and	
	\begin{align}\label{dissrho}
		&\eps\de\frac{d}{dt}(\pa_x\widetilde{u}_1,\pa^2_x\widetilde{\rho})
		+c\eps\de(\|\pa^2_x\widetilde{\rho}\|^2+\|\pa^2_x\widetilde{\phi}\|^2
		+\eps\|\pa^3_x\widetilde{\phi}\|^2)\notag\\
		\leq& C\eps\de(\|\pa^2_x(\widetilde{u},\widetilde{\theta})\|^2+\|\langle v\rangle^{-\frac{1}{2}}\pa^2_x(g,f)\|^2)+C\eps\CD(t)+C(\eta)\CE(t)
		+C(\eta)(\eps^4+\eps^2\de^2),
	\end{align}
where $C(\eta)\geq0$ depends only on $\eta$ with $C(0)=0$.
\end{lemma}	
\begin{proof}
We prove \eqref{dissrho} first since it is more general. Differentiating the second equation of \eqref{perturbeqA} in $x$ and taking the inner product with $\eps\de\pa^2_x\widetilde{\rho}$, we have	
	\begin{align}\label{dissrho1}
		&\eps\de(\pa_x\partial_t\widetilde{u}_{1},\pa^2_x\widetilde{\rho})-\de(\partial^2_x\widetilde{u}_{1},\pa^2_x\widetilde{\rho})+\de(\pa^2_x\widetilde{\phi},\pa^2_x\widetilde{\rho})
		+\eps\de(\frac{2\bar{\theta}}{3\bar{\rho}}\pa^2_x\widetilde{\rho},\pa^2_x\widetilde{\rho})+\eps\de(\pa_x(\frac{2\bar{\theta}}{3\bar{\rho}})\partial_x\widetilde{\rho}
		,\pa^2_x\widetilde{\rho})\notag\\
		=&-
		\eps\de(\pa_x[
		u_{1}\partial_x\widetilde{u}_{1}+\widetilde{u}_{1}\partial_x\bar{u}_{1}+\frac{2}{3}\partial_x\widetilde{\theta}
		+\frac{2}{3}(\frac{\theta}{\rho}-\frac{\bar{\theta}}{\bar{\rho}})\partial_x\rho+\eps^3\frac{1}{\bar{\rho}}R_2],\pa^2_x\widetilde{\rho})
		-\eps\de(\pa_x[\frac{1}{\rho}\int_{\mathbb{R}^3} v^2_{1}\partial_xG\,dv],\pa^2_x\widetilde{\rho}).
	\end{align}
Notice that the coefficient $\eps^{-1}$ in the second equation of \eqref{perturbeq} causes some singularity in the term $-\de(\partial^2_x\widetilde{u}_{1},\pa^2_x\widetilde{\rho})+\de(\pa^2_x\widetilde{\phi},\pa^2_x\widetilde{\rho})$, to which we should pay extra attention. The rest terms can be bounded as before. We first write
\begin{align*}
	\eps\de(\pa_x\partial_t\widetilde{u}_{1},\pa^2_x\widetilde{\rho})-\de(\partial^2_x\widetilde{u}_{1},\pa^2_x\widetilde{\rho})=&\eps\de([\partial_t-\frac{1}{\eps}\partial_x]\pa_x\widetilde{u}_{1},\pa^2_x\widetilde{\rho})\notag\\
	=&\frac{d}{dt}\eps\de(\pa_x\widetilde{u}_{1},\pa^2_x\widetilde{\rho})+\eps\de(\pa^2_x\widetilde{u}_{1},\pa_x[\partial_t-\frac{1}{\eps}\partial_x]\widetilde{\rho}).
\end{align*}
Using the first equation of \eqref{perturbeqA}, we arrive at
\begin{align*}
	\eps\de(\pa_x\partial_t\widetilde{u}_{1},\pa^2_x\widetilde{\rho})-\de(\partial^2_x\widetilde{u}_{1},\pa^2_x\widetilde{\rho})=\frac{d}{dt}\eps\de(\pa_x\widetilde{u}_{1},\pa^2_x\widetilde{\rho})-\eps\de(\pa^2_x\widetilde{u}_{1},\pa_x[\partial_x(\bar{\rho}\widetilde{u}_{1})+\partial_x(\widetilde{\rho}u_{1})+\eps^3R_1]).
\end{align*}
By \eqref{DefE}, \eqref{DefD}, \eqref{apriori}, \eqref{boundkdv} and \eqref{2.61AA}, we obtain
\begin{align*}
	|\eps\de(\pa^2_x\widetilde{u}_{1},\pa_x[\partial_x(\bar{\rho}\widetilde{u}_{1})+\partial_x(\widetilde{\rho}u_{1})+\eps^3R_1])|\leq C\eps\de\|\pa^2_x\widetilde{u}\|^2+C\eps\CD(t)+C(\eta)\CE(t)+C(\eta)\eps^4,
\end{align*}
which, together with the identities above, yields
\begin{align}\label{dissrho2}
	\eps\de(\pa_x\partial_t\widetilde{u}_{1},\pa^2_x\widetilde{\rho})-\de(\partial^2_x\widetilde{u}_{1},\pa^2_x\widetilde{\rho})\geq \frac{d}{dt}\eps\de(\pa_x\widetilde{u}_{1},\pa^2_x\widetilde{\rho})-C\eps\de\|\pa^2_x\widetilde{u}\|^2-C\eps\CD(t)-C(\eta)\CE(t)-C(\eta)\eps^4.
\end{align}
Now we consider the term involving $\|\pa^2_x\widetilde{\phi}\|$ in \eqref{dissrho1}. Substituting the last equation of \eqref{perturbeq} into $\de(\pa^2_x\widetilde{\phi},\pa^2_x\widetilde{\rho})$ gives
\begin{align}\label{dissrho3}
	\de(\pa^2_x\widetilde{\phi},\pa^2_x\widetilde{\rho})=\de(\pa^2_x\widetilde{\phi},-\eps^2\pa_x^4\widetilde{\phi}+\eps\pa_x^2\widetilde{\phi}+\eps^4\pa_x^2R_4)\geq \eps\de(\|\pa^2_x\widetilde{\phi}\|^2
	+\eps\|\pa^3_x\widetilde{\phi}\|^2)-C(\eta)\eps^4. 
\end{align}
At last, the rest terms in \eqref{dissrho1} can be controlled as in the proofs of Lemma \ref{le1strho} and Lemma \ref{le1stu} that
\begin{align}\label{dissrho4}
	&\big|\eps\de(\pa_x[
	u_{1}\partial_x\widetilde{u}_{1}+\widetilde{u}_{1}\partial_x\bar{u}_{1}+\frac{2}{3}\partial_x\widetilde{\theta}
	+\frac{2}{3}(\frac{\theta}{\rho}-\frac{\bar{\theta}}{\bar{\rho}})\partial_x\rho+\eps^3\frac{1}{\bar{\rho}}R_2],\pa^2_x\widetilde{\rho})
	-\eps\de(\pa_x[\frac{1}{\rho}\int_{\mathbb{R}^3} v^2_{1}\partial_xG\,dv],\pa^2_x\widetilde{\rho})\big|\notag\\
	\leq&\ka\eps\de\|\pa^2_x\widetilde{\rho}\|^2+C_\ka\eps\de(\|\pa^2_x(\widetilde{u},\widetilde{\theta})\|^2+\|\langle v\rangle^{-\frac{1}{2}}\pa^2_x(g,f)\|^2)+C\eps\CD(t)+C_\ka C(\eta)\CE(t)+C_\ka C(\eta)(\eps^4+\eps^2\de^2),
\end{align}
for any $0<\ka<1.$ Hence, \eqref{dissrho} holds by collecting \eqref{dissrho1}, \eqref{dissrho2}, \eqref{dissrho3}, \eqref{dissrho4} and choosing $\ka$ to be small.

Using very similar arguments, we can also obtain \eqref{zerodissrho} without giving all details of the proof. This completes the proof of 
Lemma \ref{lem4.8}.
\end{proof}

With all above lemmas, now we can obtain the final estimate for first order fluid part.

\begin{lemma}\label{le1stfluid}
Under the a priori assumption \eqref{apriori}, it holds that
\begin{align}\label{1stfluid}
&\|\pa_x(\widetilde{\rho},\widetilde{u},\widetilde{\theta},\widetilde{\phi})(t)\|^{2}+\eps\|\pa^2_x\widetilde{\phi}(t)\|^2+\eps\de\int^t_0\{\|\pa^2_x(\widetilde{\rho},\widetilde{u},\widetilde{\theta},\widetilde{\phi})(s)\|^{2}+\eps\|\pa^3_x\widetilde{\phi}(s)\|^{2}\}\,ds
\notag\\
\leq&C\CE(0)+C\eps^2\de^2\|\pa^2_x\widetilde{\rho}(t)\|^2+C\eps^2\de^2\|\pa^2_x(g,f)(t)\|^2+
C\eps\de\int_0^t\|\pa^2_x(g,f)(s)\|_\sigma^{2}\,ds
\notag\\
&+C(\de+\eps+\frac{\de}{\sqrt\eps})\int_0^t\CD(s)\,ds+CA\int_0^t\min\{(\frac{\eps}{\de}+\frac{\eps^2}{\de^2})\CD(s),(\frac{\eps^{3/2}}{\sqrt{\de}}+\frac{\eps^{3}}{\de})\CE(s)\}\,ds
\notag\\
&+C(\eta)\int^t_0\CE(s)\,ds+C(\eta)t(\eps^4+\eps^3\de+\eps\de^3)+C\eps^5,
\end{align}	
for any $0<t<T$, where $C(\eta)\geq0$ depends only on $\eta$ with $C(0)=0$.
\end{lemma}	
\begin{proof}
	The summation of \eqref{1strho}, \eqref{1stu}, \eqref{1stui} and \eqref{1sttheta} gives
\begin{align}\label{1fluid1}
	&\frac{1}{2}\frac{d}{dt}\{(\pa_x\widetilde{\rho},\frac{2\bar{\theta}}{3\bar{\rho}^{2}}\pa_x\widetilde{\rho})+\|\pa_x\widetilde{u}\|^{2}+(\frac{1}{\bar\rho}\pa_x\widetilde{\phi},\pa_x\widetilde{\phi})+\eps(\frac{1}{\bar\rho}\pa^2_x\widetilde{\phi},\pa^2_x\widetilde{\phi})+(\pa_x\widetilde{\theta},\frac{1}{\bar{\theta}}\pa_x\widetilde{\theta})\}
	+c\ep\de\|\pa^2_x(\widetilde{u},\widetilde{\theta})\|^{2}\notag\\
	&+\ep\de\sum^{3}_{i=1}\frac{d}{dt}\int_{\mathbb{R}}\int_{\mathbb{R}^{3}}
	\big\{\pa_x[\frac{1}{\rho}\partial_{x}(K\theta B_{11}(\frac{v-u}{\sqrt{K\theta}})\frac{1}{M}h)]
	\pa_x\widetilde{u}_i\big\}\,dv\,dx
	\notag\\
	&+\eps\de\frac{d}{dt}\int_{\mathbb{R}}\int_{\mathbb{R}^{3}}
	\big\{\pa_x[\frac{1}{\rho}\partial_{x}
	((K\theta)^{\frac{3}{2}}A_{1}(\frac{v-u}{\sqrt{K\theta}})\frac{1}{M}h)+\sum^{3}_{j=1}\frac{1}{\rho}
	\partial_{x}u_j K\theta B_{1j}(\frac{v-u}{\sqrt{K\theta}})\frac{1}{M}h]\frac{1}{\bar{\theta}}\pa_x\widetilde{\theta}\big\}\,dv\,dx
	\notag\\
	\leq& C\ka\eps\de\|\pa^2_x(\widetilde{\rho},     \widetilde{u},\widetilde{\theta})\|^{2}
	 +C_\ka(\de+\eps+\frac{\de}{\sqrt\eps})\CD(t)+C_\ka\eps\de\|\pa^2_x(g,f)\|_\sigma^2
     \notag\\
	 &+C_\ka A\min\{(\frac{\eps}{\de}+\frac{\eps^2}{\de^2})\CD(t),(\frac{\eps^{3/2}}{\sqrt{\de}}+\frac{\eps^{3}}{\de})\CE(t)\}+C_\ka C(\eta)\CE(t)+C_\ka C(\eta)(\eps^4+\eps^3\de+\eps\de^3),
\end{align}
where we have used the fact that
$$
\big|(\pa^2_x\widetilde{u}_1,\frac{2\bar{\theta}}{3\bar{\rho}}\pa_x\widetilde{\rho})+
(\frac{2\bar{\theta}}{3\bar{\rho}}\pa^2_x\widetilde{\rho},\pa_x\widetilde{u}_1)+\frac{2}{3}(\pa^2_x\widetilde{\theta},\pa_x\widetilde{u}_1)+\frac{2}{3}(\pa^2_x\widetilde{u}_1,\pa_x\widetilde{\theta})\big|=|(\pa_x\widetilde{u}_1,\pa_x(\frac{2\bar{\theta}}{3\bar{\rho}})\pa_x\widetilde{\rho})|\leq C(\eta)\eps\CE(t).
$$
Then we multiply \eqref{dissrho} with $\ka_1>0$ and add it to \eqref{1fluid1} to get
\begin{align*}
	&\frac{1}{2}\frac{d}{dt}\{(\pa_x\widetilde{\rho},\frac{2\bar{\theta}}{3\bar{\rho}^{2}}\pa_x\widetilde{\rho})+\|\pa_x\widetilde{u}\|^{2}+(\frac{1}{\bar\rho}\pa_x\widetilde{\phi},\pa_x\widetilde{\phi})+\eps(\frac{1}{\bar\rho}\pa^2_x\widetilde{\phi},\pa^2_x\widetilde{\phi})+(\pa_x\widetilde{\theta},\frac{1}{\bar{\theta}}\pa_x\widetilde{\theta})\}+\ka_1\eps\de\frac{d}{dt}(\pa_x\widetilde{u}_1,\pa^2_x\widetilde{\rho})\notag\\
	&
	+c\eps\de\{\|\pa^2_x(\widetilde{u},\widetilde{\theta})\|^{2}+\ka_1\|\pa^2_x(\widetilde{\rho},\widetilde{\phi})\|^{2}+\ka_1\eps\|\pa^3_x\widetilde{\phi}\|^{2}\}+\eps\de\frac{d}{dt}\Big\{\sum^{3}_{i=1}\int_{\mathbb{R}}\int_{\mathbb{R}^{3}}
	\pa_x[\frac{1}{\rho}\partial_{x}(K\theta B_{11}(\frac{v-u}{\sqrt{K\theta}})\frac{1}{M}h)]
	\pa_x\widetilde{u}_i
	\notag\\
	&\qquad+
	\pa_x[\frac{1}{\rho}\partial_{x}
	((K\theta)^{\frac{3}{2}}A_{1}(\frac{v-u}{\sqrt{K\theta}})\frac{1}{M}h)+\sum^{3}_{j=1}\frac{1}{\rho}
	\partial_{x}u_j K\theta B_{1j}(\frac{v-u}{\sqrt{K\theta}})\frac{1}{M}h]\frac{1}{\bar{\theta}}\pa_x\widetilde{\theta}\Big\}\,dv\,dx
	\notag\\
	\leq&C\ka_1\eps\de(\|\pa^2_x(\widetilde{u},\widetilde{\theta})\|^2)
    +C\ka\eps\de\|\pa^2_x(\widetilde{\rho},     \widetilde{u},\widetilde{\theta})\|^{2}
    +(C\ka_1+C_\ka)\eps\de\|\pa^2_x(g,f)\|_\sigma^2+
C_\ka(\de+\eps+\frac{\de}{\sqrt\eps})\CD(t)\notag\\
	&+C_\ka A\min\{(\frac{\eps}{\de}+\frac{\eps^2}{\de^2})\CD(t),(\frac{\eps^{3/2}}{\sqrt{\de}}+\frac{\eps^{3}}{\de})\CE(t)\}+C_\ka C(\eta)\CE(t)+C_\ka C(\eta)(\eps^4+\eps^3\de+\eps\de^3),
\end{align*}
which, by choosing $\ka_1$ to be small first and then taking $\ka>0$ small enough, yields
\begin{align*}
		&\frac{1}{2}\frac{d}{dt}\{(\pa_x\widetilde{\rho},\frac{2\bar{\theta}}{3\bar{\rho}^{2}}\pa_x\widetilde{\rho})+\|\pa_x\widetilde{u}\|^{2}+(\frac{1}{\bar\rho}\pa_x\widetilde{\phi},\pa_x\widetilde{\phi})+\eps(\frac{1}{\bar\rho}\pa^2_x\widetilde{\phi},\pa^2_x\widetilde{\phi})+(\pa_x\widetilde{\theta},\frac{1}{\bar{\theta}}\pa_x\widetilde{\theta})\}+\ka_1\eps\de\frac{d}{dt}(\pa_x\widetilde{u}_1,\pa^2_x\widetilde{\rho})\notag\\
	&
	+c\eps\de\{\|\pa^2_x(\widetilde{\rho},\widetilde{u},\widetilde{\theta},\widetilde{\phi})\|^{2}+\eps\|\pa^3_x\widetilde{\phi}\|^{2}\}+\eps\de\sum^{3}_{i=1}\frac{d}{dt}\int_{\mathbb{R}}\int_{\mathbb{R}^{3}}
	\big\{\pa_x[\frac{1}{\rho}\partial_{x}(K\theta B_{11}(\frac{v-u}{\sqrt{K\theta}})\frac{1}{M}h)]
	\pa_x\widetilde{u}_i\big\}\,dv\,dx
	\notag\\
	&+\eps\de\frac{d}{dt}\int_{\mathbb{R}}\int_{\mathbb{R}^{3}}
	\big\{\pa_x[\frac{1}{\rho}\partial_{x}
	((K\theta)^{\frac{3}{2}}A_{1}(\frac{v-u}{\sqrt{K\theta}})\frac{1}{M}h)+\sum^{3}_{j=1}\frac{1}{\rho}
	\partial_{x}u_j K\theta B_{1j}(\frac{v-u}{\sqrt{K\theta}})\frac{1}{M}h]\frac{1}{\bar{\theta}}\pa_x\widetilde{\theta}\big\}\,dv\,dx
	\notag\\
	\leq&
	C(\de+\eps+\frac{\de}{\sqrt\eps})\CD(t)
    +C\eps\de\|\pa^2_x(g,f)\|_\sigma^2+C(\eta)\CE(t)
    \notag\\
	&+C A\min\{(\frac{\eps}{\de}+\frac{\eps^2}{\de^2})\CD(t),(\frac{\eps^{3/2}}{\sqrt{\de}}+\frac{\eps^{3}}{\de})\CE(t)\}+ C(\eta)(\eps^4+\eps^3\de+\eps\de^3).
\end{align*}
Consider $i=1$ for the second term in the second line above. Recalling $h=\sqrt{\overline{M}}f(t,x,v)+\sqrt{\mu}g(t,x,v)$, it follows from \eqref{boundAB}, \eqref{4.5A} and \eqref{apriori} that
\begin{align*}
	&\eps\de\int_{\mathbb{R}}\int_{\mathbb{R}^{3}}
	\big\{\pa_x[\frac{1}{\rho}\partial_{x}(K\theta B_{11}(\frac{v-u}{\sqrt{K\theta}})\frac{1}{M}h)]
	\pa_x\widetilde{u}_1\big\}\,dv\,dx\notag\\
	\leq&C\eps\de\big\{\|\pa_x(\rho,u,\theta)\|_{L^\infty}(\|\pa_x(\rho,u,\theta)\|\|(f,g)\|_{L^\infty}+\|\pa_x(f,g)\|+C(\eta)\|f\|)
    \|\pa_x\widetilde{u}_1\|
    \notag\\
	&\quad+\|\pa^2_x(\rho,u,\theta)\|\|(f,g)\|_{L^\infty}\|\pa_x\widetilde{u}_1\|+\|\pa^2_x(f,g)\|\|\pa_x\widetilde{u}_1\|\big\}
    \notag\\
	\leq& \ka\|\pa_x\widetilde{u}_1\|^2+C_{\ka}\eps^2\de^2\|\pa^2_x(f,g)\|^2 +C\eps^5,
\end{align*}
for any $\ka>0$. Similarly, we obtain
\begin{align*}
	&\eps\de\sum^{3}_{i=2}\int_{\mathbb{R}}\int_{\mathbb{R}^{3}}
	\big\{\pa_x[\frac{1}{\rho}\partial_{x}(K\theta B_{11}(\frac{v-u}{\sqrt{K\theta}})\frac{1}{M}h)]
	\pa_x\widetilde{u}_i\big\}\,dv\,dx\notag\\
	&+\eps\de\int_{\mathbb{R}}\int_{\mathbb{R}^{3}}
	\big\{\pa_x[\frac{1}{\rho}\partial_{x}
	((K\theta)^{\frac{3}{2}}A_{1}(\frac{v-u}{\sqrt{K\theta}})\frac{1}{M}h)+\sum^{3}_{j=1}\frac{1}{\rho}
	\partial_{x}u_j K\theta B_{1j}(\frac{v-u}{\sqrt{K\theta}})\frac{1}{M}h]\frac{1}{\bar{\theta}}\pa_x\widetilde{\theta}\big\}\,dv\,dx
	\notag\\
	\leq&\ka\|\pa_x(\widetilde{u},\widetilde{\theta})\|^2+C_{\ka}\eps^2\de^2\|\pa^2_x(f,g)\|^2 +C\eps^5.
\end{align*}
Then taking integration in $t$ and combining with the above estimate, we get
\begin{align*}
		&\frac{1}{2}\{(\pa_x\widetilde{\rho},\frac{2\bar{\theta}}{3\bar{\rho}^{2}}\pa_x\widetilde{\rho})+\|\pa_x\widetilde{u}\|^{2}+(\frac{1}{\bar\rho}\pa_x\widetilde{\phi},\pa_x\widetilde{\phi})+\eps(\frac{1}{\bar\rho}\pa^2_x\widetilde{\phi},\pa^2_x\widetilde{\phi})+(\pa_x\widetilde{\theta},\frac{1}{\bar{\theta}}\pa_x\widetilde{\theta})\}\notag\\
	&
	+c\eps\de\int^t_0\{\|\pa^2_x(\widetilde{\rho},\widetilde{u},\widetilde{\theta},\widetilde{\phi})(s)\|^{2}+\eps\|\pa^3_x\widetilde{\phi}(s)\|^{2}\}\,ds
	\notag\\
	\leq&C\ka\|\pa_x(\widetilde{u},\widetilde{\theta})\|^2+C_{\ka}\CE(0)+C_{\ka}\eps^2\de^2\|\pa^2_x\widetilde{\rho}\|^2+C_{\ka}\eps^2\de^2\|\pa^2_x(g,f)\|^2+
	C\eps\de\int_0^t\|\pa^2_x(g,f)(s)\|_\sigma^{2}\,ds
    \notag\\
	&+C(\de+\eps+\frac{\de}{\sqrt\eps})\int_0^t\CD(s)ds+CA\int_0^t\min\{(\frac{\eps}{\de}+\frac{\eps^2}{\de^2})\CD(s),(\frac{\eps^{3/2}}{\sqrt{\de}}+\frac{\eps^{3}}{\de})\CE(s)\}\,ds
	\notag\\
	&+ C(\eta)\int^t_0\CE(s)\,ds+ C(\eta)t(\eps^4+\eps^3\de+\eps\de^3)+C\eps^5,
\end{align*}
for any  $0<\ka<1$.
Hence, the desired estimate \eqref{1stfluid} holds by choosing $\ka$ to be small.
\end{proof}

We can obtain the corresponding zero order energy estimate with similar arguments.

\begin{lemma}\label{lezerofluid}
	Under the a priori assumption \eqref{apriori}, it holds that
	\begin{align}\label{zerofluid}
		&\|(\widetilde{\rho},\widetilde{u},\widetilde{\theta},\widetilde{\phi})(t)\|^{2}+\eps\|\pa_x\widetilde{\phi}(t)\|^2+\eps\de\int^t_0\{\|\pa_x(\widetilde{\rho},\widetilde{u},\widetilde{\theta},\widetilde{\phi})(s)\|^{2}+\eps\|\pa^2_x\widetilde{\phi}(s)\|^{2}\}\,ds
		\notag\\
		\leq&C\CE(0)+C(\de+\eps+\frac{\de}{\sqrt\eps})\int_0^t\CD(s)\,ds+C(\eta)\int^t_0\CE(s)ds+Ct(\eps^4+\eps^3\de+\eps\de^3)+C\eps^4,
	\end{align}	
for any $0<t<T$, where $C(\eta)\geq0$ depends only on $\eta$ with $C(0)=0$.
\end{lemma}
\begin{proof}
The desired estimate \eqref{zerofluid} follows from \eqref{zerorho}, \eqref{zerou}, \eqref{zeroui}, \eqref{zerotheta}, \eqref{zerodissrho} and similar calculations in the proof of Lemma \ref{le1stfluid}.
\end{proof}
\subsection{Lower order estimates on non-fluid part}	
In this subsection, we start from equations \eqref{eqg} and \eqref{eqf} to obtain the energy estimates for lower order non-fluid part $\|\partial^{\alpha}_{\beta}g(t)\|_{w}^{2}$ and $\|\partial^{\alpha}_{\beta}f(t)\|^{2}$ with $\be=0$, $\al<2$ or $\alpha+|\beta|\leq 2$, $|\beta|\geq1$. 
\begin{lemma}\label{lem4.11}
Under the a priori assumption \eqref{apriori}, it holds that
	\begin{align}\label{walbeg}
		&\frac{d}{dt}\big\{C_\alpha\sum_{\alpha\leq 1}\|\partial^{\alpha}_x g\|_w^{2}+C_{\alpha,\be}\sum_{\alpha+|\beta|\leq 2,|\beta|\geq 1}\|\partial^{\alpha}_\be g\|_w^{2}\big\}+c\frac{1}{\eps\de}\big\{\sum_{\alpha\leq 1}\|\partial^{\alpha}_x g\|_{\sigma,w}^{2}+\sum_{\al+|\beta|\leq 2,|\beta|\geq 1}\|\partial^{\alpha}_\be g\|_{\sigma,w}^{2}\big\}\notag\\
		&\qquad+cq_{1}q_{2}(1+t)^{-(1+q_2)}\big\{\sum_{\alpha\leq 1}\|\langle v\rangle \partial^{\alpha}_xg\|_{w}^{2}+\sum_{\al+|\beta|\leq 2,|\beta|\geq 1}\|\langle v\rangle \partial^{\alpha}_{\beta}g\|_{w}^{2}\big\}\notag\\
		\leq& C\eps\de\|\pa^2_xg\|_{\sigma,w}^{2}+
        C(\eps+\de)\mathcal{D}(t)+
        C(\eta)\eps^2\de^2+ C(\eta)\CE(t)\notag\\
		&+C(C(\eta)+\eps)\frac{\de^{1/2}}{\eps}\big\{\sum_{\alpha\leq 1}\|\langle v\rangle \pa^\al_x g\|_w^2+\sum_{\al+|\beta|\leq 2,|\beta|\geq 1}\|\langle v\rangle \pa^\al_\be g\|_w^2\big\},
	\end{align}	
	for some positive constants $C_\alpha$ and $C_{\alpha,\be}$, where $C(\eta)\geq0$ depends only on $\eta$ with $C(0)=0$.
\end{lemma}
\begin{proof}
In the following proof, it is important to bound the derivatives of the weight function $w$ defined in \eqref{Defw}. Direct calculations show that
	\begin{equation}\label{patw}
		\partial_{t}[w^{2}(\alpha,\beta)] =-q_{1}q_{2}(1+t)^{-(1+q_2)}\langle v\rangle^2 w^{2}(\alpha,\beta),
	\end{equation}
	and
	\begin{equation}\label{pavw}
		\partial_{v_1}w(\alpha,\beta)=2(l-\al-|\beta|)\frac{v_1}{\langle v\rangle^{2}}w(\alpha,\beta)+
		q_1(1+t)^{-q_2}v_1 w(\alpha,\beta), \quad |\partial_{v_1}w(\alpha,\beta)|
		\leq C\langle v\rangle w(\alpha,\beta).
	\end{equation}
	Let $\alpha\leq 1$ when $|\be|=0$, and $\alpha+|\be|\leq 2$ when $|\be|\geq1$. Applying $\partial^{\alpha}_\be$ to \eqref{eqg} and taking the inner product with $w^2(\al,\be)\partial^{\alpha}_\be g$, one has
	\begin{align}\label{ipg}
		&(\partial_t\partial^{\alpha}_\be g,w^2(\alpha,\be)\partial^{\alpha}_\be g)+(\partial^{\alpha}_\be (v_1\pa_xg),w^2(\alpha,\beta)\partial_{\beta}^\alpha g)-\frac{1}{\eps}(\partial^{\alpha}_\be[\frac{\partial_x\phi\partial_{v_{1}}(\sqrt{\mu}g)}{\sqrt{\mu}}],w^2(\al,\be)\partial^{\alpha}_\be g)\notag\\
		&
		-\frac{1}{\eps\de}(\partial^{\alpha}_\be L_D g,w^2(\al,\be)\partial^{\alpha}_\be g)+(\partial^{\alpha}_\be[\frac{(\pa_t+v_1\pa_x)\sqrt{\overline{M}}}{\sqrt{\mu}}f],w^2(\al,\be)\partial^{\alpha}_\be g)\notag\\
		&-\frac{1}{\eps}(\partial^{\alpha}_\be[\frac{\pa_x\sqrt{\overline{M}}}{\sqrt{\mu}}f],w^2(\al,\be)\partial^{\alpha}_\be g)-\frac{1}{\eps}(\partial^{\alpha}_\be[\frac{\partial_x\phi\partial_{v_{1}}(\sqrt{\overline{M}}f)}{\sqrt{\mu}}],w^2(\al,\be)\partial^{\alpha}_\be g)
		\notag\\
		=&\frac{1}{\eps\de}(\partial^{\alpha}_\be \Gamma(g+\frac{\sqrt{\overline{M}}}{\sqrt{\mu}}f+\frac{\overline{G}}{\sqrt{\mu}},g+\frac{\sqrt{\overline{M}}}{\sqrt{\mu}}f+\frac{\overline{G}}{\sqrt{\mu}}),w^2(\al,\be)\partial^{\alpha}_\be g)\notag\\
		&
		+\frac{1}{\eps\de}(\partial^{\alpha}_\be\Gamma(\frac{M-\overline{M}}{\sqrt{\mu}},\frac{\sqrt{\overline{M}}}{\sqrt{\mu}}f)+\partial^{\alpha}_\be\Gamma(\frac{\sqrt{\overline{M}}}{\sqrt{\mu}}f,\frac{M-\overline{M}}{\sqrt{\mu}}),w^2(\al,\be)\partial^{\alpha}_\be g).
	\end{align}
We should bound each term above and begin with the first one on the left hand side. It follows from \eqref{patw} that
\begin{align}\label{1gtime}
	(\partial_{t}\partial^{\alpha}_\be g,w^2(\alpha,\be)\partial^{\alpha}_\be g)=&\frac{1}{2}\frac{d}{dt}(\partial^\alpha_\be g,w^2(\alpha,\be)\partial^\alpha_\be g)
	-\frac{1}{2}(\partial^\alpha_\be g,\partial_{t}[w^2(\alpha,\be)]\partial^\alpha_\be g)
	\notag\\
	=&\frac{1}{2}\frac{d}{dt}\|\partial^\alpha_\be g\|_{w}^2+\frac{1}{2}q_{1}q_{2}(1+t)^{-(1+q_2)}\|\langle v\rangle \partial^\alpha_\be g\|^2_{w}.
\end{align}
For the second term above, it is direct to get
\begin{align*}
	&(\partial^{\alpha}_x (v_1\pa_xg),w^2(\alpha,0)\partial_{x}^\alpha g)=0,\quad \al\leq 1,\notag\\
	&(\partial^{\alpha}_x\pa_{v_1} (v_1\pa_xg),w^2(\alpha,\beta)\partial^{\alpha}_x\pa_{v_1} g)=(\partial^{\alpha+1}_xg,w^2(\alpha,\beta)\partial^{\alpha}_x\pa_{v_1} g),\quad \be=(1,0,0),\quad \al\leq 1\notag\\
	&(\pa^2_{v_1} (v_1\pa_xg),w^2(0,\beta)\pa^2_{v_1} g)=2(\partial_x\pa_{v_1}g,w^2(0,\beta)\pa^2_{v_1} g),\quad \be=(2,0,0),\notag\\
	&(\pa_{v_1}\pa_{v_i} (v_1\pa_xg),w^2(0,\beta)\pa_{v_1}\pa_{v_i} g)=(\partial_x\pa_{v_i}g,w^2(0,\beta)\pa_{v_1}\pa_{v_i} g),\quad \be=(1,\de_{i2},\de_{i3}),\quad i=2,3,\notag\\
	&(\partial^{\alpha}_x\pa_{v_i} (v_1\pa_xg),w^2(\alpha,\beta)\partial^{\alpha}_x\pa_{v_i} g)=0,\quad \al\leq1,\quad \be=(0,\de_{i2},\de_{i3}), \quad i=2,3,
    \notag\\
	&(\pa_{v_i}\pa_{v_j} (v_1\pa_xg),w^2(0,\beta)\pa_{v_i}\pa_{v_j}g)=0,\quad \be=(0,\de_{i2}+\de_{j2},\de_{i3}+\de_{j3}), \quad i,j=2,3,
\end{align*}
where $\de_{ij}$ stands for Kronecker delta that $\de_{ij}=1$ when $i=j$ and $\de_{ij}=0$ otherwise. The above identities show all the cases for the second term in \eqref{ipg} and we only need to consider three nonzero ones among them. Using the definition of weight function \eqref{Defw}, for $\be=(1,0,0)$ and $\al= 1$, it holds that
\begin{align*}
	|(\partial^{\alpha+1}_xg,w^2(\alpha,\beta)\partial^{\alpha}_x\pa_{v_1} g)|\leq& \ka\frac{1}{\eps\de}\|\langle v\rangle^{\frac{1}{2}}w(\alpha,\beta)\partial^{\alpha}_x\pa_{v_1} g\|^2+C_{\ka} \eps\de\|\langle v\rangle^{-\frac{1}{2}}w(\alpha,\beta)\partial^{\alpha+1}_xg\|^2\notag\\
	\leq& \ka\frac{1}{\eps\de}\|\langle v\rangle^2\langle v\rangle^{-\frac{3}{2}}w(\alpha,\beta)\partial^{\alpha}_x\pa_{v_1} g\|^2+C_{\ka} \eps\de\|\langle v\rangle^{-\frac{1}{2}}w(\alpha+1,0)\partial^{\alpha+1}_xg\|^2\notag\\
	\leq&\ka\frac{1}{\eps\de}\|\langle v\rangle^{-\frac{3}{2}}w(\alpha,0)\partial^{\alpha}_x\pa_{v_1} g\|^2+C_\ka\eps\de\|\partial^2_xg\|^2_{\sigma,w}\notag\\
	\leq& C\ka\frac{1}{\eps\de}\|\partial^{\alpha}_xg\|_{\sigma,w}^2+C_\ka \eps\de\|\partial^2_xg\|^2_{\sigma,w}.
\end{align*}
For $\be=(1,0,0),$ $\al=0$, one has from \eqref{boundlandaunorm} and \eqref{DefD} that
\begin{align*}
	|(\partial^{\alpha+1}_xg,w^2(\alpha,\beta)\partial^{\alpha}_x\pa_{v_1} g)|\leq& \|\langle v\rangle^2\langle v\rangle^{-\frac{3}{2}}w(0,\beta)\pa_{v_1} g\|\|\langle v\rangle^{-\frac{1}{2}}w(0,\beta)\partial_xg\|
	\leq C\eps\CD(t).
\end{align*}
Similarly, for $\be=(2,0,0)$, one has
\begin{align*}
|(\partial_x\pa_{v_1}g,w^2(0,\beta)\pa^2_{v_1} g)|\leq C\|\langle v\rangle^2\langle v\rangle^{-\frac{3}{2}}w(0,\beta)\pa_{v_1} \pa_xg\|\|\langle v\rangle^{-\frac{1}{2}}w(0,\beta)\pa^2_{v_1}g\|
	\leq C\eps\CD(t),
\end{align*}
and for $\be=(1,\de_{i2},\de_{i3}),$ $i=2,3,$
\begin{align*}
	|(\pa_{v_1}\pa_{v_i} (v_1\pa_xg),w^2(0,\beta)\pa_{v_1}\pa_{v_i} g)|=|(\partial_x\pa_{v_i}g,w^2(0,\beta)\pa_{v_1}\pa_{v_i} g)|\leq&C\eps\CD(t).
\end{align*}
Collecting the above estimates gives
\begin{align}\label{1gtran}
	|(\partial^{\alpha}_\be (v_1\pa_xg),w^2(\alpha,\beta)\partial_{\beta}^\alpha g)|\leq 
    C\ka\frac{1}{\eps\de}\|\partial^{\alpha}_xg\|_{\sigma,w}^2+C_\ka \eps\de\|\partial^2_xg\|^2_{\sigma,w}+     C\eps\CD(t).
\end{align}
Now we turn to the third term of \eqref{ipg} and rewrite
\begin{align}\label{phif1}
	\frac{1}{\eps}(\partial_\be^{\alpha} [\frac{\partial_x\phi\partial_{v_1}(\sqrt{\mu}g)}{\sqrt{\mu}}],w^2(\al,\be)\partial_\be^{\alpha} g)
	&=\frac{1}{\eps}(\partial^\alpha_\be[\partial_x\phi\partial_{v_1}g],w^2(\al,\be)\partial_\be^\alpha g)-\frac{1}{\eps}(\partial^\alpha_\be[\frac{v_1}{4} g\partial_x\phi],w^2(\al,\be)\partial_\be^\alpha g)\notag\\
	&=I_3-I_4.
\end{align}
First consider $I_3$ for $\al=1$ where we rewrite
\begin{align*}
	I_3=&\frac{1}{\eps}(\partial_x\phi\partial_{v_1}\partial^\alpha_\beta g,w^2(\alpha,\beta)\partial_{\beta}^\alpha g)+\frac{1}{\eps}(\partial^2_x\phi\partial_{v_1}\partial^\beta_v g,w^2(\alpha,\beta)\partial_{\beta}^\alpha g).
\end{align*}
Then we use \eqref{pavw}, the H\"older's inequality, \eqref{4.5A}, \eqref{DefE}, \eqref{DefD}, \eqref{boundkdv} and \eqref{apriori} to get 
\begin{align}\label{I31}
	&|\frac{1}{\eps}(\partial_x\phi\partial_{v_1}\partial^\alpha_\beta g,w^2(\alpha,\beta)\partial_{\beta}^\alpha g)|=|\frac{1}{2\eps}(\partial_x\phi [\pa_{v_1} w^2(1,\be)]\partial_{\beta}^\alpha g,\partial_{\beta}^\alpha g)|\notag\\
	\leq& C\frac{1}{\eps}\int_{\mathbb R}\int_{{\mathbb R}^3}|\partial_x\phi\langle v\rangle w^2(1,\be)
	(\partial_x\pa^\be_v g)^2|\,dv\,dx
	\notag\\
	\leq& C\frac{1}{\eps}\int_{\mathbb R}|\partial_x\phi|(\int_{{\mathbb R}^3}\langle v\rangle^{-1} w^2(1,\be)|\partial_x\pa^\be_vg|^2\,dv)^{1/3}
	(\int_{{\mathbb R}^3}\langle v\rangle^{2}w^2(1,\be) |\partial_x\pa^\be_vg|^2\,dv)^{2/3}\,dx
	\notag\\	
	\leq& C\frac{1}{\eps}\int_{\mathbb R}\{\frac{\ka}{\de}\int_{{\mathbb R}^3}\langle v\rangle^{-1}w^2(1,\be) |\partial_x\pa^\be_vg|^2\,dv+C_\ka
	\de^{1/2}|\partial_x\phi|^{3/2}\int_{{\mathbb R}^3}\langle v\rangle^{2}w^2(1,\be) |\partial_x\pa^\be_vg|^2\,dv\}\,dx	
	\notag\\	
	\leq& C\ka\frac{1}{\eps\delta}\| \partial_x\pa^\be_v g\|_{\sigma,w}^2
	+C_\ka\frac{\de^{1/2}}{\eps}\|\partial_x\phi\|^{3/2}_{L^{\infty}}\|\langle v\rangle w(1,\be)\partial_x\pa^\be_v g\|^2
	\notag\\
	\leq& C\ka\frac{1}{\eps\delta}\| \partial_x\pa^\be_v g\|_{\sigma,w}^2
	+C_\ka(C(\eta)+\eps)\frac{\de^{1/2}}{\eps}\|\langle v\rangle w(1,\be)\partial_x\pa^\be_v g\|^2,
\end{align}
and
\begin{align}\label{I32}
	&|\frac{1}{\eps}(\partial^2_x\phi\partial_{v_1}\partial^\beta_v g,w^2(\alpha,\beta)\partial_{\beta}^\alpha g)|=|\frac{1}{\eps}(\partial^2_x\phi\partial_{v_1}\partial^\beta_v g,w^2(1,\beta)\partial_x\partial^\beta_v g)|\notag\\
	\leq&	\frac{1}{\eps}\|\partial^2_x\phi\|_{L^\infty}\|\langle v\rangle^2 \langle v\rangle^{-\frac{3}{2}} w(1,\beta)\partial_{v_1}\partial^\beta_v g\|\|\langle v\rangle^{-\frac{1}{2}}w(1,\beta)\partial_x\partial^\beta_v g\|\notag\\
	\leq&	C\frac{1}{\eps}(C(\eta)+\frac{A^{1/4}\eps}{\eps^{1/4}}\frac{A^{1/4}\eps}{(\eps^3\de^2)^{1/4}})\eps\de\CD(t)\notag\\
	\leq&	C(\eps+\de)\CD(t).
\end{align}
The case $\al=0$ is more straightforward. We apply the approach in \eqref{I31} to get
\begin{align}\label{I33}
	I_3=&|\frac{1}{\eps}(\partial_x\phi\partial_{v_1}\partial^\alpha_\beta g,w^2(\alpha,\beta)\partial_{\beta}^\alpha g)|=|\frac{1}{2\eps}(\partial_x\phi [\pa_{v_1} w^2(0,\be)]\partial_{\beta} g,\partial_{\beta} g)|\notag\\
	\leq& C\frac{1}{\eps}\int_{\mathbb R}\int_{{\mathbb R}^3}|\partial_x\phi\langle v\rangle w^2(0,\be)
	(\pa^\be_v g)^2|\,dv\,dx
	\notag\\
	\leq& C\ka\frac{1}{\eps\delta}\| \pa^\be_v g\|_{\sigma,w}^2
	+C_\ka(C(\eta)+\eps)\frac{\de^{1/2}}{\eps}\|\langle v\rangle w(0,\be)\pa^\be_v g\|^2.
\end{align}
Then it follows from \eqref{I31}, \eqref{I32} and \eqref{I33} that
\begin{align}\label{I3}
	I_3\leq C\ka\frac{1}{\eps\delta}\| \pa^\al_\be g\|_{\sigma,w}^2
	+C_\ka(C(\eta)+\eps)\frac{\de^{1/2}}{\eps}\|\langle v\rangle w(\al,\be)\pa^\al_\be g\|^2+C(\eps+\de)\CD(t).
\end{align}
Consider $I_4=\frac{1}{\eps}(\partial^\alpha_\be[\frac{v_1}{4} g\partial_x\phi],w^2(\al,\be)\partial_\be^\alpha g)$. We compute
\begin{align}\label{reI4}
	2I_4=	\left\{
	\begin{array}{rl}
		&\dis	\frac{1}{\eps}(\frac{v_1}{2} g\partial_x\phi,w^2(0,0) g),\  \al=0, \ \be=0,\\
		&\dis	\frac{1}{\eps}(\frac{v_1}{2} \pa_xg\partial_x\phi+\frac{v_1}{2} g\partial^2_x\phi,w^2(1,0)\partial_x g),\  \al=1, \ \be=0,\\
		&\dis \frac{1}{\eps}(\frac{v_1}{2} \partial_x\phi\partial_{v_1}g+\frac{1}{2} g\partial_x\phi,w^2(0,\be)\partial_{v_1} g),\  \al=0, \ \be=(1,0,0),\\
		&\dis \frac{1}{\eps}(\frac{v_1}{2} \partial_x\phi\partial_{v_1}\pa_xg+\frac{v_1}{2} \partial^2_x\phi\partial_{v_1}g+\frac{1}{2}\pa_x g\partial_x\phi+\frac{1}{2} g\partial^2_x\phi,w^2(1,\be)\partial_{v_1}\pa_x g),\  \al=1, \ \be=(1,0,0),\\
		&\dis \frac{1}{\eps}(\frac{v_1}{2} \partial_x\phi\partial_{v_i}g,w^2(0,\be)\partial_{v_i} g),\  \al=0, \ \be=(0,\de_{i2},\de_{i3}),\ i=2,3,\\
		&\dis \frac{1}{\eps}(\frac{v_1}{2} \partial_x\phi\pa_x\partial_{v_i}g+\frac{v_1}{2} \partial^2_x\phi\partial_{v_i}g,w^2(1,\be)\partial_{v_i} \pa_xg),\  \al=1, \ \be=(0,\de_{i2},\de_{i3}),\ i=2,3,\\
		&\dis \frac{1}{\eps}(\frac{v_1}{2} \partial_x\phi\partial_{v_iv_j}g,w^2(0,\be)\partial_{v_iv_j} g),\  \al=0, \ \be=(0,\de_{i2}+\de_{j2},\de_{i3}+\de_{j3}),\ i,j=2,3,\\
		&\dis \frac{1}{\eps}(\frac{v_1}{2} \partial_x\phi\partial^2_{v_1}g+ \pa_{v_1}g\partial_x\phi,w^2(0,\be)\partial^2_{v_1} g),\  \al=0, \ \be=(2,0,0),\\
		&\dis \frac{1}{\eps}(\frac{v_1}{2} \partial_x\phi\partial_{v_1v_i}g+\frac{1}{2}\pa_{v_i} g\partial_x\phi,w^2(0,\be)\partial_{v_1v_i} g),\  \al=0, \ \be=(1,\de_{i2},\de_{i3}),\ i=2,3.
	\end{array} \right.
\end{align}
By observation, all terms above have similar structure to either \eqref{I31} or \eqref{I32}. To be more specific, terms with the form $\frac{1}{\eps}(\frac{v_1}{2} \partial_x\phi\pa^\al_\be g,w^2(1,0)\pa^\al_\be g)$ can be bounded as in \eqref{I31}, and the others can be controlled in the similar way as in \eqref{I32}. We skip the details to get
\begin{align}\label{I4}
	|I_4|\leq C\ka\frac{1}{\eps\delta}\| \pa^\al_\be g\|_{\sigma,w}^2
	+C_\ka(C(\eta)+\eps)\frac{\de^{1/2}}{\eps}\|\langle v\rangle w(\al,\be)\pa^\al_\be g\|^2+C(\eps+\de)\CD(t).
\end{align}
It follows from \eqref{phif1}, \eqref{I3} and \eqref{I4} that
\begin{align}\label{phig}
	\frac{1}{\eps}(\partial_\be^{\alpha} [\frac{\partial_x\phi\partial_{v_1}(\sqrt{\mu}g)}{\sqrt{\mu}}],w^2(\al,\be)\partial_\be^{\alpha} g)\leq& C\ka\frac{1}{\eps\delta}\| \pa^\al_\be g\|_{\sigma,w}^2
    \notag\\
	&\quad+C_\ka(C(\eta)+\eps)\frac{\de^{1/2}}{\eps}\|\langle v\rangle w(\al,\be)\pa^\al_\be g\|^2+C(\eps+\de)\CD(t).
\end{align}
For the fourth term on the left hand side of \eqref{ipg}, using \eqref{coLDx} and \eqref{coLD} respectively, we have
\begin{equation}\label{Lg}
	-\frac{1}{\eps\de}(\partial^\alpha_x L_D g,w^2(\alpha,0)\partial^\alpha_x g)
	\geq c\frac{1}{\eps\de}\|\partial^\alpha_x g\|^2_{\sigma,w}-C\frac{1}{\eps\de}\sum_{\al'<\al}\|\partial^{\al'}_x g\|^2_{\sigma,w}-C\eps\CD(t),
\end{equation}
for $\al\leq1$ and 
\begin{align}\label{paLg}
	-\frac{1}{\eps\de}(\partial^\alpha_\beta L_D g,w^2(\alpha,\beta)\partial^\alpha_\beta g)
	\geq \frac{1}{\eps\de}\big(c\|\partial^\alpha_\beta g\|^2_{\sigma,w}-&\ka_1\sum_{|\beta'|=|\beta|}\|\partial^\alpha_{\beta'} g\|_{\sigma,w}^2-C_{\ka_1}\sum_{|\beta'|<|\beta|}\|\partial^\alpha_{\beta'}g\|_{\sigma,w}^2\notag\\
	&-C\sum_{\al'<\al,|\beta'|\leq|\beta|}\|\partial^{\alpha'}_{\beta'}g\|_{\sigma,w}^2\big)-C\eps\CD(t),
\end{align}
for $\al+|\be|\leq2$, $|\be|\geq 1$ and some $0<\ka_1<1$ that will be chosen later. We turn to the fifth term on the left hand side of \eqref{ipg}. Since $(\pa_t+v_1\pa_x)\sqrt{\overline{M}}$ produces exponential decay in $v$ that is smaller than $e^{-8|v|^2}$, it holds that
\begin{align}\label{tranbarM}
	&|(\partial^{\alpha}_\be[\frac{(\pa_t+v_1\pa_x)\sqrt{\overline{M}}}{\sqrt{\mu}}f],w^2(\al,\be)\partial^{\alpha}_\be g)|
    \notag\\
    &\leq C\sum_{\alpha'\leq\alpha,\be'\leq \be}\|\langle v\rangle^{2}w^2(\al,\be)\pa^{\al-\al'}_{\be-\be'}(\pa_t+v_1\pa_x)\sqrt{\overline{M}}\|_{L^\infty_xL^\infty_v}\|\langle v\rangle^{-1}\pa^{\al'}_{\be'}f\|\|\langle v\rangle^{-1}\pa^\al_\be g\|\notag\\
	&\leq C(\eta)\eps\CD(t).
\end{align}
Here we have used $q_1>0$ small enough such that $|\langle v\rangle^{m}w^2(\al,\be)\overline{M}^{\frac{1}{2}}\mu^{-\frac{1}{2}}|\leq C$ for any $m>0$.
We consider the rest two terms on the left hand side of \eqref{ipg} simultaneously. Direct calculations gives
\begin{align*}
	&-\frac{1}{\eps}(\partial^{\alpha}_\be[\frac{\pa_x\sqrt{\overline{M}}}{\sqrt{\mu}}f],w^2(\al,\be)\partial^{\alpha}_\be g)-\frac{1}{\eps}(\partial^{\alpha}_\be[\frac{\partial_x\phi\partial_{v_{1}}(\sqrt{\overline{M}}f)}{\sqrt{\mu}}],w^2(\al,\be)\partial^{\alpha}_\be g)\notag\\
	=&-\frac{1}{\eps}(\partial^{\alpha}_\be[\frac{\pa_x\sqrt{\overline{M}}}{\sqrt{\mu}}f],w^2(\al,\be)\partial^{\alpha}_\be g)+\frac{1}{2\eps}(\partial^{\alpha}_\be[\frac{\partial_x\bar{\phi}\sqrt{\overline{M}}f}{\sqrt{\mu}}\frac{v_1-\bar{u}_1}{K\bar{\theta}}],w^2(\al,\be)\partial^{\alpha}_\be g)\notag\\
	&+\frac{1}{2\eps}(\partial^{\alpha}_\be[\frac{\partial_x\widetilde{\phi}\sqrt{\overline{M}}f}{\sqrt{\mu}}\frac{v_1-\bar{u}_1}{K\bar{\theta}}],w^2(\al,\be)\partial^{\alpha}_\be g)-\frac{1}{\eps}(\partial^{\alpha}_\be[\frac{\partial_x\widetilde{\phi}\sqrt{\overline{M}}\partial_{v_{1}}f}{\sqrt{\mu}}],w^2(\al,\be)\partial^{\alpha}_\be g)\notag\\
	&-\frac{1}{\eps}(\partial^{\alpha}_\be[\frac{\partial_x\bar{\phi}\sqrt{\overline{M}}\partial_{v_{1}}f}{\sqrt{\mu}}],w^2(\al,\be)\partial^{\alpha}_\be g).
\end{align*}
The first, second and fifth terms on the right hand side above can be bounded by
\begin{align*}
	&\big|-\frac{1}{\eps}(\partial^{\alpha}_\be[\frac{\pa_x\sqrt{\overline{M}}}{\sqrt{\mu}}f]-\partial^{\alpha}_\be[\frac{\partial_x\bar{\phi}\sqrt{\overline{M}}f}{\sqrt{\mu}}\frac{v_1-\bar{u}_1}{2K\bar{\theta}}]+\partial^{\alpha}_\be[\frac{\partial_x\bar{\phi}\sqrt{\overline{M}}\partial_{v_{1}}f}{\sqrt{\mu}}],w^2(\al,\be)\partial^{\alpha}_\be g)\big|\notag\\
	\leq& C\frac{1}{\eps}\sum_{\alpha'\leq\alpha,\be'\leq \be}\{\|\langle v\rangle^6 w^2(\al,\be)\pa^{\al-\al'}_{\be-\be'}\frac{\pa_x\sqrt{\overline{M}}}{\sqrt{\mu}}\|_{L^\infty_xL^\infty_v}\|\langle v\rangle^{-3}\pa^{\al'}_{\be'}f\|\|\langle v\rangle^{-3}\pa^\al_\be g\|\notag\\
	&\qquad\qquad\qquad+\|\langle v\rangle^6 w^2(\al,\be)\pa^{\al-\al'}_{\be-\be'}[\frac{\partial_x\bar{\phi}\sqrt{\overline{M}}}{\sqrt{\mu}}\frac{v_1-\bar{u}_1}{K\bar{\theta}}]\|_{L^\infty_xL^\infty_v}\|\langle v\rangle^{-3}\pa^{\al'}_{\be'}f\|\|\langle v\rangle^{-3}\pa^\al_\be g\|\notag\\
	&\qquad\qquad\qquad+\|\langle v\rangle^6 w^2(\al,\be)\pa^{\al-\al'}_{\be-\be'}[\frac{\partial_x\bar{\phi}\sqrt{\overline{M}}}{\sqrt{\mu}}]\|_{L^\infty_xL^\infty_v}\|\langle v\rangle^{-3}\pa^{\al'}_{\be'}\pa_{v_1}f\|\|\langle v\rangle^{-3}\pa^\al_\be g\|\}\notag\\
	\leq& C(\eta)\frac{1}{\eps}\sum_{\alpha'\leq\alpha,\be'\leq \be}(\|\langle v\rangle^{-3}\pa^{\al'}_{\be'}f\|+\|\langle v\rangle^{-3}\pa^{\al'}_{\be'}\pa_{v_1}f\|)\|\pa^\al_\be g\|_{\sigma}
    \notag\\
	\leq& C(\eta)\de\CD(t),
\end{align*}
according to \eqref{4.5A}, \eqref{DefE}, \eqref{DefD} and \eqref{apriori}. On the other hand, the rest two terms are controlled by
\begin{align*}
	&\big|\frac{1}{2\eps}(\partial^{\alpha}_\be[\frac{\partial_x\widetilde{\phi}\sqrt{\overline{M}}f}{\sqrt{\mu}}\frac{v_1-\bar{u}_1}{K\bar{\theta}}],w^2(\al,\be)\partial^{\alpha}_\be g)-\frac{1}{\eps}(\partial^{\alpha}_\be[\frac{\partial_x\widetilde{\phi}\sqrt{\overline{M}}\partial_{v_{1}}f}{\sqrt{\mu}}],w^2(\al,\be)\partial^{\alpha}_\be g)\big|\notag\\
	\leq&C\frac{1}{\eps}\sum_{\alpha'=\alpha,\be'\leq \be}\{\|\langle v\rangle^6 w^2(\al,\be)\pa^{\al-\al'}_{\be-\be'}[\frac{\partial_x\widetilde{\phi}\sqrt{\overline{M}}}{\sqrt{\mu}}\frac{v_1-\bar{u}_1}{K\bar{\theta}}]\|_{L^\infty_xL^\infty_v}\|\langle v\rangle^{-3}\pa^{\al'}_{\be'}f\|\|\langle v\rangle^{-3}\pa^\al_\be g\|\notag\\
	&\qquad\qquad\qquad+\|\langle v\rangle^6 w^2(\al,\be)\pa^{\al-\al'}_{\be-\be'}[\frac{\partial_x\widetilde{\phi}\sqrt{\overline{M}}}{\sqrt{\mu}}]\|_{L^\infty_xL^\infty_v}\|\langle v\rangle^{-3}\pa^{\al'}_{\be'}\pa_{v_1}f\|\|\langle v\rangle^{-3}\pa^\al_\be g\|\}\notag\\
	&+C\frac{1}{\eps}\sum_{\alpha'<\alpha,\be'\leq \be}\{\|\langle v\rangle^6 w^2(\al,\be)\pa^{\al-\al'}_{\be-\be'}[\frac{\partial_x\widetilde{\phi}\sqrt{\overline{M}}}{\sqrt{\mu}}\frac{v_1-\bar{u}_1}{K\bar{\theta}}]\|_{L^2_xL^\infty_v}\|\langle v\rangle^{-3}\pa^{\al'}_{\be'}f\|_{L^\infty_xL^2_v}\|\langle v\rangle^{-3}\pa^\al_\be g\|\notag\\
	&\qquad\qquad\qquad+\|\langle v\rangle^6 w^2(\al,\be)\pa^{\al-\al'}_{\be-\be'}[\frac{\partial_x\widetilde{\phi}\sqrt{\overline{M}}}{\sqrt{\mu}}]\|_{L^2_xL^\infty_v}\|\langle v\rangle^{-3}\pa^{\al'}_{\be'}\pa_{v_1}f\|_{L^\infty_xL^2_v}\|\langle v\rangle^{-3}\pa^\al_\be g\|\}\notag\\
	\leq&C\de\CD(t).
\end{align*}
We combine the above two estimates to get
\begin{align}\label{pabarM}
\frac{1}{\eps}|(\partial^{\alpha}_\be[\frac{\pa_x\sqrt{\overline{M}}}{\sqrt{\mu}}f],w^2(\al,\be)\partial^{\alpha}_\be g)|+\frac{1}{\eps}|(\partial^{\alpha}_\be[\frac{\partial_x\phi\partial_{v_{1}}(\sqrt{\overline{M}}f)}{\sqrt{\mu}}],w^2(\al,\be)\partial^{\alpha}_\be g)|\leq C\de\CD(t).
\end{align}
The first term on the right hand side of \eqref{ipg} is equivalent to
\begin{align}
\label{4.90A}
&\frac{1}{\eps\de}\{
\partial^{\alpha}_\be \Gamma(\frac{\overline{G}}{\sqrt{\mu}},\frac{\overline{G}}{\sqrt{\mu}})+
(\partial^{\alpha}_\be \Gamma(\frac{\sqrt{\overline{M}}}{\sqrt{\mu}}f,\frac{\sqrt{\overline{M}}}{\sqrt{\mu}}f)
+\partial^{\alpha}_\be \Gamma(\frac{\sqrt{\overline{M}}}{\sqrt{\mu}}f,\frac{\overline{G}}{\sqrt{\mu}})
+\partial^{\alpha}_\be \Gamma(\frac{\overline{G}}{\sqrt{\mu}},\frac{\sqrt{\overline{M}}}{\sqrt{\mu}}f)
\notag\\
&+
\partial^{\alpha}_\be \Gamma(\frac{\overline{G}}{\sqrt{\mu}},g)
+\partial^{\alpha}_\be \Gamma(g,\frac{\overline{G}}{\sqrt{\mu}})
+ \partial^{\alpha}_\be \Gamma(\frac{\sqrt{\overline{M}}}{\sqrt{\mu}}f,g)
+\partial^{\alpha}_\be \Gamma(g,\frac{\sqrt{\overline{M}}}{\sqrt{\mu}}f)
+\partial^{\alpha}_\be \Gamma(g,g),w^2(\al,\be)\partial^{\alpha}_\be g)\}.
\end{align}
We use \eqref{controlpaGa}, \eqref{boundbarG0}, \eqref{boundbarG},
\eqref{DefD}, \eqref{DefE} and \eqref{apriori} to get
\begin{align*}
&\frac{1}{\eps\de}|(\partial^{\alpha}_\be \Gamma(\frac{\overline{G}}{\sqrt{\mu}},\frac{\overline{G}}{\sqrt{\mu}}),w^2(\al,\be)\partial^{\alpha}_\be g)|     \notag\\ 	\leq& C\sum_{\alpha'\leq\alpha}\sum_{\bar{\beta}\leq\beta'\leq\beta} 	\frac{1}{\eps\de}\int_{\mathbb {R}}|\mu^a \partial^{\alpha'}_{\bar{\beta}}[\frac{\overline{G}}{\sqrt{\mu}}]|_2|  \partial^{\alpha-\alpha'}_{\beta-\beta'}[\frac{\overline{G}}{\sqrt{\mu}}]|_{\sigma,w(\alpha,\beta)}|\partial^\alpha_\beta g|_{\sigma,w}\,dx     \notag\\ 	
\leq& C\sum_{ \frac{\alpha+|\beta|}{2}<\alpha'+|\bar{\beta}|}\sum_{\alpha'\leq\alpha,     \bar{\beta}\leq\beta'\leq\beta} 	\frac{1}{\eps\de}\|\mu^a \partial^{\alpha'}_{\bar{\beta}}[\frac{\overline{G}}{\sqrt{\mu}}]\|  |\partial^{\alpha-\alpha'}_{\beta-\beta'}[\frac{\overline{G}}{\sqrt{\mu}}]|_{\sigma,w(\alpha,\beta)}\|_{L^\infty_x}\|\partial^\alpha_\beta g\|_{\sigma,w}     \notag\\ &+C\sum_{\alpha'+|\bar{\beta}|\leq \frac{\alpha+|\beta|}{2}} \sum_{\alpha'\leq\alpha,\bar{\beta}\leq\beta'\leq\beta} 	\frac{1}{\eps\de}\|\mu^a \partial^{\alpha'}_{\bar{\beta}}[\frac{\overline{G}}{\sqrt{\mu}}]\|_{L^\infty_xL^2_v}\|\partial^{\alpha-\alpha'}_{\beta-\beta'}[\frac{\overline{G}}{\sqrt{\mu}}]\|_{\sigma,w(\alpha,\beta)}\|\partial^\alpha_\beta g\|_{\sigma,w} \notag\\ 
\leq& C\eps\CD(t)+C(\eta)\eps^2\de^2.
\end{align*}
Using \eqref{controlpaGa} again gives
\begin{align*}
\frac{1}{\eps\de}|(\partial^{\alpha}_\be \Gamma(g,g),w^2(\al,\be)\partial^{\alpha}_\be g)|
	\leq C\sum_{\alpha'\leq\alpha}\sum_{\bar{\beta}\leq\beta'\leq\beta}
	\frac{1}{\eps\de}\int_{\mathbb {R}}|\mu^a \partial^{\alpha'}_{\bar{\beta}}g|_2|  \partial^{\alpha-\alpha'}_{\beta-\beta'}g|_{\sigma,w(\alpha,\beta)}|\partial^\alpha_\beta g|_{\sigma,w}\,dx.
\end{align*}
Consider $\alpha+|\beta|\leq 2$, $|\beta|\geq1$, then $\alpha\leq1$. If $\alpha-\alpha'+|\beta-\beta'|=0$, then $w(\alpha,\beta)\leq w(\alpha_1,0)$ for $\alpha_1\leq 1$ and $|\beta|\geq 1$, it follows from \eqref{DefD}, \eqref{DefE} and \eqref{apriori} that
\begin{align}
\label{4.21A}
&\frac{1}{\eps\de}\int_{\mathbb {R}}|\mu^a \partial^{\alpha'}_{\bar{\beta}}g|_2|  \partial^{\alpha-\alpha'}_{\beta-\beta'}g|_{\sigma,w(\alpha,\beta)}|\partial^\alpha_\beta g|_{\sigma,w}\,dx
\notag\\
&\leq C\frac{1}{\eps\de}\|\mu^a\partial^{\alpha'}_{\bar{\beta}}g\|\||w(\alpha,\beta) \partial^{\alpha-\alpha'}_{\beta-\beta'}g|_{\sigma}\|_{L^{\infty}}\|\partial^\alpha_\beta g\|_{\sigma,w}
\notag\\
&\leq CA\eps^2\frac{1}{\eps\de}\|w(\alpha,\beta) \partial^{\alpha-\alpha'}_{\beta-\beta'}g\|^{\frac{1}{2}}_{\sigma}\|w(\alpha,\beta) \partial^{\alpha-\alpha'}_{\beta-\beta'}\partial_xg\|^{\frac{1}{2}}_{\sigma}\|\partial^\alpha_\beta g\|_{\sigma,w}
\notag\\
&\leq C\eps\CD(t).
\end{align}	
If $\alpha-\alpha'+|\beta-\beta'|=1$, then $\alpha'+|\beta'|\leq1$. 
In this case, if $\alpha'+|\beta'|=0$, then $\alpha+|\beta|=1$ and
$w(\alpha,\beta)=w(\alpha-\alpha',\beta-\beta')$, we thus have
\begin{align}
\label{4.22A}
&\frac{1}{\eps\de}\int_{\mathbb {R}}|\mu^a \partial^{\alpha'}_{\bar{\beta}}g|_2|  \partial^{\alpha-\alpha'}_{\beta-\beta'}g|_{\sigma,w(\alpha,\beta)}|\partial^\alpha_\beta g|_{\sigma,w}\,dx
\notag\\
&\leq \frac{1}{\eps\de}\||\mu^a\partial^{\alpha'}_{\bar{\beta}}g|_2\|_{L^{\infty}}\|w(\alpha,\beta) \partial^{\alpha-\alpha'}_{\beta-\beta'}g\|_{\sigma}\|\partial^\alpha_\beta g\|_{\sigma,w}
\leq C\eps\CD(t).
\end{align}	
If $\alpha'+|\beta'|=1$, then $\alpha+|\beta|=2$ and $w(\alpha,\beta)\leq w(\alpha-\alpha'+\alpha_1,\beta-\beta')$ for $\alpha_1\leq1$. if $\alpha-\alpha'=0$, we use the similar arguments as \eqref{4.21A} to get the same bound. If $\alpha-\alpha'=1$, we use the similar arguments as \eqref{4.22A} to get the same bound.
\\	
If $\alpha-\alpha'+|\beta-\beta'|=2$, then $\alpha'+|\beta'|=0$ and $w(\alpha,\beta)=w(\alpha-\alpha',\beta-\beta')$,
we use the similar arguments as \eqref{4.22A} to get the same bound.
Hence, for $\alpha+|\beta|\leq 2$ and $|\beta|\geq1$, we obtain
\begin{equation*}
\frac{1}{\eps\de}|(\partial^{\alpha}_\be \Gamma(g,g),w^2(\al,\be)\partial^{\alpha}_\be g)|
\leq C\eps\CD(t).	
\end{equation*}
The other terms in \eqref{4.90A} can be treated similarly. We thus arrive at
\begin{align}\label{Gagg}
	\frac{1}{\eps\de}|(\partial^{\alpha}_\be \Gamma(g+\frac{\sqrt{\overline{M}}}{\sqrt{\mu}}f+\frac{\overline{G}}{\sqrt{\mu}},g+\frac{\sqrt{\overline{M}}}{\sqrt{\mu}}f+\frac{\overline{G}}{\sqrt{\mu}}),w^2(\al,\be)\partial^{\alpha}_\be g)|
\leq C\eps\CD(t)+C(\eta)\eps^2\de^2.
\end{align}
By mean value theorem, \eqref{DefLocalMax}, \eqref{2.38A}, \eqref{Defpert} and \eqref{Defmu}, one has
\begin{equation*}
|M-\overline{M}|\leq C|(\widetilde{\rho},\widetilde{u},\widetilde{\theta})|\mu,
\end{equation*}
Using this and the identity
\begin{align*}
\partial_x(M-\overline{M})=&M\big(\frac{\pa_x\rho}{\rho}+\frac{(v-u)\cdot\pa_xu}{K\theta}+(\frac{|v-u|^{2}}{2K\theta}-\frac{3}{2})\frac{\pa_x\theta}{\theta} \big)
\notag\\
&-\overline{M}\big(\frac{\pa_x\bar{\rho}}{\bar{\rho}}+\frac{(v-\bar{u})\cdot\pa_x\bar{u}}{K\bar{\theta}}+(\frac{|v-\bar{u}|^{2}}{2K\bar{\theta}}-\frac{3}{2})\frac{\pa_x\bar{\theta}}{\bar{\theta}} \big),
\end{align*}
one has from \eqref{controlpaGa} that
\begin{align}\label{GaMf}
	&\frac{1}{\eps\de}|(\partial^{\alpha}_\be\Gamma(\frac{M-\overline{M}}{\sqrt{\mu}},\frac{\sqrt{\overline{M}}}{\sqrt{\mu}}f)+\partial^{\alpha}_\be\Gamma(\frac{\sqrt{\overline{M}}}{\sqrt{\mu}}f,\frac{M-\overline{M}}{\sqrt{\mu}}),w^2(\al,\be)\partial^{\alpha}_\be g)|\notag\\
	\leq& C\sum_{\alpha'\leq\alpha}\sum_{\bar{\beta}\leq\beta'\leq\beta}
	\frac{1}{\eps\de}\int_{\mathbb {R}}|\mu^a \partial^{\alpha'}_{\bar{\beta}}[\frac{M-\overline{M}}{\sqrt{\mu}}]|_2|  \partial^{\alpha-\alpha'}_{\beta-\beta'}[\frac{\sqrt{\overline{M}}}{\sqrt{\mu}}f]|_{\sigma,w(\alpha,\beta)}|\partial^\alpha_\beta g|_{\sigma,w}\,dx
\notag\\
&+ C\sum_{\alpha'\leq\alpha}\sum_{\bar{\beta}\leq\beta'\leq\beta}
	\frac{1}{\eps\de}\int_{\mathbb {R}}|\mu^a \partial^{\alpha'}_{\bar{\beta}}[\frac{\sqrt{\overline{M}}}{\sqrt{\mu}}f]|_2|  \partial^{\alpha-\alpha'}_{\beta-\beta'}[\frac{M-\overline{M}}{\sqrt{\mu}}]|_{\sigma,w(\alpha,\beta)}|\partial^\alpha_\beta g|_{\sigma,w}\,dx
  \notag\\
  \leq& C\eps\CD(t).
\end{align}
Here we have used \eqref{apriori}, \eqref{DefD} and
$$
\|\mu^a \partial_{\bar{\beta}}[\frac{M-\overline{M}}{\sqrt{\mu}}]\|
+\|\mu^a\partial_x\partial_{\bar{\beta}}[\frac{M-\overline{M}}{\sqrt{\mu}}]\|
+\|\partial_{\beta-\beta'}[\frac{M-\overline{M}}{\sqrt{\mu}}]\|_{\sigma,w(\alpha,\beta)}
+\|\partial_x\partial_{\beta-\beta'}[\frac{M-\overline{M}}{\sqrt{\mu}}]\|_{\sigma,w(\alpha,\beta)}
\leq C\eps,
$$
and $q_1>0$ small enough such that $|\langle v\rangle^{m}w^2(\al,\be)\overline{M}^{\frac{1}{2}}\mu^{-\frac{1}{2}}|\leq C$ for any $m>0$ and
\begin{align*}
&\|\partial_{\beta-\beta'}[\frac{\sqrt{\overline{M}}}{\sqrt{\mu}}f]\|_{\sigma,w(\alpha,\beta)}
+\|\partial_x\partial_{\beta-\beta'}[\frac{\sqrt{\overline{M}}}{\sqrt{\mu}}f]\|_{\sigma,w(\alpha,\beta)}
\notag\\
&\leq C\sum_{\alpha+|\beta|\leq 2,|\beta|\geq1}\|\partial^{\alpha}_{\beta}f(t)\|_{\sigma,W}
+C\sum_{\alpha\leq 1}\|\partial_x^{\alpha}f(t)\|_{\sigma,W}.
\end{align*}
Hence, it follows by \eqref{ipg}, \eqref{1gtime}, \eqref{1gtran}, \eqref{phig}, \eqref{Lg}, \eqref{paLg}, \eqref{tranbarM}, \eqref{pabarM}, \eqref{Gagg} and \eqref{GaMf} that
\begin{align}\label{albeg}
&\frac{1}{2}\frac{d}{dt}\|\partial^{\alpha}_\be g\|_w^{2}+c\frac{1}{\eps\de}\|\partial^{\alpha}_\be g\|_{\sigma,w}^{2}+\frac{1}{2}q_{1}q_{2}(1+t)^{-(1+q_2)}\|\langle v\rangle \partial^{\alpha}_{\beta}g(t)\|_{w}^{2}\notag\\
	\leq&C_{\ka_1}\frac{1}{\eps\de}\sum_{|\beta'|<|\beta|}\|\partial^\alpha_{\beta'}g\|_{\sigma,w}^2+\ka_1\frac{1}{\eps\de}\sum_{|\beta'|=|\beta|}\|\partial^\alpha_{\beta'} g\|_{\sigma,w}^2+C\frac{1}{\eps\de}\sum_{\al'<\al,|\beta'|\leq|\beta|}\|\partial^{\alpha'}_{\beta'}g\|_{\sigma,w}^2\notag\\
	&
    +C\ka\frac{1}{\eps\delta}(\| \pa^\al_\be g\|_{\sigma,w}^2+\|\partial^{\alpha}_xg\|_{\sigma,w}^2)
    +C_\ka\eps\de\|\pa^2_xg\|_{\sigma,w}^{2}+C_\ka C(\eta)\CE(t)
    \notag\\ 	&
    +C_\ka C(\eta)\eps^2\de^2+C_\ka(C(\eta)+\eps)\frac{\de^{1/2}}{\eps}\|\langle v\rangle w(\al,\be)\pa^\al_\be g\|^2
    +C(\eps+\de)\mathcal{D}(t),
\end{align}
for $\al+|\be|\leq2$, $|\be|\geq1$, and
\begin{align}\label{alg}
	&\frac{1}{2}\frac{d}{dt}\|\partial^{\alpha}_x g\|_w^{2}+c\frac{1}{\eps\de}\|\partial^{\alpha}_x g\|_{\sigma,w}^{2}+\frac{1}{2}q_{1}q_{2}(1+t)^{-(1+q_2)}\|\langle v\rangle \partial^{\alpha}_x g(t)\|_{w}^{2}\notag\\
	\leq&C(\eps+\de)\mathcal{D}(t)
    +C\frac{1}{\eps\de}\sum_{\al'<\al}\|\partial^{\al'}_x g\|^2_{\sigma,w}+C\ka\frac{1}{\eps\delta}\|\partial^{\alpha}_xg\|_{\sigma,w}^2
    \notag\\
	&+C_\ka\eps\de\|\pa^2_xg\|_{\sigma,w}^{2}+C_\ka C(\eta)\CE(t)+C_\ka C(\eta)\eps^2\de^2+C_\ka(C(\eta)+\eps)\frac{\de^{1/2}}{\eps}\|\langle v\rangle w(\al,\be)\pa^\al_\be g\|^2,
\end{align}
for $\al\leq1$. We take suitable linear combination of \eqref{alg} over all  $\al\leq1$  and choose $\ka$ to be sufficiently small to get
\begin{align*}
&\sum_{\alpha\leq 1}\{C_\alpha\frac{d}{dt}\|\partial^{\alpha}_x g\|_w^{2}+c\frac{1}{\eps\de}\|\partial^{\alpha}_x g\|_{\sigma,w}^{2}+q_{1}q_{2}(1+t)^{-(1+q_2)}\|\langle v\rangle \partial^{\alpha}_x g(t)\|_{w}^{2}\}
\notag\\ 	\leq&C(\eps+\de)\mathcal{D}(t)     	+C\eps\de\|\pa^2_xg\|_{\sigma,w}^{2}+C(\eta)\CE(t)+C(\eta)\eps^2\de^2+C(C(\eta)+\eps)\sum_{\alpha\leq 1}\frac{\de^{1/2}}{\eps}\|\langle v\rangle \pa^\al_\be g\|_w^2.
\end{align*}
which, together with the linear combination of \eqref{albeg} over  $\al+|\be|\leq2$, $|\be|\geq1$ and choose $\ka_1$  and $\ka$ to be sufficiently small, yields \eqref{walbeg}. This completes the proof of Lemma \ref{lem4.11}.


\end{proof}

Next we turn to $f$ using the equation \eqref{eqf}. Note that if we directly multiply both sides of \eqref{eqf} by $W(\al,\be)$ and try to get the weighted energy estimate, it is still needed to control the dissipation without the weight $\|\partial^{\alpha}_{\beta}f(t)\|_{\sigma}^{2}$ from the inequality 
\begin{equation}
\label{wLf}
	-(\mathcal{L}_{\overline{M}}\partial^\alpha_x f,W^2(\alpha,0)\partial^\alpha_x f)
	\geq c\|\partial^\alpha_x f\|^2_{\sigma,W}
    -C\|\partial^\alpha_x f\|^2_{\sigma},
\end{equation}
as we will see later. However, we still begin with the weighted estimates since the calculations in the proof are more general compared to the one without the weight.	
\begin{lemma}
Under the a priori assumption \eqref{apriori}, 
there exist some constants $\overline{C}'_{\alpha}>0$
and $\overline{C}_{\alpha,\be}>0$ such that
\begin{align}\label{walbef}
		&\frac{d}{dt}\big\{\sum_{\alpha\leq 1}\overline{C}'_{\alpha}\|\partial^{\alpha}_x f\|_W^{2}+\sum_{\alpha+|\beta|\leq 2,|\beta|\geq 1}\overline{C}_{\alpha,\be}\|\partial^{\alpha}_\be f\|_W^{2}\big\}+c\frac{1}{\eps\de}\big\{\sum_{\alpha\leq 1}\|\partial^{\alpha}_x f\|_{\sigma,W}^{2}+\sum_{\al+|\beta|\leq 2,|\beta|\geq 1}\|\partial^{\alpha}_\be f\|_{\sigma,W}^{2}\big\}\notag\\
		\leq& C\eps\de(\|\pa^2_xf\|_{\sigma,W}^{2}+\|\pa^2_xg\|_{\sigma,w}^{2})+C(\eps+\de)\mathcal{D}(t)+C(\eta)(\CE(t)+\eps^3\de+\eps\de^3)
        \notag\\
		&+C\frac{1}{\eps\de}\sum_{\alpha\leq 1}\|\partial^\alpha_x\mathbf{P}_1 f\|^2_\sigma+C\eps\de  \sum_{\alpha\leq 1}\|\pa^{\alpha}_x\pa_x(\widetilde{u},\widetilde{\theta})\|^{2}
        +C\frac{1}{\eps\de}\big\{\sum_{\alpha\leq 1}\|\partial^{\alpha}_x g\|_{\sigma,w}^{2}+\sum_{\al+|\beta|\leq 2,|\beta|\geq 1}\|\partial^{\alpha}_\be g\|_{\sigma,w}^{2}\big\},
	\end{align}	
where $C(\eta)\geq0$ depends only on $\eta$ with $C(0)=0$.
\end{lemma}
\begin{proof}
Let $\alpha\leq 1$ when $|\be|=0$, and $\alpha+|\be|\leq 2$ when $|\be|\geq1$. Applying $\partial^{\alpha}_\be$ to \eqref{eqf} and taking the inner product with $W^2(\al,\be)\partial^{\alpha}_\be f$, one has
\begin{align}\label{ipf}
	&(\partial_t\partial^{\alpha}_\be f,W^2(\alpha,\be)\partial^{\alpha}_\be f)+(\partial^{\alpha}_\be (v_1\pa_xf),W^2(\alpha,\beta)\partial_{\beta}^\alpha f)
	-\frac{1}{\eps\de}(\partial^{\alpha}_\be \CL_{\overline{M}} f,W^2(\al,\be)\partial^{\alpha}_\be f)
	\notag\\
	=&\frac{1}{\eps\de}(\partial^{\alpha}_\be[\frac{\sqrt{\mu}}{\sqrt{ \overline{M}}}
    L_Bg],W^2(\al,\be)\partial^{\alpha}_\be f)+(\partial^{\alpha}_\be[\frac{ P_{0}(v_{1}\partial_x(\sqrt{\overline{M}}f)+v_{1}\sqrt{\mu}\partial_xg)}{\sqrt{\overline{M}}}],W^2(\al,\be)\partial^{\alpha}_\be f)
    \notag\\
	&	-(\partial^{\alpha}_\be[\frac{1}{\sqrt{\overline{M}}} P_{1}\big\{v_{1}M(\frac{|v-u|^{2}\partial_x\widetilde{\theta}}{2K\theta^{2}}+\frac{(v-u)\cdot\partial_x\widetilde{u}}{K\theta})\big\}],W^2(\al,\be)\partial^{\alpha}_\be f)+\frac{1}{\eps}(\partial^{\alpha}_\be[\frac{\partial_x\overline{G}}{\sqrt{\overline{M}}}],W^2(\al,\be)\partial^{\alpha}_\be f)
	\notag\\
	&
	+\frac{1}{\eps}(\partial^{\alpha}_\be[\frac{\partial_x\phi\partial_{v_{1}}\overline{G}}{\sqrt{\overline{M}}}],W^2(\al,\be)\partial^{\alpha}_\be f)-(\partial^{\alpha}_\be[\frac{ P_{1}(v_1\partial_x\overline{G})}{\sqrt{\overline{M}}}],W^2(\al,\be)\partial^{\alpha}_\be f)-(\partial^{\alpha}_\be[\frac{\partial_{t}\overline{G}}{\sqrt{\overline{M}}}],W^2(\al,\be)\partial^{\alpha}_\be f).
\end{align}	
Recall $W(\al,\be)$ given by \eqref{DefW}, we have from an integration by parts that 
\begin{align*}
(\partial_{t}\partial^{\alpha}_\be f,W^2(\alpha,\be)\partial^{\alpha}_\be f)=\frac{1}{2}\frac{d}{dt}\|\partial^\alpha_\be f\|_{W}^2.
\end{align*}
Note that
		\begin{equation*}
		\partial_{v_1}W(\alpha,\beta)=2(l-\al-|\beta|)\frac{v_1}{\langle v\rangle^{2}}W(\alpha,\beta).
	\end{equation*}
Similar arguments as in the proof for \eqref{1gtran} show that
\begin{align*}
	|(\partial^{\alpha}_\be (v_1\pa_xg),W^2(\alpha,\beta)\partial_{\beta}^\alpha f)|\leq 
    C\ka\frac{1}{\eps\de}\|\partial^{\alpha}_xf\|_{\sigma,W}^2+C_\ka \eps\de\|\partial^2_xg\|^2_{\sigma,w}+
    C\eps\CD(t).
\end{align*}
For the linear operator $\CL_{\overline{M}}$, it holds by \eqref{controlpaLM} and \eqref{controlLM} that
\begin{equation*}
	-\frac{1}{\eps\de}(\partial^\alpha_x \CL_{\overline{M}} f,W^2(\alpha,0)\partial^\alpha_x f)
	\geq \frac{1}{\eps\de}(c\|\partial^\alpha_x  f\|^2_{\sigma,W}-C\|\partial^\alpha_x f\|^2_{\sigma}-C\sum_{\al'<\al}\|\partial^{\al'}_x f\|^2_{\sigma,W})-C\eps\CD(t),
\end{equation*}
for $\al\leq1$ and 
\begin{align*}
	-\frac{1}{\eps\de}(\partial^\alpha_\beta \CL_{\overline{M}} f,W^2(\alpha,\beta)\partial^\alpha_\beta f)
	\geq \frac{1}{\eps\de}\big(c\|\partial^\alpha_\beta f\|^2_{\sigma,W}-&\ka_1\sum_{|\beta'|=|\beta|}\|\partial^\alpha_{\beta'} f\|_{\sigma,W}^2-C_{\ka_1}\sum_{|\beta'|<|\beta|}\|\partial^\alpha_{\beta'}f\|_{\sigma,W}^2\big)
    \notag\\
	&-C\frac{1}{\eps\de}\sum_{\al'<\al,|\beta'|\leq|\beta|}\|\partial^{\alpha'}_{\beta'}f\|_{\sigma,W}^2-C\eps\CD(t),
\end{align*}
for $\al+|\be|\leq2$, $|\be|\geq 1$ and some $0<\ka_1<1$ that will be chosen later. We turn to the right hand side of \eqref{ipf} now. The boundedness property of $L_B$ gives that
\begin{align}\label{paLB}
	\frac{1}{\eps\de}|(\partial^{\alpha}_\be[\frac{\sqrt{\mu}}{\sqrt{\overline{M}}}L_Bg],W^2(\al,\be)\partial^{\alpha}_\be f)|\leq \ka\frac{1}{\eps\de}\|\partial^\alpha_\beta f\|^2_{\sigma,W}+C_{\ka}\frac{1}{\eps\de} \sum_{\al'\leq\al,|\beta'|\leq|\beta|}\|\partial^{\alpha'}_{\beta'}g\|_{\sigma,w}^2+C\eps\CD(t).
\end{align}
For the second term  on the right hand side of \eqref{ipf}, we have from \eqref{DefProjection} that
\begin{align}\label{P0ff}
	&(\partial^{\alpha}_\be[\frac{ P_{0}(v_{1}\partial_x(\sqrt{\overline{M}}f)+v_{1}\sqrt{\mu}\partial_xg)}{\sqrt{\overline{M}}}],W^2(\al,\be)\partial^{\alpha}_\be f)\notag\\
	=&\int_{\R}\int_{\R^3}\partial^{\alpha}_\be \big\{\overline{M}^{-\frac{1}{2}}\sum_{i=0}^{4}\langle v_{1}\partial_x(\sqrt{\overline{M}}f)+v_{1}\sqrt{\mu}\partial_xg,\frac{\chi_{i}}{M}\rangle\chi_{i}\big\}W^2(\al,\be)\partial^{\alpha}_\be f\,dv\,dx
    \notag\\
	\leq& C\int_{\R}|\langle v\rangle^{-2}(\partial^{\alpha}_x\pa_{x}f,\partial^{\alpha}_x\pa_{x}g)|_2|\langle v\rangle^{-2}W(\al,\be)\partial^{\alpha}_\be f|_2\,dx+C\eps\CD(t)+C(\eta)\CE(t)
    \notag\\
	\leq&
    \ka\frac{1}{\eps\de}\|\partial^\alpha_\beta f\|^2_{\sigma,W}+C_\ka\eps\de(\|\partial_x^{\alpha}\pa_xf\|_{\si,W}+\|\partial_x^{\alpha}\pa_xg\|_{\si,w})+C\eps\CD(t)+C(\eta)\CE(t),
\end{align}
for any $0<\ka<1$. Here we have used the fact that $\|\partial_x^{\alpha}\pa_xg\|_{\si,W}\leq\|\partial^{\alpha}\pa_xg\|_{\si,w}$. Similarly, we have
\begin{align}\label{thetauf}
	&|(\partial^{\alpha}_\be[\frac{1}{\sqrt{\overline{M}}} P_{1}\big\{v_{1}M(\frac{|v-u|^{2}\partial_x\widetilde{\theta}}{2K\theta^{2}}+\frac{(v-u)\cdot\partial_x\widetilde{u}}{K\theta})\big\}],W^2(\al,\be)\partial^{\alpha}_\be f)|
    \notag\\
	\leq& \ka\frac{1}{\eps\de}\|\partial^\alpha_\beta f\|^2_{\sigma,W}+C_\ka\eps\de\|\partial_x^{\alpha}(\pa_x\widetilde{u},\pa_x\widetilde{\theta})\|^2+C\eps\CD(t)+C(\eta)\CE(t).
\end{align}
For the rest four terms involving $\overline{G}$, we  use \eqref{boundbarG}, \eqref{boundlandaunorm}, \eqref{DefE}, \eqref{apriori} and \eqref{DefD} to obtain
\begin{align}\label{barGf}
	\frac{1}{\eps}|(\partial^{\alpha}_\be[\frac{\partial_x\overline{G}}{\sqrt{\overline{M}}}],W^2(\al,\be)\partial^{\alpha}_\be f)|
    \leq& \ka\frac{1}{\eps\delta}\| \langle v\rangle^{-3}W(\al,\be)\pa^\al_\be f\|^2+C_\ka\frac{\de}{\eps}\|\langle v\rangle^{3}W(\al,\be)\partial^{\alpha}_\be[\frac{\partial_x\overline{G}}{\sqrt{\overline{M}}}]\|^2
    \notag\\
	\leq& \ka\frac{1}{\eps\de}\|\partial^\alpha_\beta f\|^2_{\sigma,W}+C_\ka C(\eta)\frac{\de}{\eps}\eps^2\de^2(1+\|\pa^2_x(\widetilde{u},\widetilde{\theta})\|^2)\notag\\
	\leq& \ka\frac{1}{\eps\de}\|\partial^\alpha_\beta f\|^2_{\sigma,W}+ C_\ka C(\eta)\de\CD(t)+C_\ka C(\eta)\eps\de^3.
\end{align}
Likewise, it holds that
\begin{align}\label{phibarGf}
	\frac{1}{\eps}|(\partial^{\alpha}_\be[\frac{\partial_x\phi\partial_{v_{1}}\overline{G}}{\sqrt{\overline{M}}}],W^2(\al,\be)\partial^{\alpha}_\be f)|
    \leq \ka\frac{1}{\eps\de}\|\partial^\alpha_\beta f\|^2_{\sigma,W}+ C_\ka C(\eta)\de\CD(t)+C_\ka C(\eta)\eps\de^3,
\end{align}
and
\begin{align}\label{patbarGf}
	&|(\partial^{\alpha}_\be[\frac{ P_{1}(v_1\partial_x\overline{G})}{\sqrt{\overline{M}}}],W^2(\al,\be)\partial^{\alpha}_\be f)+(\partial^{\alpha}_\be[\frac{\partial_{t}\overline{G}}{\sqrt{\overline{M}}}],W^2(\al,\be)\partial^{\alpha}_\be f)|\notag\\
	\leq& \ka\frac{1}{\eps\delta}\| \langle v\rangle^{-3}W(\al,\be)\pa^\al_\be f\|^2+C_\ka\eps\de\|\langle v\rangle^{3}W(\al,\be)\partial^{\alpha}_\be[\frac{ P_{1}(v_1\partial_x\overline{G})+\partial_{t}\overline{G}}{\sqrt{\overline{M}}}]\|^2\notag\\
	\leq&  C\ka\frac{1}{\eps\de}\|\partial^\alpha_\beta f\|^2_{\sigma,W}+ C_\ka C(\eta)\de\CD(t)
    +C(\eta)\CE(t)+C_\ka C(\eta)(\eps\de^3+\eps^3\de),
\end{align}
where in the last inequality we have used \eqref{estpatpax}. 
Hence, it follows by \eqref{ipf}-\eqref{patbarGf} that
\begin{align}\label{albef}
	&\frac{d}{dt}\|\partial^{\alpha}_\be f\|_W^{2}+c\frac{1}{\eps\de}\|\partial^{\alpha}_\be f\|_{\sigma,W}^{2}-\ka_1\frac{1}{\eps\de}\sum_{|\beta'|=|\beta|}\|\partial^\alpha_{\beta'} f\|_{\sigma,W}^2-C_{\ka_1}\frac{1}{\eps\de}\sum_{|\beta'|<|\beta|}\|\partial^\alpha_{\beta'}f\|_{\sigma,W}^2\notag\\
	\leq& C\ka\frac{1}{\eps\de}\|\partial^\alpha_\beta f\|^2_{\sigma,W}+
    C_\ka(\eps+\de)\mathcal{D}(t)+C\frac{1}{\eps\de}\sum_{\al'<\al,|\beta'|\leq|\beta|}\|\partial^{\alpha'}_{\beta'}f\|_{\sigma,W}^2+C_{\ka}\frac{1}{\eps\de} \sum_{\al'\leq\al,|\beta'|\leq|\beta|}\|\partial^{\alpha'}_{\beta'}g\|_{\sigma,w}^2\notag\\
	&+ C_\ka\eps\de(\|\pa^{\alpha}_x\pa_x(\widetilde{u},\widetilde{\theta})\|^{2}+\|\pa^{\alpha}_x\pa_xf\|_{\si,W}+\|\pa^{\alpha}_x\pa_xg\|_{\si,w}^{2}+\|\pa^2_xg\|_{\si,w}^{2})+C_\ka C(\eta)(\CE(t)+\eps\de^3+\eps^3\de),
\end{align}
for $\al+|\be|\leq2$, $|\be|\geq1$, and
\begin{align}\label{alf}
	&\frac{d}{dt}\|\partial^{\alpha}_x f\|_W^{2}+c\frac{1}{\eps\de}\|\partial^{\alpha}_x f\|_{\sigma,W}^{2}\notag\\
	\leq&
    C\ka\frac{1}{\eps\de}\|\partial^\alpha f\|^2_{\sigma,W}+
    C_\ka(\eps+\de)\mathcal{D}(t)+C\frac{1}{\eps\de}\|\partial^{\alpha}_x f\|_{\sigma}^{2}+C\frac{1}{\eps\de}\sum_{\al'<\al}\|\partial^{\alpha'}_xf\|_{\sigma,W}^2+C_\ka\frac{1}{\eps\de} \sum_{\al'\leq\al}\|\partial^{\alpha'}_xg\|_{\sigma,w}^2\notag\\
	&+ 
C_\ka\eps\de(\|\pa^{\alpha}_x\pa_x(\widetilde{u},\widetilde{\theta})\|^{2}+\|\pa^{\alpha}_x\pa_xf\|_{\si,W}+\|\pa^{\alpha}_x\pa_xg\|_{\si,w}^{2}
+\|\pa^2_xg\|_{\si,w}^{2})+C_\ka C(\eta)(\CE(t)+\eps\de^3+\eps^3\de),
\end{align}
for $\al\leq1$. We take suitable linear combination of \eqref{albef} over all  $\al+|\be|\leq2$, $|\be|\geq1$, and then choose $\ka$ and $\ka_1$ to be sufficiently small to get
\begin{align*}
	&\frac{d}{dt}\overline{C}_{\alpha,\be}\sum_{\alpha+|\beta|\leq 2,|\beta|\geq 1}\|\partial^{\alpha}_\be f\|_W^{2}+c\frac{1}{\eps\de}\sum_{\al+|\beta|\leq 2,|\beta|\geq 1}\|\partial^{\alpha}_\be f\|_{\sigma,W}^{2}\notag\\
	\leq&C\frac{1}{\eps\de}\sum_{\al\leq 1}(\|\partial^\alpha_xf\|_{\sigma,W}^2+\|\partial^{\alpha}_xg\|_{\sigma,w}^2)
    +C\eps\de
 \sum_{\alpha\leq 1}\|\pa^{\alpha}_x\pa_x(\widetilde{u},\widetilde{\theta})\|^{2}
 +C\eps\de(\|\pa^2_xf\|_{\si,W}+\|\pa^2_xg\|_{\si,w}^{2})
 \notag\\
	&+C\frac{1}{\eps\de} \sum_{\alpha+|\beta|\leq 2,|\beta|\geq 1}\|\partial^{\alpha}_{\beta}g\|_{\sigma,w}^2
    +C(\eps+\de)\mathcal{D}(t)
    +C(\eta)(\CE(t)+\eps\de^3+\eps^3\de).
\end{align*}
Similarly, we have from \eqref{alf} that
\begin{align*}
	&\frac{d}{dt}
    \overline{C}_{\alpha}\sum_{\alpha\leq 1}
    \|\partial^{\alpha}_x f\|_W^{2}+c\frac{1}{\eps\de}\sum_{\alpha\leq 1}\|\partial^{\alpha}_x f\|_{\sigma,W}^{2}\notag\\
	\leq&
   C\frac{1}{\eps\de}\sum_{\alpha\leq 1}(\|\partial^{\alpha}_x f\|_{\sigma}^{2}+\|\partial^{\alpha}_xg\|_{\sigma,w}^2)
   +  C\sum_{\alpha\leq 1}\eps\de\|\pa^{\alpha}_x\pa_x(\widetilde{u},\widetilde{\theta})\|^{2}
    \notag\\
	& +C\eps\de(\|\pa^2_xf\|_{\si,W}+\|\pa^2_xg\|_{\si,w}^{2})
    +
 C(\eps+\de)\mathcal{D}(t)+
C(\eta)(\CE(t)+\eps\de^3+\eps^3\de).
\end{align*}
By the linear combination of the above two estimates, we can prove 
\eqref{walbef} hold true by using
\begin{equation}
\label{4.114A}
\sum_{\alpha\leq 1}\|\partial^\alpha_x f\|^2_{\sigma}\leq C\sum_{\alpha\leq 1}(\|\partial^{\alpha}_x\mathbf{P}_1 f\|_{\sigma}^{2} +\|\partial^\alpha_x g\|^2_{\sigma}).   
\end{equation}
\end{proof}
In the above lemma, we observe that due to the term $C\frac{1}{\eps\de}\sum_{\alpha\leq 1}\|\partial^\alpha_x \mathbf{P}_1f\|^2_{\sigma}$ in \eqref{walbef}, it is necessary to obtain the estimate on $\|\pa_x^\al \mathbf{P}_1f\|$ without the weight function, which is convenient to get since there will be many similar terms as in the proof of the above lemma.	
\begin{lemma}\label{lem4.13}
Under the a priori assumption \eqref{apriori}, it holds that
	\begin{align}\label{al01f}
		&\frac{d}{dt}\sum_{\alpha\leq 1}\|\partial^{\alpha}_x f\|^{2}+c\frac{1}{\eps\de}\sum_{\alpha\leq 1}\|\partial^{\alpha}_x\mathbf{P}_1 f\|_{\sigma}^{2}\notag\\
	\leq&
C\eps\de\sum_{\alpha\leq 1}\|\pa^{\alpha}_x\pa_x(\widetilde{u},\widetilde{\theta})\|^{2} +C\eps\de\|\pa^2_x(f,g)\|_{\sigma}^{2}+C\frac{1}{\eps\de}\sum_{\alpha\leq 1}\|\partial^{\alpha}_x g\|_{\sigma}^{2}
    \notag\\
&+C(\eps+\de)\mathcal{D}(t)+C(\eta)(\eps\de^3+\eps^3\de)+C(\eta)\CE(t),
	\end{align}	
where $C(\eta)\geq0$ depends only on $\eta$ with $C(0)=0$.
\end{lemma}	
\begin{proof}
	Let $\alpha\leq 1$. Applying $\partial^{\alpha}_x$ to \eqref{eqf} and taking the inner product with $\partial^{\alpha}_x f$, we obtain
	\begin{align}\label{L2f}
		&(\partial_t\partial^{\alpha}_x f,\partial^{\alpha}_x f)+(\partial^{\alpha}_x (v_1\pa_xf),\partial_x^\alpha f)
		-\frac{1}{\eps\de}(\partial^{\alpha}_x \CL_{\overline{M}} f,\partial^{\alpha}_x f)
		\notag\\
		=&\frac{1}{\eps\de}(\partial^{\alpha}_x[\frac{\sqrt{\mu}}{\sqrt{\overline{M}}}L_Bg],\partial^{\alpha}_x f)+(\partial^{\alpha}_x[\frac{ P_{0}(v_{1}\partial_x(\sqrt{\overline{M}}f)+v_{1}\sqrt{\mu}\partial_xg)}{\sqrt{\overline{M}}}],\partial^{\alpha}_x f)
        \notag\\
		&	-(\partial^{\alpha}_x[\frac{1}{\sqrt{\overline{M}}} P_{1}\big\{v_{1}M(\frac{|v-u|^{2}\partial_x\widetilde{\theta}}{2K\theta^{2}}+\frac{(v-u)\cdot\partial_x\widetilde{u}}{K\theta})\big\}],\partial^{\alpha}_x f)+\frac{1}{\eps}(\partial^{\alpha}_x[\frac{\partial_x\overline{G}}{\sqrt{\overline{M}}}],\partial^{\alpha}_x f)
		\notag\\
		&
		+\frac{1}{\eps}(\partial^{\alpha}_x[\frac{\partial_x\phi\partial_{v_{1}}\overline{G}}{\sqrt{\overline{M}}}],\partial^{\alpha}_x f)-(\partial^{\alpha}_x[\frac{ P_{1}(v_1\partial_x\overline{G})}{\sqrt{\overline{M}}}],\partial^{\alpha}_x f)-(\partial^{\alpha}_x[\frac{\partial_{t}\overline{G}}{\sqrt{\overline{M}}}],\partial^{\alpha}_x f).
	\end{align}	
Then it holds that
\begin{align*}
(\partial_t\partial^{\alpha}_x f,\partial^{\alpha}_x f)+(\partial^{\alpha}_x (v_1\pa_xf),\partial^{\alpha}_x f)=\frac{1}{2}\frac{d}{dt}\|\pa^\al_x f\|^2.
\end{align*}
If $\al=0$, we have from \eqref{controlLbarM} that 
$$
-\frac{1}{\eps\de}(\pa^\al_x\CL_{\overline{M}}f,\pa^\al_xf)
\geq c\frac{1}{\eps\de}\|\pa^\al_x\mathbf{P}_1f\|^{2}_{\sigma}.
$$
If $\al=1$, we have from \eqref{controlLbarM}, \eqref{DefLM} and \eqref{controlGa} that
\begin{align*}
-\frac{1}{\eps\de}(\pa^\al_x\CL_{\overline{M}}f,\pa^\al_xf)
&=-\frac{1}{\eps\de}(\CL_{\overline{M}}\pa^\al_xf,\pa^\al_xf)
-\frac{1}{\eps\de}\sum_{\al_1+\al_2+\al_3+\al'=\al,\al'<\al}(\pa^{\al_1}_x(\frac{1}{\sqrt{\overline{M}}})Q(\pa^{\al_2}_x\overline{M},\pa^{\al_3}_x(\sqrt{\overline{M}})\partial^{\al'}_xf)
\notag\\
&\quad \quad +\pa^{\al_1}_x(\frac{1}{\sqrt{\overline{M}}})Q(\pa^{\al_2}_x(\sqrt{\overline{M}})\pa^{\al'}_xf,\pa^{\al_3}_x\overline{M}),\pa^\al_xf)
\notag\\
&\geq \frac{c}{2}\frac{1}{\eps\de}\|\pa^\al_x\mathbf{P}_1f\|^{2}_{\sigma}-C(\eta)\frac{1}{\eps\de}\sum_{\al'<\al}\|\partial^{\al'}_x f\|_{\sigma}
(\|\partial^{\al}_x \mathbf{P}_1f\|_\sigma+\|\partial^{\al}_x \mathbf{P}_0f\|_\sigma)
\notag\\
		&\geq c\frac{1}{\eps\de}\|\pa^\al_x\mathbf{P}_1f\|^{2}_{\sigma}-C\frac{1}{\eps\de}(\|\mathbf{P}_1 f\|^2_{\sigma}+\|g\|^2_{\sigma}+\|\partial_x g\|^2_{\sigma}),
	\end{align*}
	where the last inequality holds from \eqref{4.114A}.
	The rest terms on the right hand side of \eqref{L2f} can be bounded in the similar way as in \eqref{paLB}, \eqref{P0ff}, \eqref{thetauf}, \eqref{barGf}, \eqref{phibarGf} and  \eqref{patbarGf}. One arrives at
	\begin{align*}
		&\big|\frac{1}{\eps\de}(\partial^{\alpha}_x[\frac{\sqrt{\mu}}{\sqrt{\overline{M}}}L_Bg],\partial^{\alpha}_x f)+(\partial^{\alpha}_x\frac{ P_{0}(v_{1}\partial_x(\sqrt{\overline{M}}f)+v_{1}\sqrt{\mu}\partial_xg)}{\sqrt{\overline{M}}},\partial^{\alpha}_x f)\notag\\
		&\quad	-(\partial^{\alpha}_x[\frac{1}{\sqrt{\overline{M}}} P_{1}\big\{v_{1}M(\frac{|v-u|^{2}\partial_x\widetilde{\theta}}{2K\theta^{2}}+\frac{(v-u)\cdot\partial_x\widetilde{u}}{K\theta})\big\}],\partial^{\alpha}_x f)+\frac{1}{\eps}(\partial^{\alpha}_x[\frac{\partial_x\overline{G}}{\sqrt{\overline{M}}}],\partial^{\alpha}_x f)
		\notag\\
		&\quad
		+\frac{1}{\eps}(\partial^{\alpha}_x[\frac{\partial_x\phi\partial_{v_{1}}\overline{G}}{\sqrt{\overline{M}}}],\partial^{\alpha}_x f)-(\partial^{\alpha}_x[\frac{ P_{1}(v_1\partial_x\overline{G})}{\sqrt{\overline{M}}}],\partial^{\alpha}_x f)-(\partial^{\alpha}_x[\frac{\partial_{t}\overline{G}}{\sqrt{\overline{M}}}],\partial^{\alpha}_x f)\big|\notag\\
		\leq&\ka\frac{1}{\eps\de}\|\partial_x^\alpha f\|^2_{\sigma}+C_\ka\eps\de
        (\|\pa^{\alpha}_x\pa_x(\widetilde{u},\widetilde{\theta})\|^{2}+\|\pa^{\alpha}_x\pa_x(f,g)\|_{\sigma}^{2})
+C_\ka(\eps+\de)\mathcal{D}(t)
        \notag\\
		&+C_\ka C(\eta)(\eps\de^3+\eps^3\de)+C_\ka C(\eta)\CE(t)+C_\ka\frac{1}{\eps\de}\sum_{\alpha\leq 1}\|\partial^{\alpha}_x g\|_{\sigma}^{2}.
	\end{align*}
Plugging the the above estimates into \eqref{L2f} and then the desired estimate \eqref{al01f} follows from 
a suitable linear combination of the resulting equation over $\al\leq1$.
This completes the proof of Lemma \ref{lem4.13}.
\end{proof}
Combining \eqref{walbef} and \eqref{al01f}, we obtain our final estimate for lower order part of $f$.
\begin{lemma}\label{le01f}
Under the a priori assumption \eqref{apriori}, there exist 
some positive constants $\overline{C}'_{\alpha}$, $\overline{C}_{\alpha,\be}$ and  sufficiently large $\overline{C}_1$ such that
\begin{align}\label{01f}
	&\frac{d}{dt}\big\{\sum_{\alpha\leq 1}\overline{C}'_{\alpha}\|\partial^{\alpha}_x f\|_W^{2}+\overline{C}_1\sum_{\alpha\leq 1}\|\partial^{\alpha}_x f\|^{2}
    +\sum_{\alpha+|\beta|\leq 2,|\beta|\geq 1}\overline{C}_{\alpha,\be}\|\partial^{\alpha}_\be f\|_W^{2}\big\}
     \notag\\
    &+c\frac{1}{\eps\de}\big\{\overline{C}_1
    \sum_{\alpha\leq 1}\|\partial^{\alpha}_x\mathbf{P}_1 f\|_{\sigma}^{2}+
    \sum_{\alpha\leq 1}\|\partial^{\alpha}_x f\|_{\sigma,W}^{2}+\sum_{\al+|\beta|\leq 2,|\beta|\geq 1}\|\partial^{\alpha}_\be f\|_{\sigma,W}^{2}\big\}\notag\\
	\leq&C\eps\de(\|\pa^2_xf\|_{\sigma,W}^{2}+\|\pa^2_xg\|_{\sigma,w}^{2})+C(\eps+\de)\mathcal{D}(t)+C(\eta)(\eps^3\de+\eps\de^3)+C(\eta)\CE(t)
    \notag\\
	&+C\eps\de  \sum_{\alpha\leq 1}\|\pa^{\alpha}_x\pa_x(\widetilde{u},\widetilde{\theta})\|^{2}      +C\frac{1}{\eps\de}\big\{\sum_{\alpha\leq 1}\|\partial^{\alpha}_x g\|_{\sigma,w}^{2}+\sum_{\al+|\beta|\leq 2,|\beta|\geq 1}\|\partial^{\alpha}_\be g\|_{\sigma,w}^{2}\big\},
\end{align}	
where $C(\eta)\geq0$ depends only on $\eta$ with $C(0)=0$.
\end{lemma}

\subsection{Second order energy}	
We now turn to the highest order energy for both fluid and non-fluid parts. We observe that the equation of $g$ in \eqref{eqg}, besides the first three terms in \eqref{eqg} which vanish after taking inner product, the others will not increase the $x$ derivatives. Hence, we estimate $g$ by \eqref{eqg} and make use of the original equation for $F$ to bound the rest parts including the fluid quantities and $f$.
\begin{lemma}
\label{lehighg}
Under the a priori assumption \eqref{apriori}, it holds that
\begin{align}\label{w2g}
&\eps^2\de^2\|\partial^{2}_x g(t)\|_w^{2}+\eps\de\int_0^{t}\|\partial^{2}_x g(s)\|_{\sigma,w}^{2}\,ds+\eps^2\de^2 q_{1}q_{2}\int_0^{t}(1+s)^{-(1+q_2)}\|\langle v\rangle \partial^{2}_xg(s)\|_{w}^{2}\,ds
\notag\\
\leq&C\CE(0)+C(\eta)t\eps^2\de^2
 +C(\eps+\de)\int_0^{t}\mathcal{D}(s)\,ds
\notag\\
&+C(\eta)\int_0^{t}\CE(s)\,ds
+C\eps^2\de^2(C(\eta)+\eps)\frac{\de^{1/2}}{\eps}\int_0^{t}\|\langle v\rangle \pa^2_x g(s)\|_w^2\,ds,
\end{align}	
where $C(\eta)\geq0$ depends only on $\eta$ with $C(0)=0$.
\end{lemma}
\begin{proof}
	Applying $\partial^{2}_x$ to \eqref{eqg} and taking the inner product with $\eps^2\de^2w^2(2,0)\partial^{2}_x g$, one has
	\begin{align}\label{2ipg}
		&\eps^2\de^2(\partial_t\partial^2_x g,w^2(2,0)\partial^2_x g)-\eps\de^2(\partial^2_x[\frac{\partial_x\phi\partial_{v_{1}}(\sqrt{\mu}g)}{\sqrt{\mu}}],w^2(2,0)\partial^2_x g)
		-\eps\de(\partial^2_x L_D g,w^2(2,0)\partial^2_x g)
        \notag\\&+\eps^2\de^2(\partial^2_x[\frac{(\pa_t+v_1\pa_x)\sqrt{\overline{M}}}{\sqrt{\mu}}f],w^2(2,0)\partial^2_x g)-\eps\de^2(\partial^2_x[\frac{\pa_x\sqrt{\overline{M}}}{\sqrt{\mu}}f],w^2(2,0)\partial^2_x g)\notag\\
		&-\eps\de^2(\partial^2_x[\frac{\partial_x\phi\partial_{v_{1}}(\sqrt{\overline{M}}f)}{\sqrt{\mu}}],w^2(2,0)\partial^2_x g)
		\notag\\
		=&\eps\de(\partial^2_x \Gamma(g+\frac{\sqrt{\overline{M}}}{\sqrt{\mu}}f+\frac{\overline{G}}{\sqrt{\mu}},g+\frac{\sqrt{\overline{M}}}{\sqrt{\mu}}f+\frac{\overline{G}}{\sqrt{\mu}}),w^2(2,0)\partial^2_x g)\notag\\
		&
		+\eps\de(\partial^2_x\Gamma(\frac{M-\overline{M}}{\sqrt{\mu}},\frac{\sqrt{\overline{M}}}{\sqrt{\mu}}f)+\partial^2_x\Gamma(\frac{\sqrt{\overline{M}}}{\sqrt{\mu}}f,\frac{M-\overline{M}}{\sqrt{\mu}}),w^2(2,0)\partial^2_x g).
	\end{align}
	It follows by \eqref{patw} that
	\begin{align}\label{2gtime}
		\eps^2\de^2(\partial_{t}\partial^2_x g,w^2(2,0)\pa^2_x g)=&\frac{1}{2}\eps^2\de^2\frac{d}{dt}(\pa^2_x g,w^2(2,0)\pa^2_x g)
		-\frac{1}{2}\eps^2\de^2(\pa^2_x g,\partial_{t}[w^2(2,0)]\pa^2_x g)
		\notag\\
		=&\frac{1}{2}\eps^2\de^2\frac{d}{dt}\|\pa^2_x g\|_{w}^2+\frac{1}{2}\eps^2\de^2q_{1}q_{2}(1+t)^{-(1+q_2)}\|\langle v\rangle \pa^2_x g\|^2_{w}.
	\end{align}
We write
\begin{align}\label{2phif1}
	\eps\de^2(\partial^2_x [\frac{\partial_x\phi\partial_{v_1}(\sqrt{\mu}g)}{\sqrt{\mu}}],w^2(2,0)\partial^2_x g)
	&=\eps\de^2(\partial^2_x[\partial_x\phi\partial_{v_1}g],w^2(2,0)\partial^2_x g)-\eps\de^2(\partial^2_x[\frac{v_1}{2} g\partial_x\phi],w^2(2,0)\partial^2_x g)\notag\\
	&=\eps^2\de^2\widetilde{I}_3-\eps^2\de^2\widetilde{I}_4.
\end{align}
 Then one has
\begin{align*}
	\eps^2\de^2\widetilde{I}_3=&\eps\de^2(\partial_x\phi\partial_{v_1}\partial^2_x g,w^2(2,0)\partial^2_x g)+2\eps\de^2(\partial^2_x\phi\partial_{v_1}\partial_x g,w^2(2,0)\partial^2_x g)+\eps\de^2(\partial^3_x\phi\partial_{v_1}g,w^2(2,0)\partial^2_x g).
\end{align*}
Similar calculation as in \eqref{I31} gives
\begin{align}\label{2I311}
	&|\eps\de^2(\partial_x\phi\partial_{v_1}\partial^2_x g,w^2(2,0)\partial^2_x g)|
	\leq C\ka\eps\de\| \partial^2_x g\|_{\sigma,w}^2
	+C_\ka\eps^2\de^2(C(\eta)+\eps)\frac{\de^{1/2}}{\eps}\|\langle v\rangle w(2,0)\partial^2_x g\|^2.	
\end{align}
By using \eqref{apriori}, \eqref{DefD}, \eqref{boundlandaunorm} and $\langle v\rangle^{2}w(2,0)\leq w(1,0)$, we obtain
\begin{align}\label{2I312}
	&|2\eps\de^2(\partial^2_x\phi\partial_{v_1}\partial_x g,w^2(2,0)\partial^2_x g)+\eps\de^2(\partial^3_x\phi\partial_{v_1}g,w^2(2,0)\partial^2_x g)|\notag\\
	\leq&	C\eps\de^2\|\partial^2_x\phi\|_{L^\infty}\|\langle v\rangle^2 \langle v\rangle^{-\frac{3}{2}} w(2,0)\partial_{v_1}\pa_x g\|\|\langle v\rangle^{-\frac{1}{2}}w(2,0)\partial^2_xg\|\notag\\
	&\qquad\qquad+C\eps\de^2\|\partial^3_x\phi\|\|\langle v\rangle^2 \langle v\rangle^{-\frac{3}{2}} w(2,0)\partial_{v_1} g\|_{L^\infty}\|\langle v\rangle^{-\frac{1}{2}}w(2,0)\partial^2_xg\|\notag\\
	\leq&	C\eps\de^2(C(\eta)+\frac{A^{1/4}\eps}{\eps^{1/4}}\frac{A^{1/4}\eps}{(\eps^3\de^2)^{1/4}}+\frac{A^{1/2}\eps^2}{(\eps^3\de^2)^{1/2}})\CD(t)\notag\\
	\leq&	C(\eps+\de)\CD(t).
\end{align}
The combination of the above three estimates gives
\begin{align}\label{2I3}
	\eps^2\de^2\widetilde{I}_3\leq C\ka\eps\de\| \partial^2_x g\|_{\sigma,w}^2
	+C_\ka\eps^2\de^2(C(\eta)+\eps)\frac{\de^{1/2}}{\eps}\|\langle v\rangle \partial^2_x g\|_w^2+C(\eps+\de)\CD(t).
\end{align}
  A direct calculation gives
\begin{align}\label{re2I4}
	\eps^2\de^2\widetilde{I}_4=	\eps\de^2(\frac{v_1}{2} \pa^2_xg\partial_x\phi+2\frac{v_1}{2} \pa_xg\partial^2_x\phi+\frac{v_1}{2} g\partial^3_x\phi,w^2(2,0)\partial^2_x g).
\end{align}
Similar arguments as in \eqref{2I311} and \eqref{2I312} show
\begin{align}\label{2I4}
	\eps^2\de^2|\widetilde{I}_4|\leq C\ka\eps\de\| \partial^2_x g\|_{\sigma,w}^2
	+C_\ka\eps^2\de^2(C(\eta)+\eps)\frac{\de^{1/2}}{\eps}\|\langle v\rangle \partial^2_x g\|_w^2+C(\eps+\de)\CD(t).
\end{align}
It holds by \eqref{2phif1}, \eqref{2I3} and \eqref{2I4} that
\begin{align}\label{2phig}
	&\eps\de^2(\partial^2_x[\frac{\partial_x\phi\partial_{v_{1}}(\sqrt{\mu}g)}{\sqrt{\mu}}],w^2(2,0)\partial^2_x g)
    \notag\\
    &\leq 
    C\ka\eps\de\| \partial^2_x g\|_{\sigma,w}^2+
    C(\eps+\de)\CD(t)
	+C_\ka\eps^2\de^2(C(\eta)+\eps)\frac{\de^{1/2}}{\eps}\|\langle v\rangle w(2,0)\partial^2_x g\|^2.
\end{align}
Using \eqref{coLDx}, the term involving the linear operator $L_D$ is controlled by
\begin{equation}\label{2Lg}
	-\eps\de(\partial^2_x L_D g,w^2(2,0)\partial^\alpha_x g)
	\geq c\eps\de\|\partial^2_x g\|^2_{\sigma,w}-C\eps\CD(t).
\end{equation}
We then use the same approach in \eqref{tranbarM} to 
bound the fourth term on the left hand side of \eqref{2ipg} as
\begin{align}\label{2tranbarM}
\eps^2\de^2|(\partial^2_x[\frac{(\pa_t+v_1\pa_x)\sqrt{\overline{M}}}{\sqrt{\mu}}f],w^2(2,0)\partial^2_x g)|\leq C(\eta)\eps\CD(t).
\end{align}
Likewise, it holds that
\begin{align}\label{barMfg}
	&|\eps\de^2(\partial^2_x[\frac{\pa_x\sqrt{\overline{M}}}{\sqrt{\mu}}f],w^2(2,0)\partial^2_x g)|\notag\\
	\leq& C\eps\de^2\sum_{\alpha'\leq 2}\{\|\langle v\rangle^6 w^2(2,0)\frac{\pa^{2-\al'}_x\pa_x\sqrt{\overline{M}}}{\sqrt{\mu}}\|_{L^\infty_xL^\infty_v}\|\langle v\rangle^{-3}\pa^{\al'}_xf\|\|\langle v\rangle^{-3}\pa^2_x g\|
	\leq C\de\CD(t),
\end{align}
and 
\begin{align}\label{phifg}
	&|\eps\de^2(\partial^2_x[\frac{\partial_x\phi\partial_{v_{1}}(\sqrt{\overline{M}}f)}{\sqrt{\mu}}],w^2(2,0)\partial^2_x g)|\notag\\
	\leq& C\eps\de^2\sum_{0<\alpha'\leq 2}\{\|\langle v\rangle^6 w^2(2,0)\mu^{1/16}\pa^{2-\al'}_x\pa_x\phi\|_{L^\infty_xL^\infty_v}\|\langle v\rangle^{-3}\mu^{-1/16}\pa^{\al'}_x\frac{\partial_{v_{1}}( \sqrt{\overline{M}}f)}{\sqrt{\mu}}\|\|\langle v\rangle^{-3}\pa^2_x g\|\notag\\
	&+C\eps\de^2\{\|\langle v\rangle^6 w^2(2,0)\mu^{1/16}\pa^{2}_x\pa_x\phi\|_{L^2_xL^\infty_v}\|\langle v\rangle^{-3}\mu^{-1/16}\pa^{\al'}_x\frac{\partial_{v_{1}}( \sqrt{\overline{M}}f)}{\sqrt{\mu}}\|_{L^\infty_xL^2_v}\|\langle v\rangle^{-3}\pa^2_x g\|\notag\\
	\leq& C\de\CD(t).
\end{align}
Furthermore, we use the argument as in \eqref{Gagg} to bound the first term on the right hand side of \eqref{2ipg} that
\begin{align}\label{2Gagg}
	\eps\de|(\partial^2_x \Gamma(g+\frac{\sqrt{\overline{M}}}{\sqrt{\mu}}f+\frac{\overline{G}}{\sqrt{\mu}},g+\frac{\sqrt{\overline{M}}}{\sqrt{\mu}}f+\frac{\overline{G}}{\sqrt{\mu}}),w^2(2,0)\partial^2_x g)|
	\leq C\eps\CD(t)+C(\eta)\eps^2\de^2.
\end{align}
Similar to \eqref{GaMf}, one has
\begin{align}\label{2GaMf}
\eps\de|(\partial^2_x\Gamma(\frac{M-\overline{M}}{\sqrt{\mu}},\frac{\sqrt{\overline{M}}}{\sqrt{\mu}}f)+\partial^{\alpha}_\be\Gamma(\frac{\sqrt{\overline{M}}}{\sqrt{\mu}}f,\frac{M-\overline{M}}{\sqrt{\mu}}),w^2(2,0)\partial^2_x g)|
\leq C\eps\CD(t).
\end{align}
Therefore, by collecting \eqref{2ipg}, \eqref{2gtime}, \eqref{2phig}, \eqref{2Lg}, \eqref{2tranbarM}, \eqref{barMfg}, \eqref{phifg}, \eqref{2Gagg} and \eqref{2GaMf}, then taking integral in $t$ of the resulting equation, we get the desired estimate \eqref{w2g}. This completes the proof of Lemma \ref{lehighg}.
\end{proof}

Note that if we apply $\pa^2_x$ to the equation of $f$ in \eqref{eqf} to get the second order energy, the $x$ derivative will increase to $\pa^3_x$ due to the terms involving $\pa_x$ on the right hand side of \eqref{eqf}. The same problem occurs in the fluid equation \eqref{perturbeq}. Thus, in order to get the highest order energy for such quantities, we should make use of the original equation \eqref{reVPL} where we first rewrite
\begin{align}
\label{4.133A}
	&\partial_tF-\frac{1}{\eps}\partial_xF+v_{1}\partial_xF
	-\frac{1}{\eps}\partial_x\phi\partial_{v_{1}}F
    -\frac{1}{\eps\de}L_{\overline{M}}(\sqrt{\overline{M}}f)\notag\\
	=&\frac{1}{\eps\de}
	\{Q(\sqrt{\overline{M}}f,M-\overline{M})+
	Q(M-\overline{M},\sqrt{\overline{M}}f)+Q(G,G)+Q(M,\overline{G}+\sqrt{\mu}g)+Q(\overline{G}+\sqrt{\mu}g,M)\},
\end{align}
where $L_{\overline{M}}(\sqrt{\overline{M}}f)=Q(\sqrt{\overline{M}}f,\overline{M})+ 	Q(\overline{M},\sqrt{\overline{M}}f)$. Then subtracting $\sqrt{\mu}\times\eqref{eqg}$ from \eqref{4.133A}, we have
\begin{align*}
	&\partial_tM+\sqrt{\overline{M}}\partial_tf-\frac{1}{\eps}\partial_xM-\frac{1}{\eps}\sqrt{\overline{M}}\partial_xf+v_{1}\partial_xM+\sqrt{\overline{M}}v_{1}\partial_xf
	-\frac{1}{\eps}\partial_x\phi\partial_{v_{1}}(M+\overline{G})\notag\\
	=&\frac{1}{\eps\de}
	L_{\overline{M}}\sqrt{\overline{M}}f-(\pa_t-\frac{1}{\eps}\pa_x+v_1\pa_x)\overline{G}+\frac{1}{\eps\de}L_M\overline{G}+\frac{1}{\eps\de}\sqrt{\mu}L_Bg,
\end{align*}
which, multiplying $1/\sqrt{\overline{M}}$, gives
\begin{align}\label{eqMf}
	&\partial_t(\frac{M}{\sqrt{\overline{M}}}+\frac{\overline{G}}{\sqrt{\overline{M}}}+f)-\frac{1}{\eps}\partial_x(\frac{M}{\sqrt{\overline{M}}}+\frac{\overline{G}}{\sqrt{\overline{M}}}+f)+v_{1}\partial_x(\frac{M}{\sqrt{\overline{M}}}+\frac{\overline{G}}{\sqrt{\overline{M}}}+f)
-\frac{1}{\eps}\frac{\partial_x\phi\partial_{v_{1}}(M+\overline{G})}{\sqrt{\overline{M}}}\notag\\
=&\frac{1}{\eps\de}
\CL_{\overline{M}}f+M(\pa_t-\frac{1}{\eps}\pa_x+v_1\pa_x)(\sqrt{\overline{M}}^{-1})+(\pa_t-\frac{1}{\eps}\pa_x+v_1\pa_x)(\frac{1}{\sqrt{\overline{M}}})\overline{G}
+\frac{1}{\eps\de}\frac{L_M\overline{G}}{\sqrt{\overline{M}}}+\frac{1}{\eps\de}\frac{\sqrt{\mu}L_Bg}{\sqrt{\overline{M}}}.
\end{align}
We use the above equation to obtain the highest order energy estimates for macroscopic quantities and $f$.
Due to similar reason as in \eqref{wLf}, it is also necessary to bound $F$ with and without the weight function. We start with the weighted estimate. 
\begin{lemma}\label{lem4.16}
Under the a priori assumption \eqref{apriori}, it holds that
	\begin{align}\label{2ndwdiss}
&\eps^2\de^2\|\pa^2_xf(t)\|^2_W+	\eps\de\int^t_0\|\pa^2_xf(s)\|^2_{\si,W}ds\notag\\
\leq	&C\CE(0)+C(\eta)\eps^2\de^2+C\eps^5
+C\eps^2\de^2\|\pa^2_x(\widetilde{\rho},\widetilde{u},\widetilde{\theta})(t)\|^2+C\ka_2\eps\de\int^t_0\|\pa^2_x(\widetilde{\rho},\widetilde{u},\widetilde{\theta})(s)\|^2\,ds
\notag\\
&+C\frac{1}{\ka_2}\eps\de\int^t_0\|(\pa^2_xg,\partial^{2}_x\mathbf{P}_1 f)(s)\|^2_{\si}\,ds
+(C\ka+C_\ka(\de+\eps+\frac{\de^2}{\eps}))\int^t_0\CD(s)\,ds+C_\ka  C(\eta)t(\eps\de^2+\de^3),
	\end{align}	
    for any $0<\ka<1$ and $0<\ka_2<1$.
\end{lemma}
\begin{proof}
Applying $\pa^2_x$ to both side of \eqref{eqMf} and taking the inner product with $\eps^2\de^2 W^2(2,0)\pa^2_x(\frac{M}{\sqrt{\overline{M}}}+\frac{\overline{G}}{\sqrt{\overline{M}}}+f)$, it holds that
\begin{align}\label{ipwF}
	&\eps^2\de^2\frac{1}{2}\frac{d}{dt}\|\pa^2_x(\frac{M}{\sqrt{\overline{M}}}+\frac{\overline{G}}{\sqrt{\overline{M}}}+f)\|_W^2
	-\eps\de^2(\pa^2_x[\partial_x\phi\partial_{v_{1}}(\frac{M+\overline{G}}{\sqrt{\overline{M}}})],W^2(2,0)\pa^2_x(\frac{M}{\sqrt{\overline{M}}}+\frac{\overline{G}}{\sqrt{\overline{M}}}+f))
	\notag\\
	=&\eps\de(\pa^2_x\CL_{\overline{M}} f,W^2(2,0)\pa^2_x(\frac{M}{\sqrt{\overline{M}}}+\frac{\overline{G}}{\sqrt{\overline{M}}}+f))
    \notag\\
    &+\eps^2\de^2(\pa^2_x(M(\pa_t-\frac{1}{\eps}\pa_x+v_1\pa_x)(\sqrt{\overline{M}}^{-1})),W^2(2,0)\pa^2_x(\frac{M}{\sqrt{\overline{M}}}+\frac{\overline{G}}{\sqrt{\overline{M}}}+f))
    \notag\\
	&+\eps^2\de^2(\pa^2_x((\pa_t-\frac{1}{\eps}\pa_x+v_1\pa_x)(\frac{1}{\sqrt{\overline{M}}})\overline{G}),W^2(2,0)\pa^2_x(\frac{M}{\sqrt{\overline{M}}}+\frac{\overline{G}}{\sqrt{\overline{M}}}+f))
    \notag\\
    &+\eps\de(\pa^2_x(\frac{L_M\overline{G}}{\sqrt{\overline{M}}}),W^2(2,0)\pa^2_x(\frac{M}{\sqrt{\overline{M}}}+\frac{\overline{G}}{\sqrt{\overline{M}}}+f))
    \notag\\
	&+\eps\de(\pa^2_x(\frac{\sqrt{\mu}L_Bg}{\sqrt{\overline{M}}}),W^2(2,0)\pa^2_x(\frac{M}{\sqrt{\overline{M}}}+\frac{\overline{G}}{\sqrt{\overline{M}}}+f))
    \notag\\
	&-\eps\de^2(\pa^2_x[\partial_x\phi(M+\overline{G})\partial_{v_{1}}\sqrt{\overline{M}}^{-1}],W^2(2,0)\pa^2_x(\frac{M}{\sqrt{\overline{M}}}+\frac{\overline{G}}{\sqrt{\overline{M}}}+f)).
\end{align}
For the electric field term, a direct calculation shows
\begin{align}\label{phi2}
	\eps&\de^2(\pa^2_x[\partial_x\phi\partial_{v_{1}}(\frac{M+\overline{G}}{\sqrt{\overline{M}}})],W^2(2,0)\pa^2_x(\frac{M}{\sqrt{\overline{M}}}+\frac{\overline{G}}{\sqrt{\overline{M}}}+f))\notag\\
	=&\eps\de^2(\partial^3_x\phi\partial_{v_{1}}(\frac{M}{\sqrt{\overline{M}}}),W^2(2,0)\pa^2_x(\frac{M}{\sqrt{\overline{M}}}+\frac{\overline{G}}{\sqrt{\overline{M}}}+f))
    \notag\\
    &+2\eps\de^2(\partial^2_x\phi\partial_{v_{1}}\pa_x(\frac{M}{\sqrt{\overline{M}}}),W^2(2,0)\pa^2_x(\frac{M}{\sqrt{\overline{M}}}+\frac{\overline{G}}{\sqrt{\overline{M}}}+f))\notag\\ &+\eps\de^2(\partial_x\phi\partial_{v_{1}}\pa^2_x(\frac{M}{\sqrt{\overline{M}}}),W^2(2,0)\pa^2_x(\frac{M}{\sqrt{\overline{M}}}+\frac{\overline{G}}{\sqrt{\overline{M}}}+f))
    \notag\\
    &+\eps\de^2(\pa^2_x[\partial_x\phi\partial_{v_{1}}(\frac{\overline{G}}{\sqrt{\overline{M}}})],W^2(2,0)\pa^2_x(\frac{M}{\sqrt{\overline{M}}}+\frac{\overline{G}}{\sqrt{\overline{M}}}+f)).
\end{align}
Notice that
\begin{align}\label{paM}
	\partial^2_xM=&M\big(\frac{\pa^2_x\rho}{\rho}+\frac{(v-u)\cdot\pa^2_xu}{K\theta}+(\frac{|v-u|^{2}}{2K\theta}-\frac{3}{2})\frac{\pa^2_x\theta}{\theta} \big)
	\nonumber\\
	&+\pa_x(M\frac{1}{\rho})\pa_x\rho+\pa_x(M\frac{v-u}{K\theta})\cdot\pa_xu+\pa_x(M\frac{|v-u|^{2}}{2K\theta^{2}}-M\frac{3}{2\theta})\pa_x\theta.
\end{align}
Then for any $k>0$, it follows from \eqref{apriori}, \eqref{4.5A}, \eqref{boundkdv}, \eqref{smalleps} and \eqref{DefD} and \eqref{DefE} that
\begin{align}\label{wM}
	\|\langle v\rangle^k\frac{\partial^2_xM}{\sqrt{\overline{M}}}\|^2+\|\langle v\rangle^k\frac{\partial^2_xM}{\sqrt{\overline{M}}}\|^2_{\si,w}\leq& C\|\partial^2_x(\rho,u,\theta)\|^2
	+C\|\partial_x(\rho,u,\theta)\|_{L^{\infty}}^2\|\partial_x(\rho,u,\theta)\|^2\notag\\
	\leq& C\|\pa^2_x(\widetilde{\rho},\widetilde{u},\widetilde{\theta})\|^2+C(\eta)+C(\eta)(\|\partial_x(\widetilde{\rho},\widetilde{u},\widetilde{\theta})\|^2+
    \|\partial_x(\widetilde{\rho},\widetilde{u},\widetilde{\theta})\|_{L^{\infty}}^2)
    \notag\\
	&+C\|\partial^2_x(\widetilde{\rho},\widetilde{u},\widetilde{\theta})\|\|\partial_x(\widetilde{\rho},\widetilde{u},\widetilde{\theta})\|^3\notag\\
	\leq& C(\eta)+C\frac{1}{\eps\de}\CD(t).
\end{align}
We can now bound the right hand side of \eqref{phi2} by
\begin{align}\label{phiMMf}
	&C\eps\de^2(\|\partial^3_x\phi\|\|\partial_{v_{1}}(\frac{M}{\sqrt{\overline{M}}})W^2(2,0)\langle v\rangle^\frac{1}{2}\|_{L^\infty_xL^\infty_v}+\|\partial^2_x\phi\|\|\partial_{v_{1}}\pa_x(\frac{M}{\sqrt{\overline{M}}})W^2(2,0)\langle v\rangle^\frac{1}{2}\|_{L^\infty_xL^\infty_v}\notag\\
	&\qquad+\|\partial_x\phi\|_{L^\infty}\|\partial_{v_{1}}\pa^2_x(\frac{M}{\sqrt{\overline{M}}})W^2(2,0)\langle v\rangle^\frac{1}{2}\|_{L^2_xL^\infty_v}+\|\pa^2_x[\partial_x\phi\partial_{v_{1}}(\frac{\overline{G}}{\sqrt{\overline{M}}})]W^2(2,0)\langle v\rangle^{\frac{1}{2}}\|)
    \notag\\
    &\times\|\langle v\rangle^{-\frac{1}{2}}\pa^2_x(\frac{M}{\sqrt{\overline{M}}}+\frac{\overline{G}}{\sqrt{\overline{M}}}+f)\|\notag\\
	\leq&C(\ka+C_\ka(\eps+\de+\frac{\de^2}{\eps}))\CD(t)+C_\ka C(\eta)(\eps\de^2+\de^3),
\end{align}
where we have used \eqref{wM}, \eqref{DefE}, \eqref{DefD}, \eqref{boundlandaunorm}, \eqref{apriori} and the fact that
\begin{align*}
	&\eps\de^2\|\partial^3_x\phi\|\|\partial_{v_{1}}\frac{M}{\sqrt{\overline{M}}}W^2(2,0)\langle v\rangle^\frac{1}{2}\|_{L^\infty_xL^\infty_v}\|\langle v\rangle^{-\frac{1}{2}}\pa^2_x(\frac{M}{\sqrt{\overline{M}}}+\frac{\overline{G}}{\sqrt{\overline{M}}}+f)\|
    \notag\\
    &\leq C\eps\de^2\|\partial^3_x\phi\|\|\langle v\rangle^{-\frac{1}{2}}\pa^2_x(\frac{M}{\sqrt{\overline{M}}}+\frac{\overline{G}}{\sqrt{\overline{M}}}+f)\|
    \notag\\
	&\leq  C\eps\de^2(C(\eta)+\|\pa^3_x\widetilde{\phi}\|)
    (\|\langle v\rangle^{-\frac{1}{2}}\pa^2_x(\frac{M}{\sqrt{\overline{M}}})\|+\|\langle v\rangle^{-\frac{1}{2}}\pa^2_x(\frac{\overline{G}}{\sqrt{\overline{M}}})\|+
    \|f\|_{\sigma})
    \notag\\
	&\leq C\ka\CD(t)+\ka\eps^2\de\|\pa^3_x\widetilde{\phi}\|^2
    +C_\ka\frac{\de^2}{\eps}\CD(t)
    +C_\ka C(\eta)
    (\de^3+\eps\de^2)\notag\\
	&\leq C(\ka+C_\ka\frac{\de^2}{\eps})\CD(t)+C_\ka C(\eta)(\eps\de^2+\de^3),
\end{align*}
and other terms can be bounded similarly as above since they have lower order. Therefore, one gets
\begin{align}\label{highphi}
	&|\eps\de^2(\pa^2_x[\partial_x\phi\partial_{v_{1}}(\frac{M+\overline{G}}{\sqrt{\overline{M}}})],W^2(2,0)\pa^2_x(\frac{M}{\sqrt{\overline{M}}}+\frac{\overline{G}}{\sqrt{\overline{M}}}+f))|\notag\\
	\leq&C(\ka+C_\ka(\eps+\de+\frac{\de^2}{\eps}))\CD(t)+C_\ka C(\eta)(\eps\de^2+\de^3).
\end{align}
To estimate the linear operator $\CL_{\overline{M}}$, one has
\begin{align}\label{LMfw}
	&-\eps\de(\pa^2_x\CL_{\overline{M}} f,W^2(2,0)\pa^2_x(\frac{M}{\sqrt{\overline{M}}}+\frac{\overline{G}}{\sqrt{\overline{M}}}+f))\notag\\
	\geq& c\eps\de\|\partial^2_x f\|^2_{\sigma,W}-C\eps\de\|\partial^2_x f\|^2_{\sigma}-C\eps\CD(t)-C(\eta)\eps^2\de^2-|\eps\de(\pa^2_x\CL_{\overline{M}} f,W^2(2,0)\pa^2_x(\frac{M}{\sqrt{\overline{M}}}))|.
\end{align}
For the last term above, integration by parts gives
\begin{align}\label{LMfw1}
	&|\eps\de(\pa^2_x\CL_{\overline{M}} f,W^2(2,0)\pa^2_x\frac{M}{\sqrt{\overline{M}}})|\notag\\
	\leq& C\eps\de\{|(\pa^2_x\CL_{\overline{M}} f,W^2(2,0)\frac{\pa^2_xM}{\sqrt{\overline{M}}})|+|(\pa_x\CL_{\overline{M}} f,W^2(2,0)\big(2\pa_x[\pa_xM \pa_x\overline{M}^{-1/2}]+\pa_x[M \pa^2_x\overline{M}^{-1/2}]\big))|\}\notag\\
	\leq&C\eps\de|(\pa^2_x\CL_{\overline{M}} f,W^2(2,0)\frac{\pa^2_xM}{\sqrt{\overline{M}}})|+C\eps\CD(t)+C(\eta)\eps^2\de^2.
\end{align}
Using \eqref{paM}, for any $0<\ka_2<1$, one gets
\begin{align}\label{LMfw2}
	&\eps\de\{|(\pa^2_x\CL_{\overline{M}} f,W^2(2,0)\frac{\pa^2_xM}{\sqrt{\overline{M}}})|\notag\\
	\leq& \eps\de|(\pa^2_x\CL_{\overline{M}} f,W^2(2,0)\frac{M}{\sqrt{\overline{M}}}\big(\frac{\pa^2_x\widetilde{\rho}}{\rho}+\frac{(v-u)\cdot\pa^2_x\widetilde{u}}{K\theta}+(\frac{|v-u|^{2}}{2K\theta}-\frac{3}{2})\frac{\pa^2_x\widetilde{\theta}}{\theta} \big))|\notag\\
	&+\eps\de|(\pa_x\CL_{\overline{M}} f,W^2(2,0)\pa_x[\frac{M}{\sqrt{\overline{M}}}\big(\frac{\pa^2_x\bar{\rho}}{\rho}+\frac{(v-u)\cdot\pa^2_x\bar{u}}{K\theta}+(\frac{|v-u|^{2}}{2K\theta}-\frac{3}{2})\frac{\pa^2_x\bar{\theta}}{\theta} \big)])|\notag\\
	&+\eps\de|(\pa_x\CL_{\overline{M}} f,W^2(2,0)\pa_x\frac{\pa_x(M\frac{1}{\rho})\pa_x\rho+\pa_x(M\frac{v-u}{K\theta})\cdot\pa_xu+\pa_x(M\frac{|v-u|^{2}}{2K\theta^{2}}-M\frac{3}{2\theta})\pa_x\theta}{\sqrt{\overline{M}}})|\notag\\
	\leq& \eps\de|(\CL_{\overline{M}}\pa^2_x f,W^2(2,0)\frac{M}{\sqrt{\overline{M}}}\big(\frac{\pa^2_x\widetilde{\rho}}{\rho}+\frac{(v-u)\cdot\pa^2_x\widetilde{u}}{K\theta}+(\frac{|v-u|^{2}}{2K\theta}-\frac{3}{2})\frac{\pa^2_x\widetilde{\theta}}{\theta} \big))|+C\eps\CD(t)+C(\eta)\eps^2\de^2\notag\\
	\leq& C\eps\de(\ka_2\|\pa^2_x(\widetilde{\rho},\widetilde{u},\widetilde{\theta})\|^2+\frac{1}{\ka_2}\|\pa^2_xf\|^2_{\si})+C\eps\CD(t)+C(\eta)\eps^2\de^2.
\end{align}
It follows from \eqref{LMfw}, \eqref{LMfw1} and \eqref{LMfw2} that
\begin{align}\label{LMffw}
	&-\eps\de(\pa^2_x\CL_{\overline{M}} f,W^2(2,0)\pa^2_x(\frac{M}{\sqrt{\overline{M}}}+\frac{\overline{G}}{\sqrt{\overline{M}}}+f))\notag\\
	\geq& c\eps\de\|\partial^2_x f\|^2_{\sigma,W}-C\eps\de(\ka_2\|\pa^2_x(\widetilde{\rho},\widetilde{u},\widetilde{\theta})\|^2+\frac{1}{\ka_2}\|\pa^2_xf\|^2_{\si})-C\eps\CD(t)-C(\eta)\eps^2\de^2.
\end{align}
The second and third terms on the right hand side of \eqref{ipwF} can be directly bounded by
\begin{align}\label{tranbarMMw}
	&\eps^2\de^2|(\pa^2_x(M(\pa_t-\frac{1}{\eps}\pa_x+v_1\pa_x)(\sqrt{\overline{M}}^{-1})),W^2(2,0)\pa^2_x(\frac{M}{\sqrt{\overline{M}}}+\frac{\overline{G}}{\sqrt{\overline{M}}}+f))|\notag\\
	&+\eps^2\de^2|(\pa^2_x((\pa_t-\frac{1}{\eps}\pa_x+v_1\pa_x)(\frac{1}{\sqrt{\overline{M}}})\overline{G}),W^2(2,0)\pa^2_x(\frac{M}{\sqrt{\overline{M}}}+\frac{\overline{G}}{\sqrt{\overline{M}}}+f))|
    \notag\\
	\leq& \eps^2\de^2\|\langle v\rangle^{-3}\pa^2_x(\frac{M}{\sqrt{\overline{M}}}+\frac{\overline{G}}{\sqrt{\overline{M}}}+f)\|
    \times \big(\|\langle v\rangle^3 \pa^2_x(M(\pa_t-\frac{1}{\eps}\pa_x+v_1\pa_x)(\sqrt{\overline{M}}^{-1}))W^2(2,0)\|
    \notag\\
	&\quad+\|\langle v\rangle^3 \pa^2_x((\pa_t-\frac{1}{\eps}\pa_x+v_1\pa_x)(\frac{1}{\sqrt{\overline{M}}})\overline{G})W^2(2,0)\| \big)\notag\\
	\leq& C(\eps+\de)\CD(t)+C(\eta)\eps\de^2.
\end{align}
Due to the good property of $\overline{G}$ in \eqref{boundbarG}, we have
\begin{align}\label{barGw}
	&|\eps\de(\pa^2_x(\frac{L_M\overline{G}}{\sqrt{\overline{M}}}),W^2(2,0)\pa^2_x(\frac{M}{\sqrt{\overline{M}}}+\frac{\overline{G}}{\sqrt{\overline{M}}}+f))|
    \notag\\
    &\leq \eps\de \|\langle v\rangle^3\pa^2_x(\frac{L_M\overline{G}}{\sqrt{\overline{M}}})W^2(2,0)\|\|\langle v\rangle^{-3}\pa^2_x(\frac{M}{\sqrt{\overline{M}}}+\frac{\overline{G}}{\sqrt{\overline{M}}}+f)\|\notag\\
	&\leq C(\de+\eps)\CD(t)+C(\eta)\eps^2\de^2.
\end{align}
For the term including $L_B$, we obtain by Cauchy-Schwarz inequality and similar arguments as in \eqref{LMfw1} and \eqref{LMfw2} that
\begin{align}\label{LBgw}
&|\eps\de(\pa^2_x(\frac{\sqrt{\mu}L_Bg}{\sqrt{\overline{M}}}),W^2(2,0)\pa^2_x(\frac{M}{\sqrt{\overline{M}}}+\frac{\overline{G}}{\sqrt{\overline{M}}}+f))|\notag\\
\leq& \eps\de\|\langle v\rangle^{-3}\pa^2_x f\|^2+C\eps\de\|\langle v\rangle^3\pa^2_x(\frac{\sqrt{\mu}L_Bg}{\sqrt{\overline{M}}})W^2(2,0)\|^2
\notag\\
&+ C\eps\de(\ka_2\|\pa^2_x(\widetilde{\rho},\widetilde{u},\widetilde{\theta})\|^2+\frac{1}{\ka_2}\|\pa^2_xg\|^2_{\si})+C\eps\CD(t)+C(\eta)\eps^2\de^2
\notag\\
\leq& C\eps\de(\|\pa^2_x f\|_\sigma^2+\|\pa^2_xg\|^2_{\si}+\frac{1}{\ka_2}
\|\pa^2_xg\|^2_{\si}+\ka_2\|\pa^2_x(\widetilde{\rho},\widetilde{u},\widetilde{\theta})\|^2)+
C\eps\CD(t)+C(\eta)\eps^2\de^2,
\end{align}
for any $0<\ka_2<1$. Here we have used the following  estimates that
\begin{align*} 
&|\eps\de(\pa^2_xL_Bg\frac{\sqrt{\mu}}{\sqrt{\overline{M}}},W^2(2,0)M\pa^2_x(\frac{1}{\sqrt{\overline{M}}})|\notag\\
&=|\eps\de(\pa_xL_Bg,\pa_x[\frac{\sqrt{\mu}}{\sqrt{\overline{M}}}W^2(2,0)M\pa^2_x(\frac{1}{\sqrt{\overline{M}}})])|
\leq C\eps\CD(t)+C(\eta)\eps^2\de^2.
\end{align*}
The last term on the right hand side of \eqref{ipwF} has similar structure as the second term on the left hand side and hence can be bounded using the arguments as in \eqref{phiMMf}, which yields
\begin{align}\label{highphi1}
	&|\eps\de^2(\pa^2_x[\partial_x\phi(M+\overline{G})\partial_{v_{1}}\sqrt{\overline{M}}^{-1}],W^2(2,0)\pa^2_x(\frac{M}{\sqrt{\overline{M}}}+\frac{\overline{G}}{\sqrt{\overline{M}}}+f))|\notag\\
	\leq&C(\ka+C_\ka(\eps+\de+\frac{\de^2}{\eps}))\CD(t)+C_\ka C(\eta)\eps\de^2+C_\ka C(\eta)\de^3.
\end{align}
Hence, it follows from \eqref{ipwF}, \eqref{highphi}, \eqref{LMffw}, \eqref{tranbarMMw}, \eqref{barGw}, \eqref{LBgw} and \eqref{highphi1} that
\begin{align*}
	&\eps^2\de^2\frac{d}{dt}\|\pa^2_x(\frac{M}{\sqrt{\overline{M}}}+\frac{\overline{G}}{\sqrt{\overline{M}}}+f)\|^2_W+	c\eps\de\|\pa^2_xf\|^2_{\si,W}\notag\\
	\leq	&C\eps\de(\ka_2\|\pa^2_x(\widetilde{\rho},\widetilde{u},\widetilde{\theta})\|^2+
    \frac{1}{\ka_2}\|\pa^2_x(g,f)\|^2_{\si})+(C\ka+C_\ka(\de+\eps+\frac{\de^2}{\eps}))\CD(t)+C_\ka  C(\eta)(\eps\de^2+\de^3),	
\end{align*}
which, by integration on time and the Cauchy-Schwarz inequality, yields
\begin{align*}
	&\eps^2\de^2\|\pa^2_x(\frac{M}{\sqrt{\overline{M}}}+\frac{\overline{G}}{\sqrt{\overline{M}}}+f)\|^2_W+	\eps\de\int^t_0\|\pa^2_xf(s)\|^2_{\si,W}\,ds
    \notag\\
	\leq	&C\CE(0)
    +C(\eta)\eps^2\de^2+C\eps^5+C\ka_2\eps\de\int^t_0\|\pa^2_x(\widetilde{\rho},\widetilde{u},\widetilde{\theta})(s)\|^2\,ds
+C\frac{1}{\ka_2}\eps\de\int^t_0\|\pa^2_x(g,f)(s)\|^2_{\si}\,ds
    \notag\\
	&+(C\ka+C_\ka(\de+\eps+\frac{\de^2}{\eps}))\int^t_0\CD(s)\,ds+C_\ka  C(\eta)t(\eps\de^2+\de^3).
\end{align*}
for any $0<\ka_2<1$.
Similar calculation as in \eqref{wM} gives
\begin{align*}
	\|\langle v\rangle^k\pa^2_x(\frac{M}{\sqrt{\overline{M}}})\|^2+\|\langle v\rangle^k\pa^2_x(\frac{M}{\sqrt{\overline{M}}})\|^2_{\si,w}
	\leq C\|\pa^2_x(\widetilde{\rho},\widetilde{u},\widetilde{\theta})\|^2+C(\eta)+C\frac{1}{\eps\de}A\eps^4.
\end{align*}
Hence, we obtain \eqref{2ndwdiss} by the above two estimates and
\begin{equation*}
\|\partial^2_x f\|^2_{\sigma}\leq C(\|\partial^{2}_x\mathbf{P}_1 f\|_{\sigma}^{2} +\|\partial^2_x g\|^2_{\sigma}).   
\end{equation*}
This completes the proof of Lemma \ref{lem4.16}.
\end{proof}
	
It remains to estimate the highest energy and dissipation without the weight. Compared to the lower order estimate, we need to get energies of fluid quantities and $f$ all from a weaker term $\|\pa^2_x(\frac{M}{\sqrt{\overline{M}}}+\frac{\overline{G}}{\sqrt{\overline{M}}}+f)\|^2$ as its lower bound in some sense, which requires us to carefully deal with the macro-micro relations between $f$, $g$ and $M$. In addition, we apply the technics from previous lemmas that cancel $O(\eps^{-1})$ in singular transport operator to the highest order, and combine with the gain of $O(\eps)$ from the derivatives of $(\bar{\rho},\bar{\theta})$, to obtain the energy of electronic field.
\begin{lemma}\label{le2ndenergy}
Under the a priori assumption \eqref{apriori}, it holds that
	\begin{align}\label{2ndenergy}
		&\eps^2\de^2\{\|\pa^2_x(\widetilde{\rho},\widetilde{u},\widetilde{\theta},\widetilde{\phi})(t)\|^{2}+\|\pa^2_xf(t)\|^{2}+\eps\|\pa^3_x\widetilde{\phi}(t)\|^{2}\}+	\eps\de\int^t_0\|\pa^2_x\mathbf{P}_1f(s)\|^2_{\si}\,ds\notag\\	\leq&C\CE(0)+C(\eta)\eps^2\de^2         +C\eps^5+C\eps^2\de^2\|\pa^2_xg(t)\|^{2}+	(C\ka+C_\ka(\de+\eps+\frac{\de^2}{\eps}))\int^t_0\CD(s)\,ds
        \notag\\
		&+C\eps\de\int^t_0(\ka_3\|\pa^2_x(\widetilde{\rho},\widetilde{u},\widetilde{\theta})(s)\|^2+\frac{1}{\ka_3}\|\pa^2_xg(s)\|^2_{\si})\,ds+C_\ka  C(\eta)t(\eps\de^2+\eps^2\de^2+\de^3),	
	\end{align}
    for any $0<\ka<1$ and $0<\ka_3<1$.
\end{lemma}	
\begin{proof}			
	Applying $\pa^2_x$ to both side of \eqref{eqMf} and taking the inner product with $\eps^2\de^2 \pa^2_x(\frac{M}{\sqrt{\overline{M}}}+\frac{\overline{G}}{\sqrt{\overline{M}}}+f)$, it holds that
\begin{align}\label{ipF}
	&\eps^2\de^2\frac{1}{2}\frac{d}{dt}\|\pa^2_x(\frac{M}{\sqrt{\overline{M}}}+\frac{\overline{G}}{\sqrt{\overline{M}}}+f)\|^2
	-\eps\de^2(\pa^2_x[\partial_x\phi\partial_{v_{1}}(\frac{M+\overline{G}}{\sqrt{\overline{M}}})],\pa^2_x(\frac{M}{\sqrt{\overline{M}}}+\frac{\overline{G}}{\sqrt{\overline{M}}}+f))
	\notag\\
	=&\eps\de(\pa^2_x\CL_{\overline{M}} f,\pa^2_x(\frac{M}{\sqrt{\overline{M}}}+\frac{\overline{G}}{\sqrt{\overline{M}}}+f))
+\eps^2\de^2(\pa^2_x(M(\pa_t-\frac{1}{\eps}\pa_x+v_1\pa_x)(\sqrt{\overline{M}}^{-1})),\pa^2_x(\frac{M}{\sqrt{\overline{M}}}+\frac{\overline{G}}{\sqrt{\overline{M}}}+f))
    \notag\\
	&+\eps^2\de^2(\pa^2_x((\pa_t-\frac{1}{\eps}\pa_x+v_1\pa_x)(\frac{1}{\sqrt{\overline{M}}})\overline{G}),\pa^2_x(\frac{M}{\sqrt{\overline{M}}}+\frac{\overline{G}}{\sqrt{\overline{M}}}+f))
    \notag\\
    &+\eps\de(\pa^2_x(\frac{L_M\overline{G}}{\sqrt{\overline{M}}}),\pa^2_x(\frac{M}{\sqrt{\overline{M}}}+\frac{\overline{G}}{\sqrt{\overline{M}}}+f))
+\eps\de(\pa^2_x(\frac{\sqrt{\mu}L_Bg}{\sqrt{\overline{M}}}),\pa^2_x(\frac{M}{\sqrt{\overline{M}}}+\frac{\overline{G}}{\sqrt{\overline{M}}}+f))
    \notag\\
	&-\eps\de^2(\pa^2_x[\partial_x\phi(M+\overline{G})\partial_{v_{1}}\sqrt{\overline{M}}^{-1}],\pa^2_x(\frac{M}{\sqrt{\overline{M}}}+\frac{\overline{G}}{\sqrt{\overline{M}}}+f)).
\end{align}
Following the approach in \eqref{phiMMf}, we arrive at
	\begin{align}\label{phiFF1}
		&-\eps\de^2(\pa^2_x[\partial_x\phi\partial_{v_{1}}(\frac{M+\overline{G}}{\sqrt{\overline{M}}})],\pa^2_x(\frac{M}{\sqrt{\overline{M}}}+\frac{\overline{G}}{\sqrt{\overline{M}}}+f))\notag\\
		\geq&-\eps\de^2(\frac{\pa^3_x\phi\partial_{v_{1}}M}{\sqrt{\overline{M}}},\frac{\pa^2_xM}{\sqrt{\overline{M}}})
        -C(\ka+C_\ka(\eps+\de+\frac{\de^2}{\eps}))\CD(t)-C_\ka C(\eta)(\eps\de^2+\de^3).
	\end{align}
Now we focus on the first term on the right hand side above, from where we need to get the energy for electric field. Split it into two parts as follows:
\begin{align}\label{splitphiMM}
	\eps\de^2(\frac{\pa^3_x\phi\partial_{v_{1}}M}{\sqrt{\overline{M}}},\frac{\pa^2_xM}{\sqrt{\overline{M}}})
	=&\eps\de^2(\pa^3_x\phi\partial_{v_{1}}M,\frac{1}{M}\pa^2_xM)+\eps\de^2(\pa^3_x\phi\partial_{v_{1}}M,[\frac{1}{\overline{M}}-\frac{1}{M}]\pa^2_xM).
\end{align}
For the first term on the right hand side, we integrate in $v$ and use the first equation in \eqref{1macro} to get
\begin{align}\label{phiMM1}
	\eps\de^2(\pa^3_x\phi\partial_{v_{1}}M,\frac{1}{M}\pa^2_xM)&=
	-\eps\de^2(\pa^3_x\phi\frac{v_1-u_1}{K\theta},\pa^2_xM)
	\notag\\
	&=-\eps\de^2(\frac{1}{K\theta}\pa^3_x\phi,\pa^2_x(\rho u_1))
	+\eps\de^2(\frac{1}{K\theta}\pa^3_x\phi,u_1\pa^2_x\rho)\notag\\
	&=\eps\de^2(\frac{1}{K\theta}\pa^3_x\phi,\pa_x\pa_t\rho)-\de^2(\frac{1}{K\theta}\pa^3_x\phi,\pa^2_x\rho)
	+\eps\de^2(\frac{1}{K\theta}\pa^3_x\phi,u_1\pa^2_x\rho).
\end{align}
We bound each term on the right hand side above. First it holds by the last equation in \eqref{1macro} that
\begin{align*}
	&\eps\de^2(\frac{1}{K\theta}\pa^3_x\phi,\pa_x\pa_t\rho)=\eps\de^2(\frac{1}{K\theta}\pa^3_x\phi,\pa_t(-\eps^2\partial^{3}_{x}\phi+\eps\pa_x\phi ))\notag\\
	=&-\eps^3\de^2\frac{1}{2}\frac{d}{dt}(\frac{1}{K\theta}\pa^3_x\phi,\partial^{3}_{x}\phi)+\eps^3\de^2\frac{1}{2}(\pa_t(\frac{1}{K\theta})\pa^3_x\phi,\partial^{3}_{x}\phi)-\eps^2\de^2(\frac{1}{K\theta}\pa^2_x\phi,\pa_t\pa^2_x\phi )-\eps^2\de^2(\pa_x(\frac{1}{K\theta})\pa^2_x\phi,\pa_t\pa_x\phi )\notag\\
	=&-\eps^2\de^2\frac{1}{2}\frac{d}{dt}\{(\frac{1}{K\theta}\pa^2_x\phi,\partial^{2}_{x}\phi)+\eps(\frac{1}{K\theta}\pa^3_x\phi,\partial^{3}_{x}\phi)\}+\eps^3\de^2\frac{1}{2}(\pa_t(\frac{1}{K\theta})\pa^3_x\phi,\partial^{3}_{x}\phi)\notag\\
	&\quad+\eps^2\de^2\frac{1}{2}(\pa_t(\frac{1}{K\theta})\pa^2_x\phi,\pa^2_x\phi )-\eps^2\de^2(\pa_x(\frac{1}{K\theta})\pa^2_x\phi,\pa_t\pa_x\phi ).
\end{align*}
Again by the last equation in \eqref{1macro}, one has
\begin{align*}
	&\de^2(\frac{1}{K\theta}\pa^3_x\phi,\pa^2_x\rho)=\de^2(\frac{1}{K\theta}\pa^3_x\phi,\pa_x(-\eps^2\partial^{3}_{x}\phi+\eps\pa_x\phi ))=\eps^2\de^2\frac{1}{2}(\pa_x(\frac{1}{K\theta})\pa^3_x\phi,\partial^{3}_{x}\phi)-\ep\de^2\frac{1}{2}(\pa_x(\frac{1}{K\theta})\pa^2_x\phi,\pa^2_x\phi).
\end{align*}
The combination of the above two identities gives
\begin{align}\label{pat-pax1}
	&\eps\de^2(\frac{1}{K\theta}\pa^3_x\phi,\pa_x\pa_t\rho)-\de^2(\frac{1}{K\theta}\pa^3_x\phi,\pa^2_x\rho)\notag\\
	=&-\eps^2\de^2\frac{1}{2}\frac{d}{dt}\{(\frac{1}{K\theta}\pa^2_x\phi,\partial^{2}_{x}\phi)+\eps(\frac{1}{K\theta}\pa^3_x\phi,\partial^{3}_{x}\phi)\}+\eps^3\de^2\frac{1}{2}(-\frac{(\pa_t-\eps^{-1}\pa_x)\theta}{K\theta^2}\pa^3_x\phi,\partial^{3}_{x}\phi)\notag\\
	&\quad+\eps^2\de^2\frac{1}{2}(-\frac{(\pa_t+\eps^{-1}\pa_x)\theta}{K\theta^2}\pa^2_x\phi,\pa^2_x\phi )-\eps^2\de^2(\pa_x(\frac{1}{K\theta})\pa^2_x\phi,\pa_t\pa_x\phi ).
\end{align}
We only need to bound the last three terms above. By the fact that $\bar{\theta}=\frac{3}{2}+\eps\theta_1+\eps^2\theta_2+\eps^3\theta_3$, one has
\begin{align}\label{pat-pax2}
&\big|\eps^3\de^2\frac{1}{2}(-\frac{(\pa_t-\eps^{-1}\pa_x)\bar\theta}{K\theta^2}\pa^3_x\phi,\partial^{3}_{x}\phi)\big|+\big|\eps^2\de^2\frac{1}{2}(-\frac{(\pa_t+\eps^{-1}\pa_x)\bar\theta}{K\theta^2}\pa^2_x\phi,\pa^2_x\phi )\big|
\notag\\
\leq&C(\eta)(\eps^3\de^2\|\pa^3_x\phi\|^2+\eps^2\de^2\|\pa^2_x\phi\|^2)
\notag\\
\leq&C(\eta)\eps^2\de^2+C\de\CD(t),
\end{align}
and
\begin{align*}
\eps^2\de^2\frac{1}{2}\big|(-\frac{(\pa_t+\eps^{-1}\pa_x)\widetilde{\theta}}{K\theta^2}\pa^2_x\phi,\pa^2_x\phi )\big|
\leq C(\eta)\eps^2\de^2+C\de\CD(t).
\end{align*}
It holds by the fourth equation in \eqref{perturbeq} that
\begin{align*}
	&\frac{\eps^3\de^2}{2}\big|(-\frac{(\pa_t-\eps^{-1}\pa_x)\widetilde{\theta}}{K\theta^2}\pa^3_x\phi,\partial^{3}_{x}\phi)\big|
    \notag\\
	=&\frac{\eps^3\de^2}{2}\big|(\frac{\frac{2}{3}\bar{\theta}\partial_x\widetilde{u}_{1}+\frac{2}{3}\widetilde{\theta}\partial_xu_{1}
		+u_1\partial_x\widetilde{\theta}+\widetilde{u}_{1}\partial_x\bar{\theta}
		-\frac{1}{\rho}u\cdot\partial_x(\int_{\mathbb{R}^3} v_{1}v G\, dv)
		+\frac{1}{\rho}\partial_x(\int_{\mathbb{R}^3}v_{1}\frac{|v|^{2}}{2}G\, dv)+\eps^3R_3
		}{K\theta^2}\pa^3_x\phi,\partial^{3}_{x}\phi)\big|
        \notag\\
		\leq& C\|\frac{2}{3}\bar{\theta}\partial_x\widetilde{u}_{1}+\frac{2}{3}\widetilde{\theta}\partial_xu_{1}
		+u_1\partial_x\widetilde{\theta}+\widetilde{u}_{1}\partial_x\bar{\theta}
		-\frac{1}{\rho}u\cdot\partial_x(\int_{\mathbb{R}^3} v_{1}v G\, dv)
		+\frac{1}{\rho}\partial_x(\int_{\mathbb{R}^3}v_{1}\frac{|v|^{2}}{2}G\, dv)+\eps^3R_3\|_{L^\infty_x}\notag\\
		&\qquad\times\eps^3\de^2\|\pa^3_x\phi\|^2,
		\end{align*}
		which, together with \eqref{apriori} and \eqref{boundkdv}, yields
		\begin{align}\label{pat-pax3}
			&\frac{\eps^3\de^2}{2}\big|(-\frac{(\pa_t-\eps^{-1}\pa_x)\widetilde{\theta}}{K\theta^2}\pa^3_x\phi,\partial^{3}_{x}\phi)\big|
		\leq C\eps\CD(t)+C(\eta)\eps^2\de^2.
\end{align}
Furthermore, it follows from \eqref{patphi} that
\begin{align}\label{pat-pax4}
	\eps^2\de^2\big|(\pa_x(\frac{1}{K\theta})\pa^2_x\phi,\pa_t\pa_x\phi )\big|\leq C\eps^2\de^2\|\pa_x\theta\|_{L^\infty_x}\|\|\pa^2_x\phi\|\|\pa_t\pa_x\phi \|\leq C(\eta)\eps^2\de^2+C\de\CD(t).
\end{align}
We collect \eqref{pat-pax1}, \eqref{pat-pax2}, \eqref{pat-pax3} and \eqref{pat-pax4} to get 
\begin{align*}
	&\eps\de^2(\frac{1}{K\theta}\pa^3_x\phi,\pa_x\pa_t\rho)-\de^2(\frac{1}{K\theta}\pa^3_x\phi,\pa^2_x\rho)\notag\\
	\leq&-\eps^2\de^2\frac{1}{2}\frac{d}{dt}\{(\frac{1}{K\theta}\pa^2_x\phi,\partial^{2}_{x}\phi)+\eps(\frac{1}{K\theta}\pa^3_x\phi,\partial^{3}_{x}\phi)\}+C(\eps+\de)\CD(t)+C(\eta)\eps^2\de^2.
\end{align*}
The last term in \eqref{phiMM1} is bounded by
\begin{align*}
	\eps\de^2\big|(\frac{1}{K\theta}\pa^3_x\phi,u_1\pa^2_x\rho)\big|\leq C\eps\de^2\|\pa^3_x\phi\|\|u_1\|_{L^\infty_x}\|\pa^2_x\rho\|\leq C_{\ka}C(\eta)(\eps^2\de^2+\eps\de^3)+C(\de+\ka+C_{\ka}\frac{\de^2}{\eps})\CD(t).
\end{align*}
The above two estimates, together with \eqref{phiMM1}, show that 
\begin{align*}
	-\eps\de^2(\pa^3_x\phi\partial_{v_{1}}M,\frac{1}{M}\pa^2_xM)
	\geq\eps^2\de^2\frac{1}{2}\frac{d}{dt}&\{(\frac{1}{K\theta}\pa^2_x\phi,\partial^{2}_{x}\phi)+\eps(\frac{1}{K\theta}\pa^3_x\phi,\partial^{3}_{x}\phi)\}
    \notag\\
	&-(C\ka+C_\ka(\de+\eps+\frac{\de^2}{\eps}))\CD(t)-C_{\ka}C(\eta)(\eps^2\de^2+\eps\de^3).
\end{align*}
The other term in \eqref{splitphiMM} can be bounded by
\begin{align*}
	\eps\de^2\big|(\pa^3_x\phi\partial_{v_{1}}M,[\frac{1}{\overline{M}}-\frac{1}{M}]\pa^2_xM)\big|
    \leq& C\eps\de^2\|\pa^3_x\phi\|\|\sqrt{\overline{M}}\partial_{v_{1}}M[\frac{1}{\overline{M}}-\frac{1}{M}]\|_{L^\infty_x}\|\frac{\pa^2_xM}{\sqrt{\overline{M}}}\|\notag\\
	\leq&C(\ka+C_\ka\frac{\de^2}{\eps}+\de)\CD(t)+C_\ka C(\eta) (\eps\de^2+\de^3).
\end{align*}
We conclude by the above two estimates that
\begin{align*}
	-\eps\de^2(\pa^3_x\phi\partial_{v_{1}}M,\frac{1}{\overline{M}}\pa^2_xM)
	\geq \eps^2\de^2\frac{1}{2}\frac{d}{dt}&\{(\frac{1}{K\theta}\pa^2_x\phi,\partial^{2}_{x}\phi)+\eps(\frac{1}{K\theta}\pa^3_x\phi,\partial^{3}_{x}\phi)\}\notag\\
	&-(C\ka+C_\ka(\de+\eps+\frac{\de^2}{\eps}))\CD(t)-C_\ka C(\eta)(\eps\de^2+\de^3),
\end{align*}
which, together with \eqref{phiFF1}, yields
\begin{align}\label{phiFF}
	&-\eps\de^2(\pa^2_x[\partial_x\phi\partial_{v_{1}}(\frac{M+\overline{G}}{\sqrt{\overline{M}}})],\pa^2_x(\frac{M}{\sqrt{\overline{M}}}+\frac{\overline{G}}{\sqrt{\overline{M}}}+f))
    \notag\\
	\geq&\eps^2\de^2\frac{1}{2}\frac{d}{dt}\{(\frac{1}{K\theta}\pa^2_x\phi,\partial^{2}_{x}\phi)+\eps(\frac{1}{K\theta}\pa^3_x\phi,\partial^{3}_{x}\phi)\}-C(\ka+C_{\ka}(\de+\eps+\frac{\de^2}{\eps}))\CD(t)-C_\ka C(\eta)(\eps\de^2+\de^3).
\end{align}
Then we perform similar calculations as in \eqref{LMfw}, \eqref{LMfw1} and \eqref{LMfw2} that
\begin{align}\label{splitLfF}
-\eps\de(\pa^2_x\CL_{\overline{M}} f,\pa^2_x(\frac{M}{\sqrt{\overline{M}}}+\frac{\overline{G}}{\sqrt{\overline{M}}}+f))\geq c\eps\de\|\pa^2_x\mathbf{P}_1 f\|^2_{\si}-\eps\de(\pa^2_x \CL_{\overline{M}}f,\frac{\pa^2_xM}{\sqrt{\overline{M}}})-C\eps\CD(t)-C(\eta)\eps^2\de^2.
\end{align}
We need to make further decomposition for the second term on the right hand side above in order to avoid the increase of second order energy. Based on \eqref{paM} and \eqref{LMfw2}, we have
\begin{align}\label{splitLfM}
	&\eps\de(\pa^2_x \CL_{\overline{M}}  f,\frac{\pa^2_x M}{\sqrt{\overline{M}}})\notag\\
	\leq&\eps\de( \CL_{\overline{M}} \pa^2_x f,\sqrt{\overline{M}}^{-1}\overline{M}\big(\frac{\pa^2_x\widetilde{\rho}}{\rho}+\frac{(v-u)\cdot\pa^2_x\widetilde{u}}{K\theta}+(\frac{|v-u|^{2}}{2K\theta}-\frac{3}{2})\frac{\pa^2_x\widetilde{\theta}}{\theta} \big))\notag\\
	&+\eps\de(\CL_{\overline{M}} \pa^2_x f,\sqrt{\overline{M}}^{-1}(M-\overline{M})\big(\frac{\pa^2_x\widetilde{\rho}}{\rho}+\frac{(v-u)\cdot\pa^2_x\widetilde{u}}{K\theta}+(\frac{|v-u|^{2}}{2K\theta}-\frac{3}{2})\frac{\pa^2_x\widetilde{\theta}}{\theta} \big))\notag\\
	&-\eps\de( \pa_x\CL_{\overline{M}} f,\pa_x[\sqrt{\overline{M}}^{-1}M\big(\frac{\pa^2_x\bar{\rho}}{\rho}+\frac{(v-u)\cdot\pa^2_x\bar{u}}{K\theta}+(\frac{|v-u|^{2}}{2K\theta}-\frac{3}{2})\frac{\pa^2_x\bar{\theta}}{\theta} \big)])\notag\\
	&-\eps\de( \pa_x\CL_{\overline{M}} f,\pa_x[\sqrt{\overline{M}}^{-1}\pa_x(M\frac{1}{\rho})\pa_x\rho+\pa_x(M\frac{v-u}{K\theta})\cdot\pa_xu+\pa_x(M\frac{|v-u|^{2}}{2K\theta^{2}}-M\frac{3}{2\theta})\pa_x\theta])\notag\\
	&+C\eps\CD(t)+C(\eta)\eps^2\de^2.
\end{align}
The first term on the right hand side above vanishes since by the property of $\CL_{\overline{M}}$. The second one is bounded by 
\begin{align*}
&\eps\de(\CL_{\overline{M}} \pa^2_x f,\sqrt{\overline{M}}^{-1}(M-\overline{M})\big(\frac{\pa^2_x\widetilde{\rho}}{\rho}+\frac{(v-u)\cdot\pa^2_x\widetilde{u}}{K\theta}+(\frac{|v-u|^{2}}{2K\theta}-\frac{3}{2})\frac{\pa^2_x\widetilde{\theta}}{\theta} \big))\notag\\
\leq& C\eps\de\|\pa^2_x f\|_\si \|\mu^{-3/4}(M-\overline{M})\|_{L^\infty_xL^2_v}\|\pa^2_x(\widetilde{\rho},\widetilde{u},\widetilde{\theta})\|\leq C\eps\CD(t).
\end{align*}
 Then the rest part can be bounded as in \eqref{LMfw2} without the appearance of terms containing second order spacial derivative, which, together with the above estimate, \eqref{splitLfF} and \eqref{splitLfM}, implies
\begin{align}\label{LfF}
-\eps\de(\pa^2_x\CL_{\overline{M}} f,\pa^2_x(\frac{M}{\sqrt{\overline{M}}}+\frac{\overline{G}}{\sqrt{\overline{M}}}+f))\geq c\eps\de\|\pa^2_x\mathbf{P}_1 f\|^2_{\si}-C\eps\CD(t)-C(\eta)\eps^2\de^2.
\end{align}
Other terms on the right hand side of \eqref{ipF} can be bounded using similar arguments as in \eqref{tranbarMMw}, \eqref{barGw}, \eqref{LBgw} and \eqref{highphi1}. We omit the details to get
\begin{align}\label{2ndnonlinear}
	&\big|\eps^2\de^2(\pa^2_x(M(\pa_t-\frac{1}{\eps}\pa_x+v_1\pa_x)(\sqrt{\overline{M}}^{-1})),\pa^2_x(\frac{M}{\sqrt{\overline{M}}}+\frac{\overline{G}}{\sqrt{\overline{M}}}+f))     \notag\\ 	&+\eps^2\de^2(\pa^2_x((\pa_t-\frac{1}{\eps}\pa_x+v_1\pa_x)(\frac{1}{\sqrt{\overline{M}}})\overline{G}),\pa^2_x(\frac{M}{\sqrt{\overline{M}}}+\frac{\overline{G}}{\sqrt{\overline{M}}}+f))     \notag\\     &+\eps\de(\pa^2_x(\frac{L_M\overline{G}}{\sqrt{\overline{M}}}),\pa^2_x(\frac{M}{\sqrt{\overline{M}}}+\frac{\overline{G}}{\sqrt{\overline{M}}}+f)) +\eps\de(\pa^2_x(\frac{\sqrt{\mu}L_Bg}{\sqrt{\overline{M}}}),\pa^2_x(\frac{M}{\sqrt{\overline{M}}}+\frac{\overline{G}}{\sqrt{\overline{M}}}+f))     \notag\\ 	&-\eps\de^2(\pa^2_x[\partial_x\phi(M+\overline{G})\partial_{v_{1}}\sqrt{\overline{M}}^{-1}],\pa^2_x(\frac{M}{\sqrt{\overline{M}}}+\frac{\overline{G}}{\sqrt{\overline{M}}}+f))
    \big|\notag\\
	\leq &C\eps\de(\ka_3\|\pa^2_x(\widetilde{\rho},\widetilde{u},\widetilde{\theta})\|^2+\frac{1}{\ka_3}\|\pa^2_xg\|^2_{\si})+(C\ka+C_\ka(\de+\eps+\frac{\de^2}{\eps}))\CD(t)+C_\ka  C(\eta)(\eps\de^2+\eps^2\de^2+\de^3),
\end{align}
for any $0<\ka<1$ and $0<\ka_3<1$.
Combining \eqref{ipF}, \eqref{phiFF}, \eqref{LfF} and \eqref{2ndnonlinear}, then taking integral on the time variable, it holds that
\begin{align}\label{highphienergy}
	&\eps^2\de^2\|\pa^2_x(\frac{M}{\sqrt{\overline{M}}}+\frac{\overline{G}}{\sqrt{\overline{M}}}+f)(t)\|^2+\eps^2\de^2\{(\frac{1}{K\theta}\pa^2_x\phi,\partial^{2}_{x}\phi)+\eps(\frac{1}{K\theta}\pa^3_x\phi,\partial^{3}_{x}\phi)\}+	\eps\de\int^t_0\|\pa^2_x\mathbf{P}_1 f(s)\|^2_{\si}\,ds\notag\\	\leq&C\CE(0)+C(\eta)\eps^2\de^2+C\eps^5+(C\ka+C_\ka(\de+\eps+\frac{\de^2}{\eps}))\int^t_0\CD(s)\,ds
    \notag\\
    &+C \eps\de\int^t_0(\ka_3\|\pa^2_x(\widetilde{\rho},\widetilde{u},\widetilde{\theta})(s)\|^2+\frac{1}{\ka_3}\|\pa^2_xg(s)\|^2_{\si})\,ds+C_\ka  C(\eta)t(\eps\de^2+\eps^2\de^2+\de^3).
\end{align}
It remains to bound the first term above. We first split it into three parts that
\begin{align*}
	\|\pa^2_x(\frac{M}{\sqrt{\overline{M}}}+\frac{\overline{G}}{\sqrt{\overline{M}}}+f)\|^{2}=&\int_{\mathbb{R}}\int_{\mathbb{R}^{3}}\big\{|\pa^2_x(\frac{M}{\sqrt{\overline{M}}})|^2+|\pa^2_xf|^{2}+|\pa^2_x(\frac{\overline{G}}{\sqrt{\overline{M}}})|^2\big\}\,dv\,dx
\notag\\	
    &+\int_{\mathbb{R}}\int_{\mathbb{R}^{3}}\big\{
    2\pa^2_x(\frac{M}{\sqrt{\overline{M}}})\pa^2_x(\frac{\overline{G}}{\sqrt{\overline{M}}})
    +2\pa_x^2(\frac{M}{\sqrt{\overline{M}}})\pa^2_xf
    +2\pa^2_x(\frac{\overline{G}}{\sqrt{\overline{M}}})
    \pa^2_xf
    \big\}\,dv\,dx.
\end{align*}
For the right hand side above, if $\pa_x$ acts on the $\overline{M}$ term, we can bound it as in \eqref{LMfw1}, which leads to
\begin{align*}
	\eps^2\de^2\|\pa^2_x(\frac{M}{\sqrt{\overline{M}}}+\frac{\overline{G}}{\sqrt{\overline{M}}}+f)\|^{2}\geq& \frac{1}{2}\eps^2\de^2\int_{\mathbb{R}}\int_{\mathbb{R}^{3}}\big\{|\frac{\pa_x^2M}{\sqrt{\overline{M}}}|^2+|\pa^2_xf|^{2}\big\}\,dv\,dx
    \notag\\
    &+\eps^2\de^2\int_{\mathbb{R}}\int_{\mathbb{R}^{3}}2\frac{\pa_x^2M}{\overline{M}}\pa^2_x(\sqrt{\overline{M}}f+\overline{G})\,dv\,dx
    -C(\eta)\eps^2\de^2-C\eps^5.
\end{align*}
Furthermore, we can claim that
\begin{align}\label{highenergy1}
	&\eps^2\de^2\|\pa^2_x(\frac{M}{\sqrt{\overline{M}}}
    +\frac{\overline{G}}{\sqrt{\overline{M}}}
    +f)\|^{2}\geq \frac{1}{2}\eps^2\de^2\int_{\mathbb{R}}\int_{\mathbb{R}^{3}}\big\{\frac{(\pa_x^2M)^2}{M}+|\pa^2_xf|^{2}+4\frac{\pa_x^2M}{M}\pa^2_x(\sqrt{\overline{M}}f+\overline{G})\big\}\,dv\,dx
    \notag\\
	&\quad\quad-C\eps^2\de^2\int_{\mathbb{R}}\int_{\mathbb{R}^{3}}\big\{|\pa_x^2M|^2|\frac{1}{\overline{M}}-\frac{1}{M}|+2|\pa_x^2M\pa^2_x(\sqrt{\overline{M}}f+\overline{G})||\frac{1}{\overline{M}}-\frac{1}{M}|\big\}\,dv\,dx
     -C(\eta)\eps^2\de^2-C\eps^5
    \notag\\
	&\quad\geq \frac{1}{2}\eps^2\de^2\int_{\mathbb{R}}\int_{\mathbb{R}^{3}}\{\frac{(\pa_x^2M)^2}{M}+|\pa^2_xf|^{2}+4\frac{\pa_x^2M\pa^2_x(\sqrt{\overline{M}}f+\overline{G})}{M}\}\,dv\,dx-C(\eta)\eps^2\de^2-C\eps^5.
\end{align}
For the first term on the right hand side above, recall from \eqref{paM} that
\begin{align}\label{defJ1J2}
	\partial^2_xM=&M\big(\frac{\pa^2_x\rho}{\rho}+\frac{(v-u)\cdot\pa^2_xu}{K\theta}+(\frac{|v-u|^{2}}{2K\theta}-\frac{3}{2})\frac{\pa^2_x\theta}{\theta} \big)
	\nonumber\\
	&+\pa_x(M\frac{1}{\rho})\pa_x\rho+\pa_x(M\frac{v-u}{K\theta})\cdot\pa_xu+\pa_x(M\frac{|v-u|^{2}}{2K\theta^{2}}-M\frac{3}{2\theta})\pa_x\theta:=J_1+J_2,
\end{align}
which leads to
\begin{align*}
	&\eps^2\de^2\int_{\mathbb{R}}\int_{\mathbb{R}^{3}}\frac{(\pa_x^2M)^2}{M}\,dv\,dx=\eps^2\de^2\int_{\mathbb{R}}\int_{\mathbb{R}^{3}}\frac{J_1^2+2J_1J_2+J_2^2}{M}\,dv\,dx.
\end{align*}
For $J_1$, it holds that
\begin{align}\label{J12}
	\eps^2\de^2\int_{\mathbb{R}}\int_{\mathbb{R}^{3}}\frac{(J_1)^{2}}{M}dvdx&=\eps^2\de^2\int_{\mathbb{R}}\int_{\mathbb{R}^{3}}M\Big\{\frac{\pa_x\rho}{\rho}+\frac{(v-u)\cdot\pa_xu}{R\theta}+(\frac{|v-u|^{2}}{2R\theta}-\frac{3}{2})\frac{\pa_x\theta}{\theta} \Big\}^{2}\,dv\,dx\notag
	\\
	&=\eps^2\de^2\int_{\mathbb{R}}\int_{\mathbb{R}^{3}}M
	\Big\{(\frac{\pa_x\rho}{\rho})^{2}+(\frac{(v-u)\cdot\pa_xu}{R\theta})^{2}+((\frac{|v-u|^{2}}{2R\theta}-\frac{3}{2})\frac{\pa_x\theta}{\theta})^{2} \Big\}\,dv\,dx\notag\\
	&\geq c\eps^2\de^2\|\pa^2_x(\widetilde{\rho},\widetilde{u},\widetilde{\theta})\|^{2}-C(\eta)\eps^2\de^2.
\end{align}
Others are lower order terms and thus can be bounded by 
\begin{align}\label{J22}
	\eps^2\de^2\int_{\mathbb{R}}\int_{\mathbb{R}^{3}}\frac{2J_1J_2+(J_2)^{2}}{M}\,dv\,dx\leq C(\eta)\eps^2\de^2+C\eps^5.
\end{align}
Collecting \eqref{J12} and \eqref{J22}, one gets
\begin{align}\label{highenergy2}
	&\eps^2\de^2\int_{\mathbb{R}}\int_{\mathbb{R}^{3}}\frac{(\pa_x^2M)^2}{M}dvdx\geq c\eps^2\de^2\|\pa^2_x(\widetilde{\rho},\widetilde{u},\widetilde{\theta})\|^{2}-C(\eta)\eps^2\de^2-C\eps^5.
\end{align}
There is still one term left on the right hand side of \eqref{highenergy1}. Using \eqref{defJ1J2} and similar argument as in \eqref{J22} we have
\begin{align*}
&\eps^2\de^2\int_{\mathbb{R}}\int_{\mathbb{R}^{3}}\frac{\pa_x^2M\pa^2_x(\sqrt{\overline{M}}f+\overline{G})}{M}\,dv\,dx\notag\\
	=&\eps^2\de^2\int_{\mathbb{R}}\int_{\mathbb{R}^{3}}\frac{\pa^2_x(\sqrt{\overline{M}}f+\overline{G})J_1}{M}dvdx+\eps^2\de^2\int_{\mathbb{R}}\int_{\mathbb{R}^{3}}\frac{\pa^2_x(\sqrt{\overline{M}}f+\overline{G})J_2}{M}\,dv\,dx\notag\\
	\leq& \eps^2\de^2\int_{\mathbb{R}}\int_{\mathbb{R}^{3}}\pa^2_x(-\sqrt{\mu}g)\big(\frac{\pa^2_x\rho}{\rho}+\frac{(v-u)\cdot\pa^2_xu}{K\theta}+(\frac{|v-u|^{2}}{2K\theta}-\frac{3}{2})\frac{\pa^2_x\theta}{\theta} \big)\,dv\,dx+C(\eta)\eps^2\de^2+C\eps^5.
\end{align*}
Here we have used $\sqrt{\overline{M}}f+\overline{G}=G-\sqrt{\mu}g$
and the fact that
$$
\int_{\mathbb{R}}\int_{\mathbb{R}^{3}}\pa^2_x G\big(\frac{\pa^2_x\rho}{\rho}+\frac{(v-u)\cdot\pa^2_xu}{K\theta}+(\frac{|v-u|^{2}}{2K\theta}-\frac{3}{2})\frac{\pa^2_x\theta}{\theta} \big)\,dv\,dx=0.
$$
Hence, we get
\begin{align}\label{highenergy3}
	&\eps^2\de^2\int_{\mathbb{R}^{3}}\int_{\mathbb{R}^{3}}\frac{\pa_x^2M\pa^2_x(\sqrt{\overline{M}}f+\overline{G})}{M}dvdx\notag\\
	\leq&\ka\eps^2\de^2\|\pa^2_x(\widetilde{\rho},\widetilde{u},\widetilde{\theta})\|^{2}+C_\ka\eps^2\de^2\|\pa^2_xg\|^2+C\eps^5+C(\eta)\eps^2\de^2.
\end{align}
We combine \eqref{highenergy1}, \eqref{highenergy2} and \eqref{highenergy3}, then choose $\ka$ to be small to get
\begin{align}\label{highenergy}
	\eps^2\de^2\|\pa^2_x(\frac{M}{\sqrt{\overline{M}}} +\frac{\overline{G}}{\sqrt{\overline{M}}}+f)\|^{2}\geq& \eps^2\de^2(c\|\pa^2_x(\widetilde{\rho},\widetilde{u},\widetilde{\theta})\|^{2}+c\|\pa^2_xf\|^{2}-C\|\pa^2_xg\|^{2}-)C(\eta)\eps^2\de^2-C\eps^5.
\end{align}
Hence, \eqref{2ndenergy} follows from \eqref{highphienergy} and \eqref{highenergy}. This completes the proof of Lemma \ref{le2ndenergy}.
\end{proof}
With the above results in this section, we can conclude our final estimates for the highest order quantities.
\begin{lemma}\label{highE}
Under the a priori assumption \eqref{apriori}, it holds that
\begin{align} 
\label{highE1}  
&\widetilde{C}_{2}\frac{1}{\ka_2}\eps^2\de^2\{\|\pa^2_x(\widetilde{\rho},\widetilde{u},\widetilde{\theta},\widetilde{\phi})(t)\|^{2}+\|\pa^2_xf(t)\|^{2}+\eps\|\pa^3_x\widetilde{\phi}(t)\|^{2}\} +\eps^2\de^2\|\pa^2_xf(t)\|^2_W 
\notag\\ 
&+	c\eps\de\int^t_0\|\pa^2_xf(s)\|^2_{\si,W}\,ds +\widetilde{C}_{2}\frac{1}{\ka_2}\eps\de\int^t_0\|\pa^2_x\mathbf{P}_1f(s)\|^2_{\si}\,ds  \notag\\	  
&+\widetilde{C}_{2}\widetilde{C}_{1}\frac{1}{\ka_3\ka_2}\eps^2\de^2\|\partial^{2}_x g(t)\|_w^{2}+\widetilde{C}_{2}\widetilde{C}_{1}\frac{1}{\ka_3\ka_2}\eps\de\int_0^{t}\|\partial^{2}_x g(s)\|_{\sigma,w}^{2}\,ds
 \notag\\
&+\widetilde{C}_{2}\widetilde{C}_{1}\frac{1}{\ka_3\ka_2}\eps^2\de^2 q_{1}q_{2}\int_0^{t}(1+s)^{-(1+q_2)}\|\langle v\rangle \partial^{2}_xg(s)\|_{w}^{2}\,ds 
\notag\\ 		  
\leq& C(\ka_3\frac{1}{\ka_2}+\ka_2)\eps\de\int^t_0\|\pa^2_x(\widetilde{\rho},\widetilde{u},\widetilde{\theta})(s)\|^2\,ds+C\frac{1}{\ka_3\ka_2}\CE(0) 
+C(\eta)\frac{1}{\ka_3\ka_2}\eps^2\de^2 +C\frac{1}{\ka_3\ka_2}\eps^5     
\notag\\          
&+C\frac{1}{\ka_3\ka_2}C(\eta)\int_0^{t}\CE(s)\,ds         +C\frac{1}{\ka_3\ka_2}\eps^2\de^2(C(\eta)+\eps)\frac{\de^{1/2}}{\eps}\int_0^{t}\|\langle v\rangle \pa^2_x g(s)\|_w^2\,ds
\notag\\	 
&+	C\frac{1}{\ka_2}(\ka+C_{\ka}\frac{1}{\ka_3}(\de+\eps+\frac{\de^2}{\eps}))\int^t_0\CD(s)\,ds        +C_{\ka} \frac{1}{\ka_3\ka_2} C(\eta)t(\eps\de^2+\eps^2\de^2+\de^3).	 
\end{align}	 
\end{lemma}
\begin{proof}
Adding \eqref{w2g}$\times\widetilde{C}_{1}\frac{1}{\ka_3}$ to \eqref{2ndenergy} and choosing a large constant $\widetilde{C}_{1}>1$, we get
\begin{align}\label{4.177} 
&\eps^2\de^2\{\|\pa^2_x(\widetilde{\rho},\widetilde{u},\widetilde{\theta},\widetilde{\phi})(t)\|^{2}+\|\pa^2_xf(t)\|^{2}+\eps\|\pa^3_x\widetilde{\phi}(t)\|^{2}\}+	\eps\de\int^t_0\|\pa^2_x\mathbf{P}_1f(s)\|^2_{\si}\,ds 
\notag\\	
&+\frac{1}{2}\widetilde{C}_{1}\frac{1}{\ka_3}\eps^2\de^2\|\partial^{2}_x g(t)\|_w^{2}
+\frac{1}{2}\widetilde{C}_{1}\frac{1}{\ka_3}\eps\de\int_0^{t}\|\partial^{2}_x g(s)\|_{\sigma,w}^{2}\,ds
\notag\\	
&+\widetilde{C}_{1}\frac{1}{\ka_3}\eps^2\de^2 q_{1}q_{2}\int_0^{t}(1+s)^{-(1+q_2)}\|\langle v\rangle \partial^{2}_xg(s)\|_{w}^{2}\,ds
\notag\\ 		
\leq& C\ka_3\eps\de\int^t_0\|\pa^2_x(\widetilde{\rho},\widetilde{u},\widetilde{\theta})(s)\|^2\,ds+C\frac{1}{\ka_3}\CE(0) +C(\eta)\eps^2\de^2         +C\eps^5      
\notag\\         
&+\frac{1}{\ka_3}C(\eta)\int_0^{t}\CE(s)\,ds         +C\frac{1}{\ka_3}\eps^2\de^2(C(\eta)+\eps)\frac{\de^{1/2}}{\eps}\int_0^{t}\|\langle v\rangle \pa^2_x g(s)\|_w^2\,ds
\notag\\	
&+	(C\ka+C_{\ka}\frac{1}{\ka_3}(\de+\eps+\frac{\de^2}{\eps}))\int^t_0\CD(s)\,ds        +C_{\ka}  \frac{1}{\ka_3}C(\eta)t(\eps\de^2+\eps^2\de^2+\de^3),	
\end{align}	
for any $0<\ka<1$ and $0<\ka_3<1$.
Adding \eqref{4.177}$\times\widetilde{C}_{2}\frac{1}{\ka_2}$ to \eqref{2ndwdiss} with $\ka_2$ as in \eqref{2ndwdiss}  and choosing a large constant $\widetilde{C}_{2}>1$, we get \eqref{highE1}.
This completes the proof of Lemma \ref{highE}.
\end{proof}

\section{Proof of Theorem \ref{TheoremKdV}}\label{SecProof}
With all the estimates in Section \ref{secKdV}, we now prove Theorem \ref{TheoremKdV} in this section.
\begin{proof}
The summation of \eqref{1stfluid} and \eqref{zerofluid} gives the total low order fluid estimates that
\begin{align}\label{lowfluid}
&\sum_{\alpha\leq 1}\{\|\pa^\al_x(\widetilde{\rho},\widetilde{u},\widetilde{\theta},\widetilde{\phi})(t)\|^{2}+\eps\|\pa^\al_x\pa_x\widetilde{\phi}(t)\|^2\}+\eps\de\sum_{1\leq\alpha\leq 2}\int^t_0\{\|\pa^\al_x(\widetilde{\rho},\widetilde{u},\widetilde{\theta},\widetilde{\phi})(s)\|^{2}+\eps\|\pa^\al_x\pa_x\widetilde{\phi}(s)\|^{2}\}\,ds
\notag\\ 
\leq&C\CE(0)+C\eps^2\de^2\|\pa^2_x\widetilde{\rho}(t)\|^2+C\eps^2\de^2\|\pa^2_x(g,f)(t)\|^2+ C\eps\de\int_0^t\|\pa^2_x(g,f)(s)\|_\sigma^{2}\,ds \notag\\ &+C(\de+\eps+\frac{\de}{\sqrt\eps})\int_0^t\CD(s)\,ds+CA\int_0^t\min\{(\frac{\eps}{\de}+\frac{\eps^2}{\de^2})\CD(s),(\frac{\eps^{3/2}}{\sqrt{\de}}+\frac{\eps^{3}}{\de})\CE(s)\}\,ds \notag\\ &+C(\eta)\int^t_0\CE(s)\,ds+C(\eta)t(\eps^4+\eps^3\de+\eps\de^3)+C\eps^5.
\end{align}	
 Adding \eqref{walbeg}$\times\widetilde{C}_{3}$ to \eqref{01f} and choosing a large constant $\widetilde{C}_{3}>1$, we get 
 \begin{align} \label{4.118A} 		 
 &\frac{d}{dt}\big\{\sum_{\alpha\leq 1}(\widetilde{C}_{3}C_\alpha\|\partial^{\alpha}_x g\|_w^{2}+ \overline{C}'_{\alpha}\|\partial^{\alpha}_x f\|_W^{2}+\overline{C}_1\|\partial^{\alpha}_x f\|^{2} ) 
 \notag\\  
 &\quad\quad+\sum_{\alpha+|\beta|\leq 2,|\beta|\geq 1}(C_{\alpha,\be}\widetilde{C}_{3}\|\partial^{\alpha}_\be g\|_w^{2} +\overline{C}_{\alpha,\be}\|\partial^{\alpha}_\be f\|_W^{2}) \big\} \notag\\  &+c\frac{1}{\eps\de}\big\{\sum_{\alpha\leq 1} (\|\partial^{\alpha}_x g\|_{\sigma,w}^{2} +\|\partial^{\alpha}_x\mathbf{P}_1 f\|_{\sigma}^{2}+\|\partial^{\alpha}_x f\|_{\sigma,W}^{2}) +\sum_{\al+|\beta|\leq 2,|\beta|\geq 1}(\|\partial^{\alpha}_\be g\|_{\sigma,w}^{2}+\|\partial^{\alpha}_\be f\|_{\sigma,W}^{2})\big\}      \notag\\ 		 &+cq_{1}q_{2}(1+t)^{-(1+q_2)}\big\{\sum_{\alpha\leq 1}\|\langle v\rangle \partial^{\alpha}_xg\|_{w}^{2}+\sum_{\al+|\beta|\leq 2,|\beta|\geq 1}\|\langle v\rangle \partial^{\alpha}_{\beta}g\|_{w}^{2}\big\} \notag\\ 		\leq& C(\eps+\de)\mathcal{D}(t)+C\eps\de(\|\pa^2_xf\|_{\sigma,W}^{2}+\|\pa^2_xg\|_{\sigma,w}^{2})+C(\eta)(\eps^3\de+\eps\de^3+\eps^2\de^2)+ C(\eta)\CE(t) \notag\\ 		 &+C\eps\de  \sum_{\alpha\leq 1}\|\pa^{\alpha}_x\pa_x(\widetilde{u},\widetilde{\theta})\|^{2}  +C(C(\eta)+\eps)\frac{\de^{1/2}}{\eps}\big\{\sum_{\alpha\leq 1}\|\langle v\rangle\pa^\al_x g\|_w^2+\sum_{\al+|\beta|\leq 2,|\beta|\geq 1}\|\langle v\rangle\pa^\al_\be g\|_w^2\big\}. \end{align}	 
 By taking integral in $t$ of \eqref{4.118A},  we can obtain 
\begin{align}\label{5.3A}
&\sum_{\alpha\leq 1}(\|\partial^{\alpha}_x f(t)\|_W^{2}+\|\partial^{\alpha}_x g(t)\|_w^{2})+\sum_{\alpha+|\beta|\leq 2,|\beta|\geq 1}(\|\partial^{\alpha}_\be f(t)\|_W^{2} +\|\partial^{\alpha}_\be g(t)\|_w^{2})
\notag\\ 	 
&+c\frac{1}{\eps\de}\int^t_0\big\{\sum_{\alpha\leq 1}(\|\partial^{\alpha}_x f(s)\|_{\sigma,W}^{2}+\|\partial^{\alpha}_x g(s)\|_{\sigma,w}^{2})+\sum_{\al+|\beta|\leq 2,|\beta|\geq 1}(\|\partial^{\alpha}_\be f(s)\|_{\sigma,W}^{2}+\|\partial^{\alpha}_\be g(s)\|_{\sigma,w}^{2})\big\}\,ds 
\notag\\ 	 &-\int^t_0(C(C(\eta)+\eps)\frac{\de^{1/2}}{\eps}-q_1q_2(1+s)^{-(1+q_2)})\big\{\sum_{\alpha\leq 1}\|\langle v\rangle \pa^\al_x g(s)\|_w^2+\sum_{\al+|\beta|\leq 2,|\beta|\geq 1}\|\langle v\rangle \pa^\al_\be g(s)\|_w^2\big\}\,ds \notag\\ 	 \leq&C\CE(0)+C(\eps+\de)\int^t_0\mathcal{D}(s)\,ds+C\eps\de\int^t_0\{\|\pa^2_xf(s)\|_{\sigma,W}^{2}+\|\pa^2_xg(s)\|_{\sigma,w}^{2}\}\,ds \notag\\ 	 
&+C\eps\de  \sum_{\alpha\leq 1}\int^t_0\|\pa^{\alpha}_x\pa_x(\widetilde{u},\widetilde{\theta})(s)\|^{2} \,ds+ C(\eta)\int^t_0\CE(s)\,ds
+C(\eta)t(\eps^3\de+\eps\de^3+\eps^2\de^2).  
\end{align} 
Adding \eqref{lowfluid}$\times\widetilde{C}_{4}$ to \eqref{5.3A}, then by choosing $\widetilde{C}_{4}>1$ large enough, we get
\begin{align}\label{5.4A} 
&\widetilde{C}_{4}\sum_{\alpha\leq 1}\{\|\pa^\al_x(\widetilde{\rho},
\widetilde{u},\widetilde{\theta},\widetilde{\phi})(t)\|^{2}+\eps\|\pa^\al_x\pa_x\widetilde{\phi}(t)\|^2\}     +\widetilde{C}_{4}\eps\de\sum_{1\leq\alpha\leq 2}\int^t_0\{\|\pa^\al_x(\widetilde{\rho},\widetilde{u},\widetilde{\theta},\widetilde{\phi})(s)\|^{2}+\eps\|\pa^\al_x\pa_x\widetilde{\phi}(s)\|^{2}\}\,ds \notag\\ 
&+\sum_{\alpha\leq 1}(\|\partial^{\alpha}_x f(t)\|_W^{2}+\|\partial^{\alpha}_x g(t)\|_w^{2})+\sum_{\alpha+|\beta|\leq 2,|\beta|\geq 1}(\|\partial^{\alpha}_\be f(t)\|_W^{2} +\|\partial^{\alpha}_\be g(t)\|_w^{2})
\notag\\ 	
&+\frac{1}{\eps\de}\int^t_0\big\{\sum_{\alpha\leq 1}(\|\partial^{\alpha}_x f(s)\|_{\sigma,W}^{2}+\|\partial^{\alpha}_x g(s)\|_{\sigma,w}^{2})+\sum_{\al+|\beta|\leq 2,|\beta|\geq 1}(\|\partial^{\alpha}_\be f(s)\|_{\sigma,W}^{2}+\|\partial^{\alpha}_\be g(s)\|_{\sigma,w}^{2})\big\}\,ds 
\notag\\ 	 
&-\int^t_0(C(C(\eta)+\eps)\frac{\de^{1/2}}{\eps}-q_1q_2(1+s)^{-(1+q_2)})\big\{\sum_{\alpha\leq 1}\|\langle v\rangle \pa^\al_x g(s)\|_w^2+\sum_{\al+|\beta|\leq 2,|\beta|\geq 1}\|\langle v\rangle \pa^\al_\be g(s)\|_w^2\big\}\,ds 
\notag\\ 
 &\leq C\CE(0)+C\eps^2\de^2\|\pa^2_x\widetilde{\rho}(t)\|^2+C\eps^2\de^2\|\pa^2_x(g,f)(t)\|^2
+C\eps\de\int^t_0\{\|\pa^2_xf(s)\|_{\sigma,W}^{2}+\|\pa^2_xg(s)\|_{\sigma,w}^{2}\}\,ds
\notag\\
\quad&+C(\de+\eps+\frac{\de}{\sqrt\eps})\int_0^t\CD(s)\,ds+CA\int_0^t\min\{(\frac{\eps}{\de}+\frac{\eps^2}{\de^2})\CD(s),(\frac{\eps^{3/2}}{\sqrt{\de}}+\frac{\eps^{3}}{\de})\CE(s)\}\,ds 
\notag\\ \quad&+C(\eta)\int^t_0\CE(s)\,ds+C(\eta)t(\eps^4+\eps^3\de+\eps\de^3+\eps^2\de^2)+C\eps^5.
\end{align} 
Adding \eqref{highE1}$\times\widetilde{C}_{5}$ to \eqref{5.4A}, 
then by choosing $\widetilde{C}_{5}>1$ large enough, we get
\begin{align*}  
&\sum_{\alpha\leq 1}\{\|\pa^\al_x(\widetilde{\rho},\widetilde{u},\widetilde{\theta},\widetilde{\phi})(t)\|^{2}+\eps\|\pa^\al_x\pa_x\widetilde{\phi}(t)\|^2
+\|\partial^{\alpha}_x f(t)\|_W^{2}+\|\partial^{\alpha}_x g(t)\|_w^{2}\} 
\notag\\  
&+\widetilde{C}_{5}\eps^2\de^2
\big\{\frac{1}{\ka_2}(\|\pa^2_x(\widetilde{\rho},\widetilde{u},\widetilde{\theta},\widetilde{\phi})(t)\|^{2}+\eps\|\pa^3_x\widetilde{\phi}(t)\|^{2}) +\|\pa^2_xf(t)\|^2_W  
+\frac{1}{\ka_3\ka_2}\|\partial^{2}_x g(t)\|_w^{2}\big\}
\notag\\  
&+\sum_{\alpha+|\beta|\leq 2,|\beta|\geq 1}(\|\partial^{\alpha}_\be f(t)\|_W^{2} +\|\partial^{\alpha}_\be g(t)\|_w^{2})
+\eps\de\sum_{1\leq\alpha\leq 2}\int^t_0\{\|\pa^\al_x(\widetilde{\rho},\widetilde{u},\widetilde{\theta},\widetilde{\phi})(s)\|^{2}+\eps\|\pa^\al_x\pa_x\widetilde{\phi}(s)\|^{2}\}\,ds
\notag\\ 	
&+\frac{1}{\eps\de}\int^t_0\big\{\sum_{\alpha\leq 1}(\|\partial^{\alpha}_x f(s)\|_{\sigma,W}^{2}+\|\partial^{\alpha}_x g(s)\|_{\sigma,w}^{2})+\sum_{\al+|\beta|\leq 2,|\beta|\geq 1}(\|\partial^{\alpha}_\be f(s)\|_{\sigma,W}^{2}+\|\partial^{\alpha}_\be g(s)\|_{\sigma,w}^{2})\big\}\,ds 
\notag\\  &+	\widetilde{C}_{5}\eps\de\int^t_0(\|\pa^2_xf(s)\|^2_{\si,W}+\|\pa^2_x\mathbf{P}_1f(s)\|^2_{\si}\,ds   +\frac{1}{\ka_3\ka_2}\|\partial^{2}_x g(s)\|_{\sigma,w}^{2})\,ds   
\notag\\ 		   
\leq& 
\widetilde{C}_{5}C(\ka_3\frac{1}{\ka_2}+\ka_2)\eps\de\int^t_0\|\pa^2_x(\widetilde{\rho},\widetilde{u},\widetilde{\theta})(s)\|^2\,ds+C\frac{1}{\ka_3\ka_2}\CE(0)  +C(\eta)\frac{1}{\ka_3\ka_2}\eps^2\de^2 +C\frac{1}{\ka_3\ka_2}\eps^5     
\notag\\           
&+CA\int_0^t\min\{(\frac{\eps}{\de}+\frac{\eps^2}{\de^2})\CD(s),(\frac{\eps^{3/2}}{\sqrt{\de}}+\frac{\eps^{3}}{\de})\CE(s)\}\,ds 
+C\frac{1}{\ka_3\ka_2}C(\eta)\int_0^{t}\CE(s)\,ds         
\notag\\	  
&+	C\frac{1}{\ka_2}(\ka+C_{\ka}\frac{1}{\ka_3}(\de+\eps+\frac{\de^2}{\eps}))\int^t_0\CD(s)\,ds        +C_{\ka} \frac{1}{\ka_3\ka_2} C(\eta)t(\eps\de^2+\eps^3\de+\eps^4+\de^3)
\notag\\ 	 
&+C\int^t_0(C(C(\eta)+\eps)\frac{\de^{1/2}}{\eps}-q_1q_2(1+s)^{-(1+q_2)})\big\{\sum_{\alpha\leq 1}\|\langle v\rangle \pa^\al_x g(s)\|_w^2+\sum_{\al+|\beta|\leq 2,|\beta|\geq 1}\|\langle v\rangle \pa^\al_\be g(s)\|_w^2\big\}\,ds 
\notag\\	   
&+C\frac{1}{\ka_3\ka_2}\eps^2\de^2\int_0^{t}
[C(C(\eta)+\eps)\frac{\de^{1/2}}{\eps}-\widetilde{C}_{5}q_{1}q_{2}
(1+s)^{-(1+q_2)}]
\|\langle v\rangle \pa^2_x g(s)\|_w^2\,ds. 	  
\end{align*}	
Recall $\CE(t)$ and $\CD(t)$ given by \eqref{DefE} and \eqref{DefD} respectively. We take $\ka_3=\ka_2^2$ and let $\ka_2$ small enough to get
\begin{align}
\label{5.6A}
&\CE(t)+\int^t_0\CD(s)\,ds\leq C\frac{1}{\ka_3\ka_2}\CE(0)  +C(\eta)\frac{1}{\ka_3\ka_2}\eps^2\de^2 +C\frac{1}{\ka_3\ka_2}\eps^5      \notag\\            &+CA\int_0^t\min\{(\frac{\eps}{\de}+\frac{\eps^2}{\de^2})\CD(s),(\frac{\eps^{3/2}}{\sqrt{\de}}+\frac{\eps^{3}}{\de})\CE(s)\}\,ds  +C\frac{1}{\ka_3\ka_2}C(\eta)\int_0^{t}\CE(s)\,ds         
    \notag\\	   
    &+	C\frac{1}{\ka_2}(\ka+C_{\ka}\frac{1}{\ka_3}(\de+\eps+\frac{\de^2}{\eps}))\int^t_0\CD(s)\,ds        +C_{\ka} \frac{1}{\ka_3\ka_2} C(\eta)t(\eps\de^2+\eps^3\de+\eps^4+\de^3) 
    \notag\\ 	  
    &+C\int^t_0(C(C(\eta)+\eps)\frac{\de^{1/2}}{\eps}-q_1q_2(1+s)^{-(1+q_2)})\big\{\sum_{\alpha\leq 1}\|\langle v\rangle \pa^\al_x g(s)\|_w^2+\sum_{\al+|\beta|\leq 2,|\beta|\geq 1}\|\langle v\rangle \pa^\al_\be g(s)\|_w^2\big\}\,ds  \notag\\	    &+C\frac{1}{\ka_3\ka_2}\eps^2\de^2\int_0^{t} [C(C(\eta)+\eps)\frac{\de^{1/2}}{\eps}-\widetilde{C}_{5}q_{1}q_{2} (1+s)^{-(1+q_2)}] \|\langle v\rangle \pa^2_x g(s)\|_w^2\,ds. 
\end{align}
Now we make the important assumption that
\begin{align}\label{deeps}
	\eps^{3-c_0}<\de<\eps^{2+c_0}\ \text{or}\ \de=\frac{1}{c_1}\eps^3\ \text{or}\ \de=c_1\eps^2,
\end{align}
for any $0<c_0<1/2$ and some $0<c_1<1$, both are independent of $\eps$ and $\de$, and $c_1$ will be determined later. 
We should point out that \eqref{deeps} is to ensure \eqref{5.8A}, \eqref{5.9A} and \eqref{5.10A} hold true.
We choose
$\eps\ll \ka\ll \ka_2$  with $\ka_3=\ka_2^2$  to require that
\begin{equation}
\label{5.8A}
C\frac{1}{\ka_2}(\ka+C_{\ka}\frac{1}{\ka_3}(\de+\eps+\frac{\de^2}{\eps}))\leq \frac{1}{2}.
\end{equation}
For $t\leq T$ with $T$ given in Theorem \ref{TheoremKdV}, we let $\eps$ and $c_1$ with $\eps\ll c_1$ small enough to require that
\begin{equation} 
\label{5.9A}
C(C(\eta)+\eps)\frac{\de^{1/2}}{\eps}-q_1q_2(1+s)^{-(1+q_2)}
\leq C(C(\eta)+\eps)(c_1^{1/2}+\eps^{c_0/2})-q_1q_2(1+T)^{-(1+q_2)}
\leq 0.
\end{equation}
By these facts, we have from \eqref{5.6A} that
\begin{align}
\label{5.10A}
\CE(t)+\frac{1}{2}\int^t_0\CD(s)\,ds\leq &C\CE(0)  +C(\eta)\eps^2\de^2 +C\eps^5  
 +C(\eta)t(\eps\de^2+\eps^3\de+\eps^4+\de^3) 
\notag\\            
&+(CA\sqrt{c_1}+C(\eta))\int_0^{t}\CE(s)\,ds.
\end{align}
By \eqref{deeps} and \eqref{5.10A}, we get
\begin{align*}
\CE(t)\leq C\CE(0) +C\eps^5+C(\eta)(t+1)\eps^4+(CA\sqrt{c_1}+C(\eta))\int_0^{t}\CE(s)\,ds,
\end{align*}
which together with the Gronwall inequality gives
\begin{align*}
	\CE(t)\leq& C[\CE(0)+\eps^5+C(\eta)(T+1)\eps^4]e^{(CA\sqrt{c_1}+C(\eta))T}.
\end{align*}
We let 
\begin{align*}
	\CE(0)\leq \eps^4,
\end{align*}
then for $t\in(0,T]$,
there exists a constant $\widetilde{C}_{6}>1$ such that
\begin{align}\label{gronwall}
	\CE(t)\leq& \widetilde{C}_{6}(1+T)e^{\widetilde{C}_{6}(1+\sqrt{c_1}A)T}\eps^4.
\end{align}
 Now we can choose $A$ as in \eqref{apriori} to be
\begin{align}\label{defA}
	A=2\widetilde{C}_{6}(1+T)e^{2\widetilde{C}_{6}T},
\end{align}
which depends only on $T$. Substituting \eqref{defA} into \eqref{gronwall}, it holds that
\begin{align*}
\CE(t)\leq& \frac{A}{2}e^{\widetilde{C}_{6}(1+\sqrt{c_1}A)T-2\widetilde{C}_{6}T}\eps^4.
\end{align*}
Then we take $c_1$ so small that
$$(1+\sqrt{c_1}A)T-2T\leq 0,$$
which yields that for $t\in(0,T]$,
\begin{equation}
\label{5.14A}
\CE(t)\leq \frac{A}{2}\eps^4.
\end{equation}
Thus, the a priori assumption  \eqref{apriori} 
can be closed since the estimate  \eqref{5.14A} is strictly stronger than \eqref{apriori}.
Recall $F=M+\overline{G}+\sqrt{\overline{M}}f+\sqrt{\mu}g$, we 
have from \eqref{5.14A}, \eqref{DefE}, Lemma \ref{lem3.1}
and \eqref{deeps} that
\begin{align*} 
\|\frac{F(t,x,v)-M_{[\bar{\rho},\bar{u},\bar{\theta}](t,x)}(v)}{\sqrt{\mu}}\|_{L_{x}^{\infty}L_{v}^{2}}
\leq&
\|\frac{M_{[\rho,u,\theta](t,x)}(v)-M_{[\bar{\rho},\bar{u},\bar{\theta}](t,x)}(v)}{\sqrt{\mu}}\|_{L_{x}^{\infty}L_{v}^{2}}
+\|\frac{\overline{G}(t,x,v)}{\sqrt{\mu}}\|_{L_{x}^{\infty}L_{v}^{2}}
\notag\\ 
&+\|\frac{\sqrt{\overline{M}}f(t,x,v)}{\sqrt{\mu}}\|_{L_{x}^{\infty}L_{v}^{2}}
+\|g(t,x,v)\|_{L^{\infty}_{x}L^{2}_{v}}
\notag\\ 
\leq& C\|(\widetilde{\rho},\widetilde{u},\widetilde{\theta})(t)\|_{L_{x}^{\infty}}+C(\eta)\eps\de+2\sqrt{A}\eps^2
\leq C\eps^2, 
\end{align*}
for any $t\in[0,T]$. Similar estimates also hold true for
$\|\frac{F(t,x,v)-M_{[\bar{\rho},\bar{u},\bar{\theta}](t,x)}(v)}{\sqrt{\mu}}\|$. This gives \eqref{2.63A}.

 By the uniform a priori estimates, the local existence of the solution, and the standard continuity argument, we can immediately derive the existence and uniqueness of solutions to the VPL system \eqref{rVPL}.
Recall $$\frac{1}{\de}:=\frac{\nu}{\eps^{3+1/2}}.$$ Then \eqref{deeps} is equivalent to \eqref{nueps}. Hence, the proof of Theorem \ref{TheoremKdV} is complete.
\end{proof}

\section{Global estimates around constant state}\label{SecGlobal}
In this section, we are devoted to the proof of our second result Theorem \ref{TheoremGlobal}. Namely, we consider the global existence of solutions near global Maxwellians for the scaled VPL system \eqref{rVPL} with the small parameter $\eps>0$. In particular, we establish the uniform a priori estimates with respect to both time and $\eps$.  

For the purpose, we still denote the macroscopic quantities $(\rho,u,\theta,\phi)(t,x)$ for the unknown function $F(t,x,v)$ by \eqref{DefMacro}. To construct solutions around constant states, we let
\begin{equation*}
(\bar{\rho},\bar{u},\bar{\theta},\bar{\phi})=(1,0,\frac{3}{2},0),
\end{equation*}
and still define the perturbation functions $(\widetilde{\rho},\widetilde{u},\widetilde{\theta},\widetilde{\phi})$ and $F=M+G$ with  $G=h=\sqrt{\overline{M}}f+\sqrt{\mu}g$ as in \eqref{Defpert}. Notice that now the Maxwellian
$$
\overline{M}=M_{[\bar{\rho},\bar{u},\bar{\theta}](t,x)}(v)
=\frac{1}{(2\pi)^{3/2}}\exp\big(-\frac{|v|^{2}}{2
}\big)
$$
is a global Maxwellian, which implies $\overline{G}=0$.
By performing similar calculations to how we get \eqref{perturbeq}, \eqref{eqg} and \eqref{eqf}, we have the following three systems. For the macroscopic perturbation, it holds that

\begin{align} 		
\label{Gperturbeq} 		
\left\{ 		
\begin{array}{rl} 			
&\dis \partial_{t}\widetilde{\rho}-\frac{1}{\eps}\pa_x \widetilde{\rho} +\partial_x(\rho\widetilde{u}_{1})=0, 			
\\ 			
&\dis \partial_t\widetilde{u}_{1}-\frac{1}{\eps}\partial_x\widetilde{u}_{1} 			+u_{1}\partial_x\widetilde{u}_{1}+\frac{2}{3}\partial_x\widetilde{\theta} 			+\frac{2}{3}\frac{\theta}{\rho}\partial_x\widetilde{\rho} 			+\frac{1}{\eps}\partial_x\widetilde{\phi} 			
=\eps\de\frac{4}{3\rho}\partial_x(\mu(\theta)\partial_xu_{1})-\frac{1}{\rho}\partial_x(\int_{\mathbb{R}^{3}} v^{2}_{1}L^{-1}_{M}\Theta\,dv),
\\ 			&\dis\partial_t\widetilde{u}_{i}-\frac{1}{\eps}\partial_x\widetilde{u}_{i}+u_{1}\partial_x\widetilde{u}_{i}=\eps\de\frac{1}{\rho}\partial_x(\mu(\theta)\partial_xu_{i})-\frac{1}{\rho}\partial_x(\int_{\mathbb{R}^3} v_{1}v_{i}L^{-1}_{M}\Theta\,dv), ~~i=2,3, 			
\\ 			
&\dis\partial_t\widetilde{\theta}-\frac{1}{\eps}\partial_x\widetilde{\theta} 			+\frac{2}{3}\bar{\theta}\partial_x\widetilde{u}_{1}+\frac{2}{3}\widetilde{\theta}\partial_xu_{1} 			+u_1\partial_x\widetilde{\theta}=\eps\de\frac{1}{\rho}\partial_x(\kappa(\theta)\partial_x\theta)+\eps\de\frac{4}{3\rho}\mu(\theta)(\partial_xu_{1})^2
\\ 			
&\dis\qquad  			
+\eps\de\frac{1}{\rho}\mu(\theta)[(\partial_xu_{2})^2+(\partial_xu_{3})^2] +\frac{1}{\rho}u\cdot\partial_x(\int_{\mathbb{R}^3} v_{1}v L^{-1}_{M}\Theta\, dv)-\frac{1}{\rho}\partial_x(\int_{\mathbb{R}^3}v_{1}\frac{|v|^{2}}{2}L^{-1}_{M}\Theta\, dv), 			
\\ 			
&\dis-\eps^2\pa_x^2\widetilde{\phi}+\eps\widetilde{\phi}
=\widetilde{\rho}.		
\end{array} \right. 	
\end{align}

For equations on $g$ and $f$ in the Caflisch's decomposition $F=M+\sqrt{\overline{M}}f+\sqrt{\mu}g$, it holds that
\begin{align}
\label{Geqg} 		
&\partial_tg-\frac{1}{\eps}\partial_xg+v_{1}\partial_xg	-\frac{1}{\eps\de}L_Dg-\frac{1}{\eps}\frac{\partial_x\phi\partial_{v_{1}}(\sqrt{\mu}g+\sqrt{\overline{M}}f)}{\sqrt{\mu}} 		
\notag\\ 		
=&\frac{1}{\eps\de}\Gamma(g+\frac{\sqrt{\overline{M}}}{\sqrt{\mu}}f,g+\frac{\sqrt{\overline{M}}}{\sqrt{\mu}}f)+\frac{1}{\eps\de}\{\Gamma(\frac{M-\overline{M}}{\sqrt{\mu}},\frac{\sqrt{\overline{M}}}{\sqrt{\mu}}f)+\Gamma(\frac{\sqrt{\overline{M}}}{\sqrt{\mu}}f,\frac{M-\overline{M}}{\sqrt{\mu}})\}, 	
\end{align} 	
and 	
\begin{align}\label{Geqf} 		&\partial_tf-\frac{1}{\eps}\partial_xf+v_{1}\partial_xf	-\frac{1}{\eps\de}\CL_{\overline{M}} f-\frac{1}{\eps\de}\frac{\sqrt{\mu}}{\sqrt{\overline{M}}}L_Bg\notag\\ 		=&\frac{P_{0}(v_{1}\partial_x(\sqrt{\overline{M}}f)+v_{1}\sqrt{\mu}\partial_xg)}{\sqrt{\overline{M}}}-\frac{1}{\sqrt{\overline{M}}}P_{1}\{v_{1}M(\frac{|v-u|^{2} 			\partial_x\widetilde{\theta}}{2K\theta^{2}}+\frac{(v-u)\cdot\partial_x\widetilde{u}}{K\theta})\}.
\end{align}
Note that three systems above have the similar structure to \eqref{perturbeq}, \eqref{eqg} and \eqref{eqf}, where $(\bar{\rho},\bar{u},\bar{\theta},\bar{\phi})=(1,0,\frac{3}{2},0)$
and $\overline{G}=0$. Hence, by performing the same calculations as in Section \ref{secKdV}, all estimates in Section \ref{secKdV} still hold. By the note under \eqref{boundkdv}, the background profile $(\bar{\rho},\bar{u},\bar{\theta},\bar{\phi})$ can be regarded as a constant KdV solution in Section \ref{secKdV} and hence we can let all $\eta$ and $C(\eta)$ in Section \ref{secKdV} be zero. Furthermore, to carry out weighted energy estimates for global-in-time existence, we use a different weight for $g$:
\begin{align}
\label{defw}
\textbf{w}(\alpha,\beta)(t,v):=\langle v\rangle^{2(l-\al-|\beta|)}\exp\left((q_1-q_2\int^t_0 q_3(\tau)\,d\tau)\langle v\rangle\right),
\quad  l\geq \al+|\beta|,
\end{align}
with
$$
q_3(t)=\frac{\de}{\eps}\|\pa_t\widetilde{\phi}(t)\|^2_{L^\infty_x}+\frac{\de}{\eps^3}\|\pa_x\widetilde{\phi}(t)\|^2_{L^\infty_x}.
$$
We change the weight $w$ into $\textbf{w}$ in the definition \eqref{DefE} and \eqref{DefD} such that
\begin{align}\label{DefEG}
	\mathcal{E}(t):=&\sum_{\alpha\leq 1}\{\|\partial^{\alpha}_x(\widetilde{\rho},\widetilde{u},\widetilde{\theta},\widetilde{\phi})(t)\|^{2}+\eps\|\pa^\al_x \pa_x\widetilde{\phi}(t)\|^2
	+\|\partial^{\alpha}_xg(t)\|_{\textbf{w}}^{2}+\|\partial^{\alpha}_xf(t)\|_W^{2}\}
	\nonumber\\
	&\hspace{0.5cm}+\eps^2 \de^2\{\|\pa^2_x(\widetilde{\rho},\widetilde{u},\widetilde{\theta},\widetilde{\phi})(t)\|^{2}+\eps\|\pa^3_x\widetilde{\phi}(t)\|^2
	+\|\pa^2_xg(t)\|_{\textbf{w}}^{2}+\|\pa^2_xf(t)\|_W^{2}\}\notag\\
	&\hspace{0.5cm}+\sum_{\alpha+|\beta|\leq 2,|\beta|\geq1}\{\|\partial^{\alpha}_{\beta}g(t)\|_{\textbf{w}}^{2}+\|\partial^{\alpha}_{\beta}f(t)\|_W^{2}\},
\end{align}
and
\begin{align}\label{DefDG}
	\mathcal{D}(t)&:=\eps\de\sum_{1\leq\alpha\leq 2}\{
	\|\partial^{\alpha}_x(\widetilde{\rho},\widetilde{u},\widetilde{\theta},\widetilde{\phi})(t)\|^{2}+\eps\|\pa^\al_x \pa_x\widetilde{\phi}\|^2\}
	+\frac{1}{\eps\de}\sum_{\alpha+|\beta|\leq 2,|\beta|\geq1}\{\|\partial^{\alpha}_{\beta}g(t)\|_{\sigma,\textbf{w}}^{2}+\|\partial^{\alpha}_{\beta}f(t)\|_{\sigma,W}^{2}\}
	\nonumber\\
	&\quad+\frac{1}{\eps\de}\sum_{\al\leq 1}\{\|\partial^{\alpha}_xg(t)\|_{\sigma,\textbf{w}}^{2}+\|\partial^{\alpha}_xf(t)\|_{\sigma,W}^{2}\}+\eps\de\sum_{\al=2}\{\|\partial^{\alpha}_xg(t)\|_{\sigma,\textbf{w}}^{2}+\|\partial^{\alpha}_xf(t)\|_{\sigma,W}^{2}\}.
\end{align}
Then we make the a priori assumption
\begin{equation}\label{Gapriori}
	\sup_{0\leq t\leq T}\mathcal{E}(t)+\int^T_0\CD(s)\,ds\leq A\eps^{4+\frac{3}{4}},
\end{equation}
with \eqref{smalleps} still holding true, where $0< T<\infty$ can be arbitray. By the similar calculations in \eqref{1patphi}, we get
\begin{align*}	\eps^2(\partial_{x}\partial_t\widetilde{\phi},\partial_{x}\partial_t\widetilde{\phi})+\eps(\partial_t\widetilde{\phi},\partial_t\widetilde{\phi}) =&(\frac{1}{\eps}\pa_x \widetilde{\rho} 		-\partial_x(\rho\widetilde{u}_{1}),\partial_t\widetilde{\phi}) 
\notag\\
=&\eps(\pa_x^2\widetilde{\phi},\partial_t\pa_x\widetilde{\phi})+(\pa_x\widetilde{\phi},\pa_t\widetilde{\phi})
-(\partial_x(\rho\widetilde{u}_{1}),\partial_t\widetilde{\phi}),
\end{align*} 	
which implies that
\begin{align*}
\eps^2\|\partial_t\pa_x\widetilde{\phi}\|^2+\eps\|\partial_t\widetilde{\phi}\|^2
\leq C\|\pa_x^2\widetilde{\phi}\|^2+
C\frac{1}{\eps}\big(\|\pa_x\widetilde{\phi}\|^2+
\|\pa_x\widetilde{u}_{1}\|^2+\|\pa_x\widetilde{\rho}\|^2\|\widetilde{u}_{1}\|_{L^\infty}^2\big)\leq C\frac{1}{\eps^2\de}\CD(t).
\end{align*}
It follows from this, \eqref{defw} and \eqref{Gapriori} that
\begin{align*}
	\int^t_0 q_3(s)\,ds&\leq \int^t_0 (\frac{\de}{\eps}\|\pa_s\widetilde{\phi}(s)\|\|\pa_s\pa_x\widetilde{\phi}(s)\|+\frac{\de}{\eps^3}\|\pa_x\widetilde{\phi}(s)\|\|\pa^2_x\widetilde{\phi}(s)\|)\,ds
    \notag\\
	&\leq C\int^t_0 (\frac{\de}{\eps}\frac{1}{\eps^{7/2}\de}\CD(s)+\frac{1}{\eps^4}\CD(s))\,ds\leq CA\eps^{\frac{1}{4}},
\end{align*}
uniformly on $0<t<T$.
Hence, $q_3(t)$ is integrable in time, and we further require that 
\begin{align*}
0<q_1-q_2\cdot CA\eps^{\frac{1}{4}}\leq q_1-q_2\int^T_0 q_3(\tau)\,d\tau
<q_1.
\end{align*}
In this section we will choose $0<q_1\ll 1$ and $q_2=\ep^{-\frac{1}{8}}$.
Recall the weighted  estimates on the linearized Landau  operator $\CL_{\overline{M}}$ with $\overline{M}$ being a global Maxwellian. 
Assume that $0\leq q_1-q_2\int^T_0 q_3(\tau)d\tau\ll 1$ in $\textbf{w}(\alpha,\beta)$, there exists $C_\ka>0$ such that
\begin{equation*}
\langle\partial^\alpha_\beta\mathcal{L}_{\overline{M}}f,\textbf{w}^2(\alpha,\beta)\partial^\alpha_\beta f\rangle
\geq |\textbf{w}(\alpha,\beta)\partial^\alpha_\beta f|_\sigma^2-\ka\sum_{|\beta_1|=|\beta|}|\textbf{w}(\alpha,\beta_1)\partial^\alpha_{\beta_1} f|_\sigma^2
-C_\ka\sum_{|\beta_1|<|\beta|}|\textbf{w}(\alpha,\beta_1)\partial^\alpha_{\beta_1} f|_\sigma^2,
\end{equation*}
and
\begin{equation*}
\langle\partial_x^\alpha\mathcal{L}_{\overline{M}}f,\textbf{w}^2(\alpha,0)\partial^\alpha g\rangle\geq |\textbf{w}(\alpha,0)\partial_x^\alpha f|_\sigma^2-C_\ka|\chi_{\ka}(v)\partial_x^\alpha f|_2^2,
\end{equation*}
where $\chi_\ka(v)$ is a general cutoff function depending on $\ka$.
We point out that they can be proved by a straightforward modification of the arguments used in \cite[Lemmas 9]{Strain-Guo}
and \cite[Lemmas 2.2-2.3]{Wang}, and we thus omit their proofs for brevity.

Note that the a priori assumption \eqref{Gapriori} naturally implies that
 $$
 \sup_{0\leq t\leq T}\mathcal{E}(t)\leq A\eps^4,
 $$ which is consistent with \eqref{apriori}. Then under the a priori assumption in \eqref{Gapriori}, we have the following lemmas by similar proofs to Lemmas \ref{le1stfluid}, \ref{le01f}, \ref{lem4.16} and \ref{le2ndenergy}. 
\begin{lemma}
Under the a priori assumption \eqref{Gapriori}, it holds that
\begin{align}
\label{Gfluid} &\|\pa_x(\widetilde{\rho},\widetilde{u},\widetilde{\theta},\widetilde{\phi})(t)\|^{2}+\eps\|\pa^2_x\widetilde{\phi}(t)\|^2+\eps\de\int^t_0\{\|\pa^2_x(\widetilde{\rho},\widetilde{u},\widetilde{\theta},\widetilde{\phi})(s)\|^{2}+\eps\|\pa^3_x\widetilde{\phi}(s)\|^{2}\}\,ds \notag\\ \leq&C\CE(0)+C\eps^2\de^2\|\pa^2_x\widetilde{\rho}(t)\|^2+C\eps^2\de^2\|\pa^2_x(g,f)(t)\|^2+ C\eps\de\int_0^t\|\pa^2_x(g,f)(s)\|_\sigma^{2}\,ds \notag\\ &+C(\de+\eps+\frac{\de}{\sqrt\eps})\int_0^t\CD(s)\,ds+CA\int_0^t\min\{(\frac{\eps}{\de}+\frac{\eps^2}{\de^2})\CD(s),(\frac{\eps^{3/2}}{\sqrt{\de}}+\frac{\eps^{3}}{\de})\CE(s)\}\,ds+C\eps^5, 
\end{align}	
and
\begin{align}\label{Gf} 	
&\frac{d}{dt}\big\{\sum_{\alpha\leq 1}\overline{C}'_{\alpha}\|\partial^{\alpha}_x f\|_W^{2}+\overline{C}_1\sum_{\alpha\leq 1}\|\partial^{\alpha}_x f\|^{2}    +\sum_{\alpha+|\beta|\leq 2,|\beta|\geq 1}\overline{C}_{\alpha,\be}\|\partial^{\alpha}_\be f\|_W^{2}\big\}      \notag\\    
&+c\frac{1}{\eps\de}\big\{\overline{C}_1     \sum_{\alpha\leq 1}\|\partial^{\alpha}_x\mathbf{P}_1 f\|_{\sigma}^{2}+     \sum_{\alpha\leq 1}\|\partial^{\alpha}_x f\|_{\sigma,W}^{2}+\sum_{\al+|\beta|\leq 2,|\beta|\geq 1}\|\partial^{\alpha}_\be f\|_{\sigma,W}^{2}\big\}
\notag\\ 	\leq&C\eps\de(\|\pa^2_xf\|_{\sigma,W}^{2}+\|\pa^2_xg\|_{\sigma,\textbf{w}}^{2})+C(\eps+\de)\mathcal{D}(t)   
\notag\\ 	
&+C\eps\de  \sum_{\alpha\leq 1}\|\pa^{\alpha}_x\pa_x(\widetilde{u},\widetilde{\theta})\|^{2}      +C\frac{1}{\eps\de}\big\{\sum_{\alpha\leq 1}\|\partial^{\alpha}_x g\|_{\sigma,\textbf{w}}^{2}+\sum_{\al+|\beta|\leq 2,|\beta|\geq 1}\|\partial^{\alpha}_\be g\|_{\sigma,\textbf{w}}^{2}\big\}, 
\end{align}	
for some positive constants $\overline{C}'_{\alpha}$, $\overline{C}_{\alpha,\be}$ and $\overline{C}_1$,
and
\begin{align}
\label{G2ndwdiss} &\eps^2\de^2\|\pa^2_xf(t)\|^2_W+	\eps\de\int^t_0\|\pa^2_xf(s)\|^2_{\si,W}\,ds
\notag\\ \leq	&C\CE(0)+C\eps^5 +C\eps^2\de^2\|\pa^2_x(\widetilde{\rho},\widetilde{u},\widetilde{\theta})(t)\|^2+C\ka_2\eps\de\int^t_0\|\pa^2_x(\widetilde{\rho},\widetilde{u},\widetilde{\theta})(s)\|^2\,ds 
\notag\\
&+C\frac{1}{\ka_2}\eps\de\int^t_0\|(\pa^2_xg,\partial^{2}_x\mathbf{P}_1 f)(s)\|^2_{\si}\,ds +(C\ka+C_\ka(\de+\eps+\frac{\de^2}{\eps}))\int^t_0\CD(s)\,ds, 	\end{align}	    
for any $0<\ka<1$ and $0<\ka_2<1$, and
\begin{align}
\label{G2ndenergy} 		
&\eps^2\de^2\{\|\pa^2_x(\widetilde{\rho},\widetilde{u},\widetilde{\theta},\widetilde{\phi})(t)\|^{2}+\|\pa^2_xf(t)\|^{2}+\eps\|\pa^3_x\widetilde{\phi}(t)\|^{2}\}+	\eps\de\int^t_0\|\pa^2_x\mathbf{P}_1f(s)\|^2_{\si}\,ds
\notag\\	\leq&C\CE(0)        +C\eps^5+C\eps^2\de^2\|\pa^2_xg(t)\|^{2}+	(C\ka+C_\ka(\de+\eps+\frac{\de^2}{\eps}))\int^t_0\CD(s)\,ds         \notag\\ 		&+C\eps\de\int^t_0(\ka_3\|\pa^2_x(\widetilde{\rho},\widetilde{u},\widetilde{\theta})(s)\|^2+\frac{1}{\ka_3}\|\pa^2_xg(s)\|^2_{\si})\,ds,	 	
\end{align}     
for any $0<\ka<1$ and $0<\ka_3<1$. 
\end{lemma}	

\begin{proof}
Using \eqref{Gperturbeq} and \eqref{Geqf}, the proof follows by performing similar calculations to Lemma \ref{le1stfluid}, Lemma \ref{le01f} and Lemma \ref{le2ndenergy}, and letting $C(\eta)=0$ with $\eta=0$ in the aforementioned results. Hence, we omit the details for brevity.
\end{proof}

It remains to bound the zero order energy for fluid quantities and get the estimates for $g$. We start with $g$ to illustrate how we construct the new parameter-dependent weight $\textbf{w}$ to control the global singular terms induced by electronic field. We focus mainly on parts that are different from the corresponding results in Section \ref{secKdV}.
\begin{lemma}\label{Gleg}
	Under the a priori assumption \eqref{Gapriori}, it holds that
	\begin{align}\label{Gwalbeg}
		&\frac{d}{dt}\big\{C_\al\sum_{\alpha\leq 1}\|e^{\frac{1}{4}\eps^{-1}\widetilde{\phi}}\partial^\alpha_\be g\|_{\textbf{w}}^2+\sum_{\alpha+|\beta|\leq 2,|\beta|\geq 1}C_{\al,\beta}\|e^{\frac{1}{4}\eps^{-1}\widetilde{\phi}}\partial^{\alpha}_\be g\|_\textbf{w}^{2}\big\}
        \notag\\
        &+c\frac{1}{\eps\de}\big\{\sum_{\alpha\leq 1}\|\partial^{\alpha}_x g\|_{\sigma,\textbf{w}}^{2}+\sum_{\al+|\beta|\leq 2,|\beta|\geq 1}\|\partial^{\alpha}_\be g\|_{\sigma,\textbf{w}}^{2}\big\}
        \notag\\
		&\qquad+cq_{2}q_3(t)\big\{\sum_{\alpha\leq 1}\|\langle v\rangle^{1/2} \partial^{\alpha}_xg(t)\|_{\textbf{w}}^{2}+\sum_{\al+|\beta|\leq 2,|\beta|\geq 1}\|\langle v\rangle^{1/2} \partial^{\alpha}_{\beta}g(t)\|_{\textbf{w}}^{2}\big\}\notag\\
		\leq&C(\eps+\de)\mathcal{D}(t)+C\eps\de\|\pa^2_xg\|_{\sigma,w}^{2}+Cq_3(t)\big\{\sum_{\alpha\leq 1}\|\langle v\rangle^{1/2} \pa^\al_x g\|_\textbf{w}^2+\sum_{\al+|\beta|\leq 2,|\beta|\geq 1}\|\langle v\rangle^{1/2} \pa^\al_\be g\|_\textbf{w}^2\big\},
	\end{align}	
	and
	\begin{align}\label{Gw2g}
		&\eps^2\de^2\frac{d}{dt}\|e^{\frac{1}{4}\eps^{-1}\widetilde{\phi}}\partial^{2}_x g\|_\textbf{w}^{2}+c\eps\de\|\partial^{2}_x g\|_{\sigma,\textbf{w}}^{2}+c\eps^2\de^2 q_{2}q_3(t)\|\langle v\rangle^{1/2} \partial^{2}_xg(t)\|_{\textbf{w}}^{2}
        \notag\\
	&\leq C(\eps+\de)\mathcal{D}(t)+C\eps^2\de^2q_3(t)\|\langle v\rangle^{1/2} \pa^2_x g\|_\textbf{w}^2,
	\end{align}	
	for any $0<\ka<1$.
\end{lemma}
\begin{proof}
	Applying $\partial^{\alpha}_\be$ to \eqref{Geqg} and taking the inner product with $\textbf{w}^2(\al,\be)e^{\frac{1}{2}\eps^{-1}\widetilde{\phi}} \partial^{\alpha}_\be g$, one has
\begin{align}\label{Gipg}
&(\partial_t\partial^{\alpha}_\be g,\textbf{w}^2(\alpha,\be)e^{\frac{1}{2}\eps^{-1}\widetilde{\phi}} \partial^{\alpha}_\be g)
-\frac{1}{\eps}(\partial_x\partial^{\alpha}_\be g,\textbf{w}^2(\alpha,\be)e^{\frac{1}{2}\eps^{-1}\widetilde{\phi}} \partial^{\alpha}_\be g)
\notag\\
&+(\partial^{\alpha}_\be (v_1\pa_xg),\textbf{w}^2(\alpha,\beta)e^{\frac{1}{2}\eps^{-1}\widetilde{\phi}} \partial_{\beta}^\alpha g)-\frac{1}{\eps}(\partial^{\alpha}_\be[\frac{\partial_x\phi\partial_{v_{1}}(\sqrt{\mu}g)}{\sqrt{\mu}}],\textbf{w}^2(\al,\be)e^{\frac{1}{2}\eps^{-1}\widetilde{\phi}} \partial^{\alpha}_\be g)
\notag\\
&-\frac{1}{\eps\de}(\partial^{\alpha}_\be L_D g,\textbf{w}^2(\al,\be)e^{\frac{1}{2}\eps^{-1}\widetilde{\phi}} \partial^{\alpha}_\be g)-\frac{1}{\eps}(\partial^{\alpha}_\be[\frac{\partial_x\phi\partial_{v_{1}}(\sqrt{\overline{M}}f)}{\sqrt{\mu}}],\textbf{w}^2(\al,\be)e^{\frac{1}{2}\eps^{-1}\widetilde{\phi}} \partial^{\alpha}_\be g)
\notag\\
=&\frac{1}{\eps\de}(\partial^{\alpha}_\be \Gamma(g+\frac{\sqrt{\overline{M}}}{\sqrt{\mu}}f,g+\frac{\sqrt{\overline{M}}}{\sqrt{\mu}}f),\textbf{w}^2(\al,\be)e^{\frac{1}{2}\eps^{-1}\widetilde{\phi}} \partial^{\alpha}_\be g)
        \notag\\
		&+\frac{1}{\eps\de}(\partial^{\alpha}_\be\Gamma(\frac{M-\overline{M}}{\sqrt{\mu}},\frac{\sqrt{\overline{M}}}{\sqrt{\mu}}f)+\partial^{\alpha}_\be\Gamma(\frac{\sqrt{\overline{M}}}{\sqrt{\mu}}f,\frac{M-\overline{M}}{\sqrt{\mu}}),\textbf{w}^2(\al,\be)e^{\frac{1}{2}\eps^{-1}\widetilde{\phi}} \partial^{\alpha}_\be g).
	\end{align}
In view of \eqref{Gapriori} and \eqref{DefEG}, one has
$$
\sup_{t\geq 0,x\in\mathbb{R}}|\widetilde{\phi}|\leq C\eps^2,\quad
e^{\frac{1}{2}\eps^{-1}\widetilde{\phi}}\sim 1.
$$
The definition of $\textbf{w}$ gives
\begin{equation*}
\partial_{t}[\textbf{w}^{2}(\alpha,\beta)] =-2q_{2}q_3(t)\langle v\rangle \textbf{w}^{2}(\alpha,\beta),
\end{equation*}
and
$$
\partial_{v_1}\textbf{w}(\alpha,\beta)=2(l-\al-|\beta|)\frac{v_1}{\langle v\rangle^{2}}\textbf{w}(\alpha,\beta)+
(q_1-q_2\int^t_0 q_3(\tau)d\tau)\frac{v_1}{\langle v\rangle} w(\alpha,\beta),
$$
with
\begin{equation}
\label{Gpavw}
|\partial_{v_1}\textbf{w}(\alpha,\beta)|\leq C \textbf{w}(\alpha,\beta).
\end{equation}
Then we have from this, \eqref{boundlandaunorm} and \eqref{defw} that
\begin{align}\label{G1gtime}
&(\partial_t\partial^{\alpha}_\be g,\textbf{w}^2(\alpha,\be)e^{\frac{1}{2}\eps^{-1}\widetilde{\phi}} \partial^{\alpha}_\be g)
\notag\\
	=&\frac{1}{2}\frac{d}{dt}(\partial^\alpha_\be g,\textbf{w}^2(\alpha,\be)e^{\frac{1}{2}\eps^{-1}\widetilde{\phi}} \partial^\alpha_\be g)
	-\frac{1}{2}(\partial^\alpha_\be g,\partial_{t}[\textbf{w}^2(\alpha,\be)]e^{\frac{1}{2}\eps^{-1}\widetilde{\phi}} \partial^\alpha_\be g)-\frac{1}{2\eps}(\partial^\alpha_\be g,\textbf{w}^2(\alpha,\be)e^{\frac{1}{2}\eps^{-1}\widetilde{\phi}} \pa_t\widetilde{\phi}\partial^\alpha_\be g)
\notag\\
	\geq&\frac{1}{2}\frac{d}{dt}\|e^{\frac{1}{4}\eps^{-1}\widetilde{\phi}}\partial^\alpha_\be g\|_{\textbf{w}}^2+\frac{1}{2}q_2q_3(t)\|\langle v\rangle^{1/2} \partial^\alpha_\be g\|^2_{\textbf{w}}-\ka\frac{1}{\eps\de}\|\langle v\rangle^{-1/2}\partial^\alpha_\be g\|^2_{\textbf{w}}-C_\ka \frac{\de}{\eps}\|\pa_t\widetilde{\phi}\|^2_{L^\infty_x}\|\langle v\rangle^{1/2}\partial^\alpha_\be g\|^2_{\textbf{w}}
    \notag\\
	\geq&\frac{1}{2}\frac{d}{dt}\|e^{\frac{1}{4}\eps^{-1}\widetilde{\phi}}\partial^\alpha_\be g\|_{\textbf{w}}^2+\frac{1}{2}q_2q_3(t)\|\langle v\rangle^{1/2} \partial^\alpha_\be g\|^2_{\textbf{w}}-C\ka\frac{1}{\eps\de}\|\partial^\alpha_\be g\|^2_{\sigma,\textbf{w}}
    -C_\ka q_3(t)\|\langle v\rangle^{1/2}\partial^\alpha_\be g\|^2_{\textbf{w}}.
\end{align}
For the second term in \eqref{Gipg}, one has
\begin{align*}
\frac{1}{\eps}(\partial_x\partial^{\alpha}_\be g,\textbf{w}^2(\alpha,\be)e^{\frac{1}{2}\eps^{-1}\widetilde{\phi}} \partial^{\alpha}_\be g)
&=-\frac{1}{\eps}(\partial^{\alpha}_\be g,\partial_x[\textbf{w}^2(\alpha,\be)e^{\frac{1}{2}\eps^{-1}\widetilde{\phi}}]\partial^{\alpha}_\be g)
 \notag\\
&\leq C\frac{1}{\eps^2}\|\pa_x\widetilde{\phi}\|^2_{L^\infty_x}
\|\langle v\rangle^{-1/2}\partial^\alpha_\be g\|^2_{\textbf{w}}
\|\langle v\rangle^{1/2}\partial^\alpha_\be g\|^2_{\textbf{w}}
\notag\\
&\leq C\ka\frac{1}{\eps\de}\|\partial^\alpha_\be g\|^2_{\sigma,\textbf{w}}
+C_\ka q_3(t)\|\langle v\rangle^{1/2}\partial^\alpha_\be g\|^2_{\textbf{w}}.
\end{align*}
Similar arguments as in \eqref{1gtran} show that
\begin{align}\label{G1gtran}
	(\partial^{\alpha}_\be (v_1\pa_xg),\textbf{w}^2(\alpha,\beta)e^{\frac{1}{2}\eps^{-1}\widetilde{\phi}}\partial_{\beta}^\alpha g)\geq& -\frac{1}{4\eps}( v_1\partial^{\alpha}_\be g,\pa_x\widetilde{\phi}\,\textbf{w}^2(\alpha,\beta)e^{\frac{1}{2}\eps^{-1}\widetilde{\phi}}\partial_{\beta}^\alpha g)
    \notag\\
	 & -C\ka\frac{1}{\eps\de}\|\partial^\alpha_\be g\|^2_{\sigma,\textbf{w}}-C_\ka \eps\de\|\partial^2_xg\|^2_{\sigma,\textbf{w}}
     -C\eps\CD(t),
\end{align}
for $\alpha+|\beta|\leq 2,|\beta|\geq 1$ and 
\begin{align}\label{G1gtran0}
	(\partial^{\alpha}_x (v_1\pa_xg),\textbf{w}^2(\alpha,0)e^{\frac{1}{2}\eps^{-1}\widetilde{\phi}}\partial_x^\alpha g)= -\frac{1}{4\eps}( v_1\partial^{\alpha}_x g,\pa_x\widetilde{\phi}\,\textbf{w}^2(\alpha,0)e^{\eps^{-1}\widetilde{\phi}}\partial_x^\alpha g),
\end{align}
for $0\leq\al\leq2$. For the fourth term on the left hand side of \eqref{Gipg}, we have
\begin{align}\label{Gphif1}
&\frac{1}{\eps}(\partial_\be^{\alpha} [\frac{\partial_x\phi\partial_{v_1}(\sqrt{\mu}g)}{\sqrt{\mu}}],\textbf{w}^2(\al,\be)e^{\frac{1}{2}\frac{1}{2}\eps^{-1}\widetilde{\phi}}\partial_\be^{\alpha} g)
\notag\\
=&\frac{1}{\eps}(\partial^\alpha_\be[\partial_x\phi\partial_{v_1}g],\textbf{w}^2(\al,\be)e^{\frac{1}{2}\eps^{-1}\widetilde{\phi}}\partial_\be^\alpha g)-\frac{1}{\eps}(\partial^\alpha_\be[\frac{v_1}{4} g\partial_x\phi],\textbf{w}^2(\al,\be)e^{\frac{1}{2}\eps^{-1}\widetilde{\phi}}\partial_\be^\alpha g)
\notag\\
=&\CI_3-\CI_4.
\end{align}
Direct calculation gives
\begin{align*}
		\CI_3=	\left\{
	\begin{array}{rl}
		&\dis\frac{1}{\eps}(\partial_x\phi\partial_{v_1}\partial^\alpha_\beta g,\textbf{w}^2(\alpha,\beta)e^{\frac{1}{2}\eps^{-1}\widetilde{\phi}}\partial_{\beta}^\alpha g),\  \al=0,
        \\
		&\dis	\frac{1}{\eps}(\partial_x\phi\partial_{v_1}\partial^\alpha_\beta g,\textbf{w}^2(\alpha,\beta)e^{\frac{1}{2}\eps^{-1}\widetilde{\phi}}\partial_{\beta}^\alpha g)+\frac{1}{\eps}(\partial^2_x\phi\partial_{v_1}\partial^\beta_v g,\textbf{w}^2(\alpha,\beta)e^{\frac{1}{2}\eps^{-1}\widetilde{\phi}}\partial_{\beta}^\alpha g),\  \al=1,
        \\
		&\dis \frac{1}{\eps}(\partial_x\phi\partial_{v_1}\partial^2_x g,\textbf{w}^2(2,0)e^{\frac{1}{2}\eps^{-1}\widetilde{\phi}}\partial^2_x g)
        +\frac{2}{\eps}(\partial^2_x\phi\partial_{v_1}\partial_x g,\textbf{w}^2(2,0)
        e^{\frac{1}{2}\eps^{-1}\widetilde{\phi}}\partial^2_x g)
        \\
       & +\frac{1}{\eps}(\partial^3_x\phi\partial_{v_1}g,\textbf{w}^2(2,0)e^{\frac{1}{2}\eps^{-1}\widetilde{\phi}}\partial^2_x g),\  \al=2.
	\end{array} \right.
\end{align*}
Applying \eqref{Gpavw} and Cauchy-Schwarz inequality, one has
\begin{align*}
	&|\frac{1}{\eps}(\partial_x\phi\partial_{v_1}\partial^\alpha_\beta g,\textbf{w}^2(\alpha,\beta)e^{\frac{1}{2}\eps^{-1}\widetilde{\phi}}\partial_{\beta}^\alpha g)|=|\frac{1}{2\eps}(\partial_x\phi\, e^{\frac{1}{2}\eps^{-1}\widetilde{\phi}} [\pa_{v_1} \textbf{w}^2(\al,\be)]\partial_{\beta}^\alpha g,\partial_{\beta}^\alpha g)|\notag\\
	\leq& C\frac{1}{\eps}\int_{\mathbb R}\int_{{\mathbb R}^3}|\partial_x\phi  \textbf{w}^2(\al,\be)
	(\partial_{\beta}^\alpha g)^2|\,dv\,dx
	\notag\\	
	\leq& C\frac{1}{\eps}\int_{\mathbb R}\{\frac{\ka}{\de}\int_{{\mathbb R}^3}\langle v\rangle^{-1}\textbf{w}^2(\al,\be) |\partial_{\beta}^\alpha g|^2\,dv+C_\ka
	\de|\partial_x\phi|^2\int_{{\mathbb R}^3}\langle v\rangle\textbf{w}^2(\al,\be) |\partial_{\beta}^\alpha g|^2\,dv\}\,dx	
	\notag\\	
	\leq& C\ka\frac{1}{\eps\delta}\| \textbf{w}(\al,\be)\partial_{\beta}^\alpha g\|_\sigma^2
	+C_\ka\frac{\de}{\eps}\|\partial_x\widetilde{\phi}\|^{2}_{L^{\infty}}\|\langle v\rangle^{1/2} \textbf{w}(\al,\be)\partial_{\beta}^\alpha g\|^2
	\notag\\
	\leq& C\ka\frac{1}{\eps\delta}\| \partial_{\beta}^\alpha g\|_{\sigma,\textbf{w}}^2
	+C_\ka\eps^2 q_3(t)\|\langle v\rangle^{1/2}\partial_{\beta}^\alpha g\|_{\textbf{w}}^2.	
\end{align*}
For the rest terms, similar arguments as in \eqref{I32} and \eqref{2I312} show that
\begin{align*}
&|\frac{1}{\eps}(\partial^2_x\phi\partial_{v_1}\partial^\beta_v g,\textbf{w}^2(\alpha,\beta)e^{\frac{1}{2}\eps^{-1}\widetilde{\phi}}\partial_{\beta}^\alpha g)|
\leq	C(\eps+\de)\CD(t),
\end{align*}
and
\begin{align*}
	&|2\frac{1}{\eps}(\partial^2_x\phi\partial_{v_1}\partial_x g,\textbf{w}^2(2,0)e^{\frac{1}{2}\eps^{-1}\widetilde{\phi}}\partial^2_x g)+\frac{1}{\eps}(\partial^3_x\phi\partial_{v_1}g,\textbf{w}^2(2,0)e^{\eps^{-1}\widetilde{\phi}}\partial^2_x g)|\leq	C\frac{1}{\eps^2\de^2}(\eps+\de)\CD(t).
\end{align*}
It follows from the above three inequalities that
\begin{align}\label{GI3}
	|\CI_3|\leq C\ka\frac{1}{\eps\delta}\| \pa^\al_\be g\|_{\sigma,\textbf{w}}^2
	+C_\ka\eps^2 q_3(t)\|\langle v\rangle^{1/2}\pa^\al_\be g\|_{\textbf{w}}^2+C(\eps+\de)\CD(t),
\end{align}
for $\be=0$, $\al<2$ or $\alpha+|\beta|\leq 2$, $|\beta|\geq1$, and
\begin{align}\label{G2I3}
	|\CI_3|\leq C\ka\frac{1}{\eps\de}\| \partial^2_x g\|_{\sigma,\textbf{w}}^2
	+C_\ka\eps q_3(t)\|\langle v\rangle^{1/2} \partial^2_x g\|_\textbf{w}^2+C\frac{1}{\eps^2\de^2}(\eps+\de)\CD(t).
\end{align}
Now we turn to $\CI_4$ which has the similar form as in \eqref{reI4} and \eqref{re2I4} by changing $w$ into $\textbf{w}$.
Note that in each line of \eqref{reI4} after replacing $w^2$ by $\textbf{w}^2 e^{\eps^{-1}\widetilde{\phi}}$, we keep the first term and others can be controlled using similar arguments as in \eqref{I32} and \eqref{2I312}, which leads to
\begin{align}\label{GI4}
	\CI_4\leq \frac{1}{\eps}(\frac{v_1}{4} \partial_x\widetilde{\phi}\pa^\al_\be g,\textbf{w}^2(\al,\be)e^{\frac{1}{2}\eps^{-1}\widetilde{\phi}}\partial^\al_\be g)
	+C(\eps+\de)\CD(t),
\end{align}
for $\be=0$, $\al<2$ or $\alpha+|\beta|\leq 2$, $|\beta|\geq1$, and
\begin{align}\label{G2I4}
	\CI_4\leq \frac{1}{\eps}(\frac{v_1}{4} \partial_x\widetilde{\phi}\pa^\al_x g,\textbf{w}^2(\al,0)
    e^{\frac{1}{2}\eps^{-1}\widetilde{\phi}}\partial^\al_x g)
	+C\frac{1}{\eps^2\de^2}(\eps+\de)\CD(t).
\end{align}
The combination of \eqref{Gphif1}, \eqref{GI3}, \eqref{G2I3}, \eqref{GI4} and \eqref{G2I4} shows
\begin{align}\label{Gphif}
	-\frac{1}{\eps}(\partial_\be^{\alpha} [\frac{\partial_x\phi\partial_{v_1}(\sqrt{\mu}g)}{\sqrt{\mu}}],\textbf{w}^2(\al,\be)e^{\frac{1}{2}\eps^{-1}\widetilde{\phi}}\partial_\be^{\alpha} g)
	&\geq \frac{1}{\eps}(\frac{v_1}{4} \partial_x\widetilde{\phi}\pa^\al_\be g,\textbf{w}^2(\al,\be)e^{\frac{1}{2}\eps^{-1}\widetilde{\phi}}\partial^\al_\be g)-C\ka\frac{1}{\eps\delta}\| \pa^\al_\be g\|_{\sigma,\textbf{w}}^2\notag\\
	&\quad
	-C_\ka\eps^2 q_3(t)\|\langle v\rangle^{1/2} \pa^\al_\be g\|_\textbf{w}^2
	-C(\eps+\de)\CD(t),
\end{align}
for $\be=0$, $\al<2$ or $\alpha+|\beta|\leq 2$, $|\beta|\geq1$, and
\begin{align}\label{G2phif}
	-\frac{1}{\eps}(\pa^2_x [\frac{\partial_x\phi\partial_{v_1}(\sqrt{\mu}g)}{\sqrt{\mu}}],\textbf{w}^2(2,0)
    e^{\frac{1}{2}\eps^{-1}\widetilde{\phi}}\pa^2_x g)
		&\geq \frac{1}{\eps}(\frac{v_1}{4} \partial_x\widetilde{\phi}\pa^\al_x g,\textbf{w}^2(2,0)e^{\frac{1}{2}\eps^{-1}\widetilde{\phi}}\partial^\al_x g)-C\ka\frac{1}{\eps\delta}\| \pa^2_x g\|_{\sigma,\textbf{w}}^2\notag\\
	&\quad
	-C_\ka\eps q_3(t)\|\langle v\rangle^{1/2} \pa^2_x g\|_\textbf{w}^2
	-C\frac{1}{\eps^2\de^2}(\eps+\de)\CD(t).
\end{align}
All other terms in \eqref{Gipg} can be bounded similarly as in the proof of Lemma \ref{LemmaL}, \eqref{Gagg}, \eqref{pabarM}, \eqref{phifg}, \eqref{2Gagg} and \eqref{2GaMf}. We omit the details to get 
\begin{equation}\label{GLg}
	-\frac{1}{\eps\de}(\partial^\alpha_x L_D g,\textbf{w}^2(\alpha,0)e^{\frac{1}{2}\eps^{-1}\widetilde{\phi}}\partial^\alpha_x g)
	\geq c\frac{1}{\eps\de}\|\partial^\alpha_x g\|^2_{\sigma,\textbf{w}}-C\frac{1}{\eps\de}\sum_{\al'<\al}\|\partial^{\al'}_x g\|^2_{\sigma,\textbf{w}}-C\eps\CD(t),
\end{equation}
for $\al\leq1$,
\begin{align}\label{GpaLg}
	-\frac{1}{\eps\de}(\partial^\alpha_\beta L_D g,\textbf{w}^2(\alpha,\beta)e^{\frac{1}{2}\eps^{-1}\widetilde{\phi}}\partial^\alpha_\beta g)
	\geq& \frac{1}{\eps\de}\big(c\|\partial^\alpha_\beta g\|^2_{\sigma,\textbf{w}}-\ka_1\sum_{|\beta'|=|\beta|}\|\partial^\alpha_{\beta'} g\|_{\sigma,\textbf{w}}^2-C_{\ka_1}\sum_{|\beta'|<|\beta|}\|\partial^\alpha_{\beta'}g\|_{\sigma,\textbf{w}}^2)
    \notag\\
	&-C\frac{1}{\eps\de}\sum_{\al'<\al,|\beta'|\leq|\beta|}\|\partial^{\alpha'}_{\beta'}g\|_{\sigma,\textbf{w}}^2-C\eps\CD(t)\big),
\end{align}
for $\al+|\be|\leq2$, $|\be|\geq 1$ and some $0<\ka_1<1$ that will be chosen later, and
\begin{equation}\label{G2Lg}
	-\frac{1}{\eps\de}(\partial^2_x L_D g,\textbf{w}^2(2,0)e^{\frac{1}{2}\eps^{-1}\widetilde{\phi}}\partial^\alpha_x g)
	\geq c\frac{1}{\eps\de}\|\partial^2_x g\|^2_{\sigma,\textbf{w}}-C\frac{1}{\eps^2\de^2}\eps\CD(t),
\end{equation}
as well as
\begin{align}\label{GRe}
&\frac{1}{\eps}(\partial^{\alpha}_\be[\frac{\partial_x\phi\partial_{v_{1}}(\sqrt{\overline{M}}f)}{\sqrt{\mu}}],\textbf{w}^2(\al,\be)e^{\frac{1}{2}\eps^{-1}\widetilde{\phi}}\partial^{\alpha}_\be g)+\frac{1}{\eps\de}(\partial^{\alpha}_\be \Gamma(g+\frac{\sqrt{\overline{M}}}{\sqrt{\mu}}f,g+\frac{\sqrt{\overline{M}}}{\sqrt{\mu}}f),\textbf{w}^2(\al,\be)e^{\frac{1}{2}\eps^{-1}\widetilde{\phi}}\partial^{\alpha}_\be g)\notag\\
	&
	+\frac{1}{\eps\de}(\pa^2_x\Gamma(\frac{M-\overline{M}}{\sqrt{\mu}},\frac{\sqrt{\overline{M}}}{\sqrt{\mu}}f)+\partial^{\alpha}_\be\Gamma(\frac{\sqrt{\overline{M}}}{\sqrt{\mu}}f,\frac{M-\overline{M}}{\sqrt{\mu}}),\textbf{w}^2(\al,\be)e^{\frac{1}{2}\eps^{-1}\widetilde{\phi}}\partial^{\alpha}_\be g)\notag\\
&\leq 
C(\eps+\de)\CD(t),
\end{align}
for $\be=0$, $\al<2$ or $\alpha+|\beta|\leq 2$, $|\beta|\geq1$, and
\begin{align}\label{G2Re}
	&\frac{1}{\eps}(\pa^2_x[\frac{\partial_x\phi\partial_{v_{1}}(\sqrt{\overline{M}}f)}{\sqrt{\mu}}],\textbf{w}^2(2,0)e^{\frac{1}{2}\eps^{-1}\widetilde{\phi}}\pa^2_x g)+\frac{1}{\eps\de}(\pa^2_x \Gamma(g+\frac{\sqrt{\overline{M}}}{\sqrt{\mu}}f,g+\frac{\sqrt{\overline{M}}}{\sqrt{\mu}}f),\textbf{w}^2(2,0)e^{\frac{1}{2}\eps^{-1}\widetilde{\phi}}\pa^2_x g)\notag\\
&
	+\frac{1}{\eps\de}(\partial^2_x\Gamma(\frac{M-\overline{M}}{\sqrt{\mu}},\frac{\sqrt{\overline{M}}}{\sqrt{\mu}}f)+\pa^2_x\Gamma(\frac{\sqrt{\overline{M}}}{\sqrt{\mu}}f,\frac{M-\overline{M}}{\sqrt{\mu}}),\textbf{w}^2(2,0)e^{\frac{1}{2}\eps^{-1}\widetilde{\phi}}\pa^2_x g)
    \notag\\
&\leq C\frac{1}{\eps^2\de^2}(\eps+\de)\CD(t).
\end{align}
Hence, the desired  estimate \eqref{Gwalbeg} follows from \eqref{Gipg}, \eqref{G1gtime}, \eqref{G1gtran}, \eqref{Gphif}, \eqref{GLg}, \eqref{GpaLg} and \eqref{GRe}, where we have used the smallness of $\ka$
and $\ka_1$. Correspondingly, \eqref{Gw2g} holds from \eqref{Gipg}, \eqref{G1gtime}, \eqref{G1gtran0}, \eqref{G2phif}, \eqref{G2Lg} and \eqref{G2Re}. This then completes the proof of Lemma \ref{Gleg}.
\end{proof}

Similarly to \eqref{5.3A}, the combination of \eqref{Gf} and \eqref{Gwalbeg} gives the following lemma.
\begin{lemma}
	Under the a priori assumption \eqref{Gapriori}, it holds that
\begin{align}\label{Glowfg}
	&\sum_{\alpha\leq 1}(\|\partial^{\alpha}_x f(t)\|_W^{2}+\|\partial^{\alpha}_x g(t)\|_\textbf{w}^{2})+\sum_{\alpha+|\beta|\leq 2,|\beta|\geq 1}(\|\partial^{\alpha}_\be f(t)\|_W^{2}+\|\partial^{\alpha}_\be g(t)\|_\textbf{w}^{2})
    \notag\\
	&+\frac{1}{\eps\de}\int^t_0\big\{\sum_{\alpha\leq 1}(\|\partial^{\alpha}_x f(s)\|_{\sigma,W}^{2}+\|\partial^{\alpha}_x g(s)\|_{\sigma,\textbf{w}}^{2})+\sum_{\al+|\beta|\leq 2,|\beta|\geq 1}(\|\partial^{\alpha}_\be f(s)\|_{\sigma,W}^{2}+\|\partial^{\alpha}_\be g(s)\|_{\sigma,\textbf{w}}^{2})\big\}ds\notag\\
	&+\int^t_0(q_{2}-C) q_3(s)\big\{\sum_{\alpha\leq 1}\|\langle v\rangle^{\frac{1}{2}} \pa^\al_x g(s)\|_\textbf{w}^2+\sum_{\al+|\beta|\leq 2,|\beta|\geq 1}\|\langle v\rangle^{\frac{1}{2}} \pa^\al_\be g(s)\|_\textbf{w}^2\big\}\,ds
    \notag\\
	&\leq C\CE(0)+C(\eps+\de)\int^t_0\mathcal{D}(s)\,ds+C\eps\de\int^t_0\{\|\pa^2_xf(s)\|_{\sigma,W}^{2}+\|\pa^2_xg(s)\|_{\sigma,\textbf{w}}^{2}\}\,ds
    \notag\\
    &\quad +C\eps\de  \sum_{\alpha\leq 1}\int^t_0\|\pa^{\alpha}_x\pa_x(\widetilde{u},\widetilde{\theta})(s)\|^{2}\,ds.
\end{align}
\end{lemma}
Using the arguments as in Lemma \ref{highE}, the highest order energy estimates  can be obtained by \eqref{G2ndwdiss}, \eqref{G2ndenergy} and \eqref{Gw2g}.
\begin{lemma}
Under the a priori assumption \eqref{Gapriori}, it holds that
\begin{align} \label{GhighE}   
&\frac{1}{\ka_2}\eps^2\de^2\{\|\pa^2_x(\widetilde{\rho},\widetilde{u},\widetilde{\theta},\widetilde{\phi})(t)\|^{2}+\|\pa^2_xf(t)\|^{2}+\eps\|\pa^3_x\widetilde{\phi}(t)\|^{2}\} +\eps^2\de^2\|\pa^2_xf(t)\|^2_W  
\notag\\  &
+\frac{1}{\ka_3\ka_2}\eps^2\de^2\|\partial^{2}_x g(t)\|_\textbf{w}^{2}+	\eps\de\int^t_0\|\pa^2_xf(s)\|^2_{\si,W}\,ds +\frac{1}{\ka_2}\eps\de\int^t_0\|\pa^2_x\mathbf{P}_1f(s)\|^2_{\si}\,ds  
\notag\\	 
&+\frac{1}{\ka_3\ka_2}\eps\de\int_0^{t}\|\partial^{2}_x g(s)\|_{\sigma,\textbf{w}}^{2}\,ds  
+\frac{1}{\ka_3\ka_2}\eps^2\de^2 q_{2}\int_0^{t}q_3(s)\|\langle v\rangle^{\frac{1}{2}} \partial^{2}_xg(s)\|_{\textbf{w}}^{2}\,ds  
\notag\\ 		   
\leq& C(\ka_3\frac{1}{\ka_2}+\ka_2)\eps\de\int^t_0\|\pa^2_x(\widetilde{\rho},\widetilde{u},\widetilde{\theta})(s)\|^2\,ds+C\frac{1}{\ka_3\ka_2}\CE(0)  +C\frac{1}{\ka_3\ka_2}\eps^5      
\notag\\           
&     +C\frac{1}{\ka_3\ka_2}\eps^2\de^2\int_0^{t}q_3(s)\|\langle v\rangle^{\frac{1}{2}} \pa^2_x g(s)\|_\textbf{w}^2\,ds+	C\frac{1}{\ka_2}(\ka+C_{\ka}\frac{1}{\ka_3}(\de+\eps+\frac{\de^2}{\eps}))\int^t_0\CD(s)\,ds.	  
\end{align}	 

\end{lemma}

In the previous section, the extra $t\eps^4$ in the zero order fluid energy estimates causes increase in $t$, see Lemma \ref{lezerofluid}. We carry out the entropy-entropy flux pair method, which brings good cancellations, to avoid such terms.

\begin{lemma}\label{le0fluid}
Under the a priori assumption \eqref{Gapriori},	it holds that
\begin{align}\label{Gzerofluid}
&\|(\widetilde{\rho},\widetilde{u},\widetilde{\theta},\widetilde{\phi})(t)\|^{2}+\eps\|\pa_x\widetilde{\phi}(t)\|^{2}
+\eps\de\int^{t}_0\{
\|\pa_x(\widetilde{\rho},\widetilde{u},\widetilde{\theta},\widetilde{\phi})(s)\|^{2}+\eps\|\pa^2_x\widetilde{\phi}(s)\|^{2}\} \,ds
\notag\\
&\leq C\CE(0)+C(\de+\eps+\frac{\de}{\sqrt\eps})\int^t_0\mathcal{D}(s)\,ds
+C\eps^5.
\end{align}
\end{lemma}
\begin{proof}
	Motivated by \cite{Liu-Yang-Yu}, we define the macroscopic entropy $S$ by
	\begin{equation*}
		-\frac{3}{2}\rho S:=\int_{\mathbb{R}^{3}}M\ln M\,dv,
	\end{equation*}
	which yields
	\begin{equation*}
		S=-\frac{2}{3}\ln\rho+\ln(2\pi K\theta)+1,\quad
		p=K\rho\theta=\frac{1}{2\pi \theta}\rho^{\frac{5}{3}}\exp (S).
	\end{equation*}
Multiplying \eqref{rVPL} by $\ln M$  and integrating over $v$, then making a direct calculation, it holds that
\begin{align}\label{6.33a}
-\frac{3}{2}\pa_t(\rho S)+\frac{1}{\eps}\frac{3}{2}\pa_x(\rho S)
-\frac{3}{2}\pa_x(\rho u_1 S)
=
-\int_{\mathbb{R}^{3}} v_1\pa_xG\ln M\,dv.
\end{align}
Here we have used
\begin{equation*}
\int_{\mathbb{R}^{3}}\pa_x\phi\partial_{v_{1}}F \ln M\,dv=\frac{\pa_x\phi}{K\theta}\int_{\mathbb{R}^{3}}F(v_1-u_1)\,dv
=\frac{\pa_x\phi}{K\theta}(\rho u_1-u_1\rho)=0.
	\end{equation*}
	We denote
	\begin{align*}
		m:=&(m_{0},m_{1},m_{2},m_{3},m_{4})^{t}=(\rho,\rho u_{1},\rho u_{2},\rho u_{3},\rho(\theta+\frac{|u|^{2}}{2}))^{t},
		\\
		n:=&(n_{0},n_{1},n_{2},n_{3},n_{4})^{t}
		=(\rho u_1,\rho u^2_{1}+p,\rho u_1u_{2},\rho u_1u_{3},\rho u_1(\theta+\frac{|u|^{2}}{2})+pu_1)^{t},
	\end{align*}
	where $(\cdot,\cdot,\cdot)^{t}$ is the transpose of the vector $(\cdot,\cdot,\cdot)$.
	Then we rewrite the conservation laws \eqref{NSmacro} as
	\begin{equation*}
		\pa_t m-\frac{1}{\eps}\pa_x m+\pa_x n=\begin{pmatrix}
			0
			\\
			-\frac{1}{\eps}\rho\pa_x\phi+\frac{4}{3}\eps\de\partial_x(\mu(\theta)\partial_xu_{1})
			-\int_{\mathbb{R}^{3}} v_{1}^2\pa_xL^{-1}_{M}\Theta\,dv
			\\
			\eps\de\partial_x(\mu(\theta)\partial_xu_{2})
			-\int_{\mathbb{R}^{3}} v_{2}v_1\pa_xL^{-1}_{M}\Theta\,dv
			\\
		\eps\de\partial_x(\mu(\theta)\partial_xu_{3})
			-\int_{\mathbb{R}^{3}} v_{3}v_1\pa_xL^{-1}_{M}\Theta\,dv
			\\
			-\frac{1}{\eps}\rho u_1\pa_x\phi
            +D-\int_{\mathbb{R}^{3}} \frac{1}{2}|v|^{2}v_1\pa_xL^{-1}_{M}\Theta\,dv
		\end{pmatrix}.
	\end{equation*}
Here we have denoted that
$$
D=\eps\de\pa_x[\kappa(\theta)\pa_x\theta+\mu(\theta)(\frac{4}{3}u_{1}\partial_xu_{1}+u_{2}\partial_xu_{2}+u_{3}\partial_xu_{3})].  
$$
Define a relative entropy-entropy flux pair $(\eta,q)(t,x)$ around the global Maxwellian $\overline{M}=M_{[\bar{\rho},\bar{u},\bar{S}]}$ 
with $\bar{S}=-\frac{2}{3}\ln\bar{\rho}+\ln(2\pi K\bar{\theta})+1$ to be
	\begin{equation*}
		\left\{
		\begin{array}{rl}
			&\eta(t,x)=\bar{\theta}\{-\frac{3}{2}\rho S+\frac{3}{2}\bar{\rho}\bar{S}+\frac{3}{2}\nabla_{m}(\rho S)|_{m=\bar{m}}(m-\bar{m})\},
			\\
			&q(t,x)=\bar{\theta}\{-\frac{3}{2}\rho u_1S
            +\frac{1}{\eps}\frac{3}{2}\rho S-\frac{1}{\eps}\frac{3}{2}\bar{\rho} \bar{S}
            +\frac{3}{2}\bar{\rho}\bar{u}_1\bar{S}+\frac{3}{2}\nabla_{m}(\rho S)|_{m=\bar{m}}\cdot[(n-\frac{1}{\eps}m)-(\bar{n}-\frac{1}{\eps}\bar{m})]\},
		\end{array} \right.
	\end{equation*}
where $\bar{m}=(\bar{\rho},\bar{\rho}\bar{ u}_{1},\bar{\rho}\bar{ u}_{2},\bar{\rho}\bar{ u}_{3},\bar{\rho}(\bar{\theta}
	+\frac{1}{2}|\bar{u}|^{2}))^{t}$. Here 
$(\bar{\rho},\bar{u},\bar{\theta},\bar{\phi})=(1,0,\frac{3}{2},0)$.
It is straightforward to check that
\begin{equation*}
\pa_{m_0}(\rho S)=S+\frac{|u|^{2}}{2\theta}-\frac{5}{3}, \quad
\pa_{m_i}(\rho S)=-\frac{u_{i}}{\theta}, \quad \mbox{i=1,2,3}, \quad \pa_{m_4}(\rho S)=\frac{1}{\theta}.
\end{equation*}
A direct calculation gives
\begin{align*}
			\eta(t,x)&=\frac{3}{2}\{\rho\theta-\bar{\theta}\rho S+\rho[(\bar{S}-\frac{5}{3})\bar{\theta}
			+\frac{|u-\bar{u}|^{2}}{2}]+\frac{2}{3}\bar{\rho}\bar{\theta}\}
			\notag\\
			&=\rho\bar{\theta}\Psi(\frac{\bar{\rho}}{\rho})+\frac{3}{2}\rho\bar{\theta}\Psi(\frac{\theta}{\bar{\theta}})
			+\frac{3}{4}\rho|u-\bar{u}|^{2},
	\end{align*}
where $\Psi(s)=s-\ln s-1$ is a strictly convex function around $s=1$. Thus, for $\bf{X}$ in any closed
	bounded region in $\sum=\{{\bf{X}}:\rho>0,~~\theta>0\}$,  there exists a constant $C>1$ such that
	\begin{equation}
		\label{6.33A}
		C^{-1}|(\widetilde{\rho},\widetilde{u},\widetilde{\theta})|^{2}\leq \eta(t,x)\leq C|(\widetilde{\rho},\widetilde{u},\widetilde{\theta})|^{2}.
	\end{equation}
The definition of $(\eta,q)(t,x)$  implies that
\begin{align}
\label{6.34A} 			
\partial_{t}\eta(t,x)+\pa_x q(t,x)=&
\bar{\theta}\{-\partial_{t}(\frac{3}{2}\rho S)-
\pa_x(\frac{3}{2}\rho u_1S)            +\frac{1}{\eps}\pa_x(\frac{3}{2}\rho S)+\frac{3}{2}\nabla_{m}(\rho S)|_{m=\bar{m}}\cdot(\partial_{t}m-\frac{1}{\eps}\pa_x m+\pa_x n)\}
\notag\\
=& -\frac{3}{2}\int_{\mathbb{R}^{3}} v_1\pa_xG\ln M\,dv
-\frac{3}{2}\frac{1}{\eps}\rho u_1\pa_x\phi+\frac{3}{2}D-\frac{3}{2}\int_{\mathbb{R}^{3}} \frac{1}{2}|v|^{2}v_1\pa_xL^{-1}_{M}\Theta\,dv
\end{align}
Here we have used \eqref{6.33a} and  $\frac{3}{2}\nabla_{m}(\rho S)|_{m=\bar{m}}=\frac{3}{2}(\ln(2\pi)-\frac{2}{3},0,0,0,\frac{2}{3})$.
Note that
\begin{equation}\label{reG1} 		
G=\eps\de L^{-1}_{M}[P_{1}(v_{1}\partial_xM)]+L^{-1}_{M}\Theta, 	\end{equation} 	with 	
\begin{equation}\label{DefTheta1} 		
\Theta:=\eps\de\partial_tG-\de\partial_xG+\eps\de P_{1}(v_{1}\partial_xG)-\de\partial_x\phi\partial_{v_{1}}G-Q(G,G). 	\end{equation}
By \eqref{reG1} and \eqref{DefTheta1}, we have from a direct calculation that
\begin{align} 
\label{6.38A} 			 
-\frac{3}{2}\int_{\mathbb{R}^{3}} v_1\pa_xG\ln M\,dv
&=-\frac{3}{2}\int_{\mathbb{R}^{3}} v_1\pa_xG[\ln\rho-\frac{3}{2}\ln(2\pi K\theta)-\frac{|v-u|^2}{2K\theta}]\,dv
\notag\\
&=\frac{3}{2}\int_{\mathbb{R}^{3}} v_1\pa_xG\frac{|v-u|^2}{2K\theta}\,dv
=\frac{3}{2}\frac{1}{2K\theta}\int_{\mathbb{R}^{3}} v_1\pa_xG(v^2-2uv)\,dv
\notag\\ 
=\frac{3}{2}\frac{1}{K\theta}\int_{\mathbb{R}^{3}} &v_1(\frac{1}{2}|v|^2-uv)\pa_xL^{-1}_{M}\Theta\,dv+\frac{3}{2}\frac{1}{2K\theta}\eps\de\int_{\mathbb{R}^{3}} v_1(|v|^2-2uv)\pa_xL^{-1}_{M}[P_{1}(v_{1}\partial_xM)]\,dv
\notag\\  
=&\frac{3}{2}\frac{1}{K\theta}\int_{\mathbb{R}^{3}} v_1(\frac{|v|^2}{2}-uv)\pa_xL^{-1}_{M}\Theta\,dv
\notag\\ 
&-\frac{3}{2}\frac{1}{K\theta}\eps\de[\partial_x(\kappa(\theta)\partial_x\theta)+\partial_x(\mu(\theta)\{\frac{4}{3}u_{1}\partial_xu_{1}+u_{2}\partial_xu_{2}+u_{3}\partial_xu_{3}\})]
\notag\\ 
&+\frac{3}{2}\frac{1}{K\theta}\eps\de
[\frac{4}{3}u_1\partial_x(\mu(\theta)\partial_xu_{1})+u_2\partial_x(\mu(\theta)\partial_xu_{2})+u_3\partial_x(\mu(\theta)\partial_xu_{3})],
\end{align}
where in the last identity we have used \eqref{id1} and \eqref{id2}.

Plugging \eqref{6.34A} into \eqref{6.38A} gives
\begin{align*}	
		&\partial_{t}\eta(t,x)+\pa_x q(t,x)+\eps\de\frac{3\bar{\theta}}{2\theta}\mu(\theta)(\frac{4}{3}(\partial_xu_{1})^2+(\partial_xu_{2})^2+(\partial_xu_{3})^2)
		+\eps\de\frac{3\bar{\theta}}{2\theta^{2}}\kappa(\theta)|\pa_x\widetilde{\theta}|^{2}+\frac{3}{2}\frac{1}{\eps}\rho\widetilde{u}_1\pa_x\widetilde{\phi}
		\nonumber\\
		=&\frac{3}{2}\eps\de\pa_x[\mu(\theta)(\frac{4}{3}u_{1}\partial_xu_{1}+u_{2}\partial_xu_{2}+u_{3}\partial_xu_{3})] +\frac{3}{2}\eps\de\pa_x(\frac{\widetilde{\theta}}{\theta}\kappa(\theta)\pa_x\theta)
		\nonumber\\
		&-\frac{3}{2}\pa_x(\frac{\widetilde{\theta}}{\theta}\int_{\mathbb{R}^{3}}(\frac{1}{2}|v|^{2}-u\cdot v)v_1L^{-1}_{M}\Theta\, dv)
		-\frac{3}{2}\pa_x(\sum^{3}_{i=1}\widetilde{u}_{i}\int_{\mathbb{R}^{3}} v_{i}v_1L^{-1}_{M}\Theta\,dv)
		\nonumber\\
		&+\frac{3}{2}\pa_x(\frac{\widetilde{\theta}}{\theta})\int_{\mathbb{R}^{3}}(\frac{1}{2}|v|^{2}-v\cdot u)v_1L^{-1}_{M}\Theta\, dv
		+\frac{9}{4}\frac{1}{\theta}\pa_x\widetilde{u}\cdot\int_{\mathbb{R}^{3}} vv_1L^{-1}_{M}\Theta\,dv
	\end{align*}
Then taking integration on $x$ gives
\begin{align}
	\label{eqeta}
&\frac{d}{dt}\int_{\R}\eta(t,x)\,dx+c\eps\de\|\pa_x(\widetilde{u},\widetilde{\theta})(t)\|^{2}+\frac{3}{2}\frac{1}{\eps}\int_{\R}\rho\widetilde{u}_1\pa_x\widetilde{\phi}\,dx
	\nonumber\\
\leq &\frac{3}{2}\int_{\R}\pa_x(\frac{\widetilde{\theta}}{\theta})\int_{\mathbb{R}^{3}}(\frac{1}{2}|v|^{2}-v\cdot u)v_1L^{-1}_{M}\Theta\, dv\,dx
	+\frac{3}{2}\int_{\R}\frac{\bar{\theta}}{\theta}\pa_x\widetilde{u}\cdot\int_{\mathbb{R}^{3}} vv_1L^{-1}_{M}\Theta\,dv\, dx.
\end{align}
For the last term on the left hand side of \eqref{eqeta}, we use the first and last equations of \eqref{Gperturbeq} to get
\begin{align*}
	\frac{1}{\eps}\int_{\R}\rho\widetilde{u}_1\pa_x\widetilde{\phi}\,dx&=\frac{1}{\eps}\int_{\R}\pa_t\widetilde{\rho}\widetilde{\phi}\,dx=\frac{1}{\eps}\int_{\R}(-\eps^2\pa^2_x\widetilde{\phi}+\eps\widetilde{\phi})\widetilde{\phi}\,dx=\eps\|\pa_x\widetilde{\phi}\|^2+\|\widetilde{\phi}\|^2.
\end{align*}
The right hand side of \eqref{eqeta} can be bounded using similar arguments as how we obtain \eqref{1stutheta} in Lemma \ref{le1stu} that
\begin{align*}
	&\Big|\frac{3}{2}\int_{\R}\pa_x(\frac{\widetilde{\theta}}{\theta})\int_{\mathbb{R}^{3}}(\frac{1}{2}|v|^{2}-v\cdot u)v_1L^{-1}_{M}\Theta\, dv\,dx
	+\frac{9}{4}\int_{\R}\frac{1}{\theta}\pa_x\widetilde{u}\cdot\int_{\mathbb{R}^{3}} vv_1L^{-1}_{M}\Theta\,dv\, dx\Big|
    \notag\\
=&\Big|\frac{3}{2}\int_{\R}\pa_x(\frac{\widetilde{\theta}}{\theta})
(K\theta)^{\frac{3}{2}}\int_{\mathbb{R}^{3}}A_{1}(\frac{v-u}{\sqrt{K\theta}})\frac{\Theta}{M}\,dv\,dx 	+\frac{3}{2}\sum^{3}_{j=1}\int_{\R}\pa_x\widetilde{u}_j\int_{\mathbb{R}^{3}}B_{1j}(\frac{v-u}{\sqrt{K\theta}})\frac{\Theta}{M}\,dv\, dx\Big|    
\notag\\
	\leq &\eps\de\frac{d}{dt}\int_{\mathbb{R}}\int_{\mathbb{R}^{3}}
	\big\{\frac{3}{2}\pa_x(\frac{\widetilde{\theta}}{\theta})
(K\theta)^{\frac{3}{2}}A_{1}(\frac{v-u}{\sqrt{K\theta}})\frac{h}{M}
   +\frac{3}{2}\sum^{3}_{j=1}\pa_x\widetilde{u}_jB_{1j}(\frac{v-u}{\sqrt{K\theta}})\frac{h}{M}\, 
    \big\}\,dv\,dx
    \notag\\
	&+C\ka\eps\de\|\pa_x(\widetilde{\rho},     \widetilde{u},\widetilde{\theta})\|^{2}+C_\ka(\de+\eps+\frac{\de}{\sqrt\eps})\CD(t).
\end{align*}
The zero order dissipation for $\widetilde{\rho}$ and $\widetilde{\phi}$ can be obtained by \eqref{zerodissrho} and we directly let $\eta$ in \eqref{zerodissrho} be zero to get
\begin{align*}
&\eps\de\frac{d}{dt}(\widetilde{u}_1,\pa_x\widetilde{\rho})
	+c\eps\de(\|\pa_x\widetilde{\rho}\|^2+\|\pa_x\widetilde{\phi}\|^2
	+\eps\|\pa^2_x\widetilde{\phi}\|^2)
	\leq C\eps\de\|\pa_x(\widetilde{u},\widetilde{\theta})\|^2+ C\eps\CD(t).
\end{align*}
Hence, plugging the above three inequalities and \eqref{6.33A} into \eqref{eqeta}, we see that \eqref{Gzerofluid} holds by taking $\ka$ small enough. This then completes the proof of Lemma \ref{le0fluid}.
\end{proof}

 With the above lemmas, we are now in a position to prove Theorem \ref{TheoremGlobal}.
 
\begin{proof}[Proof of Theorem \ref{TheoremGlobal}]
The summation of \eqref{Gfluid} and \eqref{Gzerofluid} shows that
\begin{align}
\label{Glowfluid} 
&\sum_{\alpha\leq 1}\{\|\pa^\al_x(\widetilde{\rho},\widetilde{u},\widetilde{\theta},\widetilde{\phi})(t)\|^{2}+\eps\|\pa^\al_x\pa_x\widetilde{\phi}(t)\|^2\}+\eps\de\sum_{1\leq\alpha\leq 2}\int^t_0\{\|\pa^\al_x(\widetilde{\rho},\widetilde{u},\widetilde{\theta},\widetilde{\phi})(s)\|^{2}+\eps\|\pa^\al_x\pa_x\widetilde{\phi}(s)\|^{2}\}\,ds \notag\\  
&\leq C\CE(0)+C\eps^2\de^2\|\pa^2_x\widetilde{\rho}(t)\|^2+C\eps^2\de^2\|\pa^2_x(g,f)(t)\|^2+ C\eps\de\int_0^t\|\pa^2_x(g,f)(s)\|_\sigma^{2}\,ds 
\notag\\ 
&\quad+C(\de+\eps+\frac{\de}{\sqrt\eps})\int_0^t\CD(s)\,ds+CA\int_0^t(\frac{\eps}{\de}+\frac{\eps^2}{\de^2})\CD(s)\,ds+C\eps^5. 
\end{align}	
Similar to \eqref{5.6A}, a linear combination of \eqref{Glowfluid}, \eqref{Glowfg} and \eqref{GhighE} yields
\begin{align} 
\label{Genergy0}
&\CE(t)+\int^t_0\CD(s)\,ds
+\frac{1}{\ka_3\ka_2}\eps^2\de^2\int_0^{t}(q_2q_3(s)-Cq_3(s)) \|\langle v\rangle^{\frac{1}{2}} \pa^2_x g(s)\|_\textbf{w}^2\,ds
\notag\\ 	       
&+\int^t_0(q_2q_3(s)-Cq_3(s))\big\{\sum_{\alpha\leq 1}\|\langle v\rangle^{\frac{1}{2}} \pa^\al_x g(s)\|_\textbf{w}^2+\sum_{\al+|\beta|\leq 2,|\beta|\geq 1}\|\langle v\rangle^{\frac{1}{2}} \pa^\al_\be g(s)\|_\textbf{w}^2\big\}\,ds  
\notag\\
&\leq C\frac{1}{\ka_3\ka_2}\CE(0)   +C\frac{1}{\ka_3\ka_2}\eps^5      
+CA\int_0^t(\frac{\eps}{\de}+\frac{\eps^2}{\de^2})\CD(s)\,ds  
+	C\frac{1}{\ka_2}(\ka+C_{\ka}\frac{1}{\ka_3}(\de+\eps+\frac{\de^2}{\eps}))\int^t_0\CD(s)\,ds.              
\end{align}
We set $A=2$, then let 
\begin{align}\label{Gdeeps}
	\eps^{1-c_0}<\de<\eps^{1/2+c_0}\ \text{or}\ \de=\frac{1}{c_1}\eps\ \text{or}\ \de=c_1\eps^{1/2},
\end{align}
for any $0<c_0<1/2$ and some $0<c_1<1$ that both are independent of $\eps$ and $\de$. Here $c_1$ will be determined later. 
We should point out that \eqref{Gdeeps} is to ensure 
\eqref{5.8Aa}, \eqref{5.9Aa} and \eqref{5.10Aa} hold true.
We choose $\eps\ll \ka\ll \ka_2\ll 1$  with $\ka_3=\ka_2^2$ and
$c_1\ll 1$ to require that 
\begin{equation} 
\label{5.8Aa}
CA(\frac{\eps}{\de}+\frac{\eps^2}{\de^2})\leq \frac{1}{4}, \quad
C\frac{1}{\ka_2}(\ka+C_{\ka}\frac{1}{\ka_3}(\de+\eps+\frac{\de^2}{\eps}))\leq \frac{1}{4}. 
\end{equation} 
We take $q_2=\eps^{-\frac{1}{8}}$ to require that
\begin{equation} 
\label{5.9Aa} 
q_2q_3(s)-Cq_3(s)\geq 0
\end{equation} 
By these facts and \eqref{Genergy0}, there exits a $C_1>1$ such that
\begin{align} \label{5.10Aa} 
\CE(t)+\int^t_0\CD(s)\,ds\leq C_1\CE(0)  +C_1\eps^5,   
\end{align}
for $t\in(0,T]$ with $T<+\infty$. By taking $\CE(0)\leq \frac{1}{2C_1}\eps^{4+\frac{3}{4}}$, we have
\begin{align*}
\sup_{t\in(0,T]}\CE(t)+\int^T_0\CD(s)\,ds\leq \eps^{4+\frac{3}{4}}, \quad T\in(0,+\infty).
\end{align*}
Then the a priori assumption \eqref{Gapriori} can be closed. Hence we obtain \eqref{globalenergy}. Notice that \eqref{Gdeeps} is equivalent to \eqref{Gnueps}. Then the proof of Theorem \ref{TheoremGlobal} is complete.
\end{proof}

\section{Appendix}\label{sec.app}
In this appendix, for completeness we give the proof of Lemma \ref{lemmacor} and Lemma \ref{leL}.

\subsection{Solving the corrections}
\begin{proof}[Proof of Lemma \ref{lemmacor}]
To simplify the notation, we denote $u_{1i}$ by $U_i$.
Note that for any $s\geq 2$,
\begin{align}\label{KdVsolution}
	(\rho_1,U_0=u_{10},\theta_1,\phi_0)\in L^{\infty}(-T,T;H^{s}(\mathbb{R}))
\end{align} is determined by \eqref{1orderid} and \eqref{KdV}. For $(\rho_2,U_1,\theta_2,\phi_1)$, we represent all quantities in terms of $U_1$, then we only need to solve $U_1.$ Recall $\rho_1=U_0=\theta_1=\phi_0$ and
\begin{align}\label{U0}
\pa_tU_0+\frac{1}{2}\pa_x^3U_0+\frac{3}{2}U_0\pa_xU_0=0.
\end{align}
From \eqref{mass2}, we solve $U_1-\rho_2$ by
\begin{align}\label{u1rho2}
	U_1-\rho_2&=-\int^x (\pa_t U_0+\pa_x(U_0^2)(t,\xi))d\xi=\int^x (\frac{1}{2}\pa_x^3U_0+\frac{3}{4}\pa_x(U_0^2)-\pa_x(U_0^2))d\xi\notag\\
	&=\frac{1}{2}\pa_x^2U_0-\frac{1}{4}U_0^2,
\end{align}
where $\int^x$ denotes the antiderivative and we omit the $(t,\xi)$ in the integral from later on. From \eqref{velocity2} and \eqref{temperature2}, we get
\begin{align}\label{phi1u1}
	\phi_1-U_1=-\int^x(\pa_t U_0+U_0\pa_x U_0)d\xi=\frac{1}{2}\pa_x^2U_0+\frac{1}{4}U_0^2,
\end{align}
and
\begin{align}\label{U1theta2}
U_1-\theta_2=-\int^x(\pa_t U_0+U_0\pa_x U_0+\frac{2}{3}U_0 \pa_x U_0)d\xi=\frac{1}{2}\pa_x^2U_0-\frac{1}{12}U_0^2.
\end{align}
Hence, we only need to determine $U_1$. The Sobolev regularity of $U_1$ is the same as $\rho_2$, $\phi_1$ and $\theta_2$ if $U_0$ has some higher regularity. The summation of \eqref{mass3} and \eqref{velocity3} gives 
\begin{align*}
	\pa_t U_1+\pa_t\rho_2-\pa_x\rho_3+\pa_x \phi_2+\pa_x(\rho_2U_0+U_0 U_1)+U_0\pa_x U_1+U_1\pa_x U_0+\frac{5}{3}\pa_xU_0=0,
\end{align*}
which, combined with \eqref{Poisson3}, yields
\begin{align*}
	\pa_t U_1+\pa_t\rho_2+\pa^3_x \phi_1+\pa_x(\rho_2U_0+U_0 U_1)+\pa_x (U_0U_1)+\frac{5}{3}\pa_xU_0=0.
\end{align*}
Substituting \eqref{u1rho2} and \eqref{phi1u1} into the above equation gives
\begin{align*}
	&2\pa_t U_1-\frac{1}{2}\pa_x^2\pa_tU_0+\frac{1}{4}\pa_t(U_0^2)+\pa^3_x(U_1+\frac{1}{2}\pa_x^2U_0+\frac{1}{4}U_0^2)\notag\\
	&\qquad\qquad+\pa_x((U_1-\frac{1}{2}\pa_x^2U_0+\frac{1}{4}U_0^2)U_0+U_0 U_1)+\pa_x (U_0U_1)+\frac{5}{3}\pa_xU_0=0.
\end{align*}
Denote
\begin{align}\label{defN1}
	N_1(U_0)=\frac{1}{4}\pa_x^2\pa_tU_0-\frac{1}{4}U_0\pa_tU_0-\frac{1}{4}\pa_x^5U_0-\frac{1}{8}\pa^3_x(U_0^2)-\frac{1}{2}\pa_x((-\frac{1}{2}\pa_x^2U_0+\frac{1}{4}U_0^2)U_0)-\frac{5}{6}\pa_xU_0.
\end{align}
Then
\begin{align}\label{u1}
	&\pa_t U_1+\frac{1}{2}\pa^3_xU_1+\frac{3}{2}\pa_x(U_0U_1)=N_1(U_0).
\end{align}
Direct energy estimate shows
\begin{align*}
	\frac{d}{dt}\|U_1\|^2\leq C(\|\pa_xU_0\|_{L^\infty_x}+1)\|U_1\|^2+C\|N_1(U_0)\|^2.
\end{align*}
The $L^2$ bound of $U_1$ follows from Gronwall inequality and $\|U_1\|_{H^r_x}$ can be proved iteratively by letting $s$ in \eqref{KdVsolution} be large. Therefore, we obtain unique solutions
\begin{align*}
	(\rho_2,U_1,\theta_2,\phi_1)\in L^{\infty}(-T,T;H^{r}(\mathbb{R})).
\end{align*}
Next, we determine $\rho_3,\theta_3,\phi_2$ by $U_2$. It holds by \eqref{mass3} that
\begin{align*}
	U_2-\rho_3=-\int^x(\pa_t\rho_2+\pa_x(\rho_2U_0+\rho_1 U_1))d\xi=-\rho_2U_0-\rho_1 U_1-\int^x\pa_t\rho_2d\xi.
\end{align*}
We need a more explicit form of $\int^x\pa_t\rho_2d\xi$ to prove its $L^2$ integrability. It holds by \eqref{u1rho2} and \eqref{u1} that
\begin{align*}
	\int^x\pa_t\rho_2d\xi&=\int^x\pa_t(U_1-\frac{1}{2}\pa_x^2U_0+\frac{1}{4}U_0^2)d\xi\notag\\
	&=-\frac{1}{2}\pa_x\pa_tU_0 -\frac{1}{2}\pa^2_xU_1-\frac{3}{2}U_0U_1+\frac{1}{2}\int^x U_0\pa_tU_0d\xi+\int^x N_1(U_0)d\xi.
\end{align*}
Using \eqref{U0} and integration by parts, we have
\begin{align}\label{u0patu0}
	\int^x U_0\pa_tU_0d\xi=-\int^x U_0(\frac{1}{2}\pa_x^3U_0+\frac{3}{2}U_0\pa_xU_0)d\xi=-\frac{1}{2}U_0\pa^2_xU_0+\frac{1}{4}(\pa_xU_0)^2-\frac{1}{2}\pa_x(U_0^3),
\end{align}
which, combined with \eqref{defN1}, yields
\begin{align}\label{antiN1}
	&\int^x N_1(U_0)d\xi\notag\\
	=&\frac{1}{4}\pa_x\pa_tU_0-\frac{1}{4}\int^xU_0\pa_tU_0d\xi-\frac{1}{4}\pa_x^4U_0-\frac{1}{8}\pa^2_x(U_0^2)-\frac{1}{2}((-\frac{1}{2}\pa_x^2U_0+\frac{1}{4}U_0^2)U_0)-\frac{5}{6}U_0\notag\\
	=&\frac{1}{4}\pa_x\pa_tU_0+\frac{3}{8}U_0\pa^2_xU_0-\frac{1}{16}(\pa_xU_0)^2+\frac{1}{8}\pa_x(U_0^3)-\frac{1}{4}\pa_x^4U_0-\frac{1}{8}\pa^2_x(U_0^2)-\frac{1}{8}U_0^3-\frac{5}{6}U_0.
\end{align}
Thus, the combination of the above four identities shows
\begin{align*}
	U_2-\rho_3=&-\int^x(\pa_t\rho_2+\pa_x(\rho_2U_0+\rho_1 U_1))d\xi\notag\\
	=&-\rho_2U_0-\rho_1 U_1+\frac{1}{2}\pa_x\pa_tU_0 +\frac{1}{2}\pa^2_xU_1+\frac{3}{2}U_0U_1+\frac{1}{4}U_0\pa^2_xU_0-\frac{1}{8}(\pa_xU_0)^2+\frac{1}{4}\pa_x(U_0^3)\notag\\
	&\quad-(\frac{1}{4}\pa_x\pa_tU_0+\frac{3}{8}U_0\pa^2_xU_0-\frac{1}{16}(\pa_xU_0)^2+\frac{1}{8}\pa_x(U_0^3)-\frac{1}{4}\pa_x^4U_0-\frac{1}{8}\pa^2_x(U_0^2)-\frac{1}{8}U_0^3-\frac{5}{6}U_0)\notag\\
	=&-\rho_2U_0 +\frac{1}{2}\pa^2_xU_1+\frac{1}{2}U_0U_1+\frac{1}{4}\pa_x\pa_tU_0-\frac{1}{8}U_0\pa^2_xU_0-\frac{1}{16}(\pa_xU_0)^2+\frac{1}{8}\pa_x(U_0^3)\notag\\
	&+\frac{1}{4}\pa_x^4U_0+\frac{1}{8}\pa^2_x(U_0^2)+\frac{1}{8}U_0^3+\frac{5}{6}U_0.
\end{align*} 
Then we obtain from \eqref{velocity3}, \eqref{u1} and \eqref{antiN1} that 
\begin{align*}
	U_2-\phi_2&=U_0U_1+\frac{5}{3}U_0+\int^x\pa_t U_1d\xi\notag\\
	&=U_0U_1-\frac{1}{2}\pa^2_xU_1-\frac{3}{2}(U_0U_1)+\frac{1}{4}\pa_x\pa_tU_0+\frac{3}{8}U_0\pa^2_xU_0\notag\\
	&\quad-\frac{1}{16}(\pa_xU_0)^2+\frac{1}{8}\pa_x(U_0^3)-\frac{1}{4}\pa_x^4U_0-\frac{1}{8}\pa^2_x(U_0^2)-\frac{1}{8}U_0^3+\frac{5}{6}U_0.
\end{align*}
Moreover, \eqref{temperature3} and \eqref{U1theta2} show that
\begin{align*}
	\theta_3-U_2&=\int^x(\pa_t \theta_2+U_0\pa_x \theta_2+U_1\pa_x U_0+\frac{2}{3}\theta_2 \pa_x U_0+\frac{2}{3}U_0 \pa_x U_1)d\xi\notag\\
	&=\int^x\{\pa_t (U_1-\frac{1}{2}\pa_x^2U_0+\frac{1}{12}U_0^2)+U_0\pa_x (U_1-\frac{1}{2}\pa_x^2U_0+\frac{1}{12}U_0^2)+U_1\pa_x U_0\notag\\
	&\qquad\qquad\qquad+\frac{2}{3}(U_1-\frac{1}{2}\pa_x^2U_0+\frac{1}{12}U_0^2) \pa_x U_0+\frac{2}{3}U_0 \pa_x U_1\}d\xi\notag\\
	&=-\frac{1}{2}\pa_x\pa_tU_0+\frac{5}{3}U_0U_1+\int^x\{\pa_tU_1+\frac{1}{6}U_0\pa_tU_0-\frac{1}{2}U_0\pa_x^3U_0+\frac{2}{9}U_0^2\pa_xU_0-\frac{1}{3}\pa_x^2U_0\pa_xU_0\}d\xi.
\end{align*}
For brevity, we keep the antiderivative above without writing it explicitly. We see from the arguments in \eqref{u1}, \eqref{u0patu0}, \eqref{antiN1} and integration by parts that the antiderivative can be explicitly represented in terms of $U_0$ and is in $H^r_x.$ Then we can solve $\rho_3,\theta_3,\phi_2$ by $U_2$ and their Sobolev regularities remain the same. Then we deduce the equation of $U_0$ and prove the $H^r$ bound. The equation of $U_2$ in \eqref{velocity4} gives
\begin{align*}
	\pa_t U_2+\pa_x(U_0U_2)=&\pa_x(U_3-\phi_3)-U_1\pa_x U_1-\pa_x\rho_2-\frac{2}{3}U_0\pa_xU_0-\frac{2}{3}\pa_x\theta_2\notag\\
	:=&\pa_x(U_3-\phi_3)+N_2(U_0,U_1,\rho_2,\theta_2).
\end{align*}
Then we have the $L_2$ estimate
\begin{align*}
\frac{d}{dt}\|U_2\|^2\leq C(\|\pa_xU_0\|_{L^\infty_x}+1)\|U_2\|^2+\|\pa_x(U_3-\phi_3)\|^2+C\|N_2(U_0,U_1,\rho_2,\theta_2)\|^2.	
\end{align*}
Note that for any $r>2$, we can show $N_2(U_0,U_1,\rho_2,\theta_2)\in H^r$ by letting $s$ be large enough in \eqref{KdVsolution}. The $L^2$ bound of $U_2$ follows from Gronwall inequality and $\|U_2\|_{H^r_x}$ can be proved iteratively by choosing any $U_3,\phi_3\in H^{r+1}$. Therefore, we obtain unique solutions
\begin{align*}
	(\rho_3,U_2,\theta_3,\phi_2)\in L^{\infty}(-T,T;H^{r}(\mathbb{R})).
\end{align*}
Noticing that $U_3$ and $\phi_3$ is chosen above to be any functions that have higher Solobev regularity than $U_2$, we thus obtain \eqref{boundcor}. Lemma \ref{lemmacor} is proved.
\end{proof}

\subsection{The linear operator $\CL_{\overline{M}}$}

\begin{proof}[Proof of Lemma \ref{leL}]

We first prove \eqref{controlpaL}. Denote  $\overline{M}(v)=M_{[\bar{\rho},\bar{u},\bar{\theta}](t,x)}(v)$ and	 
\begin{align*} 		 
\sigma^{ij}_{\overline{M}}(v):=(\Phi_{ij}\ast \overline{M})(v)=\int_{{\mathbb R}^3}\Phi_{ij}(v-v_{\ast})\overline{M}(v_{\ast})\,dv_{\ast}. 	\end{align*} 	 
A direct calculation gives		 
\begin{align*} 		\sigma^{ij}_{\overline{M}}(v)&=\int_{{\mathbb R}^3}\Phi_{ij}(v-v_{\ast})M_{[\bar{\rho},\bar{u},\bar{\theta}](t,x)}(v_{\ast})\,dv_{\ast} 		=\int_{{\mathbb R}^3}\Phi_{ij}(v-v_{\ast})\frac{\bar{\rho}(t,x)}{(2\pi K\bar{\theta}(t,x))^{3/2}}\exp\big(-\frac{|v_{\ast}-\bar{u}(t,x)|^{2}}{2K\bar{\theta}(t,x)}\big)\,dv_{\ast}
\notag\\ 		
&=\frac{\bar{\rho}(t,x)}{(2\pi K\bar{\theta}(t,x))^{3/2}}\int_{{\mathbb R}^3}\Phi_{ij}(v-v_{\ast}-\bar{u}(t,x))\exp\big(-\frac{|v_{\ast}|^{2}}{2K\bar{\theta}(t,x)}\big)\,dv_{\ast}\notag\\ 		&=\frac{\bar{\rho}(t,x)}{(2\pi K\bar{\theta}(t,x))^{3/2}}\sigma^{ij}_{M_{[1,0,\bar{\theta}]}}(v-\bar{u}(t,x)). 	 
\end{align*} 	 
Furthermore, we also have
\begin{align*} 		
\sigma^{ij}_{\overline{M}}(v)&=\int_{{\mathbb R}^3}\Phi_{ij}(v_{\ast})M_{[\bar{\rho},\bar{u},\bar{\theta}](t,x)}(v-v_{\ast})\,dv_{\ast} \notag\\ 		 
&=\frac{\bar{\rho}(t,x)}{(2\pi K\bar{\theta}(t,x))^{3/2}}\int_{{\mathbb R}^3}\frac{1}{|v_*|}\big(\delta_{ij}-\frac{v_{*i}v_{*j}}{|v_*|^{2}}\big)\exp\big(-\frac{|v-v_{\ast}-\bar{u}(t,x)|^{2}}{2K\bar{\theta}(t,x)}\big)\,dv_{\ast}. 	 
\end{align*}  
Similar to \cite[Lemma 4]{Strain-Guo}, we define $v^1=\frac{v}{|v|}$ and generate the orthogonal basis $v^2$ and $v^3$ such that 	 
$$ 	 v^i\cdot v^j=\de_{ij},\quad i,j=1,2,3, $$  
and 	 
$$ 	 \bar{u}=\bar{u}\cdot v^1 v^1+\bar{u}\cdot v^2 v^2+\bar{u}\cdot v^3 v^3=\bar{u}_1v^1+\bar{u}_{2}v^2+\bar{u}_{3}v^3, 	 
$$  with $\bar{u}_{2}=\bar{u}_{3}.$  Define the orthogonal $3\times 3$ matrix as	 
$$ 
\mathcal{O}=[v^1,v^2,v^3]. 	
$$ 		
Then the orthogonal transformation gives 	 
\begin{align*} 	\sigma^{ij}_{\overline{M}}(v)&=\frac{\bar{\rho}}{(2\pi K\bar{\theta})^{3/2}}\int_{{\mathbb R}^3}\frac{1}{|v_*|}\big(\delta_{ij}-\frac{(\mathcal{O}v_{*})_i (\mathcal{O}v_{*})_j}{|v_*|^{2}}\big)\exp\big(-\frac{|v-\mathcal{O}v_{\ast}-\bar{u}|^{2}}{2K\bar{\theta}(t,x)}\big)\,dv_{\ast}.  \end{align*}
Here we have used $|\mathcal{O}v_{*}|=|v_{*}|$. Also
$$  (\mathcal{O}v_{\ast})_i=v_{\ast 1}v^1_i+v_{\ast 2}v^2_i+v_{\ast 3}v^3_i, 	
$$ 	 
and 	 
\begin{align*} 	 |v-\mathcal{O}v_{\ast}-\bar{u}|^2&=|v-v_{*1}v^1-v_{*2}v^2-v_{*3}v^3-\bar{u}_1v^1-\bar{u}_{2}v^2-\bar{u}_{2}v^3|^2
\notag\\ 	&=||v|-v_{*1}-\bar{u}_{1}|^2+|v_{*2}+\bar{u}_{2}|^2+|v_{*3}+\bar{u}_{2}|^2.  
\end{align*}	 
By symmetry, one gets 
\begin{align*} 	 
\sigma^{ij}_{\overline{M}}(v)&=\frac{\bar{\rho}}{(2\pi K\bar{\theta})^{3/2}}\int_{{\mathbb R}^3}\frac{1}{|v_*|}\big(\delta_{ij}-\sum_{l,m=1}^3\frac{v_{*l}v_{*m}v^l_iv^m_j}{|v_*|^{2}}\big)\exp\big(-\frac{||v|-v_{*1}-\bar{u}_{1}|^2+|v_{*2}+\bar{u}_{2}|^2+|v_{*3}+\bar{u}_{2}|^2}{2K\bar{\theta}}\big)\,dv_{\ast} 
\notag\\ 	 &=\frac{\bar{\rho}}{(2\pi K\bar{\theta})^{3/2}}\int_{{\mathbb R}^3}\frac{1}{|v_*|}\big(\delta_{ij}-\sum_{m=1}^3\frac{|v_{*m}|^2v^m_iv^m_j}{|v_*|^{2}}\big)\exp\big(-\frac{||v|-v_{*1}-\bar{u}_{1}|^2+|v_{*2}+\bar{u}_{2}|^2+|v_{*3}+\bar{u}_{2}|^2}{2K\bar{\theta}}\big)\,dv_{\ast}.  
\end{align*}		 	
Denote 	
\begin{align*} 		 
B_0(v)&=\frac{\bar{\rho}}{(2\pi K\bar{\theta})^{3/2}}\int_{{\mathbb R}^3}\frac{1}{|v_*|}\exp\big(-\frac{||v|-v_{*1}-\bar{u}_{1}|^2+|v_{*2}+\bar{u}_{2}|^2+|v_{*3}+\bar{u}_{2}|^2}{2K\bar{\theta}}\big)\,dv_{\ast}, 
\notag\\ 		
B_m(v)&=\frac{\bar{\rho}}{(2\pi K\bar{\theta})^{3/2}}\int_{{\mathbb R}^3}\frac{1}{|v_*|}\frac{|v_{*m}|^2}{|v_*|^{2}}\exp\big(-\frac{||v|-v_{*1}-\bar{u}_{1}|^2+|v_{*2}+\bar{u}_{2}|^2+|v_{*3}+\bar{u}_{2}|^2}{2K\bar{\theta}}\big)\,dv_{\ast}, 	 
\end{align*}  
for $m=2,3$. By symmetry,  we see that 	    
 $$ 	 
B_2(v)=B_3(v). 	 
$$ 	 
Hence, by      
$$      
\de_{ij}-\frac{v_iv_j}{|v|^2}=\de_{ij}-v^1_iv^1_j=v^2_iv^2_j+v^3_iv^3_j,      
$$     
it holds that 	
\begin{align*} 		
\sigma^{ij}_{\overline{M}}(v)&=B_0(v)\de_{ij}-B_1(v)v^1_iv^1_j-B_2(v)(v^2_iv^2_j+v^3_iv^3_j) 
\notag\\ 		 
&=(B_0(v)-B_1(v))\frac{v_iv_j}{|v|^2}+(B_0(v)-B_2(v))(\de_{ij}-\frac{v_iv_j}{|v|^2})\notag\\ 		 
&:=\la_1(v)\frac{v_iv_j}{|v|^2}+\la_2(v)(\de_{ij}-\frac{v_iv_j}{|v|^2}) 	
\end{align*} 	 
with 	
\begin{align*} 		 
\la_1(v)&=\frac{\bar{\rho}}{(2\pi K\bar{\theta})^{3/2}}\int_{{\mathbb R}^3}\frac{1}{|v_*|}(1-\frac{|v_{*1}^2|}{|v_*|^2})\exp\big(-\frac{||v|-v_{*1}-\bar{u}_{1}|^2+|v_{*2}+\bar{u}_{2}|^2+|v_{*3}+\bar{u}_{2}|^2}{2K\bar{\theta}}\big)\,dv_{\ast},
\\ 		
\la_2(v)&=\frac{\bar{\rho}}{(2\pi K\bar{\theta})^{3/2}}\int_{{\mathbb R}^3}\frac{1}{|v_*|}(1-\frac{|v_{*2}^2+v_{*3}^2|}{2|v_*|^2})\exp\big(-\frac{||v|-v_{*1}-\bar{u}_{1}|^2+|v_{*2}+\bar{u}_{2}|^2+|v_{*3}+\bar{u}_{2}|^2}{2K\bar{\theta}}\big)\,dv_{\ast}, 	
\end{align*} 
where we have ignored the dependence of $\la_i(v)$ $(i=1,2)$ on the functions $(\bar{\rho}, \bar{u},\bar{\theta})$. Notice that the following elementary inequalities hold:
\begin{align*}
	\frac{1}{2}||v|-v_{*1}|^2+\frac{1}{2}|v_{*2}|^2+\frac{1}{2}|v_{*3}|^2-|\bar{u}|^2\leq ||v|-v_{*1}-\bar{u}_{1}|^2+&|v_{*2}+\bar{u}_{2}|^2+|v_{*3}+\bar{u}_{3}|^2\notag\\
	&\leq 2||v|-v_{*1}|^2+2|v_{*2}|^2+2|v_{*3}|^2+2|\bar{u}|^2.
\end{align*}
By the boundedness of $\bar{u}$, it follows that
\begin{align*}
	&C^{-1}\int_{{\mathbb R}^3}\frac{1}{|v_*|}(1-\frac{|v_{*1}^2|}{|v_*|^2})\exp\big(-\frac{||v|-v_{*1}|^2+|v_{*2}|^2+|v_{*3}|^2}{K\bar{\theta}}\big)\notag\\
	&\leq \la_1(v)\leq C\int_{{\mathbb R}^3}\frac{1}{|v_*|}(1-\frac{|v_{*1}^2|}{|v_*|^2})\exp\big(-\frac{||v|-v_{*1}|^2+|v_{*2}|^2+|v_{*3}|^2}{4K\bar{\theta}}\big)\,dv_{\ast},
\end{align*}
and
\begin{align*}
&C^{-1}\int_{{\mathbb R}^3}\frac{1}{|v_*|}(1-\frac{|v_{*2}^2+v_{*3}^2|}{2|v_*|^2})\exp\big(-\frac{||v|-v_{*1}|^2+|v_{*2}|^2+|v_{*3}|^2}{K\bar{\theta}}\big)\notag\\
	&\leq \la_2(v)\leq C\int_{{\mathbb R}^3}\frac{1}{|v_*|}(1-\frac{|v_{*2}^2+v_{*3}^2|}{2|v_*|^2})\exp\big(-\frac{||v|-v_{*1}|^2+|v_{*2}|^2+|v_{*3}|^2}{4K\bar{\theta}}\big)\,dv_{\ast},
\end{align*}
for some constant $C>1$ which depend on $\|\bar{u}\|_{L^\infty}$, $\sup \bar{\theta}$, $\inf\bar{\theta}>0$, $\sup \bar{\rho}/\bar{\theta}^{3/2}$ and $\inf\bar{\rho}/\bar{\theta}^{3/2}>0$.
Thus, we deduce that when $|v|\rightarrow \infty$,
\begin{align}\label{asymla}
\la_1(v)\sim (1+|v|)^{-3},\quad \la_2(v)\sim (1+|v|)^{-1}.
\end{align}
Define the projection to the vector $v$ as 
$$
P_v f_i=\frac{v_i}{|v|}\sum^{3}_{j=1}\frac{v_j}{|v|}g_j,
$$
then
\begin{align}\label{sififj}
\sigma^{ij}_{\overline{M}}(v)f_if_j=\la_1(v)\sum^3_{i=1}\{P_v f_i\}^2+\la_2(v)\sum^3_{i=1}\{[I-P_v] f_i\}^2.
\end{align}
Recall that $-\CL_{\overline{M}}=-A_{\overline{M}}-K_{\overline{M}}$, where
	\begin{align*}
		A_{\overline{M}}f&=\pa_{v_i}[\sigma^{ij}_{\overline{M}}\pa_{v_j}f]-\sigma^{ij}_{\overline{M}}\frac{(v_i-\bar{u}_i)(v_j-\bar{u}_j)}{4K^2\bar{\theta}^2}f+\pa_{v_i}[\sigma^{ij}_{\overline{M}}\frac{v_j-\bar{u}_j}{2K\bar{\theta}}]f,
        \\
		K_{\overline{M}}f&=-\overline{M}^{-1/2}\pa_{v_i}\Big[\overline{M}[\Phi_{ij}*\overline{M}(\pa_{v_j}f+\frac{v_j-\bar{u}_j}{2K\bar{\theta}})]\Big].
	\end{align*}
	Then we have from this and \eqref{DefW} that
	\begin{align}\label{7.12A}
		-( A_{\overline{M}} f,W^2(0,0) f)=&\int_{\R^3}W^2(0,0)\sigma^{ij}_{\overline{M}}\pa_{v_j} f\ \pa_{v_i} f\, dv 
        +\int_{\R^3}\pa_{v_i}[W^2(0,0)]\sigma^{ij}_{\overline{M}}\pa_{v_j}f f\, dv
        \notag\\
        &+\int_{\R^3}W^2(0,0)\sigma^{ij}_{\overline{M}}\frac{(v_i-\bar{u}_i)(v_j-\bar{u}_j)}{4K^2\bar{\theta}^2}|f|^2\,dv
+\int_{\R^3}W^2(0,0)\pa_{v_i}[\sigma^{ij}_{\overline{M}}\frac{v_j-\bar{u}_j}{2K\bar{\theta}}]|f|^2\,dv.
	\end{align}
By \eqref{asymla} and \eqref{sififj}, one has
\begin{align}
\label{equipaipaj}
\int_{\R^3}W^2(0,0)(v)\sigma^{ij}_{\overline{M}}\pa_{v_j} f\ \pa_{v_i} f\ dv\sim\{\Big|W(0,0)\langle v\rangle^{-\frac{3}{2}}\nabla_vf\cdot\frac{v}{|v|}\Big|_2+\Big|W(0,0)\langle v\rangle^{-\frac{1}{2}}\nabla_vf\times\frac{v}{|v|}\Big|_2\}.
\end{align}
	
Note from \cite[Lemma 3]{Guo-2002} that there is a function $\bar{\la}_1(v)$ such that
$$
\si^{ij}_{M_{[1,0,\bar{\theta}]}}v_iv_jf^2=\bar{\la}_1(v)|v|^2f^2,
$$
where $\bar{\la}_1(v)\sim (1+|v|)^{-3}$ as $|v|$ tends to infinity. Then
	\begin{align}\label{revivjf}
		&\int_{\R^3}W^2(0,0)(v)\sigma^{ij}_{\overline{M}}\frac{(v_i-\bar{u}_i)(v_j-\bar{u}_j)}{4K^2\bar{\theta}^2}| f|^2\,dv
        \notag\\
		=& \frac{\bar{\rho}}{(2\pi K\bar{\theta})^{3/2}} \int_{\R^3}W^2(0,0)(v)\sigma^{ij}_{M_{[1,0,\bar{\theta}]}}(v-\bar{u})\frac{(v_i-\bar{u}_i)(v_j-\bar{u}_j)}{4K^2\bar{\theta}^2}| f|^2\,dv
        \notag\\
		=& \frac{\bar{\rho}}{(2\pi K\bar{\theta})^{3/2}} \int_{\R^3}W^2(0,0)(v)\bar{\la}_1(v-\bar{u})|v-\bar{u}|^2| f|^2\,dv.
	\end{align}
The velocity weight $\langle v\rangle$ is equivalent to $\langle v-\bar{u}\rangle$ since
\begin{align*}
\langle v\rangle\leq C(1+|v-\bar{u}+\bar{u}|)\leq C(1+|\bar{u}|+| v-\bar{u}|)\leq C\langle v-\bar{u}\rangle,
\end{align*}
and
\begin{align*}
	\langle v-\bar{u}\rangle\leq C(1+|v-\bar{u}|)\leq C(1+|\bar{u}|+| v|)\leq C\langle v\rangle.
\end{align*} 
Choose a smooth cut-off function $\chi(v)$ such that $\chi(v)=1$ for $|v-\bar{u}|<1$ and $\chi(v)=0$ for $|v-\bar{u}|>2$, then we have from Poincare’s inequality and \eqref{equipaipaj} that
\begin{align}\label{L2large}
	\int_{|v-\bar{u}|<1}\frac{W^2(0,0)}{\langle v\rangle}|f|^2\,dv
    &\leq C\int_{\R^3}|\chi(v) f|^2\,dv\leq C\int_{\R^3}|\nabla_{v}\chi(v)\, f|^2dv+C\int_{\R^3}|\chi(v)\, \nabla_{v}f|^2\,dv
    \notag\\
	&\leq  C\int_{1\leq |v-\bar{u}|\leq 2}| f|^2\,dv+C\int_{|v-\bar{u}|\leq 2}|\nabla_{v}f|^2\,dv
    \notag\\
	&\leq C\frac{\bar{\rho}}{(2\pi K\bar{\theta})^{3/2}} \int_{\R^3}W^2(0,0)\la_1(v-\bar{u})|v-\bar{u}|^2| f|^2\,dv
    \notag\\
	&\quad+C\int_{\R^3}W^2(0,0)\sigma^{ij}_{\overline{M}}\pa_{v_j} f\ \pa_{v_i} f\, dv.
\end{align}
On the other way, by the fact that
$$
\frac{1}{\langle v\rangle}\chi_{\{|v-\bar{u}|\geq1\}}\leq\frac{(1+|v-\bar{u}|^2) \chi_{\{|v-\bar{u}|\geq1\}}}{\langle v\rangle^3}\leq C\frac{|v-\bar{u}|^2 \chi_{\{|v-\bar{u}|\geq1\}}}{\langle v-\bar{u}\rangle^3},
$$
then one has from this and  \eqref{asymla} that
\begin{align}
\label{L2bounded}
\int_{|v-\bar{u}|\geq1}\frac{W^2(0,0)}{\langle v\rangle}|f|^2\,dv\leq C\frac{\bar{\rho}}{(2\pi K\bar{\theta})^{3/2}} \int_{\R^3}W^2(0,0)\la_1(v-\bar{u})|v-\bar{u}|^2| f|^2\,dv.
\end{align}
We combine \eqref{revivjf}, \eqref{L2large} and \eqref{L2bounded} to get
\begin{align*}
|W(0,0)\langle v\rangle^{-\frac{1}{2}}f|^2_2\leq C\int_{\R^3}W^2(0,0)\sigma^{ij}_{\overline{M}}\pa_{v_j} f\ \pa_{v_i} f\, dv +C\int_{\R^3}W^2(0,0)\sigma^{ij}_{\overline{M}}\frac{(v_i-\bar{u}_i)(v_j-\bar{u}_j)}{4K^2\bar{\theta}^2}|f|^2\,dv,
\end{align*}
which, together with \eqref{equipaipaj}, yields
\begin{align}\label{term12}
&	C\int_{\R^3}W^2(0,0)\sigma^{ij}_{\overline{M}}\pa_{v_j} f\ \pa_{v_i} f\, dv +C\int_{\R^3}W^2(0,0)\sigma^{ij}_{\overline{M}}\frac{(v_i-\bar{u}_i)(v_j-\bar{u}_j)}{4K^2\bar{\theta}^2}|f|^2\,dv
\notag\\
\geq&|W^2(0,0)\langle v\rangle^{-\frac{1}{2}}f|^2_2+\Big|W^2(0,0)\langle v\rangle^{-\frac{3}{2}}\nabla_vf\cdot\frac{v}{|v|}\Big|^2_2+\Big|W^2(0,0)\langle v\rangle^{-\frac{1}{2}}\nabla_vf\times\frac{v}{|v|}\Big|^2_2
\notag\\
\geq&c|f|^2_{\si,W}.
	\end{align}
	With the similar calculation as Lemma 7 in \cite{Strain-Guo}, we get
	\begin{align}\label{term3}
&\int_{\R^3}|\pa_{v_i}[W^2(0,0)]\sigma^{ij}_{\overline{M}}\pa_{v_j} f |f\, dv+		\int_{\R^3}W^2(0,0)\pa_{v_i}[\sigma^{ij}_{\overline{M}}\frac{v_j-\bar{u}_j}{2K\bar{\theta}}]|f|^2\,dv+\int_{\R^3}W^2(0,0)(v)|K_{\overline{M}}f||f|\,dv
\notag\\      
&\leq \ka|f|^2_{\si,W}+C_{\ka}|\chi_{|v|\leq C_{\ka}}f|_2^2.
\end{align}
Then \eqref{controlpaL} holds by $-\CL_{\overline{M}}=-A_{\overline{M}}-K_{\overline{M}}$, \eqref{7.12A}, \eqref{term12} and \eqref{term3} with sufficiently small $\ka$. Moreover, \eqref{controlLbarM} can be obtained from similar arguments as above and \cite[Lemma 5]{Guo-2002}.

Similar proofs for obtaining \eqref{controlpaL} and \eqref{coLD}, together with \eqref{controlpaL} \eqref{controlGa} and \eqref{controlpaGa}, give the desired
estimates \eqref{controlpaLM} and \eqref{controlLM}. The details are omitted for brevity.  Hence, Lemma \ref{leL} is proved.   
\end{proof}

\medskip
\noindent {\bf Acknowledgment:}\,
The research of Renjun Duan was partially supported by the General Research Fund (Project No.~14303523) from RGC of Hong Kong and also by the grant from the National Natural Science Foundation of China (Project No.~12425109). Zongguang Li would like to thank the Research Centre for Nonlinear Analysis at The Hong Kong Polytechnic University for supporting his postdoc study. The research of Dongcheng Yang was supported by the National Natural Science Foundation (Project No.~ 12401276),
the Guangdong Basic and Applied Basic Research Foundation (Project No.~ SL2024A04J01013).   
The research of Tong Yang is supported by the General Research Fund of Hong Kong (Project No. 11318822). He would also like to thank the Kuok Group foundation for its generous support.

\medskip
\noindent{\bf Data availability:} The manuscript contains no associated data.

\medskip
\noindent{\bf Conflict of Interest:} The authors declare that they have no conflict of interest.



\begin{thebibliography}{99}
\bibitem{Alvarez-Samaniego}
	B. Alvarez-Samaniego and D. Lannes,  Large time existence for 3D water-waves and asymptotics. {\it Invent.
		Math.}  {\bf 171} (2008), no.~3, 485--541.



		
\bibitem{Bastdos-Golse2018}
C. Bastdos, F. Golse, T. Nguyen and R. Sentis, The Maxwell-Boltzmann approximation for ion kinetic modeling.
{\it Phys. D.} {\bf 376} (2018), 94--107.		

\bibitem{Bedrossian}
J. Bedrossian, Suppression of plasma echoes and Landau damping in Sobolev spaces by weak collisions
in a Vlasov-Fokker-Planck equation. {\it Ann. PDE.} {\bf 3} (2017), no.~2,  66 pp.

\bibitem{Bedrossian-Zhao-Zi}
J. Bedrossian, W. Zhao and R. Zi, Landau damping, collisionless limit, and stability threshold for the Vlasov-Poisson equation with nonlinear Fokker-Planck collisions. {\it Commun. Math. Phys.} {\bf 406}, 166 (2025). \href{https://doi.org/10.1007/s00220-025-05343-0}{DOI}

\bibitem{Ben}
W. Ben Youssef and D. Lannes,  The long wave limit for a general class of 2D quasilinear hyperbolic
problems. {\it Comm. Partial Differential Equations} {\bf 27} (2002), no.~5-6,  979--1020.

\bibitem{BGSS1}
F. Béthuel, P. Gravejat, J.-C. Saut, and D. Smets, On the Korteweg-de Vries long-wave approximation of the Gross-Pitaevskii equation. I. {\it Int. Math. Res. Not.} (2009), no.~14, 2700--2748.

\bibitem{BGSS2}
F. Béthuel, P. Gravejat, J.-C. Saut, and D. Smets. On the Korteweg-de Vries long-wave approximation of the Gross-Pitaevskii equation II. {\it Comm. Partial Differential Equations}
{\bf 35} (2010), no.~1, 113--164.

\bibitem{Bobylev-Potapenko}
A. V. Bobylev and I. F. Potapenko, Long wave asymptotics for the Vlasov-Poisson-Landau kinetic equation. {\it J. Stat. Phys.}
{\bf 175} (2019), no.~1, 1--18.

\bibitem{Bona}
J.~L. Bona, T. Colin and D. Lannes,  Long wave approximations for water waves. {\it Arch. Rational Mech.
	Anal.}  {\bf 178} (2005), no.~3, 373--410.
	
	\bibitem{Caflisch}
	R.~E. Caflisch,  The fluid dynamical limit of the nonlinear Boltzmann equation.
	{\it  Comm. Pure Appl. Math.} {\bf 33} (1980),  no.~5, 491--508.

\bibitem{Chaturvedi-Luk-Nguyen}
S. Chaturvedi, J. Luk, and T. Nguyen, The Vlasov-Poisson-Landau system in the weakly collisional regime.
{\it J. Amer. Math. Soc.}  {\bf 36} (2023),  no.~4, 1103--1189.

\bibitem{Chiron} 
D. Chiron and F. Rousset, The KdV/KP-I limit of the nonlinear Schr\"{o}dinger equation. {\it SIAM J. Math.
	Anal.} {\bf 42} (2010),  no.~5, 64--96.		


\bibitem{Craig}
W. Craig, An existence theory for water waves and the Boussinesq and Korteweg-de Vries scaling limits.
{\it Comm. Partial Differential Equations}  {\bf 10} (1985), no.~8, 787--1003. 

\bibitem{Duan-Li}
R.~J. Duan and Z.-G. Li, Polynomial tail solutions of the non-cutoff Boltzmann equation near local Maxwellians. Preprint.  \href{https://doi.org/10.48550/arXiv.2407.08346}{arXiv.2407.08346} 

\bibitem{DL-arma}
R.~J. Duan and S.-Q. Liu, The Boltzmann equation for uniform shear flow. {\it Arch. Rational. Mech. Anal.} {\bf 242} (2021), no.~3, 1947--2002.


\bibitem{DL-cmaa}
R.~J. Duan and S.-Q. Liu, On smooth solutions to the thermostated Boltzmann equation with deformation.  {\it Commun. Math. Anal. Appl.} {\bf 1} (2022), no.~1, 152--212.



\bibitem{DL-hard}
R.~J. Duan and S.-Q. Liu, Uniform shear flow via the Boltzmann equation with hard potentials. {\it Discrete Contin. Dyn. Syst.} {\bf 45} (2025), no.~10, 4071--4118.
		
	\bibitem{Duan-Yang-Yu}
	R.~J. Duan, D.~C. Yang and H.~J. Yu, Small Knudsen rate of convergence to rarefaction wave for the Landau equation.
	{\it Arch. Ration. Mech. Anal.} {\bf 240} (2021), no.~3, 1535--1592.
		
\bibitem{Duan-Yang-Yu-1}
R.~J. Duan, D.~C. Yang and H.~J. Yu, Asymptotics toward viscous contact waves for solutions of the Landau equation. {\it Comm. Math. Phys.}
{\bf 394} (2022), no.~1, 471--529.

\bibitem{Duan-Yang-Yu-2}
R.~J. Duan, D.~C. Yang and H.~J. Yu, Compressible fluid limit for smooth solutions to the Landau equation.
{\it Ann. Inst. H. Poincaré C Anal. Non Linéaire} (2024), published online first. 
\href{https://doi.org/10.4171/AIHPC/135}{DOI}

\bibitem{Duan-Yang-Yu-VMB}
R.~J. Duan, D.~C. Yang and H.~J. Yu, Compressible Euler-Maxwell limit for global smooth solutions to the Vlasov-Maxwell-Boltzmann system. 
{\it Math. Models Methods Appl. Sci.}  {\bf 33} (2023), no.~10, 2157--2221.


\bibitem{DYY}
R.~J. Duan, D.~C. Yang and H.~J. Yu, KdV limit for the Vlasov-Poisson-Landau system. Preprint.  \href{https://doi.org/10.48550/arXiv.2308.08863}{arXiv.2308.08863}

\bibitem{Duan-Yang-Yu-VPL}
R.~J. Duan, D.~C. Yang and H.~J. Yu, Global quasineutral Euler limit for the Vlasov-Poisson-Landau system with rarefaction waves.  Preprint. 
\href{https://doi.org/10.48550/arXiv.2212.07654}{arXiv.2212.07654}

\bibitem{Duan-Yang-Zhao-M3} 
R.~J. Duan, T. Yang and H.~J. Zhao, The Vlasov-Poisson-Boltzmann system for soft potentials.
{\it Math. Models Methods Appl. Sci.} {\bf 23} (2013),  no.~6, 979--1028.


\bibitem{Duan-Yang-Zhao}
R.~J. Duan, T. Yang and H.~J. Zhao, Global solutions to the Vlasov-Poisson-Landau system. Unpublished.  
\href{https://doi.org/10.48550/arXiv.1112.3261}{arXiv.1112.3261}

\bibitem{FG}
P. Flynn and Y. Guo, The massless electron limit of the Vlasov-Poisson-Landau system. {\it Comm. Math. Phys.} {\bf 405} (2024), no.~2, Paper No. 27, 73 pp.

	
	
\bibitem{GaMo}
C. S. Gardner and G. K. Morikawa, Similarity in the asymptotic behavior of collision-free hydromagnetic waves and water waves. Courant Institute of Mathematical Sciences Report No. NYO‐9082, 1960 (unpublished).


\bibitem{GGPS}
E. Grenier, Y. Guo, B. Pausader and M. Suzuki, Derivation of the ion equation. {\it Quart. Appl. Math.} {\bf 78} (2020), 305--332.

	\bibitem{Guo-2002}
	Y. Guo, The Landau equation in a periodic box. {\it Comm. Math. Phys.} {\bf 231} (2002), no.~3, 391--434.	
    
\bibitem{Guo-JAMS}
Y. Guo, The Vlasov-Poisson-Landau system in a periodic box.  {\it J. Amer. Math. Soc.} {\bf 25} (2012), no.~3, 759--812.


\bibitem{Guo-Jang}
Y. Guo and J. Jang, Global Hilbert expansion for the Vlasov-Poisson-Boltzmann system.  {\it Comm. Math. Phys.} {\bf 299} (2010), 
no.~2, 469--501.

\bibitem{Guo-Jang-Jiang-2010}
Y. Guo, J. Jang and N. Jiang, Acoustic limit for the Boltzmann equation in optimal scaling.  {\it Comm. Pure Appl. Math.} {\bf 63} (2010),  no.~3, 337--361.

\bibitem{Guo-Pausader}
Y. Guo and B. Pausader, Global smooth ion dynamics in the Euler-Poisson system. {\it Comm. Math. Phys.} {\bf  303} (2011), 89--125.

\bibitem{Guo-Pu} 
Y. Guo and X. Pu,  KdV limit of the Euler-Poisson system.  {\it Arch. Ration. Mech. Anal.} {\bf 211} (2014), no.~2, 673--710.


		
	\bibitem{Han-Kwan}
	D. Han-Kwan, From Vlasov-Poisson equation to Korteweg-de Vries and Zakharov-Kuznetsov.
	{\it Comm. Math. Phys.} {\bf 324} (2013), no.~3, 961--993.

\bibitem{Kenig}
C. Kenig, G. Ponce and L. Vega, Well-posedness and scattering results for the generalized
Korteweg-de Vries equation via the contraction principle. {\it Comm. Pure Appl. Math.} {\bf 46}, (1993), no. 4, 527--620.

\bibitem{Lannes} 
D. Lannes, F. Linares and J. C. Saut, The Cauchy problem for the Euler-Poisson system and derivation of the Zakharov-Kuznetsov equation,
in Studies in Phase Space Analysis with Applications to PDEs, {\it Progr. Nonlinear Differential Equations Appl. 84, Springer, New York,} 2013, pp. 181--213.

\bibitem{Lei-Liu-Xiao-Zhao} 
Y.~J. Lei, S.~Q. Liu, Q.~H. Xiao and H.~J. Zhao, Global Hilbert expansion for some nonrelativistic kinetic equations. {\it J. Lond. Math. Soc. (2)} {\bf 110} (2024), no.~5, Paper No. e70016.
	
	\bibitem{Liu-Yang-Yu} 
	T. P. Liu, T. Yang and S. H. Yu, Energy method for the Boltzmann equation. {\it  Physica D} {\bf 188} (2004), 178--192.
	
	
	\bibitem{Pu}
	 X. Pu, Dispersive of the Euler-Poisson system in higher dimensions. {\it SIAM J. Math. Anal.} {\bf 45} (2013), no.~2, 834--878.
	 
	 
	 \bibitem{Schneider}
	 G. Schneider and C.~E. Wayne, The long-wave limit for the water wave problem. I. The case of zero surface
	 tension. {\it Comm. Pure Appl. Math.} {\bf 53} (2000), no.~12, 1475--1535.
	 
	
	\bibitem{Strain-Guo} 
	R.~M. Strain and Y. Guo, Exponential decay for soft potentials near Maxwellian.
	{\it Arch. Ration. Mech. Anal.}, {\bf 187} (2008), no.~2, 287--339.
	
	\bibitem{Su-1} 
	C. Su and C. Gardner, Korteweg-de Vries equation and generalizations. III. Derivation of the
	Korteweg-de Vries equation and Burgers equation. {\it J. Math. Phys.} {\bf 10} (1969), no. 3, 536--539.
	
	\bibitem{Wang}
	Y. J. Wang, Global solution and time decay of the Vlasov-Poisson-Landau system in $R^3$. {\it SIAM J. Math. Anal.} {\bf 44} (2012), no. 5, 3281--3323.
	
	
	\bibitem{Washimi}
	H. Washimi and T. Taniuti, Propagation of ion-acoustic waves of small amplitude. {\it Phys. Rev.
		Lett.}  {\bf 17}, (1966), no.~9, 996--998.
\end{thebibliography}
	\end{document}